\pdfoutput=1

\documentclass[10pt,reqno, twoside]{amsart}
\usepackage[letterpaper,margin=1.3in,footskip=0.30in]{geometry}

\usepackage[dvipsnames]{xcolor}

\usepackage{lmodern}
\usepackage[protrusion=true,expansion=true]{microtype}

\usepackage[inline,shortlabels]{enumitem}
\setlist{noitemsep}

\usepackage{amsmath,amssymb,amsthm}
\usepackage{mathtools}
\usepackage{stmaryrd}
\usepackage{textgreek}
\usepackage[retainorgcmds]{IEEEtrantools}
\usepackage{colonequals}
\usepackage{relsize}
\usepackage[bbgreekl]{mathbbol}
\usepackage{mathrsfs}
\usepackage{tikz-cd}
\usepackage{array}

\DeclareSymbolFontAlphabet{\mathbb}{AMSb} 
\DeclareSymbolFontAlphabet{\mathbbl}{bbold}

\usepackage[pdfusetitle,colorlinks,unicode]{hyperref}
\hypersetup{citecolor=black,linkcolor=black,urlcolor=black}
\usepackage[capitalise,noabbrev]{cleveref}
\crefformat{enumi}{#2\textup{(#1)}#3}
\crefformat{equation}{\ensuremath{(#2#1#3)}}
\crefmultiformat{equation}{\ensuremath{(#2#1#3)}}{ and~\ensuremath{(#2#1#3)}}{, \ensuremath{(#2#1#3)}}{, and~\ensuremath{(#2#1#3)}}
\numberwithin{figure}{subsection}
\numberwithin{equation}{subsection}
\newtheorem{theorem}[figure]{Theorem}
\newtheorem{lemma}[figure]{Lemma}
\newtheorem{corollary}[figure]{Corollary}
\newtheorem{proposition}[figure]{Proposition}

\newtheorem{question}[figure]{Question}
\theoremstyle{definition}
\newtheorem{definition}[figure]{Definition}
\newtheorem{notation}[figure]{Notation}
\theoremstyle{definition}
\newtheorem{remark}[figure]{Remark}
\theoremstyle{definition}
\newtheorem{example}[figure]{Example}
\theoremstyle{definition}
\newtheorem{construction}[figure]{Construction}

\DeclareMathOperator{\Aff}{Aff}

\DeclareMathOperator{\Coker}{Coker}

\DeclareMathOperator{\Et}{\acute{E}t}
\DeclareMathOperator{\Ext}{Ext}
\DeclareMathOperator{\Fil}{Fil}

\DeclareMathOperator{\gr}{gr}

\DeclareMathOperator{\Hom}{Hom}\DeclareSymbolFontAlphabet{\mathbb}{AMSb} 
\DeclareSymbolFontAlphabet{\mathbbl}{bbold}

\DeclareMathOperator{\height}{height}
\DeclareMathOperator{\im}{Im}
\DeclareMathOperator{\Ker}{Ker}
\DeclareMathOperator{\Map}{Map}

\DeclareMathOperator{\Pro}{Pro}

\DeclareMathOperator{\PShv}{PShv}
\DeclareMathOperator{\QCoh}{QCoh}
\DeclareMathOperator{\rk}{rank}

\DeclareMathOperator{\Shv}{Shv}
\DeclareMathOperator{\Spec}{Spec}

\DeclareMathOperator{\Sym}{Sym}
\DeclareMathOperator{\uni}{uni}

\newcommand{\fm}{\mathfrak{m}}
\renewcommand{\AA}{\mathbf{A}}

\newcommand{\FF}{\mathbf{F}}
\newcommand{\GG}{\mathbf{G}}
\newcommand{\LL}{\mathbf{L}}
\newcommand{\NN}{\mathbf{N}}
\newcommand{\PP}{\mathbf{P}}
\newcommand{\QQ}{\mathbf{Q}}

\newcommand{\UU}{\mathbf{U}}

\newcommand{\ZZ}{\mathbf{Z}}
\newcommand{\w}{\text}

\newcommand{\Hh}{H}

\newcommand{\cC}{\mathcal{C}}

\newcommand{\cH}{\mathscr{H}}
\newcommand{\cI}{\mathscr{I}}
\newcommand{\cM}{\mathcal{M}}

\newcommand{\cO}{\mathscr{O}}

\newcommand{\cS}{\mathcal{S}}
\newcommand{\cV}{\mathscr{V}}
\newcommand{\cX}{\mathcal{X}}

\newcommand{\rU}{\mathrm{U}}

\newcommand{\Bb}{{B}}
\newcommand{\SSS}{{S}}
\newcommand{\ab}{\mathrm{ab}}
\newcommand{\Ab}{\mathrm{Ab}}
\newcommand{\Alg}{\mathrm{Alg}}
\newcommand{\alg}{\mathrm{alg}}
\newcommand{\Art}{\mathrm{Art}}
\newcommand{\coh}{\mathrm{coh}}

\newcommand{\et}{\mathrm{\acute{e}t}}

\newcommand{\Mod}{\mathrm{Mod}}

\newcommand{\op}{\mathrm{op}}
\newcommand{\perf}{\mathrm{perf}}
\newcommand{\Prism}{{\mathlarger{\mathbbl{\Delta}}}}
\newcommand{\proet}{\mathrm{pro\acute{e}t}}
\newcommand{\qc}{\mathrm{qc}}

\newcommand{\id}{\mathrm{id}}

\newcolumntype{C}[1]{>{\centering\arraybackslash}p{#1}}

\newcommand*{\llbrace}{\{\mskip-5mu\{}
\newcommand*{\rrbrace}{\}\mskip-5mu\}}
\mathchardef\mhyphen="2D

\newcommand\restr[2]{{\left.\kern-\nulldelimiterspace#1\vphantom{\big|}\right|_{#2}}}
\newcommand{\suchthat}{\;\ifnum\currentgrouptype=16 \middle\fi\vert\;}
\newcommand{\cosimp}[3]{\xymatrix@1{#1 \ar@<.4ex>[r] \ar@<-.4ex>[r] & {\ }#2 \ar@<0.8ex>[r] \ar[r] \ar@<-.8ex>[r] & {\ } #3  \cdots }} 
\usepackage[all]{xy}

\AtBeginDocument{
   \def\MR#1{}
}

\title{Unipotent homotopy theory of schemes}

\author{Shubhodip Mondal}
\author{Emanuel Reinecke}

\address[Shubhodip Mondal]{Purdue University, Mathematical Sciences Bldg, 150 N University St, West Lafayette, IN 47907, USA}
\email{mondalsh@purdue.edu}

\address[Emanuel Reinecke]{Institut des
Hautes \'Etudes Scientifiques, 35 route de Chartres, 91440 Bures-sur-Yvette, France}
\email{reinecke@ihes.fr}

\begin{document}

\begin{abstract}
Building on To\"en's work on affine stacks, we develop a certain homotopy theory for schemes, which we call ``unipotent homotopy theory.''
Over a field of characteristic $p>0$, we prove that the unipotent homotopy group schemes $\pi_i^{\mathrm{U}}(\,\cdot\,)$ introduced in our paper recover the unipotent Nori fundamental group scheme \cite{Nori1st}, the $p$-adic étale homotopy groups \cite{etalehomo}, as well as certain formal groups \cite{MR457458} introduced by Artin and Mazur. We prove a version of the classical Freudenthal suspension theorem as well as a profiniteness theorem for unipotent homotopy group schemes. We also introduce the notion of a formal sphere and use it to show that for Calabi--Yau varieties of dimension $n$, the group schemes $\pi_i^{\mathrm{U}}(\,\cdot\,)$ are derived invariants for all $i \ge 0$; the case $i=n$ is related to  recent work of Antieau and Bragg \cite{AB1} involving topological Hochschild homology. 
Using the unipotent homotopy group schemes, we establish a correspondence between formal Lie groups and certain higher algebraic structures. 
\end{abstract}
\maketitle
\vspace{-8.5mm}
\tableofcontents
\newpage
\section{Introduction}
In this paper, we develop a notion of ``unipotent homotopy theory" for schemes. For schemes over fields of positive characteristic, the resulting
unipotent homotopy group schemes contain a lot of topological as well as arithmetic information.
We will, in particular, see that the unipotent homotopy group schemes provide interesting new invariants for varieties whose cohomology is especially simple, a phenomenon perhaps more familiar in topology, for example, in the case of spheres. Before delving into the precise definitions and applications, we discuss some relevant contexts.
\vspace{2mm}

One of the most classical homotopical notions for schemes is the \'etale fundamental group introduced by Grothendieck \cite{EtaleGro}, which is a profinite group. 
 Grothendieck also conjectured in \cite[\S~X.2]{Etalegro1} that the \'etale fundamental group should admit natural enhancement of a \emph{group scheme}, which would contain more information than only the finite \'etale coverings of a scheme.
The first construction of such an object was given by Nori \cite{Nori1st}, who introduced the fundamental group scheme $\pi_1^{\mathrm{N}}(X, x)$ for a pointed scheme $(X,x)$ defined over a field. Nori's construction uses the notion of essentially finite vector bundles introduced by him and the Tannakian formalism (see, e.g., \cite{tan1,tannaka,Deligne2007}).
\vspace{2mm}

In \cite{etalehomo}, Artin and Mazur introduced the notion of \'etale homotopy type of a scheme (see also \cite{etalfry});
for some notable applications, see, e.g., \cite{etale3,etale1, etale2,etalfry,MR3549624}.
The theory of \'etale homotopy types in particular gives rise to the higher \'etale homotopy groups of a scheme, which generalize the notion of \'etale fundamental groups from \cite{EtaleGro}. However, the natural analogous question of constructing higher homotopy {group schemes} that would extend Nori's fundamental group scheme $\pi_1^{\mathrm{N}}(X, x)$ \cite{Nori1st} has not been addressed. 
\vspace{2mm}

Note that there are several fundamental difficulties in using the Tannakian formalism to define higher homotopy group schemes.
Let us briefly illustrate one of the issues. If $G$ is a commutative group scheme, then as a consequence of the Tannakian formalism, one can recover $G$ from the category $\mathrm{Rep}(G)$ of representations of $G$, or equivalently, the category of vector bundles on $BG$.
From a homotopical point of view, this recovers $G$ as ``$\pi_1^{\mathrm{N}}(K(G,1))$.'' However, a similar construction fails to recover $G$ as ``$\pi_2^{\mathrm{N}}(K(G,2))$''; this is because the category of quasi-coherent sheaves on $K(G,2)$ is trivial for \emph{all} group schemes $G$. In fact, the derived $\infty$-category $D_{\mathrm{qc}}(K(G,2))$ of quasi-coherent sheaves on $K(G,2)$ can often be trivial, i.e., $D_\qc(\Spec k)$;
for a concrete example, one can take $G = \mathbf{G}_m$ (see \cref{exaqmpleintroc}). Moreover, even if $D_{\mathrm{qc}}(K(G,2))$ is nontrivial, e.g., when $G= \mathbf{G}_a$, it is unclear in positive characteristic how to recover $G$ from it.
\vspace{2mm}

In our paper, we resolve the question of constructing higher homotopical analogues of the Nori fundamental group scheme ``up to unipotent completion.''
More precisely, we introduce the \textit{unipotent homotopy group schemes} $\pi_i^{\mathrm{U}}(X,x)$ for a pointed scheme $(X,x)$ over a field $k$ that extend the unipotent Nori fundamental group scheme from \cite{MR682517} (which is simply the unipotent completion of $\pi_1^{\mathrm{N}}(X,x)$ when $k$ has characteristic $p>0$). Our method of construction takes an entirely different approach. We completely bypass the Tannakian machinery used in the existing construction of the Nori fundamental group scheme by employing To\"en's work on higher stacks and affine stacks \cite{Toe} instead.
\vspace{2mm}

Affine stacks were introduced by To\"en as a solution to Grothendieck's schematization problem, which proposed to extend the foundations of algebraic stacks suitably in order to accommodate homotopy types. In particular, To\"en showed that affine stacks can be used to model simply connected rational and $p$-adic homotopy types.
Roughly speaking, the classical results \cite{Quillen-rational,Sullivan-infinitesimal,Bousfield-Gugenheim, MR1243609, Goerss-simplicial, Mandell} in rational and $p$-adic homotopy theory allow one to embed simply connected rational and $p$-adic homotopy types in a suitable category of cosimplicial or $E_{\infty}$-algebras. To\"en proved that algebras of the former kind can further be embedded contravariantly into the category of \emph{higher stacks};
a higher stack that lies in the essential image of this embedding is called an \emph{affine stack}.
For a more precise discussion of the definitions, see \cref{affine-stacks}. Note that the category of higher stacks contains the category of schemes as well. Our work was fueled by the following observation. 
\begin{theorem}\label{thm1intro}
    Let $(X,x)$ be a pointed scheme of finite type over a field $k$ such that $H^0 (X, \cO) \simeq k$.\footnote{
    A scheme $X$ over $k$ such that $H^0 (X, \cO) \simeq k$ will be called cohomologically connected.}
    Let $\mathbf{U}(X)$ be the affine stack which is universal with respect to the property of receiving a map from $X$. Then there is a natural isomorphism 
$$\pi_1(\mathbf{U}(X), x) \xrightarrow{\sim} \pi_1^{\mathrm{U,N}}(X,x),$$ where the right-hand side denotes the unipotent fundamental group scheme constructed by Nori. (\textit{cf.}~\cref{prop:Nori-affine-pi1})
\end{theorem}{}
Our observation in \cref{thm1intro} led us to take the following perspective: affine stacks are in some sense an abstract notion of a ``unipotent homotopy type.'' This motivates the following definitions. 
\begin{definition}
Let $X$ be a scheme over a field $k$. The affine stack $\mathbf{U}(X)$ from \cref{thm1intro} is called the \emph{unipotent homotopy type} of $X$.
\end{definition}{}

\begin{definition}\label{introdef}
    Let $(X,x)$ be a pointed, cohomologically connected scheme over $k$.
    We define 
$$\pi_i^\mathrm{U}(X) \colonequals \pi_i (\mathbf{U}(X), x).$$
The sheaves $\pi_i^\mathrm{U}(X)$ are representable by unipotent group schemes and will be called the $i$-th unipotent homotopy group scheme of $X$. 
\end{definition}{}
\begin{remark}
    For us, a group scheme $G$ over a field $k$ is unipotent if it is affine and every nonzero representation of $G$ over $k$ has a nonzero fixed vector. Note that we do not assume the group schemes to be of finite type. For example, an infinite product of $\mathbf{G}_a$ is a unipotent group scheme.
\end{remark}{}

\begin{remark}
  The notion of unipotent homotopy type can be defined for any higher stack $X$ over an arbitrary base ring, and it is determined by $R\Gamma(X, \mathscr{O})$ naturally equipped with the structure of a derived commutative ring (see \cref{aboutrings} and \cref{cope}). However, the geometric language of affine stacks plays a crucial role in our paper, which is already manifested in \cref{introdef}. Moreover, our definition via the universal property (as in \cref{thm1intro}) and the geometric approach taken in this paper play an important role in the proofs.
\end{remark}{}

Let us now explain our main results regarding the unipotent homotopy group schemes introduced above.
\begin{theorem}[Profiniteness theorem]\label{thm2intro}
Let $X$ be a proper, cohomologically connected, pointed scheme over a field $k$ of characteristic $p>0$. Then for all $i$, the unipotent homotopy group schemes $\pi_i^{\mathrm{U}}(X)$ are profinite group schemes. (\textit{cf.} \cref{profiniteunipotent}.)
\end{theorem}{}

One may think of \cref{thm2intro} as an analogue of a result of Artin--Mazur \cite[Thm.~11.1]{etalehomo} who proved that the \'etale homotopy groups of geometrically unibranch varieties are profinite. While Artin and Mazur use geometric and simplicial arguments to deduce their result from the profiniteness of Galois groups, our proof uses the theory of Frobenius modules on higher stacks that we discuss and is closer in spirit to the Riemann--Hilbert correspondence in characteristic $p>0$; \textit{cf}.~\cite{Emerton-Kisin,MR2561048,BL-RHcharp}.
Additionally, our proof requires certain results on profinite unipotent group schemes and their representations which are developed and studied in \cref{posto4}. If $k$ is of characteristic zero, then \cref{thm2intro} does not hold. Indeed, if $E$ is an elliptic curve over a characteristic zero field, then $\pi_1^{\mathrm{U}}(E) \simeq \mathbf{G}_a,$ which is not profinite.
\vspace{2mm}

Our next result shows that the unipotent homotopy group schemes actually refine the Artin--Mazur \'etale homotopy groups \cite{etalehomo} in the $p$-adic context.

\begin{theorem}[Unipotent--\'etale comparison theorem]\label{thm3intro}
Let $X$ be a proper, cohomologically connected, pointed scheme over an algebraically closed field of characteristic $p > 0$. Then for all $i$, the pro-$p$-finite completed $i$-th Artin--Mazur $p$-adic \'etale homotopy group $\pi_i^{\et,\, \mathrm{AM}}(X)_p$ (see the discussion following \cref{defetalehom}) is naturally isomorphic to the maximal pro-\'etale quotient of $\pi^{\mathrm U}_i(X)$. (\textit{cf}.~\cref{unipotent-etale-comparison}.)
\end{theorem}{}

The first step in the proof of \cref{thm3intro} is a different definition of $p$-adic \'etale homotopy group schemes introduced in \cref{defetalehom} and its comparison with $\pi_i^{\mathrm{\et, AM}}(X)_p$ established in \cref{compareetaleuni}; the latter result crucially relies on $p$-adic homotopy theory. We also need to use some Frobenius semi-linear algebra adapted to the ``derived" setting that we study in \cref{posto3}.
Additionally, we use \cref{thm2intro} and certain results on profinite group schemes developed in \cref{posto4}.
Lastly, we invoke the general theory of Milnor sequences for replete topoi, which has been studied recently by the authors in \cite{MR100}.
\vspace{2mm}

In another paper \cite{MR457458}, Artin--Mazur constructed certain commutative formal groups from algebraic varieties.
These formal groups have been a major source of interest, especially in the study of Calabi--Yau varieties;
see, e.g., \cite{Artin-supersingular,Shioda-supersingular,vdG-Kat}.
Our next results show that in many cases of interest, these formal groups can be recovered by dualizing certain unipotent homotopy group schemes. By the duality discussed in \cref{durevision}, one can ask whether the dual of an affine group scheme corresponds to a commutative formal Lie group (\cref{usethisdefasrefsays2}) or a non-commutative formal Lie group (\cref{usethisdefasrefsays}). To this end, we first prove the following general statement, which allows us to construct formal Lie groups via unipotent homotopy theory. 

\begin{theorem}\label{thm4intro}
Let $n \ge 1$ be an integer and $X$ be a pointed higher stack over a field $k$ satisfying the conditions
\begin{equation}\label{conditionintro}
H^0 (X, \cO) \simeq k, \quad H^i(X,\cO) = 0 \text{ for all } 0 < i < n, \quad \text{and} \quad H^{n+1}(X, \cO)=0.
\end{equation}
Further, let us assume that $\dim _k H^n (X, \cO) =g$.
Then the dual of $\pi_n^{\mathrm{U}}(X)$ is a commutative formal Lie group of dimension $g$ if $n>1$ and is a non-commutative formal Lie group of dimension $g$ if $n=1$. (\textit{cf.}~\cref{cof6}.)
\end{theorem}{}

Unlike the construction in \cite{MR457458}, our construction of the formal Lie groups in \cref{thm4intro} does not rely on the \'etale cohomology groups of $\mathbf{G}_m$ and is also able to recover \emph{non-commutative} formal Lie groups.
Instead, the proof of \cref{thm4intro} relies on establishing certain foundational results on representability of duals of unipotent group schemes by formal Lie groups---this is the subject of \cref{cof0}. Using our work in \cref{cof0}, we prove the following result which establishes a correspondence between formal Lie groups and certain ``higher algebraic" structures.

\begin{theorem}\label{thm5intro}
Let $k$ be a field.
The full subcategory of the $\infty$-category $(\mathrm{DAlg}^{\mathrm{ccn}}_k)_{/k}$ (see \cref{derivedvscosimp}) spanned by those $B \in (\mathrm{DAlg}^{\mathrm{ccn}}_k)_{/k} $ that satisfy $H^0(B) = k$, $\dim_k H^2 (B) = g$ and $H^* (B) \simeq \Sym^*H^2 (B)$ is equivalent to the category of commutative formal Lie groups over $k$ of dimension $g$ via the functor that sends a formal Lie group $ E \mapsto R\Gamma(K(E^\vee, 2), \cO)$. (\textit{cf.}~\cref{classifycommfgl2}.)
\end{theorem}
In \cref{classifynoncommfgl}, we prove a similar classification result for non-commutative formal Lie groups, proposing an answer to a problem posed by Nori \cite[p.~75]{MR682517}; also see \cref{classifycommfgl} for a variant of \cref{thm5intro}.
The proofs of these results rely on unipotent homotopy theory; more precisely, the inverse to the functor described in \cref{thm5intro} is obtained by taking $\pi_2$ of the affine stack associated with $B$.
\vspace{2mm}

Finally, let us state the result that compares our formal Lie groups arising from \cref{thm4intro} with the ones constructed by Artin--Mazur in a geometric context, proposing a homotopical answer to \cite[p.~92, Qn.~(a)]{MR457458}.

\begin{theorem}\label{thm6intro}
Let $n \ge 1$ be an integer.
Let $X$ be a pointed proper scheme over an algebraically closed field $k$ of characteristic $p>0$ satisfying the conditions in \cref{conditionintro}. Let $\Phi_X^n$ denote the $n$-th Artin--Mazur formal group defined in this context. 
Then if $n > 1$, $\Phi^n_X$ is naturally isomorphic to the dual $\pi^{\mathrm{U}}_n(X)^\vee$ of the $n$-th unipotent homotopy group scheme of $X$. 
If $n=1$, $\Phi^n_X$ is naturally isomorphic to $(\pi_1^{\mathrm{U}}(X)^{\mathrm{ab}})^{\vee}$ (\textit{cf.}~\cref{cycomp}.)    
\end{theorem}

Note that \cref{thm3intro} and \cref{thm6intro} show that the unipotent homotopy group schemes recover the $p$-adic \'etale homotopy groups from \cite{etalehomo} as well as the formal groups from \cite{MR457458} in many interesting cases. 
Roughly speaking, according to the principle in which homotopy groups relate to homology groups by the classical Hurewicz theorem in algebraic topology, \cref{thm6intro} says that the dual group schemes of the Artin--Mazur formal groups can now be thought of as a ``homology theory" for the unipotent homotopy theory.
While one needs to impose certain hypotheses to guarantee that certain deformation functors defined by Artin--Mazur are pro-representable by formal groups, the unipotent homotopy group schemes we consider are always representable. Further, \cref{thm6intro} makes it clear that for a scheme $X$ over a field $k$ of characteristic $p>0$, the unipotent homotopy group schemes $\pi_i^{\mathrm{U}}(X)$ contain nontrivial arithmetic information.
For example, if $X$ is a $K3$ surface, the dual of $\pi_2^{\mathrm{U}}(X)$ recovers the formal Brauer group and thus the height of $X$. 
Applying \cref{thm4intro} to certain stacks introduced in \cite{drinfeld2018stacky} (see also \cite{BL,Monrec}), in many cases of interest (e.g., $K3$ surfaces), one can moreover recover the formal groups arising from de Rham cohomology that were constructed by Artin and Mazur \cite[\S~III]{MR457458} using the multiplicative de Rham complex; see \cref{multideRham1}.
\vspace{2mm}

We discuss more examples in \cref{curveabelian}.
We show that the unipotent homotopy types of curves and abelian varieties are of the form $K(\pi_1^{\mathrm{U}}, 1)$, as one might expect. 
Further, in \cref{norisquestion}, we show using the methods from our paper that the unipotent fundamental group scheme admits a flat variation for certain families of curves and abelian varieties, giving a construction for a claim made in \cite{MR682517}.
As an application of \cref{norisquestion}, we explain in \cref{cold2} how the filtered circle from \cite{moulinos2019universal} can simply be obtained as the unipotent homotopy type of the classical family of curves that degenerates nodal curves to a cuspidal curve. 
In \cref{cold3}, we construct the unipotent fundamental group scheme of the universal curve over $\overline{\cM}_{g,1}$, which could be of possible further interest.
\vspace{2mm}

Next, we discuss applications of unipotent homotopy theory to questions surrounding derived equivalences.
In \cite{MR1818984}, Bondal--Orlov showed that if two smooth projective varieties $X$ and $Y$ with ample (anti-)canonical bundle are derived equivalent, then $X$ and $Y$ are isomorphic as varieties.
A classical question in algebraic geometry arising from the foundational work of Bondal and Orlov is the following: if $X$ and $Y$ are two derived equivalent algebraic varieties over $k$, how are certain invariants associated with $X$ and $Y$ related?
For example, a conjecture of Orlov \cite{Orlov2} asks whether the rational Chow motive $M(\cdot)_{\mathbf
Q}$ of a smooth projective variety is a derived invariant.
In positive characteristic, the work from our paper prompts the following line of investigation:
\begin{question}\label{musing}
For what class of smooth projective varieties $X$ and $Y$ defined over an algebraically closed field $k$ of characteristic $p>0$ does $X$ and $Y$ being derived equivalent imply that there is an isomorphism of unipotent homotopy types $\mathbf{U}(X) \simeq \mathbf{U}(Y)?$
\end{question}
One can also ask variants of this question by demanding that in addition to derived equivalence, certain other invariants associated to $X$ and $Y$ coincide.
Note that if $X$ and $Y$ are derived equivalent smooth projective varieties with ample (anti-)canonical bundle, then $\mathbf{U}(X) \simeq \mathbf{U}(Y)$ because \cite{MR1818984} implies that $X \simeq Y$.
In \cref{not-derived-invariant}, we show that there exist abelian threefolds $X$ for which $\pi_1^{\mathrm{U}}(X)$ and $\pi_1^{\mathrm{U}}(X^\vee)$ are not isomorphic;
we construct such abelian varieties by carefully analyzing the moduli space of certain Dieudonn\'e modules of dimension $3$ (see \cref{a=1}), extending work of Oda--Oort \cite[Prop.~4.1]{Ooo} to certain non-supersingular cases.
This shows that $\mathbf{U}(X)$ and $\mathbf{U}(X^{\vee})$ cannot be isomorphic, even though the abelian varieties $X$ and $X^\vee$ are derived equivalent.
\vspace{2mm}

\cref{musing} is particularly interesting in the case of Calabi--Yau varieties (\cref{CY}). 
Since the canonical bundle of a Calabi--Yau variety is trivial, the techniques from \cite{MR1818984} do not yield useful conclusions.
In this paper, we use our homotopical techniques to give a positive answer to \cref{musing} in the case when $X$ and $Y$ are derived equivalent Calabi--Yau varieties.
\begin{theorem}\label{thm7intro}
    Let $X$ and $Y$ be two derived equivalent Calabi--Yau varieties of dimension $n$ over an algebraically closed field $k$ of characteristic $p>0$. Then there exists an isomorphism $\mathbf{U}(X) \simeq \mathbf{U}(Y)$ of unipotent homotopy types of $X$ and $Y$. (\textit{cf.}~\cref{thmCY}.)
\end{theorem}{}

In order to state our next result, we will need a definition.

\begin{definition}\label{defncy}
   Let $\mathrm{CY}_n$ denote the $1$-category of Calabi--Yau varieties of dimension $n$ over an algebraically closed field of characteristic $p>0$. Let $\mathcal{N}\mathrm{CY}_n$ denote the $1$-category whose objects are the objects of $\mathrm{CY}_n$; for any two Calabi--Yau varieties $X$ and $Y$, the set of morphisms is defined to be the set of isomorphism classes of objects of $D_{\mathrm{perf}}(X \times_k Y)$ and the composition is given by convolution (\textit{cf.} \cite[\href{https://stacks.math.columbia.edu/tag/0G0F}{Tag~0G0F}]{stacks}). 
   Note that there is a natural functor $\mathcal{N} \colon \mathrm{CY}_n \to \mathcal{N}\mathrm{CY}_n^{\mathrm{op}}$ that sends a map $f \colon X \to Y$ of Calabi--Yau varieties to the isomorphism class of the structure sheaf of the graph $\cO_{\Gamma_f} \in D_{\mathrm{perf}} (Y \times_k X)$.
\end{definition}{}

\begin{remark}
    Any object $K \in D_{\mathrm{perf}}(X \times_k Y)$ induces a $k$-linear exact functor $D_{\mathrm{perf}}(X) \to D_{\mathrm{perf}}(Y)$, which is called the associated Fourier--Mukai transform \cite{MR607081};
    note that these Fourier--Mukai functors are generally not symmetric monoidal.
    Conversely, by a result of Orlov \cite[Thm.~2.19]{Orlov}, any $k$-linear exact equivalence arises uniquely (up to natural isomorphism) as a Fourier--Mukai functor.
\end{remark}{}
Now we can formulate the next result, which proves that in odd characteristic, the unipotent homotopy type of Calabi--Yau varieties of dimension at least $3$ is functorial even in the ``non-commutative sense.''
In particular, if $X$ and $Y$ are derived equivalent, then $\UU(X)$ and $\UU(Y)$ are canonically isomorphic in the homotopy category of higher stacks over $k$, which we denote by $\mathrm{h}\mathrm{Shv}(k)$---this gives a functorial strengthening of \cref{thm7intro}.
More precisely, we have the following statement:
\begin{theorem}\label{thm8into}
Let $k$ be an algebraically closed field of characteristic $p>2$ and $n>2$ be an integer. Let $\UU \colon \mathrm{CY}_n \to \mathrm{h}\Shv(k)$ be the functor obtained by sending $X$ to the unipotent homotopy type $\UU(X)$. Then there is a canonical functor $\widetilde{\UU}: \mathcal{N}\mathrm{CY}_n^{\mathrm{op}}\to \mathrm{h}\Shv(k)$ such that $\widetilde{\UU}\circ\mathcal{N}$ is naturally equivalent to $\UU$. (\textit{cf.}~\cref{thmCY1}.)
 \end{theorem}{}
\Cref{thm7intro} implies that if $X$ and $Y$ are two derived equivalent Calabi--Yau varieties of dimension $n$, then $\pi_i^{\mathrm{U}}(X) \simeq \pi_i^{\mathrm{U}}(Y)$ for all $i \ge 0$; in particular, one has $\pi_n^{\mathrm{U}}(X) \simeq \pi_n^{\mathrm{U}}(Y)$.
Concretely, the latter isomorphism means that the Artin--Mazur formal Lie groups of $X$ and $Y$ are isomorphic (see~\cref{thm6intro}), which had previously been observed by Antieau--Bragg \cite{AB1} using certain constructions from \cite{TR}. Our proof builds on their observation regarding the use of \cite{TR}, but requires a significant amount of new homotopical ingredients.
The main insight is that from the standpoint of unipotent homotopy theory, Calabi--Yau varieties behave in some sense like spheres. More precisely, for every $1$-dimensional commutative formal Lie group $E$, we construct a (higher) stack $S^n_{E} \colonequals \Sigma^{n-1}B E^\vee$, which one may think of as a \emph{formal $n$-sphere}.
In \cref{hach}, we prove that if $X$ is a Calabi--Yau variety of dimension $n$ with Artin--Mazur formal group $\Phi^n_X$, then $\mathbf{U}(X)$ is (noncanonically) isomorphic to the unipotent homotopy type of $S^n_{\Phi^n_X}$. Relatedly, we also prove in \cref{ordinaryCY} that the $p$-completion of the Artin--Mazur \'etale homotopy type of an ordinary Calabi--Yau variety of dimension $n$ is equivalent to the $p$-completion of the $n$-sphere; this gives an ``algebraic model" for the $p$-adic unstable homotopy of spheres. 
\vspace{2mm}

However, to construct the canonical functor $\widetilde{\UU}$ as in \cref{thm8into}, one needs further ingredients. We use our construction of the Artin--Mazur--Hurewicz class (\cref{amhclass}) as well as the notion of $\mathrm{TR}$ from \cite{TR}; additionally, we need to establish the vanishing $\pi_{n+1}^{\mathrm{U}} (S^n_{\Phi^n_X}) = 0$ for $n>2$.
The intuition behind this vanishing statement comes from the stable homotopy groups of spheres:
if $S^n$ is the usual $n$-sphere, then for $n>2$, the classical homotopy group $\pi_{n+1}(S^n) \simeq \mathbf Z/2\mathbf Z$.
If $k$ is a field of odd characteristic, one may therefore hope that the unipotent homotopy group scheme $\pi_{n+1}^{\mathrm{U}} (S^n_{\Phi^n_X}) = 0$. To turn this hope into a proof, we first prove a generalization of the classical Freudenthal suspension theorem in the world of unipotent homotopy theory---this plays an important role in understanding the unipotent homotopy theory of Calabi--Yau varieties similar to the role the classical version plays in understanding the homotopy theory of spheres.
We prove the following version of the Freudenthal suspension theorem that is general enough to be used for studying unipotent homotopy of algebraic varieties as well as homotopy theory of spaces.
\begin{theorem}[Freudenthal suspension theorem in unipotent homotopy theory]\label{thm9intro} Let $n \ge 0$ be an integer. Let $X$ be a pointed connected higher stack over a field $k$ such that $H^i (X, \mathscr{O})$ is finite-dimensional for all $i \ge 0$ and $\pi_i^{\mathrm{U}}(X) = \left \{* \right \}$ for $i \le n$. Then there are natural maps $\pi_i^{\mathrm{U}}(X) \to \pi_{i+1}^{\mathrm{U}}(\Sigma X)$ which are isomorphisms of group schemes for $i \le 2n$ and a surjection for $i = 2n+1$. (cf.\ \cref{freudenthalsusp}.)
\end{theorem}
The hardest part of the proof is the surjectivity of the map $\pi_{2n+1}^{\mathrm{U}}(X) \to \pi_{2n+2}^{\mathrm{U}}(\Sigma X)$;
here, the assumption that $H^i (X, \cO)$ is finite-dimensional is crucial.
For example, \cref{sohard} and \cref{advv} crucially rely on such finiteness assumptions and they play an important role in the proof of \cref{thm9intro}. 
To a certain extent, this is because the functor $\mathbf{U} (\cdot)$ does not preserve limits in general and we need to establish several results regarding cohomology and base change in the context of higher stacks as an essential ingredient in the proof of \cref{thm9intro}.
To this end, we also discuss the theory of quasi-coherent sheaves on affine stacks (which is not discussed in \cite{Toe}).
More generally, we concretely describe the derived $\infty$-category of quasi-coherent sheaves on a fairly general class of stacks that we call \emph{weakly affine} (see \cref{makiman34});
this description generalizes a result of Lurie in the case of affine stacks over fields of characteristic zero \cite[Prop.~4.5.2]{lurie2011quasi}.
Using \cref{makiman34}, in \cref{sec2.3}, we prove certain results regarding cohomology and base change.
\vspace{2mm}

Relying on \cref{thm9intro} along with a careful study of certain tensor constructions for group schemes (see \cref{heightformula} for the relation of the tensor product of group schemes from \cref{chris1} with the existing literature) carried out in \cref{worldcup1}, we prove the following (\textit{cf.}~\cref{retire1}):

\begin{theorem}\label{thm10intro}
    Let $X$ be a Calabi--Yau variety of dimension $n \ge 3$ over an algebraically closed field $k$ of characteristic $p>0$. Then 
$$ \pi_{n+1}^{\mathrm{U}}(X)
 \simeq
\begin{cases}
W[F] & \text{if} \,\,p=2 \,\text{and} \,\,X\,\, \text{is not weakly ordinary,}\,\, \\
\mathbf{Z}/2\mathbf{Z} & \text{if} \,\,p=2 \,\text{and} \,\,X \,\text{is weakly ordinary,} \\

0 & \text{otherwise, i.e., if}\,\, p \ne 2.
\end{cases}
$$
\end{theorem}{} 
The construction of $\widetilde{\UU}$ as in \Cref{thm8into} can now be completed by using \cref{hach} and the fact that  $\pi^{\mathrm{U}}_{n+1}(X) =0$ for an $n$-dimensional Calabi--Yau variery $X$ over an algebraically closed field of characteristic $p>2$ and $\dim X \ge 3$. The calculation in \cref{thm10intro} is really a consequence of a more general calculation for the homotopy of the formal $n$-spheres that we introduce (see \cref{haircut12}) and the fact that the unipotent homotopy type of a Calabi--Yau variety is isomorphic to a formal sphere. 
Finally, we mention two consequences of \cref{thm7intro} and \cref{thm8into}.
\begin{corollary}
   Let $X$ and $Y$ be two derived equivalent Calabi--Yau varieties of dimension $n$ over an algebraically closed field $k$ of characteristic $p>0$. Then $R\Gamma (X, \cO)$ and $R\Gamma (Y, \cO)$ are isomorphic as $E_\infty$-algebras over $k$. 
\end{corollary}{}

\begin{corollary}
    Let $k$ be an algebraically closed field of characteristic $p>2$ and $n>2$. Then there is a canonical functor from ${\mathcal{N}\mathrm{CY}_n}$ to the homotopy category of $E_\infty$-algebras over $k$ that sends $X$ to $R\Gamma(X,\cO)$.
\end{corollary}{}

\subsection*{Further remarks}
In this subsection, we briefly mention some natural directions and questions arising from our paper.
The notion of unipotent homotopy type studied in this paper can be applied to any higher stack. Recently, the stacky approach to $p$-adic cohomology theories due to Bhatt--Lurie \cite{BL} and Drinfeld \cite{Dri20} constructs certain stacks such as $X^\Prism$ (resp. $X^{\mathrm{crys}}$, $X^{\mathrm{dR}}$) which can be used to understand prismatic (resp. crystalline, de Rham) cohomology theory along with the relevant notion of coefficients.
It seems interesting to further investigate the unipotent homotopy types of these stacks as well. As an instance of this perspective, in \cite[Prop.~6.6]{soon}, a version of \cref{thm5intro} has been used to reconstruct the $1$-dimensional formal group over $(\mathrm{Spf}\, \mathbf{Z}_p)^{\Prism}$ due to Drinfeld \cite{drinew} as the dual of $\pi_2^{\mathrm{U}}((B \mu_{p^\infty})^{\Prism})$;
here $B \mu_{p^\infty}$ is the classifying stack of the $p$-divisible group $\mu_{p^\infty}$ (\textit{cf.}~\cite[Def.~3.21]{Mon21}).
\vspace{2mm}

When $X$ is a $K3$ surface over an algebraically closed field of positive characteristic, the height of the Artin--Mazur formal group of $X$ is either $\le 10$ or $\infty$ (\cite{Artin-supersingular, Milne111}). 
At present, it is unknown if there is any such numerical bound on the height in the case of Calabi--Yau varieties. Note that from the perspective of unipotent homotopy theory, every $n$-dimensional Calabi--Yau variety is equivalent to a formal $n$-sphere (see \cref{hach}). 
The question of whether there exists a Calabi--Yau variety $X$ with Artin--Mazur formal group $E$ is therefore equivalent to asking if there is a Calabi--Yau variety $X$ of dimension $n$ such that $\mathbf{U}(X) \simeq \mathbf{U}(S^n_E)$.
We wonder if there is any reasonable way to ``approximate" the stack $S^n_E$ to at least produce a smooth proper scheme $X$ (possibly of large dimension) whose $n$-th Artin--Mazur formal group would be $E$.
\vspace{2mm}

The calculation of $\pi_{n+1}^{\mathrm{U}}(X)$ for a Calabi--Yau variety $X$ of dimension $n \ge 3$ plays a crucial role in producing the \emph{canonical} isomorphism in \cref{thm8into}. It would be interesting to obtain more computations of higher unipotent homotopy group schemes of Calabi--Yau varieties. As explained before, this is equivalent to computing higher unipotent homotopy group schemes of the formal spheres. Relatedly, it would be interesting to obtain a version of the EHP sequence in the context of unipotent homotopy theory, which is not pursued in this paper;
see \cite{EHP3} for the EHP sequence in the context of $\mathbf{A}^1$-homotopy theory \cite{MV1}.
\vspace{2mm}

\cref{thm5intro} establishes a correspondence between formal Lie groups of dimension $g$ and certain higher algebraic structures. Classically, formal Lie groups of dimension $1$ are also closely connected to complex-orientable cohomology theories \cite{CF66, MR253350}. It would be interesting to investigate how \cref{thm5intro} (in the case of $g=1$) interacts with this picture. 
\vspace{2mm}

Finally, we point out that the theory of \'etale homotopy types has been used in anabelian geometry for higher dimensional varieties; see \cite{MR3549624}. 
Moreover, certain unipotent fundamental groups appear in the work of Kim \cite{Kim1, Kim2}. We hope that the unipotent homotopy type of a scheme $X$ (as well as the unipotent homotopy types of the stacks $X^\Prism$, $X^{\mathrm{crys}}$, $X^{\mathrm{dR}}$, etc.) could have possible applications in this direction as well.

\subsection*{Outline of the paper}
We start by recalling To\"en's work on affine stacks in \cref{affine-stacks}. In \cref{quasilurie}, we discuss the theory of quasi-coherent sheaves on affine stacks and describe the derived $\infty$-category of quasi-coherent sheaves of weakly affine stacks (see \cref{makiman34}). In \cref{sec2.3}, we establish certain cohomology and base change type results in the context of affine stacks. 
We prove \cref{thm1intro} in \cref{recovernori} and discuss some properties of higher unipotent homotopy group schemes in \cref{posto1}. We study pro-algebraic completions of group valued sheaves in \cref{proalgcompletion}. 
These are used in \cref{makiman} to prove the Freudenthal suspension in unipotent homotopy theory (\cref{thm9intro}). \Cref{posto3} discusses some Frobenius semilinear algebra in the derived setup. \cref{posto4} is devoted to the theory of profinite group schemes, with a special focus on the theory of profinite unipotent group schemes and their representations. The techniques developed in \cref{posto3} and \cref{posto4} are then used in \cref{spri} to prove the profiniteness theorem (\cref{thm2intro}). Subsequently, in \cref{posto2}, we prove the unipotent-\'etale comparison theorem (\cref{thm3intro}). 

\vspace{2mm}

After establishing these foundational results, we move on to discussing more concrete applications. \cref{spri1} and \cref{spri2} describe the unipotent homotopy types of curves and abelian varieties, respectively. 
In \cref{cold6}, we discuss moduli of certain Dieudonn\'e modules and construct examples showing that the unipotent fundamental group scheme is not a derived invariant for abelian threefolds (\cref{not-derived-invariant}).
In \cref{cold1}, we study unipotent fundamental group schemes of certain families and applications to a question considered by Nori (\cref{norisquestion}) and a different construction of the filtered circle (\cref{cold2}). 
In \cref{cof0}, we prove certain foundational results on the representability of duals of unipotent group schemes by formal Lie groups, which are then used in \cref{cof1} to construct certain formal Lie groups and prove \cref{thm4intro} and \cref{thm5intro}. 
In \cref{cof3}, we recover the Artin--Mazur formal group (\cref{thm6intro}) via unipotent homotopy group schemes.
In \cref{spri4}, we discuss the unipotent homotopy type of ordinary Calabi--Yau varieties. Finally, we prove our main applications \cref{thm7intro} and \cref{thm8into} regarding derived equivalent Calabi--Yau varieties in \cref{worldcup1}.

\subsection*{Notations and conventions}
As in \cite{Toe} and \cite{Lur09}, we work with a certain Grothendieck universe (containing the set of natural numbers); to deal with the size-related aspects of certain constructions, \cite{Toe} and \cite{Lur09} choose an enlargement of the Grothendieck universe, which will be kept implicit in our paper similar to \cite{Lur09}. We will freely use the theory of $\infty$-categories as developed in \cite{Lur09, luriehigher, luriespectral}. We denote by $\mathcal{S}$ the $\infty$-category of spaces, which is also called the $\infty$-category of anima or $\infty$-groupoids. If $\mathcal{C}$ is a $\infty$-category and $X,Y$ are objects of $\mathcal{C}$, we will let $\mathrm{Map}(X,Y)$ denote the associated mapping space. 
For an $\infty$-category $\mathcal{C}$ that admits finite limits, $\Pro(\mathcal{C})$ denotes the $\infty$-category of pro-objects of $\mathcal{C}$ (\cite[\S~5.3]{Lur09}). For an $E_\infty$-ring $R$, we will let $\mathrm{Mod}_R$ denote the $\infty$-category of $R$-modules. 
An $E_\infty$-ring $R$ called discrete if $\pi_i(R) \simeq 0$ for all $i \ne 0$. In such a case, $\mathrm{Mod}_R$ is simply the derived $\infty$-category of $R$-modules $D(R)$ and we will sometimes prefer the latter notation in this context to avoid confusion. While discussing $t$-structures, we will follow the homological convention. Unless otherwise mentioned, limits and colimits are taken in the $\infty$-categorical sense and the tensor products are ``derived.'' A higher stack will often simply be called a stack. 
We will denote the derived $\infty$-category of quasi-coherent sheaves on a (higher) stack $X$ by $D_{\qc}(X)$.
Given a map $f \colon X \to Y$ of (higher) stacks, $f_*$ (resp. $f^*$) will always denote the ``derived" pushforward (resp. pullback) discussed in \cref{secc2}. A higher stack $X$ over a field $k$ is called connected if $\pi_0 (X) \simeq \left \{ * \right \}$; $X$ is called cohomologically connected if $H^0 (X, \cO) \simeq k$. For a scheme $X$ over $k$, we will let $D_{\mathrm{perf}}(X)$ denote the $k$-linear triangulated $1$-category of perfect complexes; note that if $X$ is smooth, then $D_{\mathrm{perf}}(X) \simeq D^{b}_{\coh}(X)$ (\cite[\href{https://stacks.math.columbia.edu/tag/0FDC}{Tag~0FDC}]{stacks}). For a discrete commutative ring $A$, we will let $\mathrm{Alg}_A$ denote the category of discrete (commutative) $A$-algebras and $\mathrm{DAlg}_A^{\mathrm{ccn}}$ denote the $\infty$-category constructed in \cref{derivedvscosimp}. For $B \in \mathrm{DAlg}_A^{\mathrm{ccn}}$, we will denote the associated (affine) higher stack by $\Spec B$ (see \cref{defofaffinestack}).
Note that if $B \in \mathrm{Alg}_A \subset \mathrm{DAlg}_A^{\mathrm{ccn}}$, then $\Spec B$ is equivalent to the classical affine scheme associated with $B$, so there is no clash of notation. We will let $\mathbf{G}_a$ (resp.\ $\mathbf{G}_m$) denote the additive (resp. multiplicative) group scheme and $\mathbf{\widehat{G}}_a$ (resp.\ $\mathbf{\widehat{G}}_m$) denote the additive (resp. multiplicative) formal group.

\subsection*{Acknowledgments}
We are very grateful to Ben Antieau, Bhargav Bhatt and Akhil Mathew for many helpful discussions and encouragement throughout the writing of this article.
We thank Denis-Charles Cisinski, Dustin Clausen, Vladimir Drinfeld, Lars Hesselholt, Luc Illusie, Srikanth B. Iyengar, Andy Jiang, Dimitry Kubrak, Shizhang Li, Jacob Lurie, Barry Mazur, Peter Scholze, Carlos Simpson, Bertrand To\"en and Marco Volpe for useful conversations. We also thank the anonymous referees for their comments on an earlier version of this manuscript.
\vspace{2mm}

The authors acknowledge funding through the Max Planck Institute for Mathematics in Bonn, Germany, as well as NSF Grants DMS \#1801689, FRG \#1952399 (S.M.) and DMS \#1926686 (E.R.).
In the earlier stages of this project, the first named author was a graduate student at the University of Michigan and the second named author was a member at the Institute for Advanced Study.
During the preparation of the paper, the first author also visited the University of Chicago and the University of Copenhagen. In the revision stages of the paper, the first author was supported by the University of British Columbia and the second author was supported by the Institut des Hautes \'Etudes Scientifiques. We thank all these institutes for their support and hospitality.

\newpage

\addtocontents{toc}{\protect\setcounter{tocdepth}{2}}

\section{Some results on affine stacks}\label{secc2}
In this section, we discuss some foundational material on affine stacks that will be useful to us later on. In \cref{affine-stacks}, we begin by recalling the definitions of higher stacks and affine stacks following \cite{Toe} in an $(\infty,1)$-categorical language. In \cref{quasilurie}, we discuss the theory of quasi-coherent sheaves on affine stacks that did not appear in \cite{Toe}. The main result in \cref{quasilurie} is \cref{makiman34}, which extends a similar result of Lurie for affine stacks in characteristic $0$ (\cite[Prop.~4.5.2~(7)]{lurie2011quasi}) to \textit{weakly affine} (see \cref{weaklyaffine}) stacks in all characteristics.

\subsection{Definition of affine stacks}\label{affine-stacks}
This subsection provides a quick reminder on To\"en's work on affine stacks \cite{Toe}.
\begin{notation}

For a fixed discrete commutative ring $A$, we will denote the category of discrete $A$-algebras by $\Alg_A$ and the category of affine schemes by $\mathrm{Aff}_A$. Let $D(A)$ denote the derived $\infty$-category of $A$-modules. Let $\cS$ denote the $\infty$-category of spaces. Let $\PShv(A) \colonequals \mathrm{Fun}(\mathrm{Aff}_A^{\mathrm{op}}, \cS)$ denote the $\infty$-category whose objects will be called presheaves (of spaces).
\end{notation}{}

\begin{definition}\label{toposA}
We let $\Shv(A)$ denote the full subcategory of $\PShv(A)$ spanned by objects that are sheaves with respect to the \textit{fpqc} topology on $\mathrm{Aff}_A^{\mathrm{}}$ (see \cite[Ch.~6]{Lur09}). The objects of $\Shv(A)$ will be called (higher) stacks over $A$. The natural inclusion $\Shv(A) \to \PShv(A)$ preserves all small limits. 
\end{definition}{}

\begin{remark}
Similar to \cite{Toe}, by our conventions, $\mathrm{Alg}_A$ is the category of $A$-algebras in a fixed universe; thus it can be regarded as a small category (by enlarging the universe). Therefore the functor $\mathrm{Shv}(A) \to \mathrm{PShv}(A)$ admits a left adjoint given by sheafification.    
\end{remark}{}

\begin{definition}
The category $\Shv(A)$ has a final object $\left \{* \right \}$ and we let $\Shv(A)_*$ denote the category of pointed objects, i.e., the coslice category $\left \{* \right\} /\Shv(A)$. Its objects will be called pointed (higher) stacks.
\end{definition}{}

\begin{definition}\label{pancake}
Let $X$ be a pointed (higher) stack. For $n\ge 0$, we define $\pi_n(X, *) \in \Shv(A)$ to be the sheafification of the presheaf on $\Aff_A$ that sends $\Spec S$ to $\pi_n (X(\Spec S),*)$. The pointed sheaf $\pi_n(X,*)$ is naturally a sheaf of groups for $n>0$, which is further commutative for $n>1$.
For a stack $X$, we define $\pi_0 (X) \in \Shv(A)$ to be the sheafification of the presheaf on $\Aff_A$ that sends $\Spec S \mapsto \pi_0(X(\Spec S))$. A stack $X$ is called \textit{connected} if $\pi_0(X) \simeq \left \{*\right \}$. A pointed  connected stack will refer to the data of a map $p \colon \left \{* \right \} \to X$ such that $p$ is an effective epimorphism in $\Shv(A)$.
\end{definition}{}

\begin{example}\label{ELstack}
For a commutative affine group scheme $G$ over $A$, one can define the Eilenberg--MacLane stacks $K(G,i)$, which have the property that $\pi_n(K(G, i), *) \simeq G$ for $n=i$ and is trivial otherwise.
This type of stacks will play a crucial role in our paper later on. Note that $K(G,1)$ is simply the classifying stack of $G$, which will also be denoted as $BG$.
\end{example}{}

\begin{definition}
We let $\Shv(A)^\wedge$ denote the full subcategory of hypercomplete objects of $\Shv(A)$ (see \cite[ Prop.~6.5.2.13]{Lur09}). This can also be defined as the $\infty$-category associated with the model category considered in \cite[\S~1.1]{Toe}.
\end{definition}{}

\begin{definition}\label{derivedvscosimp}
We let $\mathrm{DAlg}^{\mathrm{ccn}}_A$ denote the $\infty$-category arising from the simplicial model structure defined in \cite[Thm.~2.1.2]{Toe} on the category $\Alg^\Delta_A$ of cosimplicial algebras over $A$. By \cite[Cor.~4.2.4.8]{Lur09}, $\mathrm{DAlg}^{\mathrm{ccn}}_A$ has all small limits and colimits.
\end{definition}{}

\begin{remark}\label{aboutrings}Let us explain the motivation behind the notation $\mathrm{DAlg}^{\mathrm{ccn}}_A$ for the $\infty$-category defined above. The approach to affine stacks due to To\"en uses cosimplicial algebras. In \cite{soon}, Mathew and Mondal give a generalization to the theory of affine stacks, called ``affine derived stacks", that uses derived commutative rings \cite{Rak20}. They prove that the $\infty$-category of cosimplicial rings is \textit{equivalent} to the $\infty$-category of coconnective derived commutative rings; see \cite[Cor.~3.7]{soon} (\textit{cf.}\ \cite[Cor.~5.29]{partlie}). Note that the definition of derived commutative rings is purely $\infty$-categorical and does not require model structures. There is a natural functor from the $\infty$-category of derived commutative rings to the $\infty$-category of $E_\infty$-rings that preserves small limits and colimits (see \cite[Prop.~4.2.27]{Rak20}), but not an equivalence in general.

\end{remark}{}

\begin{definition}
The natural inclusion functor $\Alg_A \to \mathrm{DAlg}^{\mathrm{ccn}}_A$ gives rise to a functor $\mathrm{Aff}_A \to (\mathrm{DAlg}^{\mathrm{ccn}}_A)^{\mathrm{op}}$.
The \emph{derived global sections functor} $R\Gamma (\cdot, \cO) \colon \PShv(A) \to (\mathrm{DAlg}^{\mathrm{ccn}}_A)^{\mathrm{op}}$ is the left Kan extension of $\mathrm{Aff}_A \to (\mathrm{DAlg}^{\mathrm{ccn}}_A)^{\mathrm{op}}$ along $\mathrm{Aff}_A \to \PShv(A)$.
If $X \in \PShv(A)$, we will call $R\Gamma(X, \cO)$ the derived global sections of $X$. 
\end{definition}{}

\begin{definition}
There is a natural functor $\mathrm{DAlg}^{\mathrm{ccn}}_A \to D(A)$ that preserves all limits.
By the adjoint functor theorem, the above functor admits a left adjoint $L \colon D(A) \to \mathrm{DAlg}^{\mathrm{ccn}}_A$. For $i \in \mathbf{Z}_{\ge 0}$, we define $\Sym A[-i] \colonequals L (A[-i])$ (\textit{cf.}\ \cite[Lem.~2.2.5]{Toe}).
\end{definition}{}

\begin{definition}[Affine stacks]\label{defofaffinestack}Let $B \in \mathrm{DAlg}^{\mathrm{ccn}}_A$.
Let $h_{B} \colon \mathrm{DAlg}^{\mathrm{ccn}}_A \to \mathcal{S}$ denote the functor corepresented by $B$. Restricting $h_{B}$ along the inclusion $\Alg_A \to \mathrm{DAlg}^{\mathrm{ccn}}_A$, we obtain a functor $\restr{h_{B}}{\mathrm{Aff}_A} \colon \mathrm{Aff}_A^{\mathrm{op}} \to \mathcal{S}$. By faithfully flat descent, it follows that $\restr{h_{B}}{\mathrm{Aff}_A}$ is a (higher) stack, which will be denoted as $\Spec B$.
\vspace{2mm}

An object of $\Shv(A)$ will be called an \textit{affine stack} over $A$ if it is isomorphic to $\Spec B$ for some $B \in \mathrm{DAlg}^{\mathrm{ccn}}_A$. 
\end{definition}{}

\begin{example}
An affine scheme is an example of an affine stack. The stacks $K(\GG_a, i)$ are all affine. Indeed, $K(\GG_a, i) \simeq \Spec \Sym A[-i]$. However, for $i>0$, the stacks $K(\GG_m, i)$ are not affine.
\end{example}{}

\begin{example}
It is worth pointing out that zero truncated affine stacks are not necessarily affine schemes. In fact, an interesting class of such examples are given by quasi-affine schemes. This is essentially a consequence of \cite[Thm.~2.6.0.2]{luriespectral} (\textit{cf.}~\cite[Thm.~2.3]{BhaHar}). Indeed, if $X$ is a quasi-affine scheme over $A$, then \cite[Thm.~2.6.0.2]{luriespectral} implies that $X(B) \simeq \mathrm{Map}_{\mathrm{CAlg}(D(A))} (R\Gamma(X, \cO), B)$ for any $B \in \Alg_A$, where $\mathrm{CAlg}(D(A))$ denotes the $\infty$-category of $E_\infty$-algebras over $A$. Since the natural functor $\mathrm{DAlg}^{\mathrm{ccn}}_A \to \mathrm{CAlg}(D(A))$ preserves limits, it has a left adjoint. This implies that $X$ is an affine stack.
\end{example}

\begin{remark}\label{hypercompletea}
It follows from the definition that the inclusion of the category of affine stacks into the category of all stacks preserves limits. It also follows that affine stacks are actually hypercomplete, i.e., they are objects of $\Shv(A)^\wedge$. 
\end{remark}{}

\begin{remark}\label{cope}
Note that by adjoint functor theorem, the functor $\Spec \colon (\mathrm{DAlg}^{\mathrm{ccn}}_A)^{\mathrm{op}} \to \PShv(A)$ is a right adjoint and the left adjoint is given by the derived global section functor. By \cite[Cor.~2.2.3]{Toe}, we have $R\Gamma(\Spec B, \cO) \simeq B$. The category of affine stacks is therefore equivalent to $(\mathrm{DAlg}^{\mathrm{ccn}}_A)^{\mathrm{op}}$.
\end{remark}{}

To\"en proved the following result that gives a classification of pointed connected affine stacks over a field in terms of concrete structures such as group schemes.
\begin{theorem}[{\cite[Thm.~2.4.1, Thm.~2.4.5]{Toe}}]\label{thmoftoen}
Let $k$ be a field and $\ast \to X$ a pointed object of $\Shv(k)^\wedge$ such that $\pi_0(X) \simeq \{*\}$. Then $X$ is an affine stack if and only if the sheaves $\pi_i(X,\ast)$ are representable by unipotent affine group schemes over $k$ for all $i > 0$.
\end{theorem}

\begin{remark}\label{connn}
By \cite[Thm.~2.4.5]{Toe}, it moreover follows that if $B \in \mathrm{DAlg}^{\mathrm{ccn}}_k$ is an augmented object such that $H^0(B) \simeq k$, then $\Spec B$ is a pointed connected stack. Note that the unipotent group schemes in \cref{thmoftoen} above are not necessarily of finite type.
\end{remark}{}

For any pointed higher stack $X$, one can associate the pro-system of $n$-truncations $(\tau_{\le n} X)$ which is called the Postnikov tower of $X$. If $X$ is a pointed and connected stack, then the fibre of $\tau_{\le n} X \to \tau_{\le n-1}X$ is given by the stack $K(\pi_n(X), n)$ from \cref{ELstack}. If $X$ is an affine stack, it follows that $X \to \varprojlim \tau_{\le n}X$ is an isomorphism (see \cite[Cor.~1.8]{MR100}). The following proposition records certain additional properties of the Postnikov tower in the context of pointed connected affine stacks over a field.

\begin{proposition}[Postnikov tower for affine stacks]\label{postnikov}
    Let $X$ be a pointed connected affine stack over $k$. In this situation
    \begin{enumerate}
        \item The stacks $\tau_{\le n}X$ are naturally pointed connected affine stacks.
        \item\label{postnikov-cohom} The natural maps $H^i (\tau_{\le n}X, \cO) \to H^i (X, \cO)$ induce isomorphisms for $i \le n$ and an injection for $i= n+1$.
        \item The natural map $X \to \varprojlim \tau_{\le n}X$ is an isomorphism.
    \end{enumerate}{}
\end{proposition}{}
\begin{proof}
    This is obtained by combining the statements appearing in the proof of Theorem 2.4.5 and Corollary 1.2.3 in \cite{Toe}.
\end{proof}{}

 We will end this section by noting the following property of affine stacks.

\begin{proposition}\label{sad}
Let $X \in \Shv(A)^\wedge$ be an affine stack. Then there is a simplicial affine scheme $X_{\bullet} \colonequals (\Spec A_n)$ such that $X \simeq \varinjlim_{[n]\in \Delta} \Spec A_n$ in $\Shv(A)^\wedge$. 
\end{proposition}{}

\begin{proof}This follows from choosing a cofibrant object of $\Alg^\Delta_A$ that models the affine stack $X$.
We refer to the proof of \cite[Thm.~2.2.9]{Toe} for more details.
\end{proof}{}

\subsection{Quasi-coherent sheaves on affine stacks}\label{quasilurie}

The main goal of this subsection is to understand the derived $\infty$-category of quasi-coherent sheaves on an affine stack.
In fact, as we prove in \cref{makiman34}, one can completely describe the derived $\infty$-category of quasi-coherent sheaves on a fairly general class of pointed connected higher stacks, which we call \emph{weakly affine}. 
\begin{definition}[Weakly affine stacks]\label{weaklyaffine}
Let $k$ be a field.
A pointed connected higher stack $X$ over $k$ is \emph{weakly affine} if 
\begin{enumerate}
    \item\label{weaklyaffine-simplicial} There exists a simplicial affine scheme $X_{\bullet} \colonequals (\Spec A_n)$ such that $X \simeq \varinjlim_{[n] \in \Delta}\Spec A_n$ in $\Shv(k)^\wedge$.
    \item\label{weaklyaffine-unipotent} $\pi_1(X, *)$ is representable by a unipotent affine group scheme over $k$.
\end{enumerate}{}
\end{definition}{}
\begin{example}
    By \cref{thmoftoen} and \cref{sad}, any pointed connected affine stack over a field $k$ is weakly affine. However, the class of weakly affine stacks is more general.
    For example, note that $K(G,n)$ satisfies \cref{weaklyaffine-simplicial} for any affine commutative group scheme $G$. 
    Indeed, under the Dold--Kan correspondence, the $D(k)$-valued presheaf $G[n]$ on $\Aff_k$ corresponds to a simplicial presheaf whose terms are finite products of $G$ and hence are representable by affine schemes.
    One can thus write $K(G,n)$ as a colimit of a simplicial affine scheme in $\Shv(k)^{\wedge}$.
    Therefore, $K(G, n)$ is weakly affine for $n \ge 2$.
    On the other hand, if $G$ is not unipotent, $K(G,1)$ is not weakly affine because it does not satisfy \cref{weaklyaffine-unipotent}.
\end{example}{}
A description of quasi-coherent sheaves and other derived categories associated with affine stacks does not appear explicitly in \cite{Toe}. Therefore, we will begin by documenting their constructions and some basic properties here. 
\vspace{2mm}

Given a discrete commutative ring $A$, we denote by $D(A)$ the derived $\infty$-category of $A$-modules.
This is a presentable stable $\infty$-category and is equipped with a $t$-structure whose connective part $D(A)_{\ge 0}$ (resp.\ coconnective part $D(A)_{\le 0}$) is the full subcategory of $D(A)$ spanned by those $M \in D(A)$ for which $\Hh^i(M) = 0$ for all $i > 0$ (resp. $i < 0$); see \cite[Prop.~1.3.5.9, Prop.~1.3.5.21]{luriehigher}.
The tensor product of chain complexes equips $D(A)$ with the structure of a symmetric monoidal $\infty$-category in the sense of \cite[Def.~2.0.0.7]{luriehigher}.
\vspace{2mm}

We fix a base ring $A$ as before and consider $\Shv(A)$ as in \cref{toposA}.
For any $X \in \Shv(A)$, we set $\Shv(X) \colonequals \Shv(A)_{/X}$.
\begin{definition}[\textit{cf.}~{\cite[Not.~6.3.5.16]{Lur09}}]
A \emph{$D(\ZZ)$-valued sheaf on $X$} is a functor $\Shv(X)^{\op} \to D(\ZZ)$ between $\infty$-categories that preserves all limits. The $\infty$-category  of $D(\ZZ)$-valued sheaves on $X$ is denoted by $\Shv_{D(\ZZ)}(X)$.
\end{definition}
One can prove exactly as in \cite[Prop.~1.3.4.6]{luriespectral} that $\Shv_{D(\ZZ)}(X)$ is a stable $\infty$-category equipped with a symmetric monoidal structure.
\begin{remark}\label{connective-truncation}
 Consider the limit preserving functor $D(\ZZ) \to \cS$ obtained as a composition of the truncation $D(\ZZ) \to D(\ZZ)_{\ge 0}$ with the forgetful functor $D(\ZZ)_{\ge 0} \to \mathcal{S}$.
Postcomposition with $D(\ZZ) \to \cS$ induces a functor $\tau_{\ge 0} \colon \Shv_{D(\ZZ)}(X) \to \Shv(X)$. 
Similarly, postcomposition with the shift $[n] \colon D(\ZZ) \to D(\ZZ)$ on the stable $\infty$-category $D(\ZZ)$ induces the shift functor $[n] \colon \Shv_{D(\ZZ)}(X) \to \Shv_{D(\ZZ)}(X)$.
\end{remark}
\begin{definition}
Let $\mathscr{F}$ be a $D(\ZZ)$-valued sheaf on $X$.
For any $n \in \ZZ$, we define the \emph{$n$-th cohomology sheaf of $\mathscr{F}$} as 
\[ \cH^n(\mathscr{F}) \colonequals (\tau_{\le 0} \tau_{\ge 0}(\mathscr{F}[n])) \in \Shv(X) \]
where $\mathscr{F}[n]$ and $\tau_{\ge 0}$ are defined as in \cref{connective-truncation} and $\tau_{\le 0} \colon \Shv(X) \to \tau_{\le 0} \Shv(X)$ is the $0$-truncation functor \cite[Prop.~5.5.6.18]{Lur09}.
In particular, $\cH^n(\mathscr{F})$ is a discrete object of $\Shv(X)$ (in the sense of \cite[Def.~5.5.6.1]{Lur09}).
\end{definition}
\begin{proposition}\label{t-structure-sheaves}
We consider the following two full subcategories of $\Shv_{D(\ZZ)}(X)$:
\begin{enumerate}
    \item The full subcategory $\Shv_{D(\ZZ)}(X)_{\ge 0} \subset \Shv_{D(\ZZ)}(X)$ spanned by those $F \in \Shv_{D(\ZZ)}(X)_{}$ such that the sheaves $\cH^n(F) = 0$ for $n > 0$;
    \item The full subcategory $\Shv_{D(\ZZ)}(X)_{\le 0} \subset \Shv_{D(\ZZ)}(X)$ spanned by those $F \in \Shv_{D(\ZZ)}(X)_{}$ such that $\tau_{\ge 0}F \in \Shv(X)$ is a discrete object.
\end{enumerate}{}
 Then $(\Shv_{D(\ZZ)}(X)_{\ge 0}, \Shv_{D(\ZZ)}(X)_{\le 0})$
defines a right complete $t$-structure on $\Shv_{D(\ZZ)}(X)$.
\end{proposition}
\begin{proof}
Follows in a way entirely similar to \cite[Prop.~1.3.2.7]{luriespectral}. 
\end{proof}

\begin{definition}\label{discrete-structure-sheaf}
In \cref{defofaffinestack}, we constructed a colimit preserving functor $\Spec \colon \mathrm{DAlg}^{\mathrm{ccn}}_A \to \Shv(A)^{\op}$.
By the adjoint functor theorem, $\Spec$ is left adjoint to a limit preserving functor $\Shv(A)^{\mathrm{op}} \to \mathrm{DAlg}^{\mathrm{ccn}}_A$.
Concretely, this functor is determined by sending an affine scheme $\Spec B$ to $B$.
For $X \in \Shv(A)$, it restricts to a limit preserving functor $\cO_X \colon \Shv(X) \to \mathrm{DAlg}^{\mathrm{ccn}}_A$, which can be viewed as an object of $\Shv_{D(\ZZ)}(X)$;
if no confusion is likely to arise, we simply denote it by $\cO$.
It follows from the construction that $\cO_X$ lies in the heart of the $t$-structure from \cref{t-structure-sheaves}.
Further, $\cO_X$ has the structure of a commutative algebra object (\cite[\S~2.1.3]{luriehigher}) of $\Shv_{D(\ZZ)}(X)$.
\end{definition}

\begin{definition}
The $\infty$-category $\Mod(\cO_X)$ is the category $\Mod_{\cO_X}(\Shv_{D(\ZZ)}(X))$ of modules over the commutative algebra object $\cO_X$ of $\Shv_{D(\ZZ)}(X)$ (\cite[\S~4.5.1]{luriehigher}).
\end{definition}

\begin{lemma}\label{t-structure-Mod}
We consider the following two full subcategories of $\Mod(\cO_X)$:
\begin{enumerate}
    \item The full subcategory $\Mod(\cO_X)_{\ge 0} \subset \Mod(\cO_X)$ spanned by those $F \in \Mod(\cO_X)$ such that $\cH^n(F) = 0$ for $n > 0$;
    \item The full subcategory $\Mod(\cO_X)_{\le 0} \subset \Mod(\cO_X)$ spanned by those $F \in \Mod(\cO_X)$ such that the underlying $D(\ZZ)$-valued sheaf is an object of $\Shv_{D(\ZZ)}(X)_{\le 0}$.
\end{enumerate}
Then $(\Mod(\cO_X)_{\ge 0}, \Mod(\cO_X)_{\le 0})$ defines a right complete $t$-structure on $\Mod(\cO_X)$.
\end{lemma}
\begin{proof}
Follows in a way entirely similar to \cite[Prop.~2.1.1.1]{luriespectral}.
\end{proof}
Next, we define the derived category of quasi-coherent sheaves;
\textit{cf.} also \cite[\S~4.5]{lurie2011quasi} for a similar account in a slightly different setting.
\begin{definition}\label{defqcoh}
Let $X \in \Shv(A)$.
One defines the \emph{derived $\infty$-category of quasi-coherent sheaves} on $X$ to be $D_{\qc}(X) \colonequals \varprojlim_{\Spec T \to X} D(T)$, where the limit ranges over all maps from affine schemes $\Spec T$ to $X$ (over $A$) and $D(T)$ denotes the derived $\infty$-category of $T$-modules.
\end{definition}{}

\begin{remark}\label{sad1}
Note that if $X^\wedge \in \Shv(A)^\wedge$ denotes the hypercompletion of $X \in \Shv(A)$, we have a natural equivalence of categories $D_{\qc}(X^\wedge) \simeq D_{\qc}(X)$.
More generally, \cref{defqcoh} makes sense for any $X \in \PShv(A)$ and there is an equivalence $D_{\qc}(X) \to D_{\qc}(X^\wedge)$ where $X^\wedge \in \Shv(A)^\wedge$ is the hypercompletion of $X$;
see \cite[Rem.~6.2.3.3]{luriespectral}.
We use the notion of derived $\infty$-categories associated with presheaves in only one place in this paper in \cref{heartrep}.
\end{remark}{}

\begin{remark}\label{happy}
Let $X$ be a weakly affine stack over a field $k$.
Let $X_{\bullet} \colonequals (\Spec A_n)$ be a simplicial affine scheme such that $X \simeq \varinjlim_{[n]\in \Delta} \Spec A_n$ in $\Shv(k)^\wedge$.
Then it follows from \cref{sad1} that $D_{\qc}(X) \simeq \varprojlim_{[n] \in \Delta} D(A_n)$.
In particular, $D_{\qc}(X)$ is a presentable stable $\infty$-category;
\textit{cf.} \cite[Prop.~6.2.3.4]{luriespectral}.
\end{remark}
Note that the category $D_{\qc}(X)$ has all small colimits and for any map $f \colon X \to X'$ in $\Shv(A)$, there is a pullback functor $f^* \colon D_{\qc}(X') \to D_{\qc}(X)$ that preserves small colimits \cite[Prop.~6.2.3.4]{luriespectral}.
Moreover, for each affine scheme $\Spec T$, faithfully flat descent gives a natural fully faithful functor $\lambda_T \colon D(T) \simeq D_\qc(\Spec T) \to \Mod(\cO_{\Spec T})$ that preserves all small colimits.
Since $X \mapsto \Mod(\cO_X)$ is the right Kan extension of its restrictions along the Yoneda embedding $\Aff_A \hookrightarrow \Shv(A)$ (\textit{cf.} \cite[Rem.~2.1.0.5]{luriespectral}) and essentially by definition so is $X \mapsto D_\qc(X)$ (\textit{cf.}~\cite[Prop.~6.2.1.9]{luriespectral}), one obtains a natural colimit preserving, fully faithful functor $\lambda_X \colon D_{\qc}(X) \to \Mod(\cO_X)$ for all $X \in \Shv(A)$. Roughly speaking, this means that an object of $D_\qc(X)$ can be viewed as an fpqc sheaf of $\cO_X$-modules, in analogy with the classical situation.
\begin{definition}\label{qc-t-structure}
Let $X$ be a pointed connected stack over a field $k$.
\begin{enumerate}
    \item Let $D_{\qc}(X)_{\ge 0} \subset D_\qc(X)$ be the full subcategory spanned by those $K \in D_{\qc}(X)$ such that $\lambda_X(K) \in \Mod(\cO_X)_{\ge 0}$.
    \item Let $D_{\qc}(X)_{\le 0} \subset D_\qc(X)$ be the full subcategory spanned by those $K \in D_{\qc}(X)$ such that $\lambda_X(K) \in \w{Mod}(\cO_X)_{\le 0}$.
\end{enumerate}
\end{definition}
In \cref{switz2}, we show that \cref{qc-t-structure} actually defines a $t$-structure; note that this is not formal and we crucially use the assumption that $X$ is a pointed connected stack over a field $k$;
that is, there is an fpqc effective epimorphism $\Spec k \to X$.
\begin{proposition}\label{switz1}
Let $X$ be a pointed connected stack over a field $k$ and $K \in D_\qc(X)$.
The following are equivalent:
\begin{enumerate}
    \item\label{basel} $K \in D_{\qc}(X)_{\ge 0}$ (resp.\ $K \in D_{\qc}(X)_{\le 0}$).
    \item\label{zurich} $u^*K \in D(k)_{\ge 0}$ (resp.\ $u^*K \in D(k)_{\le 0}$) for the fpqc effective epimorphism $u \colon \Spec k \to X$.
    \item\label{bern} $v^*K \in D(T)_{\ge 0}$ (resp.\ $v^*K \in D(T)_{\le 0}$) for all maps $v \colon \Spec T \to X$.
\end{enumerate}
\end{proposition}
\begin{proof}
Note that for any $L \in D(k)$, we have $L \in D(k)_{\ge 0}$ (resp.\ $L \in D(k)_{\le 0}$), if and only if $\lambda_k(L) \in \Mod(\cO_{\Spec k})_{\ge 0}$ (resp.\ $\lambda_k(L) \in \Mod(\cO_{\Spec k})_{\le 0}$).
Since $u$ is an effective epimorphism, one can use the definition of $\Mod(\cO_X)_{\ge 0}$ via cohomology sheaves (resp. use \cite[Prop.~6.2.3.17]{Lur09}, for example) to obtain the equivalence of $\cref{basel}$ and $\cref{zurich}$.
\vspace{2mm}

The implication $\cref{bern}$$\implies$$\cref{zurich}$ is clear; for the implication  $\cref{zurich}$$\implies$$\cref{bern}$, note that since $u \colon \Spec k \to X$ is an fpqc effective epimorphism, any map $v \colon \Spec T \to X$ admits a faithfully flat cover $\Spec T' \to \Spec T$ which fits into the square
\[ \begin{tikzcd}
\Spec T' \arrow[d] \arrow[r] & \Spec k \arrow[d,"u"] \\
\Spec T \arrow[r,"v"] & X.                        
\end{tikzcd} \]
Then $(u^*K) \otimes_k T' \simeq (v^*K) \otimes_T T'$.
This shows that $u^*K \in D(k)_{\ge 0}$ (resp.\ $u^*K \in D(k)_{\le 0}$) if and only if $v^*K \in D(T)_{\ge 0}$ (resp.\ $v^*K \in D(T)_{\le 0}$) because both $T \to T'$ and $k \to T'$ are faithfully flat.
\end{proof}
\begin{proposition}\label{switz2}
Let $X$ be a pointed connected stack over a field $k$.
Then the subcategories $(D_{\qc}(X)_{\ge 0}, D_{\qc}(X)_{\le 0})$ form a right complete $t$-structure on $D_{\qc}(X)$.
\end{proposition}{}
\begin{proof}
Since $(\Mod(\cO_X)_{\ge 0}, \Mod(\cO_X)_{\le 0})$ is a right complete $t$-structure and $\lambda_X$ is a fully faithful functor between stable $\infty$-categories which preserves small (co)limits, it suffices to show the following:
given $K \in D_\qc(X)$, the coconnective truncation of $\lambda_X(K)$ lies in the essential image of $\lambda_X$.
\vspace{2mm}

To analyze the coconnective truncation of $K$, we will prove that for any maps $v \colon \Spec T \to X$ and $a \colon \Spec S \to \Spec T$, the natural morphism
\begin{equation}\label{coconn}
(\tau_{\le 0}(v^*K)) \otimes_T S \longrightarrow \tau_{\le 0}((v \circ a)^* K)
\end{equation}
in $D(S)$ is an equivalence.
Once we show this, we naturally obtain an object of $D_\qc(X)$ denoted as $ \tau_{\le 0}K$, which is essentially uniquely determined by the requirement that for any map $v \colon \Spec T \to X$, we have $v^* (\tau_{\le 0} K) \simeq \tau_{\le 0} (v^* K)$. Since $\lambda_X(\tau_{\le 0} K)$ is coconnective, we obtain a natural map $\tau_{\le 0}\lambda_X(K) \to \lambda_X(\tau_{\le 0}K)$.
This natural map is an isomorphism because the functor $\Mod(\cO_X) \to \Shv_{D(\ZZ)}(X)$ is conservative and the coconnective truncation in $\Shv_{D(\ZZ)}(X)$ is given by applying the coconnective truncation in $\PShv_{D(\ZZ)}(X)$ (which is induced from the coconnective truncation on $D(\ZZ)$) and sheafifying.
Therefore, the coconnective truncation of $\lambda_X(K)$ lies in the essential image of $\lambda_X$, as desired.
\vspace{2mm}

It remains to prove that \cref{coconn} is an isomorphism.
For this, we can proceed similarly to the proof of \cref{switz1} using the fpqc effective epimorphism $u \colon \Spec k \to X$.
The map $v$ admits a faithfully flat cover $u' \colon \Spec T' \to \Spec T$ which fits into the diagram
\begin{equation}\label{trunc19}
\begin{tikzcd}
\Spec S' \arrow[r] \arrow[d] & \Spec T' \arrow[d] \arrow[r] & \Spec k \arrow[d] \\
\Spec S \arrow[r] & \Spec T \arrow[r] & X               \end{tikzcd} 
\end{equation}
where $S' \colonequals T' \otimes_T S$.
\vspace{2mm}

By the faithful flatness of $S \to S'$, it suffices to show that \begin{equation}\label{tru67}
    (\tau_{\le 0}(v^*K)) \otimes_T S' \simeq (\tau_{\le 0}((v \circ a)^* K)) \otimes_S S'.
\end{equation} Note that since $k$ is a field, the maps $k \to T'$ and $k \to S'$ are both flat. Since tensoring along flat maps is $t$-exact, by chasing through \cref{trunc19}, one directly verifies that the source and target of \cref{tru67} are both naturally isomorphic to $(\tau_{\le 0}(u^*K)) \otimes_k S'$.
This finishes the proof.
\end{proof}
\begin{remark}
By the construction of the $t$-structure in \cref{qc-t-structure}, the natural functor $\lambda_X \colon D_\qc(X) \to \Mod({\cO_X})$ is $t$-exact.
\end{remark}{}
\begin{definition}
For a pointed connected stack $X$ over a field $k$, we let $\QCoh(X) \colonequals D_\qc(X)^\heartsuit$ be the heart of the $t$-structure from \cref{switz2}.
\end{definition}{}

\begin{remark}\label{flatheart}
    Let $X$ be a pointed connected stack over a field $k$.
    Using faithfully flat descent (see the proof of \cref{switz1}), it follows that if $K \in \mathrm{QCoh}(X)$, and $v \colon \Spec T \to X$ is any map, then $v^* K$ corresponds to a discrete \emph{flat} module over $T$.
\end{remark}{}

\begin{proposition}\label{heartrep}
Let $X$ be a pointed connected stack over a field $k$.
Then the truncation $X \to \tau_{\le 1}X$ induces an equivalence of categories 
$$ \QCoh(\tau_{\le 1}X) \simeq \QCoh(X).$$
\end{proposition}{}

\begin{proof}
The stack $X$ is classified by a left fibration $\mathcal{X} \to \Alg_k$ under the straightening equivalence (which is an $\infty$-categorical analogue of the Grothendieck construction, see e.g., \cite[\S~20]{joyal}, \cite[\S 2.2]{Lur09}). Let $\Mod^\heartsuit$ denote the category whose objects are pairs 
$(R, M)$ where $R$ is a $k$-algebra and $M \in D(R) ^\heartsuit$.
A morphism $(R_1,M_1) \to (R_2, M_2)$ is given by a pair of morphisms $f \colon R_1 \to R_2$ and $g \colon  M_1 \otimes_{R_1} R_2 \to M_2$ in the symmetric monoidal category $D(R_2)^\heartsuit$.
This way, we obtain a natural left fibration $q \colon \Mod^{\heartsuit} \to \Alg_k$. We note that $\Mod^\heartsuit$ is a $1$-category. Let $\mathrm{flMod}^{\heartsuit} \subset \mathrm{Mod}^{\heartsuit}$ denote the full subcategory spanned by $(R,M) \in \mathrm{Mod}^{\heartsuit}$ such that $M$ is a flat module over $R$. By restriction, we obtain a left fibration $q' \colon \mathrm{flMod}^{\heartsuit} \to \mathrm{Alg}_k$.
\vspace{2mm}

By \cref{flatheart}, $\QCoh(X)$ is equivalent to the full subcategory of $\w{Fun}_{\w{Alg}_k}(\mathcal{X}, \mathrm{flMod}^\heartsuit)$ spanned by those functors that send every morphism of $\mathcal{X}$ to a $q'$-coCartesian morphism.
\vspace{2mm}

Let $X' \colon \Alg_k \to \mathcal{S}$ be the composition of the functor $X \colon \Alg_k \to \mathcal{S}$ with $\tau_{\le 1} \colon \mathcal{S} \to \mathcal{S}$.
We denote the left fibration that classifies $X'$ by $\mathcal{X}' \to \Alg_k$.
Since $\mathrm{flMod}^\heartsuit$ is a $1$-category and the natural map $\cX \to \cX'$ induces an equivalence on homotopy categories, we have
\[ \w{Fun}_{\w{Alg}_k}(\mathcal{X'}, \mathrm{flMod}^\heartsuit) \simeq \w{Fun}_{\w{Alg}_k}(\mathcal{X}, \mathrm{flMod}^\heartsuit). \]
One can identify the full subcategory of $\w{Fun}_{\w{Alg}_k}(\mathcal{X'}, \mathrm{flMod}^\heartsuit)$ spanned by those functors that send every morphism in $\cX'$ to a $q'$-coCartesian morphism with a full subcategory of $D_{\qc}(X')\colonequals \varprojlim_{\Spec T \to X'} D(T)$ (\textit{cf.}~\cref{sad1}). Let us denote this full subcategory by $D_{\qc}^{\mathrm{fl}, \heartsuit}(X')$. Note that there is a natural map $X' \to \tau_{\le 1} X$ which induces an isomorphism $D_{\qc}(\tau_{\le 1}X) \simeq D_{\qc}(X')$ by \cref{sad1}. One observes that under this equivalence, the full subcategory $D_{\qc}^{\mathrm{fl}, \heartsuit}(X') \subset D_{\qc}(X')$ identifies with $\QCoh(\tau_{\le 1}X)$.
This proves the desired equivalence of categories
\[ \QCoh (\tau_{\le 1}X) \simeq \QCoh(X). \qedhere \]
\end{proof}{}
\begin{remark}\label{chicago14}
As a consequence of \cref{heartrep}, the set of global sections of any quasi-coherent sheaf on a pointed connected stack $X$ can be computed on $\tau_{\le 1}X$.
\end{remark}{}
\begin{remark}
Let $X$ be a pointed connected stack such that $D_{\qc}(X)$ is a presentable stable $\infty$-category. Then the functor $\lambda_X \colon D_{\qc}(X) \to \Mod({\cO_X})$ is colimit preserving and therefore admits a right adjoint which one may call ``coherator" and denote by 
$$Q_X \colon \w{Mod}(\cO_X) \to D_{\qc}(X).$$
As a right adjoint to a $t$-exact functor, $Q_X$ is left $t$-exact.
\end{remark}{}

 \begin{construction}\label{jap22}
If $B$ is any $E_\infty$-ring, one can consider the $\infty$-category $\w{Mod}_B$ of $B$-modules. Then $\w{Mod}_B$ has a $t$-structure such that $(\w{Mod}_B)_{\ge 0}$ is generated by $B$ under extensions and small colimits (see \cite[Prop.~1.4.4.11]{luriehigher}). By construction (see e.g., \cite[Rmk~1.2.1.3]{luriehigher}), under this $t$-structure, an object $M \in (\w{Mod}_B)_{\le 0}$ if and only if $\mathrm{Map}_{\mathrm{Mod}_B}(B, M[-1] ) \simeq 0$, i.e., the underlying spectrum of $M$ is coconnective. This implies that the $t$-structure on $\w{Mod}_B$ is right complete.
 \end{construction}{}

\begin{construction}\label{sad2}
Let $X$ be a pointed connected stack over a field $k$.
Every map of stacks $\Spec T \to X$ induces a map $R\Gamma(X, \cO) \to T$ in $\mathrm{DAlg}^{\mathrm{ccn}}_k$. This defines a pullback functor $\w{Mod}_{R\Gamma(X, \cO)} \to D(T)$ which further gives rise to a functor 

$$U^* \colon \w{Mod}_{R\Gamma(X, \cO)} \to \varprojlim_{\Spec T \to X}D(T) \simeq D_{\qc}(X).$$

For $M \in \w{Mod}_{R\Gamma(X, \cO)}$, we will denote the image of $M$ under the above functor by $\widetilde{M} \in D_{\qc}(X)$.
Note that $\widetilde{R\Gamma(X, \cO)} = \cO_X$ under the above construction. The functor $U^*$ preserves all small colimits. By the adjoint functor theorem, $U^*$ must have a right adjoint $U_* \colon D_{\qc}(X) \to \w{Mod}_{R\Gamma(X, \cO)}$ which identifies with the functor $\mathscr{F} \mapsto R\Gamma(X, \mathscr{F})$, when the latter is regarded as an $R\Gamma(X, \cO)$-module. Additionally, we point out that $U^*$ is a symmetric monoidal functor for the natural symmetric monoidal structures on $\Mod_{R\Gamma(X, \cO)}$ and $D_{\w{qc}}(X)$.
\end{construction}{}

\begin{proposition}
In the above set up, the functor $$U^* \colon \Mod_{R\Gamma(X, \cO)} \to D_{\qc}(X)$$ is $t$-exact.
\end{proposition}{}

\begin{proof}
Since $U^*$ preserves small colimits and $U^*(R\Gamma(X, \mathscr O)) = \cO$, it follows from the construction of the $t$-structure on $\w{Mod}_{R\Gamma(X, \cO)}$ that $U^*$ is right $t$-exact. The left $t$-exactness of $U^*$ follows from the following lemma, which finishes the proof of the proposition.
\end{proof}
\begin{lemma}[{\cite[Cor.~4.1.12]{lurie2011quasi}}]\label{lurielemma}
Let $A \to B$ be a map of $E_\infty$-rings over $k$. Moreover, let us assume that $A$ (resp. $B$) has the property that the spectrum underlying $A$ (resp. $B$) is coconnective and the structure map $k \to A$ (resp. $k \to B$) induces isomorphism $k \xrightarrow{\sim} \pi_0(A)$ (resp. $k \xrightarrow{\sim} \pi_0 (B) $). Let $M$ be a left $A$-module such that $\pi_i(M) = 0$ for $i>0$. Then the homotopy groups $\pi_i (B \otimes_A{M})$ vanish for $i>0$ and $\pi_0(M) \to \pi_0(B \otimes_{A}M)$ is injective.
\end{lemma}{}

\begin{lemma}\label{filco}
Let $X$ be a pointed connected weakly affine stack over a field $k$. The functor $$U_* \colon D_{\qc}(X) \to \Mod_{R\Gamma(X, \cO)}$$ from \cref{sad2} is left $t$-exact and therefore induces a functor denoted as $$ U^{\mathrm{ccn}}_* \colon D_{\qc}(X)_{\le 0} \to (\Mod_{R\Gamma(X, \cO)})_{\le 0}$$ which preserves filtered colimits.
\end{lemma}{}

\begin{proof}
Since $X$ is weakly affine, we can choose a simplicial affine scheme $X_{\bullet} \colonequals (\Spec A_{n})$ such that $X \simeq \varinjlim_{[n]\in \Delta} \Spec A_n$ in $\Shv(\Alg_A)^\wedge$. Further, by \cref{sad1}, it follows that for any $\mathscr{F} \in D_{\qc}(X)_{\le 0}$, we have $$U^{\mathrm{ccn}}_*(\mathscr{F}) \simeq R\Gamma(X, \mathscr{F}) \simeq \varprojlim_{[n] \in \Delta}R\Gamma (\Spec A_n, \mathscr F).$$ Therefore, we will be done by the following well-known lemma.
\begin{lemma}
Let us consider the $\infty$-category $\mathrm{Fun}(\Delta, D(\mathbf{Z})_{\le 0})$, where $\Delta$ denotes the simplex category. Let $I$ be a filtered category and $T \colon I \to \mathrm{Fun}(\Delta, D(\mathbf{Z})_{\le 0})$ be a functor. Then $\varinjlim_{i \in I}\varprojlim_{[n] \in \Delta} T(i) \simeq \varprojlim_{[n]\in \Delta} \varinjlim_{i \in I} T(i)$.
\end{lemma}{}
\begin{proof}Follows from e.g.,~\cite[Cor.~4.3.7]{lurie2011quasi}.
\end{proof}{}
This finishes the proof of \cref{filco}.
\end{proof}{}

Finally, we are ready to prove our main result of this section, which gives a natural description of the derived category of quasi-coherent sheaves on pointed connected weakly affine stacks. The result below will be used in the proof of \cref{sohard139}.

\begin{proposition}\label{makiman34}
Let $X$ be a pointed connected weakly affine stack over a field $k$. Then $U^*$ from \cref{sad2} induces an equivalence of categories 
$$U^*_{\mathrm{ccn}} \colon (\Mod_{R\Gamma(X, \cO)})_{\le 0} \simeq D_{\qc}(X)_{\le 0}.$$

\end{proposition}{}
\begin{proof}
By construction, the functors $U^*_{\mathrm{ccn}}$ and $U^{\mathrm{ccn}}_*$ forms an adjoint pair where $U^*_{\mathrm{ccn}}$ is the left adjoint. 
\vspace{2mm}

First, we show that the functor $U^*_{\mathrm{ccn}}$ is conservative. In order to see this, let $v \colon \Spec k \to X$ denote the map from the point which is an fpqc effective epimorphism. It would be enough to show that if $M \in (\Mod_{R\Gamma(X, \cO)})_{\le 0}$ is such that $v^* (U^*(M)) \simeq 0$, then $M \simeq 0$. This amounts to showing that if $M \otimes_{R\Gamma(X, \cO)} k \simeq 0$ for some $M \in (\Mod_{R\Gamma(X, \cO)})_{\le 0}$, then $M \simeq 0$. If that were false, one can find some $N \in (\Mod_{R\Gamma(X, \cO)})_{\le 0}$ such that $\pi_0 (N) \ne 0$ and $N \otimes_{R\Gamma(X, \cO)} k \simeq 0$. However, since $X$ is a pointed and connected stack, $R\Gamma(X, \cO)$ is naturally augmented and satisfies the assumptions in \cref{lurielemma}; this gives a contradiction.
\vspace{2mm}

Since $U^*_{\mathrm{ccn}}$ is conservative, to prove that $U^*_{\mathrm{ccn}}$ is an equivalence, it would be enough to prove that the counit $U^*_{\mathrm{ccn}}U^{\mathrm{ccn}}_* \to \mathrm{id}$ is an isomorphism. Since $D_{\qc}(X)$ is right complete and $U^{\mathrm{ccn}}_*$ preserves filtered colimits by \cref{filco}, it would be enough to prove that for $\mathscr{F} \in \QCoh(X)$, we have $U^*_{\mathrm{ccn}}U^{\mathrm{ccn}}_* (\mathscr{F}) \simeq \mathscr{F}$. By \cref{heartrep}, we have an equivalence $\QCoh(X) \simeq \QCoh(\tau_{\le 1} X)$. Since $X$ is a pointed connected weakly affine stack, $\tau_{\le 1} X \simeq B \pi_1(X)$, where $\pi_1(X)$ is a unipotent affine group scheme over $k$. Therefore, one can identify $\mathscr{F}$ with a representation of $\pi_1(X)$. Since $\pi_1(X)$ is unipotent, any finite-dimensional representation has a filtration where the graded pieces correspond to the trivial representation. Since $U^{\mathrm{ccn}}_*$ preserves filtered colimits, in order to check $U^*_{\mathrm{ccn}}U^{\mathrm{ccn}}_*(\mathscr{F}) \simeq \mathscr{F}$, we can assume that $\mathscr{F}$ comes from a finite-dimensional representation of $\pi_1(X)$ (see \cite[\S~3.3]{waterhouse}). Therefore, by the unipotence of $\pi_1(X)$, we reduce to checking the same when $\mathscr{F} = \cO$, which is clear.
\end{proof}{}

\begin{corollary}\label{iy2}
    Let $X$ be a pointed connected weakly affine stack over $k$. Let $M, N \in D_{\mathrm{qc}}(X)_{\le 0}$. Then we have a natural isomorphism $R\Gamma (X, M) \otimes_{R\Gamma(X, \cO)} R\Gamma(X, N) \xrightarrow[]{\sim} R\Gamma(X, M \otimes N)$.
\end{corollary}{}

\begin{proof}
    Follows in a manner similar to the proof of \cref{makiman34} by reducing to the case $N = \cO$.
\end{proof}{}

\begin{remark}
\cref{makiman34} extends a result of Lurie (\cite[Prop.~4.5.2.(7)]{lurie2011quasi}) for affine stacks in characteristic $0$ to weakly affine stacks in any characteristic. Note that $D_{\w{qc}}(X)$ is always left complete; however, $\Mod_{R\Gamma (X, \cO)}$ (see \cref{jap22}) need not be left complete. For a concrete example, one may take $X$ to be $B \GG_a$ over a perfect field of characteristic $p>0$, see \cite[Prop.~3.1]{hallrenam} and \cite[Rmk.~3.0.7]{Monrec}. This shows that \cref{makiman34} does not in general extend to an equivalence $\Mod_{R\Gamma(X, \cO)} \simeq D_{\w{qc}}(X)$ of $\infty$-categories.
\end{remark}{}

\begin{remark}
Our formulation of \cref{weaklyaffine} only applies to pointed connected stacks: 
even though the conditions \cref{weaklyaffine-simplicial} and \cref{weaklyaffine-unipotent} in \cref{weaklyaffine} hold for any separated scheme $X$, the definition does not apply (unless $X \simeq \Spec k$), and \cref{makiman34} need not hold. For a concrete example, one can take $X = \PP^1_k$. It would be interesting to isolate a more general class of stacks (containing e.g., affine schemes and pointed connected weakly affine stacks) for which \cref{makiman34} holds.
\end{remark}{}

\begin{example}\label{exaqmpleintroc}
Since $K(\mathbf{G}_m,n)$ is weakly affine for $n \ge 2$, by using \cref{makiman34}, one can conclude that $D_{\mathrm{qc}}(K(\mathbf{G}_m,n)) \simeq D(k)$ for $n \ge 2$. This follows from the fact that $R\Gamma (K(\mathbf{G}_m,n), \cO) \simeq k$ for $n \ge 1$. 
\end{example}{}

\begin{example}
Let $k$ be a perfect field of characteristic $p>0$. Let $W[F]$ denote the group scheme over $k$ that is obtained by considering kernel of the Frobenius operator $F$ on the ring scheme $W$ of $p$-typical Witt vectors. In this case, \cref{makiman34} extends to give an isomorphism $$\Mod_{R\Gamma(B W[F], \cO)} \simeq D_{\w{qc}}(B W[F]).$$ The latter claim follows from \cite[Lem.~3.0.3]{Monrec}; note that (by \cite[Lem. 3.2.6]{Dri20} or \cite[Prop.~2.4.10]{Monda}) $W[F] \simeq \GG_a^\sharp$ in \textit{loc.~cit}. In this case, one further has an isomorphism $D_{\w{qc}}(BW[F]) \simeq D_{\w{qc}}(\widehat{\mathbf G}_a)$, which can be seen by a suitable version of Cartier duality (\textit{cf.}~\cref{dualizinggpsch}). See \cite[\S 3.5]{BL} for a variant of the latter isomorphism and its application in $p$-adic Hodge theory.
\end{example}{}

\begin{example}
We will use \cref{makiman34} to extend the isomorphism $D_{\w{qc}}(BW[F]) \simeq D_{\w{qc}}(\widehat{\mathbf G}_a)$ to the case of the higher stacks $K(W[F], n)$ for $n>1$. Instead of using Cartier duality as in the case of $n=1$, we will appeal to Koszul duality for $n>1$ (see \cite[\S~14.1, Prop.~14.1.3.2]{luriespectral}). To this end, by \cref{makiman34}, $D_{\w{qc}}(K(W[F], n))$ is equivalent to the left completion of $\Mod_{R\Gamma(K(W[F], n), \cO)}$. One computes directly that the Koszul dual of $R\Gamma (K(W[F],n), \cO)$ is $\Sym k[n-1]$, where $\Sym k[n-1]$ is the free animated ring on a generator in degree $(n-1)$. By Koszul duality, we obtain an equivalence $D_{\w{qc}} (K(W[F], n)) \simeq \Mod_{\Sym k[n-1]}$. For $n=2$, the latter fact was proven before in \cite[\S~4.2]{moulinos2019universal} using different arguments.
\end{example}{}

\subsection{Cohomological Eilenberg--Moore spectral sequence for higher stacks}\label{sec2.3}
Now we will establish some results on cohomology and base change in the context of higher stacks.
More precisely, given a pullback square of higher stacks
\begin{center}
 \begin{tikzcd}
Y' \arrow[r, "f'"] \arrow[d, "g'"] & X' \arrow[d, "g"] \\
Y \arrow[r, "f"]                   & X               
\end{tikzcd}
\end{center}
over a field $k$, we will investigate when the natural map
\begin{equation}\label{makiman1}
    R\Gamma(Y,\cO)\otimes_{R\Gamma(X, \cO)} R\Gamma(X', \cO) \to R\Gamma(Y', \cO)
\end{equation}
is an isomorphism of $E_\infty$-algebras.
Whenever \cref{makiman1} is an isomorphism, we obtain an Eilenberg--Moore type spectral sequence $E_2^{p,q}= \mathrm{Tor}_p ^{H^* (X, \cO)}(H^* (Y, \cO), H^* (X', \cO))_q \implies H^* (Y', \cO)$; see, e.g., \cite[Prop.~7.2.1.19]{luriehigher}.
\vspace{2mm}

Note that \cref{makiman1} is only likely to be an isomorphism when one carefully imposes certain conditions on the maps and the stacks involved in the above pullback square.
For example, the desired isomorphism is already false when $X = \Spec k$ is a point, $X' = \AA^1_k$ and $Y$ is a scheme given by an infinite disjoint union of $\Spec\, k$.
Even if $X'$ and $Y$ are both points, but $X = B\GG_m$, the map \cref{makiman1} cannot be an isomorphism since $R\Gamma(B\GG_m, \cO) \simeq k$. In fact, the following example shows that \cref{makiman1} fails to be an isomorphism even when all the stacks involved are smooth proper schemes over $k$.
\begin{example}
Let $\pi \colon X \to \mathbf{P}^1$ be an elliptic $K3$ surface over an algebraically closed field $k$, i.e., $\pi$ is a flat surjection from a K3 surface $X$ and there exists a closed point $t \colon \Spec k \to \mathbf{P}^1$ such that the fibre $X_t$ over $t$ is a smooth integral curve of genus $1$.
In this case, the map \cref{makiman1} becomes $R\Gamma(X, \cO) \to R\Gamma (X_t, \cO)$, which cannot be an isomorphism because $H^2(X,\cO) \neq 0$.
\end{example}{}
In \cref{iy3} below, we give a general criterion for when the map \cref{makiman1} is an isomorphism in the world of pointed connected stacks.
In the context of pointed connected affine stacks, we reformulate this criterion in \cref{iy5}.
Note that in this context, our assumptions are slightly more general than those of \cite[Prop.~3.10]{Bzvi} and \cite[Prop.~9.1.5.7]{luriespectral} on certain related Beck--Chevalley transformations being isomorphisms (see also \cite[Prop.~A.1.5]{harper}).
\begin{proposition}\label{iy3}
    Let $X$ be a pointed connected weakly affine stack over $k$.
    Let $X',Y$ be pointed connected (higher) stacks.
    Consider a pullback diagram (of pointed stacks) of the form
\begin{center}
 \begin{tikzcd}
Y' \arrow[r, "f'"] \arrow[d, "g'"] & X' \arrow[d, "g"] \\
Y \arrow[r, "f"]                   & X              
\end{tikzcd}
\end{center}
such that $f^*g_* \cO \simeq g'_* f^{\prime *} \cO$ in $D_{\mathrm{qc}}(Y)$ and either one of the following two properties holds:
\begin{enumerate}
    \item\label{iy3-colim} $f_*$ preserves filtered colimits when viewed as a functor $f_* \colon \QCoh(Y) \to D_{\mathrm{qc}}(X)_{\le 0}$.
    \item\label{iy3-fin} For each $i \ge 0$, $\mathscr{H}^i (g_* \cO) \in \QCoh(X)$ corresponds to a finite-dimensional representation of $\pi_1(X)$ under the equivalence of \cref{heartrep}.
\end{enumerate}{}
Then the natural map $$R\Gamma(Y,\cO)\otimes_{R\Gamma(X, \cO)} R\Gamma(X', \cO) \to R\Gamma(Y', \cO)$$ is an isomorphism.
\end{proposition}{}
\begin{proof}
 Let us set $K \coloneqq g_* \mathscr{O} \in {D_{\qc}(X)}_{\le 0}$. We will start by proving that $K$ satisfies the projection formula
\begin{equation}\label{iy1} f_* \mathscr{O} \otimes{K} \simeq f_* f^* K .
\end{equation}More generally, let $M \in D_{\qc}(X)$ and consider the natural map
$$\theta_M \colon f_* \mathscr{O} \otimes{M} \to f_* f^* M.$$ When $M = \mathscr{O}$, the map $\theta_M$ is clearly an isomorphism. We claim that $\varphi_M$ is an isomorphism when $M = \mathscr{H}^i (g_* \cO)$. Note that $\mathscr{H}^i (g_* \cO)$ corresponds to a representation of the unipotent affine group scheme $\pi_1(X)$. Since every such representation is a filtered colimit of finite-dimensional representations (see \cite[\S~3.3]{waterhouse}), under the hypothesis of either $(1)$ or $(2)$ in the proposition, it would be enough to prove that $\theta_M$ is an isomorphism when $M \in \QCoh(X)$ corresponds to a finite-dimensional representation of $\pi_1(X)$. However, since $X$ is weakly affine, by definition, $\pi_1(X)$ is unipotent. Therefore, $M$ admits a finite filtration where the graded pieces are isomorphic to the structure sheaf $\cO$ (\textit{cf}.~\cite[Prop.~IV.2.2.5]{MR0302656}). Thus the claim follows.
\vspace{2mm}

Now we proceed onto showing that \cref{iy1} is an isomorphism. Let $n \ge 1$ be an integer. Let $K_n \coloneqq \tau_{\ge -n} K$. Since $K_n$ can be expressed as a finite limit of certain shifts of $\mathscr{H}^i (g_* \cO)$, the previous paragraph implies that $\theta_M$ is an isomorphism for $M= K_n$. Further, there is a natural map $K_n \to K$ and let $K^n$ denote the cofibre so that there is a cofibre sequence $K_n \to K \to K^n$.
This induces a cofibre sequence 
$$\w{cofib}\,\theta_{K_n} \to \w{cofib}\,\theta_{K} \to \w{cofib}\,\theta_{K^n}$$ for each $n \ge 1$. Since $\theta_{K_n}$ is already noted to be an isomorphism, one has $\w{cofib}\,\theta_{K_n} \simeq 0$. This implies that there is an isomorphism 
$$ \w{cofib}\,\theta_{K} \simeq \w{cofib}\,\theta_{K^n},$$ for each $n \ge 1$.
\vspace{2mm}

Note that by construction $K^n \in D_{\qc}(X)_{\le -(n+1)}$. Therefore, $f_* f^*K^n \in D_{\qc}(X)_{\le -(n+1)}$. By pulling back along the effective epimorphism $u \colon \Spec k \to X$ (see \cref{switz1}), one verifies that $f_* \mathscr{O} \otimes K^n \in D_{\qc}(X)_{\le - (n+1)}$. This implies that $\w{cofib}\,\theta_K \simeq \w{cofib}\,\theta_{K^n} \in D_{\qc}(X)_{\le -n}$. Since the $t$-structure of $D_{\qc}(X)$ is right complete and $n \ge 1$ was arbitrary, it follows that $\w{cofib}\,\theta_K \simeq 0$. Therefore, we have an isomorphism $\theta_K \colon f_* \mathscr{O} \otimes{K} \simeq f_* f^* K$, as desired in \cref{iy1}.
\vspace{2mm}

By \cref{iy2}, it follows that $R\Gamma(X,f_* \mathscr{O} \otimes {K}) \simeq R\Gamma (Y,\mathscr{O}) \otimes_{R\Gamma(X, \mathscr{O})} R\Gamma(X', \cO)$. Therefore, we have $R\Gamma (Y,\mathscr{O}) \otimes_{R\Gamma(X, \mathscr{O})} R\Gamma(X', \cO) \simeq R\Gamma (X,f_* f^* g_* \cO) \simeq R\Gamma (X,f_* g'_* \cO)$, where the last isomorphism follows from the hypothesis $f^*g_* \cO \simeq g'_* f'^* \cO$. This yields the desired assertion.
\end{proof}{}
\begin{remark}\label{ig90}
In the setup of \cref{iy3}, if $X$ and $Y$ are both assumed to be pointed connected weakly affine stacks, then hypothesis (\ref{iy3-colim}) is automatically satisfied.
This follows from \cref{makiman34}. Therefore, in this case, if we have $f^* g_* \cO \simeq g'_* f'^* \cO$, \cref{iy3} shows that the natural map $R\Gamma(Y,\cO)\otimes_{R\Gamma(X, \cO)} R\Gamma(X', \cO) \to R\Gamma(Y', \cO)$ is automatically an isomorphism.
\end{remark}{}
\begin{corollary}\label{iy5}
Let $g \colon X' \to X$ be a map of pointed connected affine stacks over $k$. 
Let $Y$ be a pointed connected stack.
Consider a pullback diagram (of pointed stacks) of the form
\[ \begin{tikzcd}
Y' \arrow[r, "f'"] \arrow[d, "g'"] & X' \arrow[d, "g"] \\
Y \arrow[r, "f"]                   & X               
\end{tikzcd} \]
such that either one of the following two properties holds:
\begin{enumerate}
    \item $f_*$ preserves filtered colimits when viewed as a functor $f_* \colon \QCoh(Y) \to D_{\mathrm{qc}}(X)_{\le 0}$.
    \item For each $i \ge 0$, $\mathscr{H}^i (g_* \cO) \in \QCoh(X)$ corresponds to a finite-dimensional representation of $\pi_1(X)$ under the equivalence \cref{heartrep}.
\end{enumerate}{}
Then the natural map $$R\Gamma(Y,\cO)\otimes_{R\Gamma(X, \cO)} R\Gamma(X', \cO) \to R\Gamma(Y', \cO)$$ is an isomorphism.    
\end{corollary}{}
\begin{proof}
By \cref{iy3} and \cref{ig90}, we only need to prove that $f^*g_* \cO \simeq g'_* f'^* \cO$ in $D_{\mathrm{qc}}(Y)$.
To this end, it would be enough to show that the $\cO$-module pushforward of the structure sheaf along the map $g \colon X' \to X$ is automatically quasi-coherent. In order to see this, since $X$ is pointed and connected, we may compute the $\cO$-module pushforward of the structure sheaf of $X'$ by restricting to the Čech nerve of the fpqc epimorphism $u \colon \Spec k \to X$. Let us denote $Z \colonequals X' \times _X \Spec k$. Then the $\cO$-module pushforward can be represented by the cosimplicial $\cO$-module  $$\xymatrix{
  R\Gamma(Z, \cO) \ar[r]<1.5pt>\ar[r]<-1.5pt>  &  R\Gamma(\Omega X \times_k Z, \cO)   \ar[r]<3pt>\ar[r]\ar[r]<-3pt> & R\Gamma(\Omega X \times_k \Omega X \times_k Z, \cO)  \cdots\
}$$over the simplicial higher stack
$$\xymatrix{
 \cdots  \Omega X \times_k \Omega X   \ar[r]<3pt>\ar[r]\ar[r]<-3pt>  &  \Omega X  \ar[r]<1.5pt>\ar[r]<-1.5pt> & \Spec k\
}$$ which realizes $X$. Since affine stacks are closed under limits, the stacks $\Omega X^n$ and $\Omega X^n \times_k Z$ are all affine stacks for $n \ge 0$. Further, if $X_1$ and $X_2$ are any two affine stacks over a field $k$, it follows that $R\Gamma (X_1 \times_k X_2, \cO) \simeq R\Gamma (X, \cO) \otimes_k R\Gamma (X, \cO)$. These observations, along with the description of the derived category of quasi-coherent sheaves on $\Omega X$ from \cref{makiman34} and faithfully flat descent imply that the $\cO$-module pushforward is already quasi-coherent. This finishes the proof.
\end{proof}{}

\begin{corollary}\label{iy10}
Let $g \colon X' \to X$ be a map of pointed connected affine stacks over $k$. 
Let $Y$ be a pointed connected weakly affine stack.
Consider a pullback diagram (of pointed stacks) of the form
\begin{center}
 \begin{tikzcd}
Y' \arrow[r, "f'"] \arrow[d, "g'"] & X' \arrow[d, "g"] \\
Y \arrow[r, "f"]                   & X.               
\end{tikzcd}
\end{center}
Then the natural map $$R\Gamma(Y,\cO)\otimes_{R\Gamma(X, \cO)} R\Gamma(X', \cO) \to R\Gamma(Y', \cO)$$ is an isomorphism.  
\end{corollary}

\begin{proof}
Follows from \cref{ig90} and \cref{iy5}.    
\end{proof}{}
\begin{corollary}\label{rent14}
Let $X$ be a pointed connected affine stack over $k$. Let $\Omega X \colonequals \Spec k \times_{X} \Spec k$ denote the loop stack of $X$. Then the natural map $k \otimes_{R\Gamma(X,\cO)}k \to R\Gamma(\Omega X, \cO)$ is an isomorphism of $E_\infty$-algebras.
\end{corollary}{}
Note that the $k \otimes_{R\Gamma(X,\cO)}k$ in \cref{rent14} refers to the pushout of $E_\infty$-algebras (as opposed to pushout in $\mathrm{DAlg}^{\mathrm{ccn}}_k$), so the statement does not simply follow from the fact that the functor $\Spec$ takes colimits in $\mathrm{DAlg}^{\mathrm{ccn}}_k$ to limits.
We will need the following variant in \cref{makiman}.
\begin{proposition}\label{sohard139}
Let $X$ be a pointed connected affine stack over $k$ such that $\pi_1(X)$ is trivial and $H^i(X, \cO)$ is finite-dimensional for all $i$. Let $Y$ be a pointed connected stack.
Then any pullback diagram (of pointed stacks)
\[ \begin{tikzcd}
Y' \arrow[r, "f'"] \arrow[d, "g'"] & * \arrow[d, "g"] \\
Y \arrow[r, "f"]                   & X               
\end{tikzcd} \]
induces an isomorphism 
$$k \otimes_{R\Gamma(X, \cO)}R\Gamma(Y, \cO) \xrightarrow[]{\sim} R\Gamma(Y', \cO)$$ of $E_\infty$-algebras.
\end{proposition}
\begin{proof}
As in the proof of \cref{iy5}, we have $f^*g_* \cO \simeq g'_* f'^* \cO$ in $D_{\mathrm{qc}}(Y)$ because $X$ is an affine stack.
By \cref{iy3}, it would be enough to check that for each $i \ge 0$, $\mathscr{H}^i (g_* \cO)$ corresponds to a finite-dimensional representation of $\pi_1(X)$. Since $\pi_1(X)$ is trivial, $\mathscr{H}^i (g_* \cO) \simeq \mathscr{O}^{\oplus S_i}$ for some indexing set $S_i$. It would be enough to show that $S_i$ is a finite set for each $i \ge 0$. For the sake of contradiction, let us choose $m \ge 0$ such that $m$ is minimal with respect to the property that $S_m$ is infinite. The maps $* \xrightarrow{g} X \xrightarrow{h} *$ gives rise to an $E_2$-spectral sequence $E_2^{p,q}=H^p(X, \mathscr{H}^qg_* \cO) \Rightarrow H^{p+q}(*,\cO)$. Note that $E_2^{p,q} = E_r^{p,q}=0$ for all $r \ge 2$ if either $p<0$ or $q<0$. By the minimality of $m$ and the fact that $H^i (X, \cO)$ is finite-dimensional for all $i$, it follows that $E_2^{p,q}$ is finite-dimensional as a $k$-vector space for all $q<m$. Therefore, $E_r^{p,q}$ is finite-dimensional for all $q<m$ and $r \ge 2$. However, note that $E_2^{0,m}$ is infinite-dimensional by our choice of $m$. Combining the last two observations implies that $E_r^{0,m}$ is infinite-dimensional for all $r \ge 2$.
Since $E_{r}^{0,m} \simeq E_{m+2}^{0,m}$ for $r \ge m+2$, it follows that $E_{\infty}^{0,m}$ is infinite-dimensional as well. The latter observation implies that $H^m(*, \cO)$ is infinite-dimensional, which gives a contradiction.
\end{proof}{}

\newpage

\section{Unipotent homotopy types of schemes}\label{uhtfs}

In this section, we introduce the unipotent homotopy type of a scheme, which is the main object of interest in our paper.

\begin{definition}\label{defuht}
Let $X$ be a higher stack over $\Spec A$. Then there exists a higher stack denoted as $\mathbf{U}(X)$ which is equipped with a map $X \to \mathbf{U}(X)$ that is universal with respect to maps from $X$ to an affine stack. We will call $\mathbf{U}(X)$ the \textit{unipotent homotopy type} of $X$. Such a universal map exists because the inclusion of the $\infty$-category of affine stacks inside the $\infty$-category of higher stacks admits a left adjoint (see \cref{cope}).
\end{definition}{}

\begin{remark}
If $X$ is pointed, then $\UU(X)$ is also naturally pointed by the functoriality of the above construction. By \cref{cope}, it follows that $\UU(X) \simeq \Spec R\Gamma(X, \cO)$, where $R\Gamma(X, \cO)$ is naturally viewed as an object of $\mathrm{DAlg}^{\mathrm{ccn}}_A$.
\end{remark}

\begin{remark}
In \cite[\S~2.3]{Toe}, $\UU(X)$ is denoted by $(X \otimes A)^{\uni}$.
Let us explain our motivation behind calling $\UU(X)$ the unipotent homotopy type of $X$.
Our naming is prompted by the new realization that the nature of information captured by $\UU(X)$ reflects homotopical information about $X$. In \cref{recovernori} below, we make this claim precise by recovering the unipotent fundamental group scheme introduced by Nori from $\UU(X)$. 
\end{remark}

\begin{remark}\label{cold19}
    Let $X$ be an object of the $\infty$-category $\cS$. Then $X$ naturally defines an object in $\PShv(A)$ by considering the constant functor $X \colon \Alg_A \to \mathcal{S}$. After sheafification, we obtain a higher stack that we will denote by $\underline{X}$. We will simply use $\mathbf{U}(X)$ to denote the unipotent homotopy type of $X$. We set $C^* (X, A) \colonequals R\Gamma(\underline{X}, \cO) \in \mathrm{DAlg}^{\mathrm{ccn}}_A$ and call it the singular cochains of $X$ (with values in $A$).
\end{remark}{}

    \subsection{Recovering the unipotent Nori fundamental group scheme}\label{recovernori}
Let $X$ be a scheme of finite type over a field $k$ and $x \in X(k)$.
Assume that $H^0(X,\cO) \simeq k$. We will show that the unipotent homotopy type $\UU(X)$ of $X$ can be used to recover Nori's unipotent fundamental group scheme from \cite{MR682517}. As a concrete consequence, we will see (\cref{maincorro}) that one can write down a \emph{formula} that describes Nori's unipotent fundamental group scheme. We begin by recalling Nori's definition of the fundamental group scheme which makes use of the Tannakian formalism.
\begin{definition}[{\cite[\S~IV.1]{MR682517}}]\label{torsornori}
The \emph{unipotent Nori fundamental group scheme} $\pi^{\mathrm{U,N}}_1(X,x)$ of a pointed scheme $(X,x)$ is the (unipotent) affine group scheme over $k$ associated under Tannaka duality with the Tannakian category consisting of the tensor category of unipotent vector bundles on $X$ and the fibre functor coming from pullback along $x$.
\end{definition}

Using the above construction of $\pi_1^{\mathrm{U,N}}(X,x)$ and \cite[\S~1,~Prop.~2.9]{MR682517}, one obtains a tautological pointed $\pi_1^{\mathrm{U,N}}(X,x)$-torsor denoted as $(P,p) \to (X,x)$ which can be classified by the following universal property.
\begin{lemma}[{\cite[Prop.~IV.1]{MR682517}}]\label{lemma:Nori-universal-property}
For any unipotent affine group scheme $G$ over $k$ and any pointed $G$-torsor $(Q,q) \to (X,x)$, there exists a unique homomorphism $\rho \colon \pi_1^{\mathrm{U,N}}(X,x) \to G$ of affine group schemes over $k$ and a unique morphism of pointed $X$-schemes $f \colon (P,p) \to (Q,q)$ such that the diagram
\begin{equation}\label{equation:twisted-torsor-map} \begin{tikzcd}
P \times \pi_1^{\mathrm{U,N}}(X,x) \arrow[r] \arrow[d,"{(f,\rho)}"] & P \arrow[d,"f"] \\
Q \times G \arrow[r] & Q
\end{tikzcd} \end{equation}
commutes, where the horizontal arrows are given by the natural action maps.
\end{lemma}
We will rephrase this statement in the language of stacks that is more suitable to our applications (see \cref{lem:Nori-universal-property-restated}).
To that end, we begin by reinterpreting some of the objects in the statement of Lemma~\ref{lemma:Nori-universal-property}.
\begin{lemma}\label{lemma:pointed-classifying-stack}
Let $G$ be an affine group scheme over $k$.
Let $BG$ be the associated classifying stack, pointed by the natural map $\tau \colon \Spec k \to BG$.
Then the groupoid of $G$-torsors $(Q,q) \to (X,x)$ is equivalent to the groupoid of morphisms of pointed stacks $(X,x) \to (BG,\tau)$.
\end{lemma}
\begin{proof}
Indeed, a pointed morphism $f \colon (X,x) \to (BG,\tau)$ defines a $G$-torsor $Q \to X$ obtained by pulling back the $G$-torsor $\tau \colon \Spec k \to BG$ along $f$. Let $Q_x \colonequals Q \times_{X,x} \Spec k \to \Spec k$ denote the pullback of $Q$ along $x \colon \Spec k \to X$. Now, the fact that $f$ is a pointed map (in the $(2,1)$-categorical sense) is equivalent to the data of an isomorphism of $G$-torsors $\alpha \colon Q_x \xrightarrow{\sim} G$ over $\Spec k$. But such an isomorphism corresponds to the data of a rational point $q \in Q(k)$ (the preimage of the identity element $e \in G(k)$ under $\alpha$). This gives the desired statement.
\end{proof}
\begin{remark}\label{lemma:group-homomorphism-classifying-stack}
We recall that if $G$ and $H$ are affine group schemes, then the set of morphisms of group schemes $G \to H$ is equivalent to morphisms of pointed stacks $BG \to BH$. Let us recall how to obtain an $H$-torsor from a given $G$-torsor $P \to S$ and a morphism of group schemes $\rho \colon G \to H$. One can equip $P \times H$ with the (left) $G$-action $g\cdot(p,h) \colonequals (p \cdot g^{-1},\rho(g)\cdot h)$.
Then $P \times^G H \colonequals (P \times H)/G \to S$ with the natural right $H$-action on the second factor is a $H$-torsor:
in case $P = S \times G$ with multiplication action by $G$ on the right, this follows from the fact that the morphism
\[ G \times H \xrightarrow{(\rho,\id)} H \times H \xrightarrow{\textrm{mult}} H \]
descends to an isomorphism $\alpha \colon G \times^G H \xrightarrow{\sim} H$;
the general case follows by working locally on the base.
\end{remark}

\begin{remark}
\Cref{lemma:pointed-classifying-stack} shows that the tautological pointed $\pi_1^{\mathrm{U,N}}(X,x)$-torsor $(P,p) \to (X,x)$ from \cref{torsornori} corresponds to a pointed morphism $N \colon (X,x) \to B\pi_1^{\mathrm{U,N}}(X,x)$. In the lemma below, we record a universal property of $N \colon (X,x) \to B\pi_1^{\mathrm{U,N}}(X,x)$ formulated in the category of pointed stacks. 
\end{remark}{}
\begin{lemma}\label{lem:Nori-universal-property-restated}
For any unipotent affine group scheme $G$ over $k$ and any pointed map $(X,x) \to BG$, there exists a unique morphism of pointed stacks $r \colon B\pi_1^{\mathrm{U,N}}(X,x) \to BG$ and a uniquely commutative diagram
\begin{equation}\label{equation:pointed-stack-map}
\begin{tikzcd}
(X,x) \arrow[d, "N"'] \arrow[rrd]           &  &    \\
{B\pi_1^{\mathrm{U,N}}(X,x)} \arrow[rr, "r"'] &  & BG
\end{tikzcd} \end{equation}
of pointed stacks.
\end{lemma}
\begin{proof}
This is essentially a restatement of \cref{lemma:Nori-universal-property}. A pointed map $(X,x) \to BG$ (by \cref{lemma:pointed-classifying-stack}) is equivalent to a pointed $G$-torsor $(Q,q) \to (X,x)$. By the universal property in \cref{lemma:Nori-universal-property}, we equivalently obtain a map $\rho \colon \pi_1^{\mathrm{U,N}}(X,x) \to G $ and $f \colon (P,p) \to (Q,q)$ which fits into the diagram \cref{equation:twisted-torsor-map}. By \cref{lemma:group-homomorphism-classifying-stack}, the map $\rho \colon \pi_1^{\mathrm{U,N}}(X,x) \to G$ is equivalent to a map of pointed stacks $r \colon B\pi_1^{\mathrm{U,N}}(X,x) \to BG$, providing the unique morphism of pointed stack desired in \cref{lem:Nori-universal-property-restated}. Following the contracted product construction of Lemma~\ref{lemma:group-homomorphism-classifying-stack} and using \cref{lemma:pointed-classifying-stack}, the commutativity of (\ref{equation:pointed-stack-map}) (in the $(2,1)$-categorical sense) is equivalent to the data of an isomorphism of $G$-torsors $\varphi \colon P \times^{\pi_1^{\mathrm{U,N}}(X,x)} G \xrightarrow{\sim} Q$ that maps the class of the point $(p,e)$ to $q$.
\vspace{2mm}

However, using the map $f \colon (P,p) \to (Q,q)$, one obtains a map
\[ P \times G \xrightarrow{(f,\id)} Q \times G \xrightarrow{\textrm{act}} Q, \]
which sends $(p,e)$ to $q$, is compatible with the action of $G$ on the right, and by the commutativity of the diagram (\ref{equation:twisted-torsor-map}), descends to a morphism (and hence an isomorphism) of right $G$-torsors $\varphi \colon P \times^{\pi_1^{\mathrm{U,N}}(X,x)} G \xrightarrow{\sim} Q$ that maps the class of the point $(p,e)$ to $q$. To see the uniqueness of $\varphi$, note that conversely, an isomorphism of right $G$-torsors $P \times^{\pi^{\mathrm U}_1(X,x)} G \xrightarrow{\sim} Q$ preserving base points as before defines a map
\[ f \colon P \xrightarrow{(\id,e)} P \times G \to P \times^{\pi^{\mathrm U}_1(X,x)} G \simeq Q \]
which makes the natural diagram such as (\ref{equation:twisted-torsor-map}) commutative and maps $p$ to $q$. Thus the universal property of \cref{lemma:Nori-universal-property} implies the uniqueness of $\varphi$. This finishes the proof.
\end{proof}

\begin{proposition}\label{prop:Nori-affine-pi1}
Let $(X,x)$ be a pointed cohomologically connected scheme of finite type over $k$.
Then there is a natural isomorphism of sheaves of groups $\pi_1(\UU(X), *) \xrightarrow{\sim} \pi_1^{\mathrm{U,N}}(X)$.
\end{proposition}
\begin{proof}
Since $H^0 (X, \cO) \simeq k$, the unipotent homotopy type $\UU(X)$ is a pointed connected stack (\cref{connn}). Therefore, by \cref{thmoftoen}, the sheaf $\pi_1(\UU(X),*)$ is representable by a unipotent affine group scheme. Note that by considering $1$-truncation, we have a natural morphism $(X,x) \to \UU(X) \to \tau_{\le 1} \UU(X) \simeq B\pi_1(\UU(X),*)$ of pointed stacks.
Now, let $G$ be a unipotent affine group scheme and $f\colon X \to  BG$ be a pointed map. By \cref{thmoftoen}, $BG$ is an affine stack. By the universal property of the unipotent homotopy type (see \cref{defuht}), we have a unique factorization of $f$ through a morphism of pointed stacks $f'\colon \UU(X) \to BG$. 
Further, the stack $BG$ is $1$-truncated, so $f'$ further factors uniquely through $\UU(X) \to \tau_{\le 1} \UU(X) \simeq B\pi_1(\UU(X), *)$ to give a (unique) map $r \colon B\pi_1 (\UU(X), *) \to BG$ of pointed stacks such that we have a uniquely commutative diagram

\begin{center}
    \begin{tikzcd}
(X,x) \arrow[d] \arrow[rrd]           &  &    \\
B\pi_1 (\UU(X), *) \arrow[rr, "r"'] &  & BG
\end{tikzcd}
\end{center}
of pointed stacks. Therefore, by the universal property from \cref{lem:Nori-universal-property-restated}, we are done.
\end{proof}

\subsection{The higher unipotent homotopy group schemes}\label{posto1}

In \cref{prop:Nori-affine-pi1}, we recovered the unipotent Nori fundamental group scheme as $\pi_1$ of the unipotent homotopy type of the scheme $X$. In particular, this description of the unipotent Nori fundamental group scheme bypasses the Tannakian formalism that was used by Nori. In this subsection, our goal is to introduce and record some basic properties of the higher unipotent homotopy group schemes $\pi_i (\mathbf{U}(X))$. In particular, we will see that the higher unipotent homotopy group schemes as introduced in \cref{affinestacks} satisfy a product formula (\cref{productfor}), birational invariance (\cref{birational-invariance}) and an analogue of the Hurewicz theorem (\cref{hurewicz2}). 
\begin{notation}\label{cohconn}
We will call a higher stack $X$ over $k$ \emph{cohomologically connected} if $H^0(X,\cO)\simeq k$.
Note that if $X$ is a geometrically reduced and geometrically connected proper scheme over $k$, then it is cohomologically connected (see, e.g., \cite[\href{https://stacks.math.columbia.edu/tag/0FD2}{Tag~0FD2}]{stacks}).
\end{notation}
\begin{definition}[Unipotent homotopy group schemes]\label{affinestacks}
Let $X$ be a pointed, cohomologically connected scheme over $\Spec k$.
In this setup, the sheaves $\pi_i (\UU(X), *)$ are representable by affine unipotent group schemes over $k$ by \cref{thmoftoen} and \cref{connn}.
We define the $i$-th unipotent homotopy group schemes of $X$ to be $$\pi_i^{\mathrm{U}}(X) \colonequals \pi_i (\UU(X), *).$$

\end{definition}{}

The above definition can be formulated in the generality of higher stacks:

\begin{definition}
    Let $X$ be a pointed, cohomologically connected higher stack over $\Spec k$.
    We define the $i$-th unipotent homotopy group schemes of $X$ to be $$\pi_i^{\mathrm{U}}(X) \colonequals \pi_i (\UU(X), *).$$
\end{definition}{}

\begin{remark}
By \cref{prop:Nori-affine-pi1}, for a pointed cohomologically connected scheme $X$ over a field, $\pi^{\mathrm U}_1(X)$ agrees with the unipotent Nori fundamental group scheme $\pi_1^{\mathrm{U,N}}(X)$ from \cref{torsornori}.
However, note that the unipotent homotopy type $\UU(X)$ can be defined for any higher stack over an arbitrary base. One may therefore also contemplate the homotopy sheaves $\pi_i(\UU(X), *)$ for any pointed higher stack $X$, even though they might not be representable in general (see \cref{weirdexample}).
\end{remark}

Now we prove some basic properties of the unipotent homotopy group schemes defined above. Some of these results extend the results known for the unipotent Nori fundamental group scheme. However, the techniques we use to prove them are quite different. 
\begin{proposition}[Product formula]\label{productfor}
Let $X$ and $Y$ be two pointed and cohomologically connected quasicompact and quasiseparated schemes over $k$.
Then $\pi_i^{\mathrm{U}} (X \times Y) \simeq \pi_i^{\mathrm{U}}(X) \times \pi_i^{\mathrm{U}} (Y)$.
\end{proposition}{}
\begin{proof} This is a consequence of the K\"unneth formula in cohomology, which implies that $R\Gamma(X, \cO) \otimes_k R\Gamma(Y, \cO) \simeq R\Gamma (X \times Y, \cO)$ in $\mathrm{DAlg}^{\mathrm{ccn}}_k$.
Indeed, the previous isomorphism yields an isomorphism $\mathbf{U}(X\times Y) \simeq \mathbf{U}(X) \times \mathbf{U}(Y)$ of pointed stacks.
Therefore, the product formula now follows from the facts that taking homotopy groups commutes with products of spaces and that sheafification commutes with finite limits. 
\end{proof}{}

\begin{remark}We point out that as a consequence of our definition of the unipotent homotopy groups, we obtain a simpler proof of the product formula even in the case of $i=1$ which appears in \cite[Lem.~IV.8]{MR682517}.
\end{remark}{}

\begin{proposition}[Birational invariance]\label{birational-invariance}
Let $f \colon X \to Y$ be a birational morphism of smooth proper pointed schemes over $k$. Then $f$ induces isomorphisms $\pi_i^{\mathrm{U}} (X) \simeq \pi_i^{\mathrm{U}}(Y)$ of group schemes over $k$.
\end{proposition}{}
\begin{proof} 
It suffices to prove that $R^i f_* \cO_X=0$ for $i>0$ and $\cO_Y \simeq R^0 f_* \cO_X$. Therefore, the proposition follows from \cite[Thm.~2]{MR2923726}.
\end{proof}{}
\begin{proposition}\label{simplyconnected}Let $X$ be a pointed cohomologically connected higher stack over $k$ such that $H^1(X, \cO)=0$. Then $\pi_1^{\mathrm U}(X)$ is the trivial group scheme over $k$. Conversely, if $\pi_1^{\mathrm U}(X)$ is trivial, then $H^1(X, \cO)=0$.
\end{proposition}{}
\begin{proof}
By \cref{connn}, the 1-truncation of $\mathbf{U}(X)$ is equivalent to $B \pi_1^{\mathrm U}(X)$.
Therefore,
$$ H^1(X, \cO) \simeq H^1(\mathbf{U}(X), \cO) \simeq H^1 (B \pi_1^{\mathrm U}(X), \GG_a ) \simeq \Hom(\pi_1^{\mathrm U}(X), \GG_a), $$
where the middle isomorphism follows from \cref{postnikov} and the last one from fpqc descent along $\Spec k \to B\pi^\rU_1(X)$.
Since $\pi_1^{\mathrm U}(X)$ is unipotent, $H^1(X, \cO)=0$ implies that $\pi_1^{\mathrm U}(X)$ must be the trivial group scheme.
The converse also follows from the above isomorphisms.
\end{proof}{}
\begin{remark}
In the case of schemes, the isomorphism $H^1 (X, \cO) \simeq \Hom(\pi_1^{\mathrm U}(X), \GG_a)$ already appears in \cite[Prop.~2]{MR682517} as a result of Tannakian principles;
this is enough to prove \cref{simplyconnected}.
We formulated our proof in the language of stacks to motivate a more general assertion appearing in \cref{hurewicz} below, in which we reinterpret this isomorphism from the perspective of the Hurewicz theorem in algebraic topology. We will begin with the following lemma.
\end{remark}

\begin{lemma}\label{gait1}Let $G$ be a commutative affine group scheme over $k$ and $m \ge 1$ be an integer. Then $H^m(K(G,m), \cO) \simeq \Hom(G, \GG_a)$ and there is a natural injection  $\Ext^1 (G, \GG_a) \to H^{m+1}(K(G,m), \cO)$ which is an isomorphism if $m \ge 2$. Also, $H^i (K(G, m), \cO) =0$ for $0 <i < m$.
\end{lemma}
\begin{proof}
This can be seen by a spectral sequence argument similar to \cite[Rmk.~3.17]{Mon21}.
Alternatively, we note that $H^{m+1}(K(G,m), \cO)= \pi_0 \w{Map}(K(G,m), K(\GG_a, m+1 ))$.
The latter may also be computed as homotopy classes of maps $K(G,m) \to K(\GG_a, m+1)$ of \emph{pointed} higher stacks (e.g., using \cref{usefullemma113} and the fact that $H^n(\Spec k,\GG_a) = 0$ for all $n>0$).
By delooping repeatedly, one can also compute that as homotopy classes of maps $G \to B \GG_a$ of $E_m$-group stacks\footnote{i.e., an $E_m$-group like object in the $\infty$-category of stacks, see for e.g., \cite[Thm.~5.2.6.10]{luriehigher}.} which gives the claim that $H^{m+1}(K(G, m), \cO) = \mathrm{Ext}^1(G, \GG_a)$ for $m \ge 2$. The claims that $H^m(K(G, m), \cO) \simeq \mathrm{Hom}(G, \GG_a)$ and $H^i (K(G,m), \cO)=0$ for $0<i <m$ follow similarly.
\end{proof}{}

\begin{proposition}[Hurewicz theorem for affine stacks]\label{hurewicz}
Let $X$ be a pointed connected affine stack over $k$.
Let $n \ge 0$ be a nonnegative integer.
Then the following two statements are equivalent:
\begin{enumerate}
    \item  $H^i(X, \cO)$ is trivial for $ i \le n$.\footnote{\label{trivial-convention}
    For $n=0$, we use the convention that $H^0 (X, \cO)$ being trivial means that $H^0 (X, \cO) \simeq k$.}
    \item $\pi_i(X)$ is trivial for $ i \le n$.
\end{enumerate}
Moreover, in such a situation, we have $H^{n+1}(X, \cO) \simeq\Hom (\pi_{n+1}(X), \GG_a)$ and there is a natural injection $\Ext^1(\pi_{n+1}(X), \GG_a) \hookrightarrow H^{n+2}(X, \cO)$.
\end{proposition}
\begin{proof}
First, we note that since $X$ is by assumption pointed and connected, by descent along $\Spec k \to X$, it follows that $H^0 (X, \cO) \simeq k$.
Thus, the two statements are equivalent when $n=0$.
Next, we will check their equivalence when $n \ge 1$.
\vspace{2mm}

To begin, let us assume that $\pi_i(X)$ is trivial for $0\le i \le n$.
Using the assumption, it follows that the $n$-truncation $\tau_{\le n} X$ from \cref{postnikov} must be naturally isomorphic to the point.
Since $H^i (\tau_{\le n}X, \cO) \simeq H^i (X, \cO)$ for $i \le n$, we obtain the desired conclusion. \vspace{2mm}

Now conversely, we assume that $H^i (X, \cO)=0$ for $1 \le i \le n$. Let us consider the Postnikov tower $(\tau_{\le n} X)$ as in \cref{postnikov}. It would be enough to prove that $\tau_{\le n} X$ is naturally isomorphic to the point.
We will prove via induction on $i$ that $* \simeq \tau_{\le i}X$ for $0 \le i \le n$. Since $X$ is connected, the claim $*\simeq \tau_{\le 0}X$ holds. Assuming that the claim is true for $\tau_{\le {i-1}}X$, it follows that $\tau_{\le i}X \simeq K(\pi_i (\tau_{\le i}X), i) \simeq K(\pi_i(X), i)$. By \cref{postnikov}, the natural map $H^i (\tau_{\le i}X, \cO) \to H^i (X, \cO)$ is an isomorphism. Using the computation from \cref{gait1}, we obtain a natural isomorphism $\Hom(\pi_i(X) , \GG_a) \xrightarrow{\sim} H^i (X, \cO)$.
If $i \le n$, the assumption $H^i (X, \cO)=0$ then gives $\Hom(\pi_i(X) , \GG_a)=0$. Since $\pi_i(X)$ is an affine unipotent group scheme, this implies that $\pi_i(X)$ is the trivial group scheme. Therefore, we see that $\tau_{\le i}X$ is indeed naturally isomorphic to the point for $i \le n$. In particular, $\tau_{\le n}X$ is naturally isomorphic to the point. This finishes the proof of the claim.
\vspace{2mm}

Now we prove the last part of the proposition. Let $n \ge 0$ be as given. Under the above equivalent assumptions, $\tau_{\le n}X$ is naturally isomorphic to the point. Therefore, it follows that $\tau_{\le {n+1}}X \simeq K(\pi_{n+1}(X), n+1)$. By \cref{postnikov}, the map $H^{n+1}(\tau_{\le n+1}X, \cO) \to H^{n+1}(X, \cO)$ is an isomorphism and the map $H^{n+2} (\tau_{\le n+1}X, \cO) \to H^{n+2}(X, \cO)$ is an injection. Applying \cref{gait1} now finishes the proof of \cref{hurewicz}.
\end{proof}{}

\begin{corollary}[Hurewicz theorem for unipotent homotopy group schemes]\label{hurewicz2}Let $X$ be a pointed, cohomologically connected higher stack over $k$.
Let $n \ge 0$ be a nonnegative integer.
Then the following two statements are equivalent:
\begin{enumerate}
    \item  $H^i(X, \cO)$ is trivial for $i \le n$.
    \item $\pi_i^{\mathrm{U}}(X)$ is the trivial group scheme for $ i \le n$.
\end{enumerate}
Moreover, in such a situation, we have $H^{n+1}(X, \cO) \simeq\Hom (\pi_{n+1}^\mathrm{U}(X), \GG_a)$ and there is an injection $\Ext^1(\pi_{n+1}^\mathrm{U}(X), \GG_a) \hookrightarrow H^{n+2}(X, \cO)$.

\end{corollary}{}
\begin{proof}Follows from applying \cref{hurewicz} to $\mathbf{U}(X)$.
\end{proof}{}

\begin{lemma}\label{algebraicdesc} Let $X$ be a pointed connected affine stack over $k$ such that $H^0 (X, \cO) \simeq k$. Then the affine scheme $\Spec\,\pi_0 (k \otimes_{R\Gamma(X, \cO)} k)$ has a natural group scheme structure under which it is isomorphic to $\pi_1 (X)$.
\end{lemma}
\begin{proof}
By \cref{hypercompletea} and \cref{rent14}, the loop stack $\Omega X$ is an affine stack and $R\Gamma (\Omega X, \cO) \simeq k \otimes_{R\Gamma(X, \cO)}k$, respectively.
The universal property of mapping into affine schemes shows that giving a map from $\Omega X \to \Spec S$ is equivalent to giving a map $\Spec\,H^0 (\Omega X, \cO) = \Spec\,\pi_0 (k \otimes_{R\Gamma(X, \cO)} k) \to \Spec S$.
\vspace{2mm}

On the other hand, since the affine scheme $\Spec S$ is $0$-truncated as a stack, giving a map from $\Omega X \to \Spec S$ is also equivalent to giving a map $\pi_0 (\Omega X) \to \Spec S$. Now we note that $\pi_0 (\Omega X) = \pi_1 (X)$. Since we know that $\pi_1 (X)$ is representable by a unipotent affine group scheme and $\Spec S$ was an arbitrary affine scheme, it follows that $\pi_1 (X) \simeq \Spec\,\pi_0 (k \otimes_{R\Gamma(X, \cO)} k)$ as affine schemes. This equips $\Spec\,\pi_0 (k \otimes_{R\Gamma(X, \cO)} k)$ with the structure of an affine group scheme as desired and proves the claim.
\end{proof}{}
Combining \cref{algebraicdesc} above with \cref{prop:Nori-affine-pi1}, we obtain the following explicit algebraic description of $\pi_1^{\mathrm U}(X)$.
\begin{corollary}\label{maincorro}
Let $X$ be a pointed cohomologically connected scheme (of finite type) over a field $k$.
    Then we have an isomorphism of group schemes
$\pi_1^{\mathrm U}(X) \simeq \pi_1^{\mathrm{U,N}}(X) \simeq \Spec \pi_0 (k \otimes_{R\Gamma(X, \mathscr O)}k)$.
\end{corollary}{}

We end this subsection with some examples.

\begin{example}Let $\PP^n_k$ be the projective $n$-space over a field $k$. Then $\pi_i^{\mathrm{U}}(\PP^n_k)$ is the trivial for $i \ge 0$. If $X$ is a hypersurface in $\PP^n_k$, then it follows from \cref{hurewicz2} that $\pi_i^{\mathrm{U}} (X)$ is trivial for $i <n-1$.
\end{example}{}

\begin{remark}
In \cite[p.~93]{MR682517}, Nori proves that any complete normal rational variety $X$ over a field $k$ has a trivial fundamental group scheme. This implies that the unipotent fundamental group scheme of such a variety is also trivial. Indeed, if $k$ has positive characteristic this follows from \cite[Prop.~IV.3]{MR682517}. If $k$ has characteristic zero, by applying the Leray spectral sequence to a resolution of singularities, it follows that $H^1(X, \cO)=0$. Thus \cref{simplyconnected} implies that $\pi_1^{\mathrm{U}}(X)$ is trivial. The following example records the fact that this need not be true for the higher unipotent homotopy group schemes $\pi_i^{\mathrm{U}}(X)$ for $i>1$. 
\end{remark}
\begin{example}
Let $k = \overline{\FF}_p$ and $C \subset \PP^2_k$ be a smooth cubic curve.
Let $f \colon X \to \PP^2_k$ be a blowup in $10$ points of $C$ and $D \subset X$ be the strict transform of $C$.
In this case, we have $(D.D) = -1$; therefore, the contraction $\pi \colon X \to Y$ of $D$ is a birational morphism onto a normal projective surface $Y$ by \cite[Thm.~2.9]{MR146182}.
\vspace{2mm}

By the theorem on formal functions, we have $R^if_* \cO_X = 0$ for all $i > 0$ and thus $R\Gamma(X,\cO_X) \simeq R\Gamma(\PP^2_k,\cO_{\PP^2_k}) = k$.
On the other hand, the Leray spectral sequence furnishes an exact sequence
\[ 0 \to H^1(Y,\pi_*\cO_X) \to H^1(X,\cO_X) \to H^0(Y,R^1\pi_*\cO_X) \to H^2(Y,\pi_*\cO_X) \to H^2(X,\cO_X). \]
Since $\pi_*\cO_X \simeq \cO_Y$ by Stein factorization (see e.g.,\ \cite[\href{https://stacks.math.columbia.edu/tag/0AY8}{Tag~0AY8}]{stacks}), $H^1(Y,\cO_Y) \simeq H^1(Y,\pi_*\cO_X) \simeq 0$.
Moreover, $R^1\pi_*\cO_X$ is the skyscraper sheaf at the point $\pi(D)$ with fibre $H^1(D,\cO_D)$. Indeed, it is supported on $\pi(D)$, and by the theorem on formal functions, 
$$ (R^1\pi_*\cO_X)^\wedge_{\pi(D)} \simeq \varprojlim_n H^1(D,\cO_X\otimes_{\cO_Y} \cO_Y/\fm^n_{\pi(D)}) \simeq \varprojlim_n H^1(D,\cO_{nD}). $$
Now, Serre duality and the identity $\deg_D \cO_D(nD) = (nD.D) = -n$ yield $H^1(D,\cO_D(-nD)) = H^0(D,\cO_D(nD)) = 0$ for all $n \in \ZZ_{>0}$, so induction on $n$ with the short exact sequence
\[ 0 \to \cO_D(-nD) \to \cO_{(n+1)D} \to \cO_{nD} \to 0 \]
proves that $H^1(D,\cO_{nD}) \simeq H^1(D,\cO_D)$.
We conclude that
\[ H^2(Y,\cO_Y) \simeq H^2(Y,\pi_*\cO_X) \simeq H^0(Y,R^1\pi_*\cO_X) \simeq H^1(D,\cO_D) \simeq k. \]
Hence, \cref{hurewicz2} shows that $\pi^{\mathrm U}_1(Y) \simeq *$, but $\pi^{\mathrm U}_2(Y) \not\simeq *$. 
This example also shows that unipotent homotopy group schemes are not birational invariants for singular varieties (\textit{cf.}~\cref{birational-invariance}). 
\end{example}

\subsection{Pro-algebraic completion of group valued sheaves}\label{proalgcompletion}
When studying potentially non-affine higher stacks, one often encounters non-representable homotopy sheaves. In this section, we study a certain kind of pro-algebraic completion of (pre)sheaves of groups.
Using this, we define in \cref{chris1} the tensor product of commutative group schemes. We also introduce a variant of the wedge product in \cref{weakwedge}, called the \emph{weak wedge square}, which plays an important role in \cref{worldcup1} in understanding certain unipotent homotopy group schemes of Calabi--Yau varieties in positive characteristic.
\vspace{2mm}

Note that for any group scheme $G$ over a field, we have a maximal unipotent quotient $G^{\mathrm{uni}}$ such that the map $G \to G^{\mathrm{uni}}$ is universal among maps from $G$ to unipotent group schemes (see e.g.,\ \cite[Lem.~1.5.4]{Toe}). We will now generalize this construction to arbitrary presheaves of groups $G$.
In what follows, we always work over a base field $k$. By the Yoneda lemma, we have a natural functor from the category of affine group schemes over $k$ to the category of presheaves of groups on $\mathrm{Aff}_k$ which is an inclusion of categories.
Since the category of affine group schemes has all limits and they are preserved by this natural functor, it must have a left adjoint which we denote by $H \to H^{\alg}$. Note that $H^{\mathrm{alg}}$ may be obtained as the inverse limit over diagrams $H \to G$, where $G$ is a finite type affine group scheme and the map is a morphism of presheaves of groups.
In general, the map $H \to H^{\mathrm{alg}}$ is not an fpqc surjection, nor does it induce an isomorphism on global sections. These phenomena can already be observed when $H$ is the constant sheaf of groups $\mathbf{Z}$; see \cref{teach}.
\vspace{2mm}

\begin{example}
If $H$ is the constant presheaf associated with a finite group $G$, then $H^{\w{alg}}$ is the induced affine group scheme structure on $\coprod_{G} \Spec k$.
\end{example}{}

Given a presheaf of groups $H$, one can similarly construct a map $H \to H^{\mathrm{uni}}$ which is initial among maps to unipotent group schemes.
It follows from the universal properties that $H^{\mathrm{uni}} \simeq (H^{\mathrm{alg}})^{\mathrm{uni}}$.
The next lemma gives an alternative way to describe $H^{\mathrm{uni}}$, which also allows us to compute its ring of global sections.
\begin{lemma}\label{garten2}
Let $H$ be any (pre)sheaf of groups. Then $\pi_1 (\mathbf{U}(BH)) \simeq H^{\mathrm{uni}}$.   
\end{lemma}{}
\begin{proof}
By analyzing universal properties of mapping to a $1$-truncated affine stack, we obtain an equivalence between pointed maps from $BH \to BG $ and pointed maps $\tau_{\le 1} (\mathbf{U}(BH)) \to BG$ for a unipotent group scheme $G$. This gives a bijection between maps of sheaves of groups $H \to G$ and maps of sheaves of groups $\pi_1(\mathbf{U}(BH)) \to G$, yielding the required statement.
In the proof, we used that $\pi_0 (\mathbf{U}(BH)) \simeq *$, which follows from the fact that $H^0 (BH, \cO) \simeq k$ since $* \to BH$ is an fpqc epimorphism.
\end{proof}{}

\begin{corollary}
Let $H$ be any (pre)sheaf of groups.
Then
\[ H^{\mathrm{uni}} \simeq \Spec \pi_0 (k \otimes_{R\Gamma(BH,\cO)}k). \]
\end{corollary}
\begin{proof}
Follows from \cref{garten2} and \cref{algebraicdesc}.
\end{proof}

\begin{lemma}\label{whoknew}
Let $G$ be any (pre)sheaf of abelian groups. Then $\pi_n (\UU(K(G,n))) \simeq (G^{\mathrm{uni}})^{\mathrm{ab}}$ for $n \ge 2$. 
\end{lemma}{}

\begin{proof}
Since $H^i (K(G,n), \cO)=0$ for $i <n$, \cref{hurewicz} implies that 
$\tau_{\le n}\mathbf{U}(K(G,n)) \simeq K (\pi_n (\mathbf{U}(K(G,n))), n)$.
The claim now follows as in the proof of \cref{garten2} from delooping and the universal property of $(G^{\mathrm{uni}})^\ab$.
\end{proof}{}

\begin{example}\label{teach}
If $k$ has characteristic $p>0$, then $\mathbf{Z}^{\mathrm{uni}}$ is the profinite affine group scheme $\mathbf{Z}_p$. 
On the other hand, if $k$ has characteristic zero, then $\mathbf{Z}^{\w{uni}}$ is $\GG_{a, k}$.
\end{example}{}

\Cref{teach} shows that the map $H \to H^{\mathrm{uni}}$ is not an fpqc surjection in general.
However, if $H \to H^{\w{alg}}$ is an fpqc surjection, then so is $H \to H^{\w{uni}}$ since $H^{\w{alg}} \to (H^{\w{alg}})^{\uni}= H^{\w{uni}}$ is a quotient map of group schemes.
Below, we give a criterion for $H \to H^{\w{alg}}$ being an fpqc surjection that is often easy to verify.

\begin{proposition}\label{rightexact}
Let $H_1 \to H_2 \to H_3 \to 0$ be an exact sequence of abelian fpqc sheaves. Then the following two sequences are exact.

\begin{enumerate}
    \item $(H_1^{\mathrm{alg}})^{\mathrm{ab}} \to (H_2^{\mathrm{alg}})^{\mathrm{ab}} \to (H_3^{\mathrm{alg}})^{\mathrm{ab}} \to 0$
    
    \item $(H_1^{\mathrm{uni}})^{\mathrm{ab}} \to (H_2^{\mathrm{uni}})^{\mathrm{ab}} \to (H_3^{\mathrm{uni}})^{\mathrm{ab}} \to 0$
\end{enumerate}{}
\end{proposition}{}

\begin{proof}
Follows from the universal properties. 
\end{proof}{}

\begin{corollary}\label{tttttt}
Let $H$ be an fpqc sheaf of abelian groups.
Assume there exists a commutative affine group scheme $G$ and a surjection of fpqc sheaves $G \to H$.
Then $H \to (H^{\mathrm{alg}})^{\mathrm{ab}}$ is a surjection of fpqc sheaves. 
Consequently, $H \to (H^{\mathrm{uni}})^{\mathrm{ab}}$ is an fpqc surjection as well.
\end{corollary}

\begin{proof}
\Cref{rightexact} shows that $(G^{\w{alg}})^{\w{ab}} \to (H^{\w{alg}})^{\w{ab}}$ is surjective. Since $G$ is a commutative affine group scheme, the natural map $G \to (G^{\w{alg}})^{\w{ab}}$ is an isomorphism. 
Thus, the composition $G \to H \to (H^{\w{alg}})^{\w{ab}}$ is surjective. This implies that $H \to (H^{\w{alg}})^{\w{ab}}$ is surjective. The surjectivity of $H \to (H^{\w{uni}})^{\w{ab}}$ follows.
\end{proof}

Suppose that there is a surjection of sheaves $G \to G_0$ where $G$ is a commutative affine group scheme. One can then ask if $G \to {G}^{\w{alg}}_0$ is surjective and further if we can describe the kernel $K$ of $G \to G^{\w{alg}}_0$ explicitly.
The following construction gives an answer to this question in the case when the kernel of the map $G \to G_{0}$ can be generated by a map from an affine scheme $T$ to $G$. 

\begin{construction}\label{Milne}
Let $T$ be an affine scheme and $G$ be a commutative affine group scheme over $k$. 
Let $t$ be a $k$-rational point of $T$. 
Let $f \colon (T,t) \to (G, 1_{G})$ be a map of pointed schemes over $k$ whose image is stable under the inverse morphism of $G$. Then there exists an initial map $u \colon G \to H$ of affine group schemes with the property that $u \circ f = 0$:
indeed, if we let $[\mathrm{Im}(f)]$ denote the sub(pre)sheaf of groups of $G$ generated by the image of the map $f$, then $H=(G/ [\mathrm{Im}(f)])^{\mathrm{alg}}$.

\vspace{2mm}
Let us now describe an explicit construction of $H$. Let $G/ [\w{Im}(f)]\to S$ be a map of fpqc sheaves of groups where $S$ is an affine group scheme. The data of such a map is equivalent to giving a map of affine group schemes $G \to S$ whose kernel contains the scheme theoretic image of the maps
\[ T^n \to G^n \to G \]
for all $n \ge 1$, where the last maps $G^n \to G$ are induced by the multiplication on $G$.
The scheme theoretic image of the map $T^n \to G$ determines an ideal sheaf $\mathscr{I}_n$ of the affine scheme $G$. By our assumptions, it follows that $\mathscr{I}_n$ is a decreasing sequence of ideal sheaves.
Let $\mathscr{I} \colonequals \bigcap_n \cI_n$. Then $\mathscr{I}$ defines a closed subscheme $K$ of $G$, and in fact, $K$ is a subgroup scheme of $G$. Note that $K$ is still contained in the kernel of the map $G \to S$ since the kernel is a closed subgroup scheme of $G$. Therefore, we obtain a natural map $G/K \to S$ of affine group schemes.
Using the universal properties, $H = (G/[\w{Im}(f)])^{\w{alg}} = G/K$. This gives an explicit description of the kernel $K$ and shows the surjectivity of the map $G \to (G/[\w{Im}(f)])^{\w{alg}}$. Note that since $G$ is commutative, it follows that in this situation, $(((G/[\w{Im}(f)])^{\w{alg}})^{\w{ab}}) = (G/[\w{Im}(f)])^{\w{alg}}$.
\end{construction}
\begin{remark}\label{useinfreu}
The above construction can be used to understand $H^\w{alg}$ explicitly when $H$ is obtained as a quotient of a bilinear map $G_1 \times G_2 \to G$ of affine commutative group schemes. To give a more flexible treatment (see \cref{com}), we introduce the following constructions that we will need later.
Note that for a non-commutative affine group scheme $G$, \cref{Milne} still applies to the commutator map $f \colon (G \times G, 1_G \times 1_G) \to (G, 1_G)$;
this shows for example that the natural map $G \to G^\ab$ is surjective, where $G^{\mathrm{ab}} \simeq (G/[\mathrm{Im}(f)])^{\mathrm{alg}}$ is the abelianization.
\end{remark}{}

\begin{definition}[Tensor product of group schemes]\label{chris1} 
Let $G$ and $H$ be two commutative affine group schemes over $k$. Let $\underline{G}$ and $\underline{H}$ be the sheaves of abelian groups represented by $G$ and $H$ and $\underline{G} \otimes_{\mathbf{Z}} \underline{H}$ be their (ordinary) tensor product. We define $G \otimes H$ to be the affine group scheme $((\underline{G} \otimes_{\mathbf{Z}} \underline{H})^{\w{alg}})^{\w{ab}}$. It is clear that $G \otimes H$ satisfies the universal property of the tensor product in the category of commutative affine group schemes.

\end{definition}{}
\begin{definition}[Wedge product of group schemes]
Let $G$ be a commutative affine group scheme over $k$. Let $\underline{G}$ be the sheaf of abelian groups represented by $G$. We define $\wedge ^n G $ to be the affine group scheme $((\wedge ^n\underline{G})^{\w{alg}})^{\w{ab}}$.
\end{definition}{}

\begin{definition}[Weak wedge square of groups]
Let $P$ be an abelian group. There is a natural endomorphism $\varphi \colon P \otimes_{\mathbf{Z}} P \to P \otimes_{\mathbf{Z}} P$ on the (ordinary) tensor product determined by
\[ p \otimes p' \mapsto p\otimes p' + p' \otimes p. \]
We define $P \curlywedge_{\mathbf{Z}} P \colonequals \Coker(\varphi)$ and call it the \emph{weak wedge square} of $P$. There is a natural surjective map $P \curlywedge_{\mathbf{Z}} P \to P \wedge P$. 
\end{definition}{}

\begin{example}
Note that $\mathbf{Z} \curlywedge_{\mathbf{Z}} \mathbf{Z} \simeq \mathbf{Z}/ 2 \mathbf{Z}$.
\end{example}{}

\begin{remark}
If $V$ is a vector space over a field $k$ of characteristic $p>2$, one can define $V \curlywedge V$ as the cokernel of the map $V \otimes_{k} V \to V \otimes_{k} V$ determined by $v \otimes v' \mapsto v \otimes v' + v' \otimes v$.
In this case, $V \curlywedge V \simeq V \wedge V$. More generally, if $P$ is an abelian group, then $(P \curlywedge_{\mathbf Z} P) \otimes_{\mathbf{Z}} \mathbf{Z}[\frac 1 2]  \to (P \wedge_{\mathbf Z} P) \otimes_{\mathbf{Z}} \mathbf{Z}[\frac 1 2]$ is an isomorphism.
\end{remark}{}

\begin{definition}[Weak wedge square of group schemes]\label{weakwedge} Let $G$ be a commutative affine group scheme over $k$. Let $\underline{G}$ be the sheaf of abelian groups represented by $G$. We define $G \curlywedge G$ to be $((\underline{G} \curlywedge \underline{G})^{\w{alg}})^{\w{ab}}$. 
\end{definition}{}

\begin{remark}
Let $G$ be a commutative affine group scheme. Note that there is a natural bilinear map $G \times G \to G \otimes G$ of commutative affine group schemes.
\end{remark}{}

\begin{remark}\label{garten}
Let $G$ be a commutative group scheme. Let $\underline{G}$ be the sheaf of abelian groups represented by $G$.
The endomorphism of $\underline{G} \otimes _{\mathbf{Z}} \underline{G}$ that sends $g \otimes g' \mapsto g \otimes g' + g' \otimes g$ induces an endomorphism $\varphi \colon G \otimes G \to G \otimes G$. By \cref{rightexact}, it follows that $\Coker(\varphi) \simeq G \curlywedge G$.
\end{remark}{}

\begin{remark}\label{com}
Let $f \colon G_1 \times G_2 \to G$ be a bilinear map of commutative group schemes. 
Let $\underline{G_1}, \underline{G_2}$ and $\underline{G}$ be the sheaves represented by $G_1, G_2$ and $G$, respectively.
We obtain a map of abelian sheaves $\underline{f} \colon \underline{G_1} \otimes_{\mathbf{Z}} \underline{G_2} \to \underline{G}$. In the notation of \cref{Milne}, the sheaf $G/ [\w{Im}(f)]$ is isomorphic to $\Coker(\underline{f})$. By \cref{rightexact} and \cref{Milne}, it follows that $$(G/ [\w{Im}(f)])^{\w{alg}} \simeq ((G/ [\w{Im}(f)])^{\w{alg}})^{\w{ab}} \simeq ((\Coker(\underline{f}))^{\w{alg}})^{\w{ab}} \simeq \Coker(G_1 \otimes G_2 \to G).$$
\end{remark}{}
The above constructions give rise to a rich source of computations that is interesting to pursue in its own right. 
However, we do not discuss them here in detail since we only need some of these computations for our desired applications (see \cref{worldcup1}); instead, we mention only some of these examples.
\begin{example}
Even if $G$ and $H$ are unipotent commutative group schemes, their tensor product $G \otimes H$ is typically not unipotent.
For example, let $G = H = \alpha_p$ over a perfect field of characteristic $p>0$. 
Since $\alpha_p$ is self-dual under Cartier duality, there is a nontrivial bilinear pairing $\alpha_p \times \alpha_p \to \GG_m$ which yields a nontrivial map $\alpha_p \otimes \alpha_p \to \GG_m$. This implies that $\alpha_p \otimes \alpha_p$ cannot be unipotent.
\end{example}{}

\begin{example}\label{compute098}
Let us again work over a perfect field of characteristic $p >0$. We will explain how to compute the unipotent completion $(\alpha_p \otimes \alpha_p)^{\w{uni}}$ of $\alpha_p \otimes \alpha_p$.
Let $W[F]$ denote the group scheme underlying the kernel of Frobenius on the group scheme of $p$-typical Witt vectors $W$; it is also dual to the formal Lie group $\widehat{\GG}_a$ (\textit{cf.}~\cite[37.3.4]{MR2987372}). Note that $W[F]$ naturally has the structure of a non-unital ring scheme. There is also a map $[\,\cdot\,] \colon \alpha_p \to W[F]$ given by the multiplicative lift.
Sending $(x,y)$ to $[x]\cdot [y]$ gives a map $u \colon \alpha_p \times \alpha_p \to W[F]$. We claim that $u$ is bilinear. Note that $n \colonequals [x+y] - [x] - [y]$, although nonzero, lies in the kernel of the group homomorphism $W[F] \to \alpha_p$. Therefore, as observed in the proof of \cite[Lemma~B.2]{LM21}, one has $n \cdot m = 0$ for any (scheme theoretic point) $m$ of $W[F]$. In particular, it follows that $([x+y] - [x] - [y])[z]=0$, which shows that $u$ is bilinear. This constructs a map $\overline{u} \colon (\alpha_p \otimes \alpha_p)^{\w{uni}} \to W[F]$ of group schemes. Now, we note that $\dim \Hom(W[F], \GG_a)=1$ and the induced map $\Hom(W[F], \GG_a) \to \Hom(\alpha_p \otimes \alpha_p, \GG_a)$ is an isomorphism. Since $W[F]$ is unipotent, it follows that $\overline{u}$ is surjective. Since $\w{Ext}^1 (W[F], \GG_a)=0$ (\cref{fine}), by a long exact sequence chase, it also follows that $\overline{u}$ is injective. Thus $(\alpha_p \otimes \alpha_p)^{\w{uni}} \simeq W[F]$.
\end{example}{}

\begin{example}\label{zvi1}
Let $k$ be a perfect field of characteristic $p > 0$.
When $p \ne 2$, one has $(\alpha_p \curlywedge \alpha_p)^{\mathrm{uni}} = 0$ (\textit{cf.}~\cref{haircut11}). On the other hand, when $p=2$, we have $(\alpha_2 \curlywedge \alpha_2)^{\w{uni}} \simeq (\alpha_2 \wedge \alpha_2)^{\w{uni}} \simeq W[F]$. This follows for example by noting that the isomorphism $(\alpha_2 \otimes \alpha_2)^{\w{uni}} \to W[F] $ constructed above naturally factors through the surjection $(\alpha_2 \otimes \alpha_2)^{\w{uni}} \to (\alpha_2 \wedge \alpha_2)^{\w{uni}}$.
\end{example}{}

\subsection{The Freudenthal suspension theorem in unipotent homotopy theory}\label{makiman}
In this section, we prove a version of the Freudenthal suspension theorem in the context of unipotent homotopy theory. Our result gives a broad generalization of the classical Freudenthal suspension theorem \cite{Freu} for spaces. This version of the Freudenthal suspension theorem can be used to study unipotent homotopy groups of suspensions of arbitrary higher stacks and not just affinizations of spaces.
In particular, our version of the Freudenthal suspension theorem plays an important role in computing certain homotopy group schemes of Calabi--Yau varieties similar to the role played by the classical Freudenthal suspension theorem in computing homotopy groups of spheres.
The ``surjection of group schemes'' part of the Freudenthal suspension theorem in unipotent homotopy theory (see \cref{freudenthalsusp}) presents additional subtleties and a large portion of the work in \cref{secc2} and \cref{uhtfs} is used to settle this part. As preparations towards proving the theorem, we will first note the following constructions and then delve into a series of lemmas.

\begin{construction}\label{whitehead}
We recall that if $Y \in \mathcal{S}$ is a space equipped with a base point, one classically has the (functorial) bilinear maps $$W_{k,l} \colon \pi_k(Y) \times \pi_l(Y) \to \pi_{k+l-1}(Y)$$ that are called Whitehead products \cite{MR4123}. They satisfy the graded symmetry condition $W_{k,l} (u,v) = (-1)^{k\cdot l}W_{k,l}(v,u)$ for $k,l \ge 2$. These maps also satisfy a graded Jacobi identity that we do not recall here. Now let $X$ be a pointed higher stack.
The functoriality of the Whitehead brackets implies that we have (functorial) bilinear maps $$W_{k,l} \colon \pi_k(X) \times \pi_l(X) \to \pi_{k+l-1}(X).$$
We will again call these maps Whitehead products. They also satisfy similar graded symmetry conditions and a graded Jacobi identity.
\end{construction}{}

\begin{construction}[Whitehead product on unipotent homotopy groups]\label{wh}
If $X$ is a pointed higher stack, applying \cref{whitehead} to the unipotent homotopy type of $X$, we obtain Whitehead products 
$$W_{k,l} \colon \pi_k^{\w{U}}(X) \times \pi_l^{\w{U}}(X) \to \pi_{k+l-1}^{\w{U}}(X) $$ on unipotent homotopy group schemes. These maps are again bilinear and satisfy graded symmetry conditions and a graded Jacobi identity.
\end{construction}{}

\begin{lemma}\label{c}
Let $X$ and $Y$ be two pointed higher stacks.
Assume that $X$ is $m$-connected and $Y$ is $n$-connected.
Then their smash product $X \wedge Y$ is $(m+n+1)$-connected. Moreover,\footnote{one takes abelianization of $\pi_{n+1}(X)$ (resp. $\pi_{m+1}(X)$) when $n=0$ (resp. $m=0$).} $$\pi_{m+n+2}(X \wedge Y) \simeq \pi_{m+1}(X) \otimes_{\mathbf{Z}} \pi_{n+1} (Y).$$
\end{lemma}{}

\begin{proof}
The result follows from the same statement at the level of spaces (which for example can be seen by using the classical Hurewicz theorem and computing homology) and applying sheafification. Indeed, let $\underline{X}$ be the fibre of the map $X \to \tau_{\le m}^{\mathrm{pre}} X$ computed in the $\infty$-category of presheaves of spaces. Let $\underline{Y}$ be constructed similarly.
Then $\underline{X}$ and $\underline{Y}$ are an $m$-connected and $n$-connected presheaf, respectively.
It follows that $\underline{X} \wedge \underline{Y}$ is $(m+n+1)$-connected and $\pi_{m+n+2}(\underline{X}\wedge \underline{Y}) \simeq \pi_m(\underline{X}) \otimes_{\mathbf Z} \pi_n (\underline{Y})$ in the category of presheaves.
Applying sheafification and noting that the sheafification of $\underline{X}$ and $\underline{Y}$ recovers $X$ and $Y$, respectively, now yields the desired result.
\end{proof}{}

\begin{lemma}\label{f}Let $n \ge 0$ be a fixed integer. Let $X$ be a pointed $n$-connected higher stack. Then there are natural maps $$\pi_i (X) \to \pi_{i+1} (\Sigma X)$$which are isomorphisms for $i \le 2n$ and a surjection for $i=2n+1$.
\end{lemma}{}

\begin{proof}
The natural maps above are induced via the map $X \to \Omega \Sigma X$ arising from adjunction. The desired result then follows from the classical Freudenthal suspension theorem for spaces. 
Indeed, let $\underline{X}$ be the fibre of the map $X \to \tau_{\le n}^{\mathrm{pre}}X$ computed in the $\infty$-category of presheaves of spaces. 
Then $\underline{X}$ is an $n$-connected presheaf.
By the classical Freudenthal suspension theorem for spaces, it follows that the natural maps $\pi_i (\underline{X}) \to \pi_{i+1}(\Sigma \underline{X})$ of presheaves are isomorphisms for $i \le 2n$ and a surjection for $i = 2n+1$. Applying sheafification and noting that sheafification of $\underline{X}$ recovers $X$ now yields the desired result.
\end{proof}{}

\begin{remark}
    Note that \cref{f} does not imply the Freudenthal suspension theorem in unipotent homotopy theory (\cref{freudenthalsusp}): even if $X$ is an affine stack, $\Sigma X$ can be far from being an affine stack, e.g., it can have nonrepresentable homotopy sheaves (see \cref{lentil}).
\end{remark}{}

\begin{lemma}\label{sohard}
Let $X$ be a pointed connected affine stack over $k$ such that $\pi_1(X)$ is trivial and $H^i(X, \cO)$ is finite-dimensional for all $i$. Let $Y$ be a pointed connected stack.
Then any pullback diagram 
\begin{center}
 \begin{tikzcd}
Y' \arrow[r, "f'"] \arrow[d, "g'"] & * \arrow[d, "g"] \\
Y \arrow[r, "f"]                   & X               
\end{tikzcd}
\end{center}
induces a pullback diagram
\begin{center}
    \begin{tikzcd}
\UU(Y') \arrow[r, "f'"] \arrow[d, "g'"] & * \arrow[d, "g"] \\
\UU(Y) \arrow[r, "f"]                   & \UU(X).               
\end{tikzcd}
\end{center}{}

\end{lemma}{}

\begin{proof}
This is a consequence of \cref{sohard139}. Indeed, note that there are natural maps \begin{equation}\label{sohard11}
    k \otimes_{R\Gamma(X, \cO)} R\Gamma(Y, \cO) \to k \coprod_{R\Gamma(X, \cO)} R\Gamma (Y, \cO) \to R\Gamma(Y', \cO),
\end{equation}{}
where the middle term denotes the pushout in $\mathrm{DAlg}^{\mathrm{ccn}}_k$. 
By \cref{sohard139}, the composite map is an isomorphism. 
Replacing $Y$ by $\UU(Y)$ (which is again naturally a pointed connected higher stack) and $Y$ by $Y'' \colonequals \UU(Y) \times_{X} \left \{*\right \}$, another application of \cref{sohard139} shows that the composition $k \otimes_{R\Gamma(X, \cO)} R\Gamma(\UU(Y), \cO) \to k \coprod_{R\Gamma(X, \cO)} R\Gamma (\UU(Y), \cO) \to R\Gamma(Y'', \cO)$ is an isomorphism. However, since $\UU(Y)$ is an affine stack, it follows that the latter map $k \coprod_{R\Gamma(X, \cO)} R\Gamma (\UU(Y), \cO) \to R\Gamma(Y'', \cO)$ is an isomorphism. 
Thus, the former map $$k \otimes_{R\Gamma(X, \cO)} R\Gamma(\UU(Y), \cO) \to k \coprod_{R\Gamma(X, \cO)} R\Gamma (\UU(Y), \cO)$$ is an isomorphism as well.
Using the natural isomorphism $R\Gamma(\UU(Y), \cO)\simeq R\Gamma(Y, \cO)$ and the maps in \cref{sohard11}, we now see that the map $\displaystyle{k \coprod_{R\Gamma(X, \cO)} R\Gamma (Y, \cO) \to R\Gamma(Y', \cO)}$ in \cref{sohard11} is an isomorphism. This proves the desired statement.
\end{proof}{}

\begin{lemma}\label{advv}
Let $X$ be a pointed connected affine stack over a field $k$ such that $\pi_1(X)$ is trivial and $H^i (X, \cO)$ is a finite-dimensional $k$-vector space for all $i \ge 0$. 
Then $H^i (\tau_{\le n} X, \cO)$ is finite-dimensional for all $n$ and all $i$.
\end{lemma} 

\begin{proof}
The claim is immediate for $n \le 1$ by our assumptions on $X$. Our goal is to inductively prove the following statement.

\begin{itemize}
    \item For all $u \ge 2$, $\mathrm{Hom}(\pi_u(X), \mathbf{G}_a)$ is finite-dimensional and for all $i \ge 0$, $H^i (\tau_{\le u-1}X, \mathscr{O})$ is finite-dimensional.
\end{itemize}{}
Note that \cref{postnikov} and \cref{gait1} yield an isomorphism $\mathrm{Hom} (\pi_2 (X), \GG_a) \simeq H^2 (X, \cO)$. This implies base case $u=2$ of our desired assertion. For the inductive step, let us assume that the assertion is proven for $u=n$.
We first establish that $H^i (\tau_{\le n}X, \cO)$ is finite-dimensional for all $i$.
The pullback diagram
\begin{center}
    \begin{tikzcd}
{K(\pi _n(X), n)} \arrow[d] \arrow[r] & * \arrow[d]                                    \\
\tau_{\le n}X \arrow[r]               & \tau_{\le n-1}X                                
\end{tikzcd}
\end{center}
gives a spectral sequence
\begin{equation}\label{E2spect12}
    E_2^{i,j} = H^i (\tau_{\le n-1} X, \mathscr{H}^j (K(\pi_n(X), n), \cO)) \implies H^{i+j} (\tau_{\le n} X, \cO).
\end{equation}Here, $\mathscr{H}^j (K(\pi_n(X), n), \cO)$ is a quasi-coherent sheaf on $\tau_{\le n-1} X$ whose pullback to the point is the vector space ${H}^j (K(\pi_n(X), n), \cO)$.
Now we invoke the following lemma.
\begin{lemma}\label{finiteness}
Let $G$ be a commutative unipotent group scheme over $k$ such that $\mathrm{Hom}(G, \GG_a)$ is finite-dimensional.
Then $H^j (K(G,n), \cO)$ is finite-dimensional for all $j$.
\end{lemma}{}
\begin{proof}
We proceed by induction on $n$.
When $n=1$, the finiteness statement is proven later in \cref{genfunc}; it is a consequence of the fact that the dual of the Hopf algebra $\cO(G)$ (in the sense of \cref{dualizinggpsch}) is a commutative noetherian local ring and the Betti numbers of any commutative noetherian local ring are finite.
\vspace{2mm}

For the inductive step, let us assume that the statement is proven for a fixed $n$.
Denote by $G^i$ the $i$-fold self-product of $G$.
Since $K(G^i,n)$ is an affine stack for all $i \ge 0$, the K\"unneth formula shows that $H^j (K(G^i,n),\cO)$ is finite-dimensional.
By fpqc descent along $* \to K(G,n+1)$, we obtain a spectral sequence
\[ E^{i,j}_1 = H^j(K(G^i,n),\cO) \implies H^{i+j}(K(G,n+1),\cO). \]
Since all terms in the $E_1$-page are finite-dimensional by the induction hypothesis, the spectral sequence gives the desired statement for $n+1$.
\end{proof}{}
Since $\tau_{\le 1}X \simeq *$ by assumption, it follows from \cref{heartrep} that $\mathscr{H}^j (K(\pi_n(X), n), \cO) \simeq H^j(K(\pi_n(X),n),\cO) \otimes_k \cO$ as quasi-coherent sheaves on $\tau_{\le n-1}X$. On the other hand, \cref{finiteness} above shows that $H^j(K(\pi_n(X),n),\cO)$ is a finite-dimensional $k$-vector space.
Therefore, all the terms in the $E_2$-page of the spectral sequence \cref{E2spect12} are finite-dimensional by the induction hypothesis. This shows that $H^i (\tau_{\le n} X, \mathscr{O})$ is finite-dimensional as desired.
\vspace{2mm}

Now we move on to proving that $\mathrm{Hom}(\pi_{n+1}(X), \GG_a)$ is finite-dimensional.
The pullback diagram
\begin{center}
     \begin{tikzcd}
{K(\pi_{n+1}(X), n+1)} \arrow[r] \arrow[d] & * \arrow[d]   \\
\tau_{\le n+1}X \arrow[r]                 & \tau_{\le n}X
\end{tikzcd}
 \end{center}{}
gives as before a spectral sequence
$$E_2^{i,j}= H^i (\tau_{\le n}X, \mathscr{H}^j (K(\pi_{n+1}(X), n+1), \cO)) \implies H^{i+j} (\tau_{\le n+1}X, \cO).$$
It is enough to prove that $H^{n+1}(K(\pi_{n+1}(X), n+1), \cO)$ is finite-dimensional:
by \cref{gait1}, this implies that $\mathrm{Hom}(\pi_{n+1}(X), \GG_a)$ is finite-dimensional.
Assume, for the sake of contradiction, that $H^{n+1}(K(\pi_{n+1}(X), n+1), \cO)$ is infinite-dimensional.
Since $\pi_1(X)$ is trivial, that implies that $E_2^{0,n+1}$ is infinite-dimensional.
By \cref{gait1}, $\mathscr{H}^j (K(\pi_{n+1}(X), n+1), \cO))=0$ for $0<j \le n$, which gives $E_r^{i,j}=0$ for $0<j \le n$. Also, $E_r^{i,j}=0$ for $i<0$ or $j<0$.
This implies that $E_r ^{0,n+1} = E_2^{0, n+1}$ for $2 \le r \le n+2$. On the $(n+2)$-nd page of the spectral sequence we have a potentially nonzero differential $$ E_{n+2}^{0, n+1} \to E_{n+2}^{n+2, 0}.$$
Note that $E_{2}^{n+2, 0}= H^{n+2}(\tau_{\le n} X, \cO)$ is finite-dimensional, since we have already shown that $H^i (\tau_{\le n}X, \cO)$ is finite-dimensional for all $i$. Therefore, $E_{n+2}^{n+2,0}$ is finite-dimensional.
Since we assumed that $E_2^{0,n+1}$ is infinite-dimensional, this shows that $E_{n+3}^{0,n+1}$ is also infinite-dimensional.
Moreover, $E_{n+3}^{0,n+1} = E_{\infty}^{0,n+1}$ and so the latter term is infinite-dimensional as well.
By \cref{postnikov}, $H^{n+1}(\tau_{\le n+1}X, \cO) \simeq H^{n+1}(X, \cO)$, which is finite-dimensional by assumption. Since $E_{\infty}^{0,n+1}$ is now a subquotient of the finite-dimensional vector space $H^{n+1}(\tau_{\le n+1}X, \cO)$, we reach a contradiction that finishes the proof.
\end{proof}{}

\begin{proposition}\label{freu1}
Let $X$ be a pointed connected stack over a field $k$ and $n \ge 1$ be a fixed integer such that the following conditions hold:
\begin{enumerate}
    \item $\pi_1(X)$ is trivial and $H^i (X, \cO)$ is finite-dimensional for all $i$; 
    \item $\tau_{\le n} X$ is an affine stack.
\end{enumerate}{}
Then $\pi_{n+1} (\UU(X)) \simeq ((\pi_{n+1}(X))^{\mathrm{uni}})^{\mathrm{ab}}$.
\end{proposition}{}

\begin{proof}
From the universal property of mapping into pointed connected $r$-truncated affine stacks for $r \ge 0$, we obtain $$\tau_{\le r}\UU(X) \simeq \tau_{\le r}\UU(\tau_{\le r} X).$$ 
By \cref{advv}, $H^i(\tau_{\le n} \UU(X), \cO)$ is finite-dimensional for all $i$. This implies that the vector space $H^i(\tau_{\le n} \UU(\tau_{\le n} X), \cO)$ is finite-dimensional for all $i$. However, $\tau_{\le n} X$ is already an affine stack by assumption. Thus, we obtain that $H^i (\tau_{\le n}X, \cO)$ is finite-dimensional for all $i$.
By \cref{sohard}, the pullback diagram
\begin{center}
   
\begin{tikzcd}
{K(\pi_{n+1}(X), n+1)} \arrow[rr] \arrow[d] &  & * \arrow[d]    \\
\tau_{\le n+1} X \arrow[rr]                 &  & \tau_{\le n} X
\end{tikzcd}
\end{center}{}
induces the following pullback diagram:
\begin{center}
\begin{tikzcd}
{\UU(K(\pi_{n+1}(X), n+1))} \arrow[rr] \arrow[d] &  & * \arrow[d]    \\
\UU(\tau_{\le n+1} X) \arrow[rr]                 &  & \UU(\tau_{\le n} X).
\end{tikzcd}
\end{center}{}
Since $\UU(\tau_{\le n} X) \simeq \tau_{\le n} X$, we obtain $$\pi_{n+1}(\UU(K(\pi_{n+1}(X), n+1))) \simeq \pi_{n+1}(\UU(\tau_{\le n+1} X)) \simeq \pi_{n+1}(\UU(X)).$$
Now, invoking \cref{whoknew} finishes the proof.
\end{proof}{}

\begin{proposition}[Freudenthal suspension theorem for affine stacks]\label{freudenthalsusp1} Let $X$ be a pointed connected affine stack over $k$ and $n \ge 0$ be a fixed integer such that the following conditions hold:

\begin{enumerate}
    \item $H^i (X, \cO)$ is finite-dimensional for all $i \ge 0$.
    
    \item $X$ is $n$-connected, i.e., $\pi_i (X)$ is trivial for $i \le n$.
\end{enumerate}{}
Let us consider the affine stack $\UU(\Sigma X)$. Then there are natural maps of group schemes
$$\pi_i (X) \to \pi_{i+1} (\UU(\Sigma X))$$ 
which are isomorphisms for $i \le 2n$ and a surjection for $i= 2n+1$.
\end{proposition}{}

\begin{proof}
First, we note that by adjunction there are natural maps $X \to \Omega \Sigma X \to \Omega \UU(\Sigma X)$ which induces the desired maps on homotopy group schemes.
\vspace{2mm}

For any pointed connected stack $Y$ and any integer $ r \ge 1$, by the universal property of mapping into pointed connected $r$-truncated affine stacks, we have
\begin{equation}\label{truncate-twice}
\tau_{\le r}\UU(Y) \simeq \tau_{\le r}\UU(\tau_{\le r} Y).
\end{equation}
By \cref{f}, the natural maps $\pi_i (X) \to \pi_{i+1}(\Sigma X)$ are isomorphisms for $i \le 2n$.
Thus, the sheaves of homotopy groups of the stack $\tau_{\le 2n+1} (\Sigma X)$ are all representable by unipotent affine group schemes, since the same holds for $X$. 
Therefore, $\tau_{\le 2n+1} \Sigma X$ is an affine stack (see \cref{thmoftoen}).
By \cref{truncate-twice}, it follows that $\tau_{\le 2n+1}(\UU(\Sigma X)) \simeq \tau_{\le 2n+1} (\Sigma X)$, which yields the first part in the claim of the proposition, i.e., $\pi_i(X) \to \pi_{i+1} (\UU(\Sigma X))$ is an isomorphism for $i \le 2n$. 
\vspace{2mm}

Now we proceed to prove that $\pi_{2n+1}(X) \to \pi_{2n+2} (\UU(\Sigma X))$ is a surjection. By \cref{f}, we have a surjection $\pi_{2n+1} (X) \to \pi_{2n+2} (\Sigma X) $ of fpqc sheaves. In fact, by Whitehead's EHP sequence \cite[Prop.~2.9]{EHP1} (\textit{cf.}~\cite{EHP2}), it extends to an exact sequence of fpqc sheaves
\begin{equation}\label{EHP}
     \pi_{2n+2} (X \wedge X) \to \pi_{2n+1}(X) \to \pi_{2n+2} (\Sigma X) \to 0.
\end{equation}
Since $X$ is $n$-connected, \cref{c} shows that the first map in \cref{EHP} induces a bilinear map which identifies with the Whitehead product $W_{n,n} \colon \pi_{n+1} (X) \times \pi_{n+1}(X) \to \pi_{2n+1} (X)$ on homotopy groups. 
When $n=0$, one can directly see that it identifies with the commutator map. Further, when $n=0$, it follows from \cref{whoknew} that $\tau_{\le 2} (\UU(\Sigma X)) \simeq K(((\pi_2(\Sigma X)^{\w{uni}})^{\w{ab}}), 2)$.
Recall that since $X$ is an affine stack, $\pi_1 (X)$ is a unipotent affine group scheme. Using the description of the map $W_{0,0}$ and the exact sequence \cref{EHP}, one sees that the natural map $\pi_1(X) \to  \pi_2(\UU (\Sigma X)) \simeq \pi_1(X)^{\mathrm{ab}}$ is surjective (see \cref{useinfreu});
this yields the surjectivity claim in the case $n=0$. Note that in the $n=0$ case, we do not need the finiteness assumption on $H^i (X, \cO)$. In the following, we assume $n \ge 1$.
\vspace{2mm}

Let us set $X' \colonequals \Sigma X$. Then $X'$ is a pointed connected stack.
Since $X$ is connected, $X'$ is $1$-connected and therefore $H^0 (X, \cO)=k$ and $H^1(X, \cO)=0$.
By adjunction and \cref{usefullemma113}, we have $H^{i}(X', \cO) \simeq H^{i-1}(X, \cO)$ for $i \ge 2$. Further, as noted before, $\tau_{\le 2n+1} X'$ is an affine stack and $\pi_1(X')$ is trivial. By \cref{freu1}, we obtain that $\pi_{2n+2} (\UU(X')) \simeq ((\pi_{2n+2}(X'))^{\w{uni}})^{\w{ab}}$.
We know from \cref{EHP} that $\pi_{2n+2}(X')$ receives a natural surjection from $\pi_{2n+1}(X)$, which is a commutative unipotent affine group scheme since $X$ is an affine stack and $n \ge 1$. Therefore, by \cref{tttttt}, the natural map $\pi_{2n+1}(X) \to ((\pi_{2n+2} (X'))^{\w{uni}})^{\w{ab}} \simeq \pi_{2n+2}(\UU(X')) \simeq \pi_{2n+2}(\UU(\Sigma X))$ is a surjection, which finishes the proof.
\end{proof}

\begin{corollary}[Freudenthal suspension theorem in unipotent homotopy theory]\label{freudenthalsusp} Let $X$ be a pointed connected affine stack over $k$ and $n \ge 0$ be a fixed integer such that the following conditions hold:
\begin{enumerate}
    \item $H^i (X, \cO)$ is finite-dimensional for all $i \ge 0$.
    
    \item  $\pi_i^{\mathrm{U}} (X)$ is trivial for $i \le n$.
\end{enumerate}{}
Then there are natural maps of group schemes
$$\pi_i^{\mathrm{U}} (X) \to \pi_{i+1}^{\mathrm{U}} (\Sigma X)$$ 
which are isomorphisms for $i \le 2n$ and a surjection for $i= 2n+1$.
\end{corollary}{}
\begin{proof}
    Follows by applying \cref{freudenthalsusp1} to the unipotent homotopy type $\UU(X)$ of $X$.
\end{proof}{}

\begin{remark}
In the notation of \cref{Milne} and the proof of \cref{freudenthalsusp1}, it also follows that $$(\pi_{2n+1} (X))/ [\mathrm{Im}(W_{n,n})])^{\mathrm{alg}} \simeq \pi_{2n+2} (\UU(\Sigma X)),$$ since the left-hand side is already commutative and unipotent for $n \ge 1$. Thus, the bilinear map $W_{n,n}$ gives an explicit way to understand $\pi_{2n+2}(\UU(\Sigma X))$.
\end{remark}

\begin{example}\label{lentil}
We point out that even if $X$ is a pointed connected affine stack, it is not true in general that $\Sigma X$ is again an affine stack.
Indeed, when $X = BG$ for some commutative unipotent group scheme $G$, one can compute that $\pi_3 (\Sigma BG)$ is not representable in general:
The Hopf fibre sequence 
\[ BG \to \Sigma (BG \wedge BG) \to \Sigma BG \]
proves that $\pi_3 (\Sigma BG) \simeq \pi_3 (\Sigma (BG \wedge BG))$. Since $BG \wedge BG$ is $1$-connected, it follows that $\pi_3 (\Sigma (BG \wedge BG)) \simeq \pi_2 (BG \wedge BG)$.
\Cref{c} then yields $\pi_2 (BG \wedge BG) = G \otimes_{\mathbf{Z}} G$. 
\end{example}{}

We end this subsection with the following calculation that will be useful in \cref{homotopyformalsphere}.

\begin{example}[Explicit computation of the Whitehead product]\label{computewhite}
Let $G$ be an abelian group. By using the Hopf fibre seqeuence $ BG \to \Sigma (BG \wedge BG) \to \Sigma BG$, one again computes that $\pi_2(\Sigma BG) \simeq G$ and $\pi_3 (\Sigma BG)  \simeq G \otimes G$. 
Our goal is to explicitly compute the Whitehead product $W_{2,2} \colon G \times G \to G \otimes G. $ By bilinearity, this corresponds to a map (which we again denote as) $W_{2,2} \colon G \otimes G \to G \otimes G$. In order to calculate this, we will freely make use of the results from \cite{WH1}; in particular, the machinery of the ``universal quadratic functor" introduced by Whitehead \cite[Ch.~2]{WH1}. 
For an abelian group $A$, let $Q_2(A)$ denote Whitehead's universal quadratic functor, which is the free abelian group on the generators $\gamma(a)$ subject to the following relations:
\begin{enumerate}
    \item $\gamma(-a)= \gamma(a) $ for all $a \in A$.
    \item $\gamma(a+b+c) - \gamma (a+b) - \gamma (b+c) - \gamma (a+c) + \gamma(a) + \gamma(b) + \gamma(c) = 0$ for all $a,b,c \in A$.
\end{enumerate}{}
We let $\beta(a,b) \colonequals \gamma (a+b) - \gamma(a) - \gamma(b) \in Q_2(A)$. There is a natural map $A \otimes A \to Q_2 (A)$ which is determined by sending $a \otimes b$ to $\beta (a,b)$.
If $X$ is an $r$-truncated space and $r \ge 1$, then Whitehead constructs a map $Q_2 (\pi_2(X)) \to \pi_3(X)$ (\textit{cf.}~\cite[Sec.~2.3]{MW1}) such that the composite map $\pi_2(X) \otimes \pi_2(X) \to Q_2 (\pi_2(X)) \to \pi_3(X)$ identifies with the Whitehead product. Applying this to the case of $X = \Sigma BG$, the map $q \colon Q_2 (G) \to G \otimes G$ in this case is given by $\gamma (g) \mapsto g \otimes g$ (\textit{cf}. \cite[Prop.~4.10]{Lod1}). The composite map $W_{2,2} \colon G \otimes G \to G \otimes G$ is given by $g \otimes h \mapsto q(\beta(g,h)) = q(\gamma(g+h) - \gamma (g) - \gamma(h)) = (g+h) \otimes (g+h) - g \otimes g - h \otimes h = g \otimes h + h \otimes g$.
\end{example}{}

\newpage

\section{\'Etale homotopy groups via unipotent homotopy theory}
In this section, we prove a profiniteness result (see \cref{profiniteunipotent}) for the unipotent homotopy group schemes that were introduced in \cref{posto1}. 
In \cref{posto2}, we show that the unipotent homotopy group schemes refine the \'etale homotopy groups due to Artin--Mazur \cite{etalehomo} in the $p$-adic context.
\Cref{posto3} and \cref{posto4} are devoted to developing some notions and techniques that are necessary for proving these results.

\subsection{Some Frobenius semilinear algebra}\label{posto3}

We begin by discussing some Frobenius semilinear algebra that will become important later.
\begin{lemma}\label{lem:Frobenius-invariants}
Let $k$ be a perfect field of characteristic $p>0$ and let $\varphi$ be the Frobenius on $k$.
Let $V$ be a finite-dimensional $k$-vector space and $F$ be a Frobenius semilinear endomorphism of $V$.
Then there is a decomposition $V = V_s \oplus V_n$ into $F$-invariant subspaces such that $F_s \colonequals \restr{F}{V_s}$ is bijective and $F_n \colonequals \restr{F}{V_n}$ is nilpotent. 
When $k$ is algebraically closed, the natural map $V^{F_s=\id}_s \otimes_{\FF_p} k \to V_s$ is an $F_s$-equivariant isomorphism, where the endomorphism $F_s$ on the domain is given by $\id \otimes \varphi$.
\end{lemma}
Since the statement is classical (see, e.g.,\ \cite[p.~38]{MR0098097}), we only give a brief sketch of an argument here for the convenience of the reader.
\begin{proof}
Since $k$ is perfect, $\im F^N$ is a $k$-subspace of $V$ for all $N \in \ZZ_{\ge 0}$.
Moreover, as $V$ is finite-dimensional, there exists an integer $N > 0$ such that $\im F^{N + 1} = \im F^N$.
Then $V_s \colonequals \im F^N$ and $V_n \colonequals \Ker F^N$ give the desired decomposition.
\vspace{2mm}

For the second part, without loss of generality, we assume that $F$ is bijective and thus $V_s = V$ and $F_s = F$.
We can find a fixed vector $\tilde{v} \neq 0$ of $F$ as follows:
Start with any $v \in V$.
Since $V$ is finite-dimensional, we can choose a minimal $n$ such that there exists a linear relation of the form $F^{n+1}(v) = \sum^n_{i=0} \lambda_i F^i(v)$ for some $\lambda_i \in k$.
Let $\alpha_n \neq 0$ be a root of the separable polynomial $P(z) = z - \sum^n_{i=0} \lambda^{p^{n-i}}_iz^{p^{n-i+1}}$ and define $\alpha_{i-1}$ recursively as the unique $p$-th root of $\alpha_i - \alpha^p_n \lambda_i$ for $1 \le i \le n$.
Then $\tilde{v} \colonequals \sum^n_{i=0} \alpha_i F^i(v) \neq 0$ by the minimality of $n$ and $F(\tilde{v}) = \lambda_0\alpha^p_n + \sum^n_{i=1}(\lambda_i\alpha^p_n + \alpha^p_{i-1})F^i(v) = \tilde{v}$.
Using $\tilde{v}$, the statement follows by taking quotients and induction on the dimension of $V$.
\end{proof}
The first part of \cref{lem:Frobenius-invariants} has the following consequence.
\begin{corollary}\label{Frobenius-iterates}
Let $k$ be a perfect field of characteristic $p > 0$.
Let $V$ be a finite-dimensional $k$-vector space and $F$ a Frobenius semilinear endomorphism of $V$.
Equip $V^\flat \colonequals \varprojlim_F V$ and $V_\perf \colonequals \varinjlim_F V$ with their natural $k$-vector space structures coming from the isomorphisms $\varprojlim_F k \simeq k$ and $k \simeq \varinjlim_F k$, respectively.
Then the composition of the natural maps $V^\flat \to V \to V_\perf$ is an isomorphism of $k$-vector spaces.
\end{corollary}
\begin{proof}
It suffices to show that $V^\flat \to V \to V_\perf$ is bijective.
By \cref{lem:Frobenius-invariants}, there exists a natural decomposition $V = V_s \oplus V_n$ such that $F$ is bijective on $V_s$ and nilpotent on $V_n$. Our claim follows from the observation that $\varprojlim_F V_n \simeq 0$ as well as $\varinjlim_F V_n \simeq 0$.
\end{proof}

\begin{corollary}\label{ch}
Let $k$ be an algebraically closed field and $V$ be a finite-dimensional vector space over $k$ equipped with a Frobenius semilinear operator $F$. Then the natural map $V ^{F= \mathrm{id}} \otimes_{\FF_p} k \to \varinjlim_{F} V$ is an isomorphism compatible with the Frobenius semilinear operators.
\end{corollary}{}
\begin{proof}
This follows from combining \cref{lem:Frobenius-invariants} and \cref{Frobenius-iterates}.
\end{proof}{}

Let $\mathrm{DAlg}^{\mathrm{ccn}}_k$ denote the $\infty$-category from \cref{derivedvscosimp}.
There is a natural functor $\mathrm{DAlg}^{\mathrm{ccn}}_k \to D(k)_{\le 0}$. Any object $R \in \mathrm{DAlg}^{\mathrm{ccn}}_k$ admits a semilinear Frobenius endomorphism $F$. We let $R^{F=\id}$ denote the Frobenius fixed points of $R$, i.e., the equalizer of $F \colon R \to R$ and $\id \colon R \to R$, which is naturally an object of $\mathrm{DAlg}^{\mathrm{ccn}}_{\FF_p}$.
\begin{lemma}\label{cosimplicial-Frobenius-invariants}
Let $k$ be a perfect field of characteristic $p > 0$.
Let $R \in \mathrm{DAlg}^{\mathrm{ccn}}_k$ such that $H^i(R)$ is a finite-dimensional $k$-vector space for all $i \ge 0$.
Then the composition of the natural maps
\[ \alpha \colon R^\flat \colonequals \varprojlim_F R \to R \to R_\perf \colonequals \varinjlim_F R \]
is an isomorphism.
\end{lemma}
\begin{proof}
It suffices to prove that $H^i(\alpha)$ is an isomorphism for all $i \ge 0$. Let us fix an $i$. The Frobenius $F$ induces a Frobenius semilinear endomorphism of $H^i(R)$, which we denote again by $F$.
First, we will show that the natural maps $H^i(R^\flat) \to \varprojlim_F H^i(R)$ and $\varinjlim_F H^i(R) \to H^i(R_\perf)$ are isomorphisms.
\vspace{2mm}

To prove the above assertion, we recall that the natural functor $\mathrm{DAlg}^{\mathrm{ccn}}_k \to D(k)$ preserves limits and filtered colimits. Since $H^{i-1}(R)$ is finite-dimensional, we have $R^1\varprojlim_F H^{i-1}(R) = 0$; thus, the Milnor sequence
\[ 0 \to R^1\varprojlim_F H^{i-1}(R) \to H^i(\varprojlim_F R) \to \varprojlim_F H^i(R) \to 0 \]
shows that the natural map $H^i(R^\flat) = H^i(\varprojlim_F R) \to \varprojlim_F H^i(R)$ is an isomorphism.
On the other hand, taking cohomology commutes with filtered colimits, so $\varinjlim_F H^i(R) \to H^i(R_\perf)$ is an isomorphism as well, yielding the desired assertion. Now \cref{Frobenius-iterates} implies directly that $H^i(\alpha)$ is an isomorphism. This finishes the proof.
\end{proof}

\begin{definition}
Let $k$ be a perfect field of characteristic $p>0$. We let $k_{\sigma}[F]$ denote the Frobenius twisted polynomial ring over $k$ in the variable $F$, i.e., we have the relation $F\cdot c= c^p \cdot F$.
\end{definition}{}

\begin{definition}\label{torsion}
We call a left $k_{\sigma} [F]$-module $M$ \textit{torsion} if every $m \in M$ is contained in a $k_{\sigma} [F]$-submodule $V_m$ such that $V_m$ is finite-dimensional as a $k$-vector space.
\end{definition}{}

\begin{remark}
Note that if $M$ is a torsion $k_{\sigma}[F]$-module, then every element $m \in M$ is indeed killed by an element of $k_{\sigma}[F]$. Thus, \cref{torsion} is consistent with the usual definition of torsion modules. 
Every finitely generated $k_{\sigma}[F]$-submodule of a torsion module is a finite-dimensional $k$-vector space on which $F$ acts as a semilinear operator.
\end{remark}{}

\begin{lemma}\label{torsininv}
Let $k$ be an algebraically closed field of characteristic $p > 0$.
Let $M$ be a torsion $k_{\sigma}[F]$-module. Then the natural map $M^{F= \mathrm{id}} \otimes_{\mathbf{F}_p} k \to \varinjlim_{F}M$ is an isomorphism of $k_{\sigma}[F]$-modules, where the left-hand side is equipped with the $k_{\sigma}[F]$-module structure coming from the Frobenius map on $k$.
\end{lemma}{}

\begin{proof}
We note that any finitely generated $k_{\sigma}[F]$-submodule of $M$ is finite-dimensional as a $k$-vector space. Since $M$ is a filtered colimit of its finitely generated submodules and taking $F$-invariants commutes with filtered colimits, we are done by \cref{ch}.
\end{proof}{}

For any object $R \in \mathrm{DAlg}^{\mathrm{ccn}}_k$, the Frobenius endomorphism $F$ of $R$ gives $H^i (R)$ for each $i \ge 0$ naturally the structure of a $k_{\sigma}[F]$-module. 
Further, since $k$ is perfect, $R_{\w{perf}} = \varinjlim_{F} R$ can naturally be viewed as an object of $\mathrm{DAlg}^{\mathrm{ccn}}_k$.

\begin{proposition}\label{jacket1}
Let $k$ be an algebraically closed field of characteristic $p > 0$. 
Let $R \in \mathrm{DAlg}^{\mathrm{ccn}}_k$ be such that $H^i (R)$ is a torsion $k_{\sigma}[F]$-module for each $i \ge 0$.
Then the natural map $R^{F= \mathrm{id}} \otimes_{\FF_p} k \to R_{\mathrm{perf}}$ in $\mathrm{DAlg}^{\mathrm{ccn}}_k$ is an isomorphism. 
\end{proposition}{}

\begin{proof}
It is enough to prove that the maps $H^i (R^{F= \w{id}}) \otimes_{\mathbf{F}_p} k \to H^i (R_{\w{perf}}) \simeq \varinjlim_F H^i (R)$ induced by the natural map $R^{F= \mathrm{id}} \otimes_{\FF_p} k \to R_{\mathrm{perf}}$ in $\mathrm{DAlg}^{\mathrm{ccn}}_k$ are isomorphisms for $i \ge 0$. We note that since $H^i (R)$ is torsion and filtered colimits are exact, the maps $F- \w{id} \colon H^i (R) \to H^i (R)$ are surjective \cite[\href{https://stacks.math.columbia.edu/tag/0A3L}{Tag~0A3L}]{stacks}.
It follows that for all $i \ge 0$, we have $H^i (R ^{F=\id}\otimes_{\FF_p} k) \simeq H^i (R)^{F=\id} \otimes_{\FF_p} k$. Therefore, since $H^i(R)$ is torsion, we are done by \cref{torsininv}.
\end{proof}{}

\subsection{Foundations on profinite group schemes}\label{posto4}
In this section, we discuss the theory of profinite group schemes. 
In \cref{profinitenesschar}, we establish a necessary and sufficient criterion for profiniteness of a unipotent group scheme $G$ in terms of the $k_{\sigma}[F]$-module $\mathrm{Hom}_k(G, \GG_a)$. In \cref{F-representation}, we introduce the notion of $F$-representations, which, roughly speaking, are representations of a group scheme $G$ (over a field of positive characteristic) on vector spaces equipped with a Frobenius semilinear operator that are required to satisfy the appropriate compatibility conditions. We give some examples to show that the notion of $F$-representations is a very naturally occurring notion in positive characteristic algebraic geometry since the vector spaces arising as cohomology groups are often equipped with a Frobenius semilinear operator (see \cref{chicago13}). In \cref{fixedprof}, we prove a certain ``permanence of finiteness" property of $F$-representations of profinite group schemes. These notions and results are used in \cref{profiniteunipotent}, where we prove the profiniteness of unipotent homotopy group schemes in positive characteristic.
\vspace{2mm}

Throughout this subsection, we work over a field $k$ of characteristic $p>0$. 
All group schemes are assumed to be affine (but not commutative) unless otherwise mentioned.

\begin{definition}\label{profdef}
A group scheme $G$ is called \emph{profinite} if $G$ is the inverse limit of some inverse system of finite group schemes. 
\end{definition}{}

\begin{remark}\label{proobj}
Let $H$ be a finite group scheme and let $\left \{G_i \right\} _{i \in I}$ be an inverse system of finite group schemes.
Since $H$ is affine and finitely presented, $\Hom(\varprojlim_{i \in I}G_i, H ) \simeq \varinjlim_{i \in I} \Hom(G_i, H)$. Therefore, the category of pro-objects of the category of finite group schemes embeds fully faithfully into the category of all affine group schemes. However, this property is quite special to affine algebraic geometry; for example, the category of profinite sets (i.e., the category of pro-objects of finite sets) does not embed fully faithfully into the category of all sets. This dichotomy occurs because while finite group schemes are cocompact objects in the category of affine group schemes, finite sets are not cocompact objects in the category of sets. 
\end{remark}{}

\begin{definition}
An affine group scheme $G$ over $k$ is called \emph{pro-\'etale} if $G$ is pro-\'etale as a scheme over $\Spec k$.
\end{definition}{}

\begin{remark}
Alternatively, one could also define a pro-\'etale group scheme over a field to be the inverse limit of an inverse system of finite \'etale group schemes.
By \cite[\href{https://stacks.math.columbia.edu/tag/092Q}{Tag~092Q}]{stacks}, this gives a notion equivalent to the definition above. In particular, it follows that every pro-\'etale group scheme over a field is profinite.
\end{remark}{}

\begin{remark}[Profinite group cohomology as cohomology of group schemes]\label{compare}Let $T$ be a finite set. Let $S(T) \colonequals \Spec (\coprod_T k)$. This construction extends to give a fully faithful, product preserving functor $S$ from the category of profinite sets to the category of affine schemes over $k$. More precisely, a profinite set $T$ is mapped to $S(T) \colonequals \Spec C(T,k)$, where $C(T,k)$ is the ring of locally constant functions from $T$ to the field $k$. It follows that if $G$ is a profinite group, then $S(G)$ is an affine pro-\'etale group scheme. Further, letting $C^i (G,k)$ denote the ring of locally constant functions from $G^i$ to $k$, we have an isomorphism $C^i (G, k) \simeq \cO(S(G))^{\otimes {i}}$, where $\cO(S(G))$ denotes global sections on the group scheme $S(G)$. Therefore, one notes that the complex computing $R\Gamma(B (S(G)), \cO)$ obtained via faithfully flat descent along $* \to S(G)$ is identical to the complex computing continuous group cohomology. Thus one obtains a natural isomorphism $H^i_{\w{cont}} (G, k) \simeq H^i (B(S(G)), \cO)$.
\end{remark}{}

\begin{example}[Free pro-$p$-group scheme]\label{cold91}
If $G$ is the free pro-$p$-finite group on $g$ generators, then $S(G)$ will be called the free pro-$p$ group scheme over $k$ on $g$ generators. One can similarly define the free commutative pro-$p$-group scheme on $g$ generators as well, which is simply given by the group scheme $\mathbf{Z}_p^{\oplus g}$.
\end{example}{}

\begin{construction}[Maximal profinite quotient]\label{I9}
For every group scheme $G$, the category $\mathcal{C}_G$ of arrows of group schemes $G \to H$ with the property that $H$ is a finite group scheme and the map is a surjection is cofiltered.
Since the trivial group scheme is a finite group scheme, $\mathcal{C}_G$ is nonempty.
In order to show that $\mathcal{C}_G$ is cofiltered, let us first start with two surjective maps of group schemes $f_1 \colon G \to H_1$ and $f_2 \colon G \to H_2$ whose targets are finite group schemes.
In the category of group schemes, we can form kernels as well as quotients by closed normal subgroup schemes. We note that $K \colonequals \Ker(f_1) \cap \Ker(f_2)$ is a closed normal subgroup scheme of $G$.
We have a surjection $G \to G/K$ with factors the maps $f_1$ and $f_2$.
We claim that $G/K$ is a finite group scheme. To see this, we note that $\Ker(f_1)/K$ is a closed subgroup scheme of $H_2$ and is therefore a finite group scheme. On the other hand, we have an extension of group schemes
$$0 \to \Ker(f_1)/K \to  G/K \to H_1 \to 0.$$ This shows that $G/K$ is a finite group scheme since it is an extension of finite group schemes. Lastly, we need to check that if $G \to H_1$ and $G \to H_2$ are two objects of $\mathcal{C}_G$, and $u,v \colon H_2 \to H_1$ are two maps, then there is an object $G \to H_3$ and maps $w \colon H_3 \to H_2 $ such that $uw= vw$. However, one observes that in $\mathcal{C}_G$, there is at most one arrow between any two objects; thus this condition is automatically satisfied. This shows that $\mathcal{C}_G$ is cofiltered.
\vspace{2mm}

There is a natural functor from $\mathcal{C}_G$ to the category of affine group schemes that sends $(G \to H)$ to $H$.
The limit over this diagram indexed by $\mathcal{C}_G$ in the category of affine group schemes will be called the \textit{maximal profinite quotient of $G$} and denoted by $G^{\w{pft}}$. We have a natural surjection $G \to G^{\w{pft}}$.
\end{construction}{}

\begin{proposition}\label{whywhy}
Let $G$ be an affine commutative group scheme. Then any map from $G$ to a profinite group scheme factors uniquely through $G \to G^{\mathrm{pft}}$.
\end{proposition}{}

\begin{proof} 
To prove the assertion, it is enough to show that for a finite group scheme $H$, any map $G \to H$ factors as $G \to G^{\w{pft}} \to H$ for a uniquely determined map $G^{\w{pft}} \to H$. Since the image of the map $G \to H$ is a finite group scheme that $G$ surjects onto, by construction, there is a natural map $G^{\w{pft}} \to H$ such that there is a factorization $G \to G^{\w{pft}} \to H$.
The uniqueness of the map $G^{\w{pft}} \to H$ follows immediately from the surjectivity of $G \to G^{\w{pft}}$.
\end{proof}{}

\begin{remark}\label{visa}
If $G$ is a pro-\'etale group scheme, then it follows that the absolute Frobenius on $G$ is a bijection. If $G$ is profinite, one can show that the converse is also true. Indeed, let $G \simeq \varprojlim G_i$, where each $G_i$ is a finite group scheme and let us assume that the absolute Frobenius is a bijection on $G$. One can further assume that all the transition maps $G_i \to G_j$ are surjective. This implies that the absolute Frobenius must induce an injection on the global sections of the $G_i$; 
consequently, the $G_i$ are reduced and thus must be \'etale. This shows that $G$ is pro-\'etale. \end{remark}{}

\begin{construction}[Maximal pro-\'etale quotient]\label{maximal-proetale}
 For every group scheme $G$, we can define the category $\mathcal{E}_G$ of arrows of group schemes $G \to H$ with the property that $H$ is an (affine) \'etale group scheme. Since the category of affine commutative \'etale group schemes forms an abelian subcategory of the category of affine commutative group schemes, one can argue as in \cref{I9} to check that $\mathcal{E}_G$ is cofiltered. There is a natural functor from $\mathcal{E}_G$ to the category of affine commutative group schemes that sends $(G \to H)$ to $H$. 
 The limit over this diagram indexed by $\mathcal{E}_G$ in the category of affine group schemes will be called the \textit{maximal pro-\'etale quotient of $G$} and denoted by $G^{\proet}$. We have a natural surjection $G \to G^{\proet}$.
\end{construction}{}

\begin{proposition}
Let $G$ be an affine commutative group scheme. Then any map from $G$ to an affine \'etale group scheme factors uniquely through $G \to G^{\proet}$.
\end{proposition}{}

\begin{proof}
Follows in a way similar to the proof of \cref{whywhy}.
\end{proof}{}

\begin{construction}[Perfection of group schemes]\label{group-scheme-perfection} Let $k$ be a field of chracteristic $p$ fixed as before. Let $G$ be an affine group scheme over $k$. Then $G$ is equipped with the absolute Frobenius map $F \colon G \to G$ (which is not $k$-linear). Let us define $$G^{\w{perf}}_{k_{\w{perf}}} \colonequals \varprojlim_{F} G.$$We note that $G^{\w{perf}}_{k_{\w{perf}}}$ is naturally a group scheme over $\Spec k_{\w{perf}}$. When $k$ is perfect, $G^{\w{perf}}_{k_{\w{perf}}}$ can be naturally viewed as an affine group scheme over $\Spec k$, which we will simply denote by $G^{\w{perf}}$ and call the \emph{perfection} of the group scheme $G$.
\vspace{2mm}

Alternatively, one can define $\varphi_* G$ to be a group scheme over $k$ where the structure map to $\Spec k$ is given by composing $G \to \Spec k$ with the absolute Frobenius $\varphi \colon \Spec k \to \Spec k$. Then $\varphi_* G$ is naturally equipped with the structure of a group scheme. Further, one has a $k$-linear Frobenius map $\varphi_* G \to G$, which is a morphism of group schemes. Iterating this Frobenius map gives an inverse system of group schemes over $k$, whose inverse limit is naturally isomorphic to $G^{\w{perf}}$. There is a natural map $G^{\w{perf}} \to G$ of affine group schemes over $k$.
\end{construction}{}

\begin{remark}
Note that the perfection of a group scheme need not be profinite in general. For example, $\GG_a^{\w{perf}}$ is not profinite.
\end{remark}{}
\begin{remark}\label{profet}
If $G$ is profinite, it follows that $G^{\w{perf}}$ is profinite as well. By \cref{visa}, $G^{\w{perf}}$ is actually a pro-\'etale group scheme in this case.
\end{remark}{}

\begin{remark}\label{limcolim}
Note that for any algebra $S$ of characteristic $p>0$, we can consider the \textit{tilt} of $S$ denoted as $S^\flat \colonequals \varprojlim_F S$, where $F \colon S \to S$ is the Frobenius map. If $S$ is a $k$-algebra over a perfect field $k$, then $S^\flat$ is naturally a $k$-algebra as well. Further, $\varinjlim_F S$ is also naturally a $k$-algebra and there is a $k$-algebra map $S^\flat \to \varinjlim_F S$. Let us now additionally assume that $S$ is a finite-dimensional $k$-algebra.
Then the map $S^\flat = \varprojlim_F S \to \varinjlim_F S$ is an isomorphism by \cref{Frobenius-iterates}.
Note that such an assertion is clearly false without the assumption that $S$ is finite-dimensional over $k$ as can be seen by taking $S = k[x]$.
\end{remark}{}

\begin{remark}\label{tired} For an affine group scheme $G$ over a perfect field $k$, we have constructed a natural map of affine group schemes $G ^{\w{perf}} \to G$. Let $\cO(G)$ denote the global sections of $G$. When $G$ is finite, by \cref{limcolim}, $\Spec \cO(G)^\flat$ has a group scheme structure over $k$, which we will denote by $G^\flat$. There is a natural map $G \to G^\flat$ of group schemes. Further, by \cref{limcolim}, the composition $G^\w{perf} \to G \to G^\flat$ is an isomorphism. This implies that there is a \textit{split} surjection $G \to G^{\w{perf}}$.

\end{remark}{}

\begin{proposition}\label{tired1}
Let $G$ be a profinite group scheme over a perfect field $k$. Then $G^{\mathrm{perf}}$ is isomorphic to the maximal pro-\'etale quotient of $G$.
\end{proposition}{}

\begin{proof} Let us first suppose that $G$ is a finite group scheme and show the claim. If $H$ is any \'etale group scheme over $k$, then $\cO(H)$ must be a perfect ring. Thus, any map $G \to H$ of group schemes must factor uniquely as $G \to G^\flat \to H$. By \cref{tired}, this implies that the map $G \to G^{\w{perf}}$ constructed before identifies with the maximal (pro-)\'etale quotient of $G$.
\vspace{2mm}

Now we return to the general case where we may assume that $G = \varprojlim_{i \in I} G_i$ for an inverse system $\left \{G_i \right\} _{i \in I}$ of finite group schemes. Let us suppose that $H$ is a finite \'etale group scheme. Then, by \cref{proobj}, we have $\Hom(\varprojlim_{i \in I} G_i, H) \simeq \varinjlim_{i \in I} \Hom(G_i, H) \simeq \varinjlim_{i \in I} \Hom(G_i^{\w{perf}}, H) \simeq \Hom(\varprojlim_{i \in I}G_i^{\w{perf}}, H)$. This constructs a natural bijection of sets $\Hom(G, H) \simeq \Hom(G^{\w{perf}}, H)$. Since $G^{\w{perf}}$ is pro-\'etale (\cref{profet}), we obtain a natural map $G \to G^{\w{perf}}$ which identifies as the maximal pro-\'etale quotient of $G$.
\end{proof}

\begin{corollary}If $G$ is a profinite group scheme over a perfect field $k$, then the map $G \to G^{\proet}$ is a split surjection.
\end{corollary}{}
\begin{proof}
Follows from \cref{tired1} and its proof. Indeed, the map to the maximal pro-\'etale quotient identifies with $G \to G^{\w{perf}}$ and a section is provided by the natural map $G^{\w{perf}} \to G$.
\end{proof}{}

Note that \cref{proobj} shows that a group scheme being profinite is a property and not any extra data. For the rest of this section, we investigate conditions on a group scheme that characterise profiniteness.

\begin{proposition}\label{tautop}
A group scheme $G$ is profinite (resp. pro-\'etale) if and only if every finite type quotient of it is a finite (resp. finite \'etale) group scheme. 
\end{proposition}{}

\begin{proof}
If all the finite type quotients of $G$ are finite (resp. finite \'etale), then it follows that $G$ is profinite (resp. pro-\'etale). The converse is a consequence of the following algebraic observation: a finitely generated subalgebra of an ind-finite (resp. ind-\'etale) algebra is finite (resp. finite \'etale).
\end{proof}{}

\begin{corollary}\label{piano}
If $G$ is a group scheme over a field $k$, then $G$ is profinite (resp. pro-\'etale) if and only if the base change of $G$ to a field extension of $k$ is profinite (resp. pro-\'etale).
\end{corollary}{}
\begin{proof}
Follows from \cref{tautop}.
\end{proof}{}

While working over perfect fields of characteristic $p>0$, it is desirable to obtain more functorial characterizations of profiniteness involving the Frobenius operator. To this end, we first prove a permanence property under the perfection operation. This also gives a converse to \cref{profet}.
\begin{proposition}
A group scheme $G$ over a perfect field $k$ is profinite if and only if $G^{\mathrm{perf}}$ is profinite.
\end{proposition}{}
\begin{proof}
If $G$ is profinite, then $G^{\w{perf}}$ is clearly also profinite. Conversely, for the sake of contradiction, let us assume that $G^{\w{perf}}$ is profinite while $G$ is not. In this case, there must exist a surjection $G \to G'$ where $G'$ is of finite type yet not finite.
Then $G^{\w{perf}} \to G'^{\w{perf}}$ is also a surjection. Since $G^{\w{perf}}$ is profinite, by an application of \cref{tautop}, $G'^{\w{perf}}$ is also profinite. This implies that $G'^{\w{perf}}$ is zero-dimensional as an affine scheme. Since $G'^{\w{perf}} \to G'$ induces a homeomorphism on the underlying (Zariski) topological spaces, we get that $G'$ is also zero-dimensional. Since $G'$ is of finite type over a field, this implies that $G'$ must be finite, which gives the contradiction and proves the converse.
\end{proof}{}

Combining the above proposition with \cref{visa}, we obtain the following characterization of profiniteness that Bhargav Bhatt pointed out to us.
\begin{corollary}\label{bhattob}
A group scheme $G$ over a perfect field $k$ is profinite if and only if $G^{\mathrm{perf}}$ is pro-\'etale. 
\end{corollary}{}

\begin{proposition}
Any closed subgroup, quotient or extension of profinite (resp. pro-\'etale) group schemes over a field $k$ is profinite (resp. pro-\'etale). 
\end{proposition}{}

\begin{proof}
We can assume by \cref{piano} that the base field $k$ is algebraically closed. If $k$ has characteristic $p>0$, we can apply the perfection functor from \cref{group-scheme-perfection} and \cref{bhattob} to reduce to only checking the assertion in the case of pro-\'etale group schemes. If $k$ is of characteristic zero, then the profinite group schemes are the same as pro-\'etale group schemes. The assertion in the case of pro-\'etale group schemes over an algebraically closed field can be checked directly and all of them follow from \cite[\href{https://stacks.math.columbia.edu/tag/0CKQ}{Tag~0CKQ}]{stacks}.
\end{proof}{}

\begin{remark}
Note that a group scheme $G$ being pro-\'etale is a condition on the scheme underlying $G$:
it means that the diagonal map $G \times_k G \to G$ is faithfully flat. Thus, \cref{bhattob} gives a characterization of profinite group schemes that can be expressed without talking about inverse systems as in \cref{profdef} or all finite type quotients as in \cref{tautop}. 
\end{remark}{}

In practice, \cref{bhattob} can be difficult to check. When the group scheme $G$ is unipotent, we give another characterization in \cref{profinitenesschar}, which is purely in terms of the Frobenius operators and linear algebra. This characterization will be used in our paper. First, let us note some lemmas.

\begin{lemma}\label{li}
Let $G$ be a finite type unipotent group scheme over an algebraically closed field $k$ such that $\dim G >0$. Then there is a surjection $G \to \GG_a$.
\end{lemma}{}

\begin{proof}
If $k$ has characteristic $0$, then the statement follows by considering a normal series of $G$ whose successive quotients are all $\GG_a$. From now on, we assume that $k$ has characteristic $p>0$. By considering the image of a large enough power of the Frobenius map, we can furthermore assume that $G$ is smooth.
\vspace{2mm}

Let us pick a maximal normal series $G = G_0 \supset G_1 \supset \ldots \supset G_i \supset G_{i+1} \supset \ldots = \left \{ * \right \}$, and an integer $i$ minimal with respect to the property that $G_i/G_{i+1}$ is $\GG_a$. By construction, $G_n/ G_{n+1}$ is a finite algebraic subgroup of $\GG_a$ for $n<i$. Therefore, the quotient $G/G_{i+1}$ is $1$-dimensional and $\GG_a$ is a closed normal subgroup of $G/G_{i+1}$.
Thus, we can reduce to the case where $G$ is $1$-dimensional and $\GG_a$ is a closed normal subgroup of $G$. Since the quotient $H \colonequals G/\GG_a$ is a unipotent group scheme and $\w{Aut}\,\GG_a = \GG_m$, we see that $\GG_a$ is central in $G$.
We note that $H$ is a finite discrete group scheme since we are over an algebraically closed field. Further, since $\GG_a$ is central in $G$, the commutator map $G \times G \to G$ factors through $H \times H \to G$. By a result of Baer \cite{Baer} (see \cite[\S~2, Appx.]{bor1}), there is a large enough $r$ such that image of the map $(H \times H)^{\times r} \to G^{\times r} \to G$ is the commutator subgroup $[G,G]$ of $G$.
Here, the last map is the $r$-fold multiplication map of $G$. Since $H$ is a finite group scheme, it follows that $[G,G]$ is zero-dimensional. Thus, the abelianization $G^{\w{ab}}$ of $G$ is also $1$-dimensional. Therefore, we may assume that $G$ is commutative. Now, we can also use the Verschiebung operator $V$ defined on $G$. Since $G$ is unipotent, $V^n=0$ for a suitably large $n$. Thus, since $G$ is $1$-dimensional, one obtains by devissage that $G/VG$ is also $1$-dimensional. Therefore, we can now reduce to the case where $G$ is commutative, $1$-dimensional, contains $\GG_a$ and is killed by $V$. Since $G$ is $1$-dimensional, commutative and killed by $V$, by the classification of commutative unipotent group schemes that are annihilated by the Verschiebung operator (see, e.g., \cite[IV, \S~3 Thm.~6.6]{MR0302656}), $\Hom(G, \GG_a)$ is infinite-dimensional as a $k$-vector space. Since $G/ \GG_a$ is a finite group scheme, it follows that there must exist a map $G \to \GG_a$ such that the composition $\GG_a \to G \to \GG_a$ is nonzero. However, any nonzero endomorphism of $\GG_a$ is a surjection, which finishes the proof.
\end{proof}{}

\begin{lemma}\label{annoy}
Let $G$ be a finite type unipotent group scheme over a field $k$ of characteristic $p>0$.
Then $\mathrm{Hom}(G, \GG_a)$ is a finite-dimensional $k$-vector space if and only if $G$ is a finite group scheme. 
\end{lemma}

\begin{proof}
Without loss of generality, we may assume that the base field $k$ is algebraically closed.
If $G$ is finite, then $\Hom(G, \GG_a)$ is clearly finite-dimensional. The converse follows from \cref{li} and the observation that $\Hom(\GG_a, \GG_a)$ is infinite-dimensional (spanned by powers of the Frobenius map).
\end{proof}{}

\begin{lemma}\label{piano2}
If $H$ is a finite type unipotent group scheme over $k$ such that $H^{\ab}$ is finite, then $H$ must be finite. 
\end{lemma}{}
\begin{proof}
Once again, we may assume that the base field $k$ is algebraically closed.
For the sake of contradiction, let us assume that $\dim H >0$, yet $H^{\w{ab}}$ is finite-dimensional. By \cref{li}, $H^{\w{ab}}$ must surject onto $\GG_a$, which gives a contradiction.
\end{proof}{}

\begin{example}
We point out that the unipotent assumption in \cref{piano2} is important. The lemma is false for all semisimple algebraic groups of positive dimension since their abelianization is trivial. For a concrete example, one may take $\mathrm{SL}_n$.
\end{example}{}

\begin{remark}
We recall that a unipotent group scheme is called split if it admits a subnormal series whose graded pieces are isomorphic to $\GG_a$. Note that if $G$ is a commutative unipotent group scheme of dimension $n$, then \cref{li} inductively implies that there is a surjection from $G$ to an $n$-dimensional, split unipotent group scheme. 
\end{remark}{}

\begin{lemma}\label{tvshow}
Let $G$ be a unipotent group scheme over $k$ of characteristic $p > 0$.
Then $G$ is profinite if and only if the abelianization $G^{\mathrm{ab}}$ of $G$ is profinite.
\end{lemma}{}

\begin{proof}
If $G$ is profinite, then $G^{\mathrm{ab}}$ is also profinite, since $G^{\mathrm{ab}}$ is a quotient of $G$. For the converse, we use \cref{tautop}. Let $H$ be a finite type quotient of $G$. Then $H^{\w{ab}}$ is a finite type quotient of $G^{\w{ab}}$. Since $G^{\w{ab}}$ is profinite, it follows that $H^{\w{ab}}$ is finite. By \cref{piano2}, this means that $H$ is finite and we are done.
\end{proof}{}

As a final preparation, we recall some notations. Let $k_{\sigma}[F]$ denote the Frobenius twisted polynomial ring over $k$ in the variable $F$; i.e., we have the relation $F\cdot c= c^p \cdot F$. There is a natural isomorphism $\Hom_k(\GG_a, \GG_a) \simeq k_{\sigma}[F]$. It induces a natural $k_{\sigma}[F]$-module structure on $\Hom_{k}(G, \GG_a)$. Moreover, the functor that sends a group scheme $G$ to $\Hom_{k}(G, \GG_a)$ gives an anti-equivalence between the category of commutative unipotent group schemes of finite type that are annihilated by the Verschiebung and the category of finitely generated left $k_{\sigma}[F]$-modules \cite[IV, \S~3, Cor.~6.7]{MR0302656}.
\vspace{2mm}

We recall (see \cref{torsion}) that a left $k_{\sigma} [F]$-module $M$ is called \textit{torsion} if every $m \in M$ is contained in a $k_{\sigma} [F]$-submodule $V_m$ such that $V_m$ is finite-dimensional as a $k$-vector space.

\begin{proposition}\label{profinitenesschar}
Let $G$ be a unipotent group scheme (not assumed to be commutative or of finite type) over a field $k$ of characteristic $p>0$. Then $G$ is profinite if and only if $\mathrm{Hom}(G, \GG_a)$ is a torsion $k_{\sigma}[F]$-module.
\end{proposition}{}

\begin{proof}
Note that if $G$ is profinite, then $\Hom(G, \GG_a)$ is indeed torsion as a $k_{\sigma}[F]$-module.
\vspace{2mm}

For the converse, by \cref{tvshow}, we can assume that $G$ is commutative. Let $H$ be a finite type quotient of $G$. We will show that $H$ is finite, which will imply the proposition by \cref{tautop}. Since $H^{\w{}}$ is commutative, it is equipped with the Verschiebung operator $V$. Note that $\GG_a$ is killed by $V$. Therefore, 
$\Hom(H, \GG_a) \simeq \Hom(H^{\w{}}/V H^{\w{}}, \GG_a)$. This implies that $\Hom(H, \GG_a)$ is a finitely generated $k_{\sigma}[F]$-module. Since $\Hom(H, \GG_a)$ is also a submodule of the torsion $k_{\sigma}[F]$-module $\Hom(G, \GG_a)$, it follows that $\Hom(H, \GG_a)$ is a finite-dimensional $k$-vector space. Now \cref{annoy} implies that $H$ must be finite, so we are done.\end{proof}{}

For the remainder of this section, we record some basic properties of $k_{\sigma}[F]$-modules that will be useful to us later on. Note that a $k_{\sigma}[F]$-module is simply a $k$-vector space $V$ equipped with a Frobenius semilinear endomorphism $F$. We will consider $V$-representations of a group scheme $G$ that respects the operator $F$.

\begin{lemma}\label{eatsan1}
Let $k$ be a perfect field of characteristic $p>0$.
\begin{enumerate}
   \item If $M$ is a finitely generated, torsion free $k_{\sigma }[F]$-module, then $M$ is free.
   \item If $M$ is a finitely generated $k_{\sigma}[F]$-module, then $M$ is a direct sum of a torsion module and a free module.
\end{enumerate}{}
\end{lemma}{}

\begin{proof}
The proof follows in a way similar to the case of $k[X]$-modules;
\textit{cf.}~\cite[\S~3.8]{MR0008601}.
\end{proof}{}

\begin{lemma}\label{torsioncolimit}
Let $k$ be a perfect field and $V$ be a $k$-vector space equipped with a Frobenius semilinear operator $F$. Then $V$ is torsion as a $k_{\sigma}[F]$-module if and only if $\varinjlim_{F} V$ is torsion as a $k_{\sigma}[F]$-module.\end{lemma}{}

\begin{proof}
If $V$ is torsion as a $k_{\sigma}[F]$-module, then the same is true for $\varinjlim_{F} V$. For the converse, we show that if $V$ has an element that is not torsion, then $\varinjlim_{F} V$ is not a torsion $k_{\sigma}[F]$-module. By \cref{eatsan1}, we can assume that $V= k_{\sigma}[F]$ with its natural structure as a $k_{\sigma}[F]$-module. However, then $\varinjlim_{F} k_{\sigma}[F] \simeq k_{\sigma}[F, F^{-1}]$ is not a torsion module, finishing the proof.
\end{proof}{}

\begin{remark}[Tensor product of Frobenius semilinear operators]\label{tensorfrobmod}
If $V_1$ and $V_2$ are two $k$-vector spaces equipped with Frobenius semilinear operators $F_1$ and $F_2$, then $V_1 \otimes_k V_2$ is naturally equipped with a Frobenius semilinear operator that sends $v_1 \otimes v_2$ to $F_1(v_1) \otimes F_2 (v_2)$. We will denote this operator on $V_1 \otimes_k V_2$ by $F_1 \otimes F_2$. When $k$ is perfect, we have $\varinjlim_{F_1 \otimes F_2}V_1 \otimes_k V_2 \simeq \varinjlim_{F_1} V_1 \otimes_{k} \varinjlim_{F_2} V_2$.
\end{remark}{}

\begin{remark}\label{torsionkF}
If $V_1$ and $V_2$ are torsion $k_{\sigma}[F]$-modules, then $V_1 \otimes_k V_2$ is also a torsion $k_{\sigma}[F]$-module, where the $k_{\sigma}[F]$-module structure on the latter is given by the description in \cref{tensorfrobmod}.
\end{remark}{}

\begin{definition}[$F$-representation]\label{F-representation}
Let $k$ be a perfect field and $V$ be a $k$-vector space equipped with a Frobenius semilinear endomorphism $F$. Let $G$ be a group scheme over $k$ and let $r \colon G \to \w{GL}_V$ be a representation.
Then $r$ determines a map $\rho \colon V \to \cO(G) \otimes_{k} V$, where $\cO(G)$ is the Hopf algebra of global sections of $G$.
We say that $r$ defines an \textit{$F$-representation} if the following diagram commutes:

    \begin{equation}\label{F-rep1}
    \begin{tikzcd}
V \arrow[d, "F"] \arrow[rr,"\rho"] &  & \cO(G) \otimes_k V \arrow[d, "\text{Frob} \otimes F"] \\
V \arrow[rr,"\rho"]                &  & \cO(G) \otimes_k V                                   
\end{tikzcd}
\end{equation}{}
Here, $\w{Frob}$ denotes the Frobenius map on $\cO(G)$. In such a situation, we will use $(r, G, V, F)$ to denote the associated $F$-representation.
\end{definition}{}

\begin{remark}\label{F-repgeo}
Let us explain the notion of an $F$-representation from a more geometric perspective.
First, we note that if $\sigma$ denotes the Frobenius on $k$, then the data of a Frobenius semilinear operator is naturally equivalent to the data of a map $\sigma^* V \to V$. The data of a representation $V$ of the group scheme $G$ over $k$ can be viewed as the data of a quasi-coherent sheaf $\mathscr{V}$ on $BG$ equipped with an isomorphism $u^* \mathscr{V} \simeq V$, where $u\colon \Spec k \to BG$ is the natural map. Further, if the underlying $k$-vector space $V$ is equipped with a Frobenius semilinear operator $F$, then the representation $V$ of $G$ from above is an $F$-representation if there is a map $T \colon \w{Frob}^* \mathscr{V} \to \mathscr{V}$ such that when $T$ is pulled back along the natural map $u \colon \Spec k \to BG$, it naturally recovers the operator $F$ on $V$ under the isomorphism $u^* \mathscr{V} \simeq V$. Note that if there exists a map $T$ as above, it is uniquely determined; therefore, the map $T$ really captures the property that the map $\sigma^* V \to V$ is equivariant with respect to the group scheme $G$ and the Frobenius on $G$, which we have concretely spelt out in \cref{F-rep1}.
\end{remark}{}

\begin{remark}
Let $\mathscr{V}$ denote the quasi-coherent sheaf on $BG$ associated with an $F$-representation $V$ of $G$  as described in \cref{F-repgeo}. Then the cohomology groups $H^i (BG, \mathscr{V})$ for $i \ge 0$ are naturally equipped with the structure of $k_{\sigma}[F]$-modules.
\end{remark}{}

\begin{remark}
If $(r_1, G, V_1, F_1)$ is an $F_1$-representation and $(r_2, G, V_2, F_2)$ is an $F_2$-representation, then the direct sum $V_1\oplus V_2$ equipped with the semilinear operator $F_1 \oplus F_2$ is an $F_1 \oplus F_2$-representation of $G$. Also, $V_1 \otimes_{k} V_2$ equipped with the semilinear operator $F_1 \otimes F_2$ is naturally an $F_1 \otimes F_2$-representation of $G$.
\end{remark}{}

\begin{construction}[Perfection of $F$-representations]\label{perfFrep} Let $G$ be a group scheme over a perfect field $k$. Let $V$ be a $k$-vector space equipped with a Frobenius semilinear operator $F$. Further, let us assume that $V$ is equipped with the structure of an $F$-representation of $G$, denoted as $r \colon G \to \w{GL}_V$. In this case, we can consider the perfection $G^{\w{perf}}$ of $G$ (see \cref{group-scheme-perfection}), whose underlying ring of functions is given by taking the colimit over the Frobenius on $\cO(G)$. Let us denote $V_{\w{perf}} \colonequals \varinjlim_{F} V$, equipped with the structure of a $k$-vector space (coming via the isomorphism $k \simeq k_{\w{perf}}$, since $k$ is perfect).
Further, $V_{\w{perf}}$ is canonically equipped with a Frobenius semilinear operator that we denote as $F_{\w{perf}}$. By using the commutative diagram in \cref{F-rep1}, one sees that $V_{\w{perf}}$ is naturally an $F_{\w{perf}}$-representation of $G^{\w{perf}}$, which we denote by $r_{\w{perf}} \colon G^{\w{perf}} \to \w{GL}_{V_{\w{perf}}}$; we will say that $(r_{\w{perf}}, G^{\w{perf}}, V_{\w{perf}}, F_{\w{perf}})$ is the perfection of $(r, G, V, F)$.
\end{construction}{}
Before we proceed further, we record some examples showing that $F$-representations are ubiquitous in positive characteristic representation theory and algebraic geometry.

\begin{example}[Trivial representation]
The trivial $1$-dimensional representation $V$ of a group scheme $G$ over a perfect field $k$ is naturally equipped with the structure of an $F$-representation, where the Frobenius semilinear operator on the vector space $V = k$ is given by the absolute Frobenius $\sigma \colon k \to k$.
\end{example}{}

\begin{example}[Regular representation]
If $G$ is a group scheme over a perfect field $k$, then the ring of regular functions $\cO(G)$ has a natural semilinear operator $F$ given by the absolute Frobenius on $\cO(G)$. This equips the regular representation $\cO(G)$ of $G$ with the structure of an $F$-representation.
\end{example}{}

\begin{example}[Group action on geometric objects]\label{chicago13}
Let $G$ be an affine group scheme over a perfect field $k$ and let $X$ be any qcqs scheme over $k$ equipped with an action $G \times X \to X$ of $G$. For a fixed $i \ge 0$, the cohomology group $H^i (X, \cO)$ is naturally equipped with a Frobenius semilinear operator which we denote by $F$. Further, $H^i (X, \cO)$ is naturally equipped with the structure of an $F$-representation of $G$. This example remains valid even when $X$ is replaced by any higher stack $Y$ that satisfies $R\Gamma (G \times Y, \cO)  \simeq R\Gamma (Y, \cO) \otimes_{k} \cO(G)$. Examples of these types will play an important role in this paper.
\end{example}{}

\begin{proposition}\label{finitenessFmod}
Let $G$ be a profinite group scheme over a perfect field $k$. Let $\mathscr{V}$ be the quasi-coherent sheaf on $BG$ attached to an $F$-representation $(\rho, G, V, F)$. If $V$ is a torsion $k_{\sigma}[F]$-module, then $H^i (BG, \mathscr{V})$ is a torsion $k_{\sigma}[F]$-module for all $i \ge 0$.
\end{proposition}{}

\begin{proof}
Since $G$ is profinite, it follows that $\cO(G)$ is a torsion $k_{\sigma}[F]$-module where the action of $F$ is induced by the absolute Frobenius on $\cO(G)$. Since $V$ is also a torsion $k_{\sigma}[F]$-module, using \cref{torsionkF} and faithfully flat descent along $\Spec k \to BG$, we see that $H^i(BG, \mathscr{V})$ is computed as cohomology of a complex of torsion $k_{\sigma}[F]$-modules. This gives the claim.
\end{proof}{}

\begin{proposition}\label{fixedprof}
Let $G$ be a profinite unipotent group scheme over a perfect field $k$. Let $V$ be a $k_{\sigma}[F]$-module with the structure of an $F$-representation of $G$. If $V^G$ is a torsion $k_{\sigma}[F]$-module, then $V$ must be a torsion $k_{\sigma}[F]$-module. 
\end{proposition}{}
\begin{proof}

By \cref{torsioncolimit}, in order to prove the statement we can pass to the perfection of the $F$-representation in the sense of \cref{perfFrep}. Therefore, we can without loss of generality assume that $G$ is a perfect group scheme. \vspace{2mm}

For the sake of contradiction, assume that $V$ is not $k_{\sigma}[F]$-torsion.
Let $V_{\w{tor}} \subset V$ be the $k_{\sigma}[F]$-submodule of torsion elements of $V$.
By construction, $V/V_{\w{tor}}$ is a nonzero torsion free $k_{\sigma}[F]$-module.
We claim that $V_{\w{tor}}$ is a $F$-subrepresentation of $G$.
\vspace{2mm}

Let $\rho \colon V \to \cO(G) \otimes_k V$ be the natural map associated with the $F$-representation.
Showing that $V_{\w{tor}}$ is an $F$-subrepresentation amounts to showing that $\rho (V_{\w{tor}}) \subseteq \cO(G) \otimes_k V_{\w{tor}}$.
To that end, let $v \in V_{\w{tor}}$ and $\rho(v) = \sum_{i \in I} g_i \otimes v_i$, where $g_i \in \cO(G)$, $v_i \in V$ and $I$ is a finite set. Without loss of generality, one can assume that the $g_i$'s are linearly independent. Since $V$ is an $F$-representation, the map $\rho$ is a map of $k_{\sigma}[F]$-modules, where the $k_{\sigma}[F]$-module structure on the target is given by tensor product (see \cref{tensorfrobmod} and \cref{F-representation}, \cref{F-rep1}). In particular, it preserves torsion elements. That is, the set $\left \{\sum_{i \in I} \w{Frob}^n (g_i) \otimes F^n(v_i) \suchthat n \in \mathbf{N} \right \} $ spans a finite-dimensional subspace of $\cO(G) \otimes_k V$, which we will denote by $P$. We claim that the span of the set $\left \{ F^n(v_i) \suchthat i \in I,\, n \in \mathbf{N} \right \} $, which we denote as $T$, must be a finite-dimensional subspace of $V$. This would imply that $v_i \in V_{\w{tor}}$, so that $V_{\w{tor}}$ is indeed an $F$-subrepresentation of $G$. \vspace{2mm}

Indeed, to see the above claim, let us use $W$ to denote the span of $\left \{\w{Frob}^n (g_i) \suchthat n \in \mathbf{N},\, i \in I \right \}$. Then there is a natural injection of vector spaces $P \to W \otimes_k T$. Let us assume on the contrary that $T$ is infinite-dimensional. Then there exists an $l \in I$ such that the set $\left \{ F^n (v_l) \suchthat n \in \mathbf{N}\right \}$ is linearly independent. Let $T_{l,n}$ be the span of $F^n (v_l)$ for a fixed $n$. Therefore, there exists a direct sum decomposition $T \simeq (\bigoplus_{n \in \mathbf N}T_{l,n}) \oplus T'$. This induces a projection map $P \to \bigoplus_{n \in \mathbf{N}}(W \otimes_k T_{l,n})$. Since $P$ is finite-dimensional, there is a fixed $n_0 \in \mathbf{N}$ such that the induced projection maps $\pi_{l,n} \colon P \to W \otimes_k T_{l,n} \simeq W$ are zero maps for all $n \ge n_0$. Let $u_{l,n} \colon T \to T_{l,n} \simeq k$ be the projection map induced by the direct sum decomposition of $T$ above. Then $\pi_{l,n} \left (\sum _{i \in I} \w{Frob}^n (g_i) \otimes F^n (v_i)\right ) = \sum_{i \in I} u_{l,n}(F^n(v_i)) \cdot \w{Frob}^n (g_i) \in W$. However, if $n \ge n_0$, we must have $\sum_{i \in I} u_{l,n}(F^n(v_i))\cdot \w{Frob}^n (g_i)=0$. Since the $g_i$'s are linearly independent and the group scheme $G$ is perfect, this implies that $u_{l,n}(F^n (v_i))=0$ for all $i \in I$. However, by construction, $u_{l,n}(F^n (v_l)) = 1$, which gives the desired contradiction and proves the claim. 
\vspace{2mm}

As discussed in \cref{F-repgeo}, the $F$-representations $V$ and $V_{\w{tor}}$ of $G$ correspond to quasi-coherent sheaves $\mathscr{V}$ and $\mathscr{V}_{\w{tor}}$ on $BG$.
We have a short exact sequence of quasi-coherent sheaves 
\[ 0 \to \mathscr{V}_{\w{tor}} \to \mathscr{V} \to \mathscr{V}/ \mathscr{V}_{\w{tor}} \to 0. \]
Since pullback along the natural cover $u \colon \Spec k \to BG$ is exact, we have $u^* (\mathscr{V}/\mathscr{V}_{\w{tor}}) \simeq V/V_{\w{tor}}$.
On cohomology, this induces an exact sequence of $k$-vector spaces
\[ 0 \to H^0 (BG, \mathscr{V}_{\w{tor}}) \to H^0 (BG, \mathscr{V}) \to H^0 (BG, \mathscr{V}/\mathscr{V}_{\w{tor}}) \to H^1 (BG, \mathscr{V}_{\w{tor}}). \]
Further, each of the maps is a map of $k_{\sigma}[F]$-modules. By \cref{finitenessFmod}, $H^0 (BG, \mathscr{V}_{\w{tor}})$ and $H^1 (BG, \mathscr{V}_{\w{tor}})$ are torsion $k_{\sigma}[F]$-modules. Note that since $G$ is unipotent, $H^0 (BG, \mathscr{V}/\mathscr{V}_{\w{tor}}) \ne 0$. Now, $H^0 (BG, \mathscr{V}/\mathscr{V}_{\w{tor}})$ is a nonzero $k_{\sigma}[F]$-submodule of $V/V_{\w{tor}}$ and therefore must be torsion free (\cref{eatsan1}). However, that implies that $V^G \simeq H^0 (BG, \mathscr{V})$ cannot be torsion, which contradicts our hypothesis. Thus, we are done.\end{proof}{}

\begin{example}
Note that \cref{fixedprof} fails if the group scheme $G$ is not assumed to be profinite. One can see this by taking $G= \GG_a$ and $V$ to be the regular representation of $\GG_a$ equipped with the structure of an $F$-representation. In this case, $V^G$ is $1$-dimensional and yet $V$ is not torsion as a $k_{\sigma}[F]$-module.
\end{example}{}

\subsection{Profiniteness of unipotent homotopy group schemes}\label{spri} 
In this section, we use our previous work to give a proof of the following result:

\begin{proposition}[Profiniteness theorem]\label{profiniteunipotent}
Let $(X,x)$ be a proper, cohomologically connected scheme over a field $k$ of characteristic $p>0$ equipped with a $k$-rational point $x \in X(k)$. Then the unipotent homotopy group schemes $\pi_i^{\mathrm{U}}(X)$ are profinite for all $i$.
\end{proposition}{}

When $i=1$, \cref{profiniteunipotent} recovers a result of Nori \cite[Prop.~IV.3]{MR682517}.
Nori deduced the profiniteness of $\pi_1^{\mathrm U}(X)$ by using the isomorphism $H^1 (X, \cO) \simeq \Hom(\pi_1^{\mathrm U}(X), \GG_a)$ and the fact that $H^1(X, \cO)$ is finite-dimensional since $X$ is proper.
However, a similar isomorphism does not hold for $i>1$;
this presents difficulties in proving the profiniteness of $\pi_i^{\mathrm{U}}(X)$ for $i>1$ in a similar manner. We will overcome some of these difficulties via an elaborate spectral sequence argument coming from the Postnikov tower and using the theory of quasi-coherent sheaves on affine stacks developed in our paper.
\vspace{2mm}

Another key step in the proof of the profiniteness of $\pi_i^{\mathrm{U}}(X)$ for $i \ge 0$ is the realization that in positive characteristic, for a unipotent group scheme $G$, while $\Hom(G, \GG_a)$ being a finite-dimensional $k$-vector space is enough to deduce profiniteness of $G$ (as proven and used by Nori), it is not a \emph{necessary} requirement. In \cref{profinitenesschar}, we deduced a necessary and sufficient criterion for profiniteness in terms of torsion $k_{\sigma}[F]$-modules. This suggests one to aim for a more flexible statement using torsion $k_{\sigma}[F]$-modules rather than finite-dimensional $k$-vector spaces. Further, in order to approach \cref{profiniteunipotent} by using the Postnikov tower as described in the previous paragraph, one would run into controlling certain representations of $\pi_1^{\mathrm U}(X)$. This is where the notion of $F$-representations play a crucial role---the $\pi_1^{\mathrm U}(X)$-representations one needs to control are all naturally $F$-representations and can be controlled systematically by using \cref{fixedprof}. However, we will also need a generalization of the notion of $F$-representations in the context of ``$n$-gerbes.'' To that end, we make the following definitions.

\begin{definition}[Frobenius modules on $n$-gerbes]\label{ias1}
Let $Y$ be a pointed, connected, $n$-truncated stack over a perfect field $k$. Let $\w{Frob}_Y$ denote the absolute Frobenius on $Y$. Let $\mathscr{V}$ be a quasi-coherent sheaf on $Y$ equipped with the data of a morphism $T \colon \w{Frob}_Y^* \mathscr{V} \to \mathscr{V}$. These data will be called a \emph{Frobenius module} on $Y$.
\end{definition}{}

\begin{remark}
Note that as explained in \cref{F-repgeo}, a Frobenius module on $Y$ as above may be thought of as an ``$F$-representation" of the group stack $\Omega Y$.
\end{remark}{}

\begin{definition}
Let $Y$ be a pointed, connected, $n$-truncated stack over a perfect field $k$ and let $T \colon \w{Frob}_Y^* \mathscr{V} \to \mathscr{V}$ be a Frobenius module on $Y$. We will call this Frobenius module \textit{torsion} if it is a torsion $k_{\sigma}[F]$-module after pulling back along the morphism $u \colon \Spec k \to Y$.
\end{definition}{}

\begin{example}
If $Y$ is as above, then the structure sheaf $\cO$ on $Y$, naturally viewed as a Frobenius module, is an example of a torsion Frobenius module.
\end{example}{}

Our first lemma is a generalization of  \cref{finitenessFmod}, which will be used in the proof of \cref{profiniteunipotent}.

\begin{lemma}\label{chicago12}
Let $X$ be a pointed, connected $n$-truncated affine stack over a perfect field $k$ such that $\pi_i (X)$ is representable by a profinite group scheme for all $i$. Let $\mathscr{V}$ be a torsion Frobenius module on $Y$. Then $H^i (X, \mathscr{V})$ is a torsion $k_{\sigma}[F]$-module for each $i \ge 0$.
\end{lemma}{}

\begin{proof}
We argue by induction on $n$; the case $n=1$ follows from \cref{finitenessFmod}. Let us assume that $n \ge 2$ and the result is proven for all $(n-1)$-truncated stacks. 
The pullback diagram
\begin{center}
\begin{tikzcd}
K(\pi _n(X), n) \arrow[d, "q"] \arrow[r] & * \arrow[d]                                    \\
\tau_{\le n}X \arrow[r, "p"]               & \tau_{\le n-1}X                               
\end{tikzcd}
\end{center} 
gives an $E_2$-spectral sequence $E_2^{i,j} = H^i (\tau_{\le n-1} X, \mathscr{H}^j (K(\pi_n(X), n), q^*\mathscr{V})) \implies H^{i+j} (\tau_{\le n} X, \mathscr{V})$ that naturally lives in the abelian category of $k_{\sigma}[F]$-modules (by functoriality). 
Note that $\mathscr{H}^j (K(\pi_n(X), n), q^*\mathscr{V}))$ is naturally a Frobenius module on $\tau_{\le n-1}X$ whose pullback to the point is the $k_{\sigma}[F]$-module $H^j (K(\pi_n(X), n), q^* \mathscr{V})$. Further, $q^* \mathscr{V}$ is a torsion Frobenius module on $K(\pi_n(X), n)$. Therefore, by induction and the spectral sequence, it would be enough to prove that $H^j(K(\pi_n(X), n), \mathscr{V')}$ is a torsion $k_{\sigma}[F]$-module for any torsion Frobenius module $\mathscr{V'}$ on $K(\pi_n(X), n)$. However, that follows by applying faithfully flat descent along $\Spec k \to K(\pi_n(X), n)$ and induction on $n$; \textit{cf.}~\cref{finiteness}.
\end{proof}{}

\begin{proposition}\label{chicago16}
Let $k$ be a field of characteristic $p>0$. Let $X$ be a pointed, cohomologically connected higher stack over $k$ such that $H^i (X, \cO)$ is a torsion $k_{\sigma}[F]$-module for all $i \ge 0$. Then $\pi_i^{\mathrm{U}}(X)$ is profinite for all $i$.
\end{proposition}{}

\begin{proof}
One can assume that $X$ is a pointed connected affine stack. Further, by \cref{piano}, we may assume that the field $k$ is algebraically closed. By \cref{profinitenesschar}, we need to prove that $\Hom(\pi_i (X), \GG_a)$ is a torsion $k_{\sigma}[F]$-module. Note that $\pi_0(X) = \left \{* \right \}$. When $i=1$, we have $H^1 (X, \cO) \simeq \Hom(\pi_1(X), \GG_a)$. 
Therefore, $\pi_i(X)$ is indeed profinite for $i=1$. We will use strong induction to prove the assertion that for all $u \ge 1$, $\mathrm{Hom}(\pi_u(X), \GG_a)$ is a torsion $k_{\sigma}[F]$-module. To this end, let us fix $n \ge 1$ and assume that the former assertion has been proven for all $u \le n$. We will check the assertion for $u=n+1$. We first observe that $H^i (\tau_{\le n}X, \cO)$ is a torsion $k_{\sigma}[F]$-module for all $i$. Indeed, since $\pi_u(X)$ is unipotent for all $u \ge 1$, by \cref{profinitenesschar}, it follows that for $u \le n$, the group scheme $\pi_u(X) \simeq \pi_u(\tau_{\le n} X)$ is profinite. Therefore, by \cref{chicago12} we obtain that $H^i (\tau_{\le n}X, \cO)$ is a torsion $k_{\sigma}[F]$-module for all $i$.
\vspace{2mm}

We proceed to show that $\Hom(\pi_{n+1}(X), \GG_a)$ is a torsion $k_{\sigma}[F]$-module. The diagram
\begin{center}
 \begin{tikzcd}
K(\pi_{n+1}(X), n+1) \arrow[r] \arrow[d,"q"] & * \arrow[d]   \\
\tau_{\le n+1}X \arrow[ r,"p"]                 & \tau_{\le n}X
\end{tikzcd}
 \end{center}{}
gives an $E_2$-spectral sequence in the abelian category of $k_{\sigma}[F]$-modules (by functoriality of the Leray spectral sequence)
$$E_2^{i,j}= H^i (\tau_{\le n}X, \mathscr{H}^j (K(\pi_{n+1}(X), n+1), \cO)) \implies H^{i+j} (\tau_{\le n+1}X, \cO).$$ Here, $\mathscr{H}^j (K(\pi_{n+1}(X), n+1), \cO)$ is naturally a Frobenius module on $\tau_{\le n}X$ whose pullback to the point is the $k_{\sigma}[F]$-module $H^j (K(\pi_{n+1}(X), n+1),  \cO)$. By \cref{heartrep}, it follows that $\mathscr{H}^j (K(\pi_{n+1}(X), n+1), \cO)$ corresponds to an $F$-representation of the unipotent group scheme $\pi_1(X)$ whose underlying $k_{\sigma}[F]$-module is $H^j (K(\pi_{n+1}(X), n+1),  \cO)$; the latter $F$-representation can also be described as arising from the natural action of $\pi_1(X)$ on $K(\pi_{n+1}(X), n+1)$ via \cref{chicago13}. \vspace{2mm}

It would be enough to prove that $H^{n+1}(K(\pi_{n+1}(X), n+1), \cO)$ is a torsion $k_{\sigma}[F]$-module, since by \cref{gait1}, that would imply that $\Hom(\pi_{n+1}(X), \GG_a)$ is a torsion $k_{\sigma}[F]$-module as well. We assume on the contrary that $H^{n+1}(K(\pi_{n+1}(X), n+1), \cO)$ is not torsion. By \cref{chicago14} and \cref{fixedprof} regarding $F$-representations, that implies that $E_2^{0,n+1}$ is not torsion. Note that by \cref{gait1}, $\mathscr{H}^j (K(\pi_{n+1}(X), n+1), \cO))=0$ for $0<j \le n$, which implies that $E_r^{i,j}=0$ for $0<j \le n$. 
Also, $E_r^{i,j}=0$ for $i<0$ or $j<0$. This implies that $E_r ^{0,n+1} = E_2^{0, n+1}$ for $2 \le r \le n+2$. 
On the $(n+2)$-nd page of the spectral sequence, we have a potentially nonzero differential $$ E_{n+2}^{0, n+1} \to E_{n+2}^{n+2, 0}.$$
Note that $E_{2}^{n+2, 0}= H^{n+2}(\tau_{\le n} X, \cO)$ is a torsion $k_{\sigma}[F]$-module, since $H^i (\tau_{\le n}X, \cO)$ is a torsion $k_{\sigma}[F]$-module for all $i$. Therefore, $E_{n+2}^{n+2,0}$ is torsion as well.
This implies that if one assumes that $E_2^{0,n+1}= E_{n+2}^{0,n+1}$ is not a torsion $k_{\sigma}[F]$-module, then $E_{n+3}^{0,n+1}$ would not be a torsion $k_{\sigma}[F]$-module either. We note that $E_{n+3}^{0,n+1} = E_{\infty}^{0,n+1}$ and so the latter term would also not be a torsion $k_{\sigma}[F]$-module. By \cref{postnikov}, it follows that $H^{n+1}(\tau_{\le n+1}X, \cO) \simeq H^{n+1}(X, \cO)$ and is therefore torsion by assumption. Since $E_{\infty}^{0,n+1}$ is a subquotient of the torsion $k_{\sigma}[F]$-module $H^{n+1}(\tau_{\le n+1}X, \cO)$, we reach a contradiction which finishes the proof.
\end{proof}{}

\begin{proof}[Proof of \cref{profiniteunipotent}]{}
Follows from \cref{chicago16}.
\end{proof}

\subsection{Recovering the \texorpdfstring{$p$}{p}-adic \'etale homotopy groups}\label{posto2}

In this section, we finally address the question of recovering the $p$-adic \'etale homotopy groups defined by Artin--Mazur \cite{etalehomo} from the unipotent homotopy group schemes introduced in our paper by proving the unipotent-\'etale comparison theorem (see \cref{unipotent-etale-comparison}). 
In order to do so, we give a different way to define the $p$-adic \'etale homotopy groups in \cref{defetalehom} and compare
them with the definition due to Artin--Mazur involving pro-objects.
\vspace{2mm}

To this end, let us fix a proper, cohomologically connected, pointed scheme $(X,x)$ over an algebraically closed field $k$ of characteristic $p > 0$.
 Let $\varphi$ denote the Frobenius automorphism of $k$.
 In this setup, we can consider the \'etale cohomology $R\Gamma_{\et}(X, \FF_p) \xrightarrow{x^*} R\Gamma_\et(\Spec k,\FF_p) \simeq \FF_p$ which can be viewed as an object of the $\infty$-category $(\mathrm{DAlg}^{\mathrm{ccn}}_{\mathbf{F}_p})_{/\mathbf{F}_p}$. We will omit the base point from the notation if this does not cause confusion.
 \begin{proposition}\label{discrete-etale-homotopy}
 In the above setup, the group schemes $\pi^{\mathrm U}_i(\Spec\,R\Gamma_{\et}(X, \FF_p) \otimes_{\FF_p} k)$ are pro-\'etale.
 \end{proposition}
 \begin{proof}
 By \cref{chicago16}, the group schemes $\pi^{\mathrm U}_i(\Spec\,R\Gamma_{\et}(X, \FF_p) \otimes_{\FF_p} k)$ are profinite. 
 By \cref{bhattob}, it is therefore enough to prove that the absolute Frobenius on them is an isomorphism. 
 This absolute Frobenius is induced from the morphism $R\Gamma_\et(X,\FF_p) \otimes_{\FF_p} \varphi^*k \to R\Gamma_\et(X,\FF_p) \otimes_{\FF_p} k$ which is further induced by the absolute Frobenius on $X$. However, by \cite[\href{https://stacks.math.columbia.edu/tag/03SN}{Tag~03SN}]{stacks}), the Frobenius morphism $R\Gamma_\et(X,\FF_p) \otimes_{\FF_p} \varphi^*k \to R\Gamma_\et(X,\FF_p) \otimes_{\FF_p} k$ is an isomorphism.
 \end{proof}
 In light of \cref{discrete-etale-homotopy} and the natural isomorphism $R\Gamma_{\et}(X, \FF_p) \otimes_{\FF_p} k \xrightarrow[]{\sim} R\Gamma_{\et} (X, k)$, we can make the following definition:

 \begin{definition}[$p$-adic \'etale homotopy group schemes]\label{defetalehom}
Let $(X, x)$ be a proper, cohomologically connected, pointed scheme over $k$. We define the $i$-th \emph{$p$-adic \'etale homotopy group scheme} of $X$ to be the group scheme
 \[ \pi^\et_i(X)_p \colonequals \pi_i(\Spec\,R\Gamma_{\et}(X,k)). \]
 \end{definition}

The theory of \'etale homotopy groups was originally introduced by Artin--Mazur in \cite{etalehomo}. Since \cref{defetalehom} defines $p$-adic \'etale homotopy group (schemes) in a different way, we now discuss the compatibility with the classical notion.
 \vspace{2mm}

 Recall that for a locally noetherian scheme $X$ over $k$, Artin--Mazur construct a pro-object in the homotopy category of $\mathcal{S}$.
 In fact, using Friedlander's construction \cite{etalfry}, this pro-object can be refined to a pro-object in the $\infty$-category $\mathcal{S}$.
 We denote this pro-object by $\Et(X) \in \Pro(\mathcal{S})$ and call it the \emph{\'etale homotopy type} of $X$.
 The singular cochain functor (see \cref{cold19}) $C^*( \,\cdot\, ,\FF_p) \colon \mathcal{S} \to (\mathrm{DAlg}^{\mathrm{ccn}}_{\FF_p})^\op$ naturally extends to a functor $C^*( \,\cdot\, ,\FF_p) \colon \Pro(\mathcal{S}) \to (\mathrm{DAlg}^{\mathrm{ccn}}_{\FF_p})^\op$; concretely, for a pro-object $(Y_i)_{i \in I}$, it is given by $C^*\bigl((Y_i),\FF_p) = \varinjlim_{i \in I} C^*(Y_i,\FF_p)$.
 For the \'etale homotopy type, we have a natural isomorphism $C^*(\Et(X), \FF_p) \simeq R\Gamma_{\et}(X, \FF_p)$.
 \vspace{2mm}

 Next, we consider the pro-$p$-finite completion $\Et(X)^\wedge_p$ of $\Et(X)$, which can be defined as follows.
 Let $\mathcal{S}^{p\mhyphen\mathrm{fc}}$ denote the category of $p$-finite spaces (see, e.g., \cite[Def.~2.4.1]{lurie2011derived}).
 The natural inclusion functor $\mathcal{S}^{p\mhyphen\mathrm{fc}} \hookrightarrow \mathcal{S}$ induces a functor $\Pro(\mathcal{S}^{p\mhyphen\mathrm{fc}}) \to \Pro(\mathcal{S})$ which admits a left adjoint $(\,\cdot\,)^\wedge_p \colon \Pro(\mathcal{S}) \to \Pro(\mathcal{S}^{p\mhyphen\mathrm{fc}})$, called pro-$p$-finite completion.
 We call the pro-$p$-finite completion $\Et(X)^\wedge_p$ of $\Et(X)$ the \emph{$p$-adic \'etale homotopy type}.
 The pro-$p$-finite group $\pi_i ^{\et,\, \mathrm{AM}} (X)_p \colonequals \pi_i (\Et(X)^\wedge_p)$ will be called the $i$-th \emph{Artin--Mazur $p$-adic \'etale homotopy group} of $X$.
 \vspace{2mm}

 We recall that the category of profinite groups embeds fully faithfully into the category of affine group schemes over $k$ (\textit{cf}.~\cref{compare}). Therefore, one may attach an affine group scheme to the profinite group $\pi_i ^{\et,\, \mathrm{AM}} (X)_p$. Finally, we are ready to compare $\pi_i^{\et} (X)_p$ and $\pi_i^{\et,\, \mathrm{AM}}(X)_p$.
 \begin{proposition}\label{compareetaleuni}
Let $X$ be a proper, cohomologically connected, pointed scheme over an algebraically closed field $k$ of characteristic $p>0$.
Then the $p$-adic \'etale homotopy group scheme $\pi_i^{\et}(X)_p$ from \cref{defetalehom} is the affine group scheme corresponding to the pro-$p$-finite group $\pi_i ^{\et,\, \mathrm{AM}}(X)_p$.
 \end{proposition}{}
 
 \begin{proof}
 The proof follows by an application of $p$-adic homotopy theory.
 By the exactness of filtered colimits, we have 
 \[ H^i (\Et(X), \mathbf{F}_p) \simeq \pi_0 \mathrm{Map}_{\Pro(\mathcal{S})}(\Et(X), K (\mathbf{Z}/p \mathbf{Z}, i)). \]
 Since the Eilenberg--MacLane space $K (\mathbf{Z}/p \mathbf{Z}, i)$ is $p$-finite, the universal property of pro-$p$-finite completion yields a natural isomorphism
 \[ C^* (\Et(X)^\wedge_p, \mathbf{F}_p) \simeq C^* (\Et(X), \mathbf{F}_p) \simeq R\Gamma_{\et}(X, \mathbf{F}_p). \]
 Therefore, it follows that $$C^* (\Et(X)^\wedge_p, k) \simeq C^* (\Et(X)^\wedge_p, \mathbf{F}_p) \otimes_{\mathbf{F}_p}k \simeq R\Gamma_{\et} (X,k).$$ Since $\Et(X)^\wedge_p$ is pro-$p$-finite, the claim now follows from \cite[Thm.~2.5.3]{Toe} and the Milnor sequences from \cite[Prop.~3.22]{MR100}.
 \end{proof}{}
 
 \begin{remark}[$p$-adic homotopy theory]\label{p-adichomotopytheory}
 Let $k$ be an algebraically closed field of characteristic $p>0$. 
 Then as a consequence of $p$-adic homotopy theory (\textit{cf.}~\cite{MR1243609,Mandell,lurie2011derived}),
 the functor
 \[ \Pro(\mathcal{S}^{p\mhyphen\mathrm{fc}}) \to (\mathrm{DAlg}^{\mathrm{ccn}}_k)^{\op}, \quad Y \mapsto C^* (Y, \mathbf{F}_p) \otimes_{\mathbf{F}_p} k \]
 is fully faithful.
 Indeed, the fact that the functor $\mathcal{S}^{p\mhyphen\mathrm{fc}} \to (\mathrm{DAlg}^{\mathrm{ccn}}_k)^{\op}$ is fully faithful follows from \cite[Thm.~2.5.1]{Toe}.
 To show the full faithfulness of $\Pro(\mathcal{S}^{p\mhyphen\mathrm{fc}}) \to (\mathrm{DAlg}^{\mathrm{ccn}}_k)^{\op}$, it would therefore be enough to show that for a $p$-finite space $X$, the object $C^*(X, k)$ is a compact object of $\mathrm{DAlg}^{\mathrm{ccn}}_k$.
 By the technique used in the proof of \cite[Lemma 2.5.11]{lurie2011derived}, it is enough to prove the claim when $X = K(\mathbf{Z}/p \mathbf Z, n)$. 
 In this case, one can use the Artin--Schreier sequence $0 \to \mathbf{Z}/p \mathbf{Z} \to \GG_a \xrightarrow[]{x \mapsto x^p -x} \GG_a \to 0$, which induces a fibre sequence of affine stacks and thus an isomorphism $C^* (K(\mathbf{Z}/p \mathbf{Z},n), k) \simeq k \coprod_{R\Gamma(K(\GG_a, n), \mathscr{ O})} R\Gamma(K(\GG_a, n), \mathscr{ O})$ in $\mathrm{DAlg}^{\mathrm{ccn}}_k$. However, $R\Gamma(K(\GG_a, n)) \simeq \Sym_k k[-n]$, which is a compact object. 
 This implies that $C^* (K(\mathbf{Z}/p \mathbf{Z},n), k)$ is also a compact object, as desired.
 \end{remark}{}

\begin{remark}
Note that there is a natural functor from $\Pro(\mathcal{S}^{p\mhyphen\mathrm{fc}}) \to \mathcal{S}$ which sends an object $Y$ to $\mathrm{Map}_{\Pro(\mathcal{S})}(*,Y)$.
This functor is not fully faithful unless one restricts to certain simply connected pro-$p$-finite spaces (see, e.g., \cite[Thm.~3.3.3]{lurie2011derived}).
Roughly speaking, the latter functor sends a pro-$p$-finite space $(Y_i)_{i \in I} \in \Pro(\mathcal{S}^{p\mhyphen\mathrm{fc}})$ to $\varprojlim_{i \in I} Y_i \in \mathcal{S}$; however, it does not remember the extra information that $\varprojlim_{i \in I} Y_i$ was an inverse limit of $p$-finite spaces. In some sense, this is analogous to forgetting the inverse limit topology on a profinite set and merely viewing it as a set. We point out that the notion of ``condensed anima" or ``pyknotic spaces" resolves this issue: 
the natural functor from $\Pro(\mathcal{S}^{p\mhyphen\mathrm{fc}})$ to condensed anima is fully faithful;
see \cite[Ex.~3.3.10]{pyknotic}.
One may therefore think of $\Et(X)^\wedge_p$ of a scheme $X$ as a condensed anima in a lossless manner.
\end{remark}{}

Now we will show that the unipotent homotopy group schemes recover the theory of $p$-adic \'etale homotopy groups as introduced by Artin--Mazur (see the discussion after \cref{defetalehom}).
In light of \cref{compareetaleuni}, it would be enough to recover $\pi_i^{\et}(X)_p$ from the unipotent homotopy group schemes $\pi^{\mathrm U}_i(X)$. To this end, we will begin by constructing natural maps $\pi^{\mathrm U}_i(X) \to \pi_i^{\et}(X)_p$ for each $i \ge 0$, which we will call \'etale comparison maps. In \cref{unipotent-etale-comparison}, we will show that one can recover $\pi_i^{\et}(X)_p$ from $\pi^{\mathrm U}_i(X)$ via these \'etale comparison maps in a precise manner.
\begin{construction}\label{etale-comparison}
The inclusion of \'etale sheaves $\FF_p \hookrightarrow \cO$ induces a map $R\Gamma_\et(X, \FF_p ) \to R\Gamma(X,\cO)$ in $\mathrm{DAlg}^{\mathrm{ccn}}_k$. Therefore, by adjunction, we obtain a map
\[ R\Gamma_\et(X,k) \simeq R\Gamma_\et(X,\FF_p)\otimes_{\FF_p}k \to R\Gamma(X,\cO) \]
in $\mathrm{DAlg}^{\mathrm{ccn}}_k$, or equivalently, a map $\Spec\,R\Gamma(X,\cO) \to \Spec\,R\Gamma_\et(X,k)$ of affine stacks over $k$.
Therefore, we get natural \emph{\'etale comparison maps}
\[ \pi^{\mathrm U}_i(X) \simeq \pi_i(\Spec\,R\Gamma(X,\cO)) \to \pi_i(\Spec\,R\Gamma_\et(X,k)) \simeq \pi^\et_i(X)_p \]
for all $i$.
\end{construction}
\begin{proposition}[Unipotent-\'etale comparison theorem]\label{unipotent-etale-comparison}
Let $(X,x)$ be a proper, cohomologically connected, pointed scheme over an algebraically closed field $k$ of characteristic $p > 0$. Then the \'etale comparison maps $\pi^{\mathrm U}_i(X) \to \pi^\et_i(X)_p$ from \cref{etale-comparison} identify $\pi^\et_i(X)_p$ with the maximal pro-\'etale quotient of $\pi^{\mathrm U}_i(X)$ from \cref{maximal-proetale} for all $i$.
\end{proposition}
\begin{proof}
Note that the absolute Frobenius on $X$ induces a Frobenius semilinear endomorphism $F$ on $R\Gamma(X,\cO)$. Using the Artin--Schreier sequence
\[ 0 \to \FF_p \to \cO_X \xrightarrow{1-F} \cO_X \to 0, \]
we see that $R\Gamma_\et(X,\FF_p) \simeq R\Gamma(X,\cO)^{F = \id}$.
Since $\Hh^i(X,\cO)$ is a finite-dimensional $k$-vector space for all $i \ge 0$, \cref{jacket1} implies that the natural map
\begin{equation}\label{sw1}
    R\Gamma_\et(X,k) \simeq R\Gamma_\et(X,\FF_p) \otimes_{\FF_p} k \to  R\Gamma(X,\cO)_\perf 
\end{equation}
is an isomorphism.
Let us define $\mathbf{U}(X)^{\mathrm{perf}} \colonequals \varprojlim_F \UU(X)$, naturally viewed as a stack over $k$.
Since $\Spec$ is a right adjoint (see \cref{cope}), we have $\Spec\,R\Gamma(X,\cO)_\perf \simeq \UU(X)^{\mathrm{perf}}$. Using \cref{etale-comparison}, \cref{sw1} implies that the composition
\begin{equation}
    \pi_i(\mathbf{U}(X)^{\mathrm{perf}}) \to \pi^{\mathrm U}_i(X) \to \pi^\et_i(X)_p
\end{equation}
is an isomorphism for all $i \ge 0$.
Since the absolute Frobenius induces an isomorphism on $\pi_i(\mathbf{U}(X)^{\mathrm{perf}})$, the perfection functor from \cref{group-scheme-perfection} yields a factorization of the above map 
$$\pi_i(\mathbf{U}(X)^{\mathrm{perf}}) \to \pi^{\mathrm U}_i(X)$$
via the natural map $\pi^{\mathrm U}_i(X)^\perf \to \pi^{\mathrm U}_i(X)$ such that the composition of the maps
\begin{equation}\label{miraculous-splitting}
 \pi_i(\mathbf{U}(X)^{\mathrm{perf}}) \to \pi^{\mathrm U}_i(X)^\perf \to \pi^{\mathrm U}_i(X) \to  \pi^\et_i(X)_p 
 \end{equation}
is again an isomorphism.
\vspace{2mm}

Since $\pi^{\mathrm U}_i(X)$ is profinite (\cref{chicago16}), in order to prove that $\pi_i^\et (X)_p$ is the maximal pro-\'etale quotient of $\pi^{\mathrm U}_i(X)$, it would be enough to show that the composite map $\pi^{\mathrm U}_i(X)^{\mathrm{perf}} \to \pi^{\mathrm U}_i(X) \to \pi_i^\et (X)_p$ is an isomorphism (see \cref{tired1}).
Therefore, it suffices to prove that the map $\pi_i(\mathbf{U}(X)^{\mathrm{perf}}) \to \pi^{\mathrm U}_i(X)^\perf$ above is an isomorphism. Note that \cref{miraculous-splitting} already implies that the latter map is injective. 
By the Milnor sequence from \cite[Prop.~3.22]{MR100}, it follows that the natural map $\pi_i(\mathbf{U}(X)^{\mathrm{perf}}) \simeq \pi_i (\varprojlim_F \UU(X)) \to \pi_i^{\mathrm{U}}(X)^{\perf}$ is also surjective. This finishes the proof.
\end{proof}

\begin{remark}[Weakly ordinary varieties]\label{ordinary-unipotent-homotopy}
A proper variety $X$ over a perfect field $k$ of characteristic $p > 0$ is called \emph{weakly ordinary} if the (absolute) Frobenius endomorphism $F$ on $X$ induces a $k$-semilinear isomorphism $F^* \colon H^i(X,\cO) \simeq H^i(X,\cO)$ for all $i \ge 0$.
For weakly ordinary $X$, the natural map $R\Gamma(X, \cO) \to R\Gamma(X, \cO)_{\mathrm{perf}}$ is an isomorphism.
If $k$ is algebraically closed, \cref{jacket1} then implies that the natural map $R\Gamma_{\et}(X, k) \to R\Gamma(X, \cO)$ must be an isomorphism. 
Therefore, the \'etale comparison maps $\pi^{\mathrm U}_i(X) \to \pi_i^{\et}(X)_p$ from \cref{unipotent-etale-comparison} are all isomorphisms when $X$ is assumed to be weakly ordinary.
However, in general, these maps are very far from being isomorphisms;
for example, this is already the case when $X$ is a supersingular elliptic curve.
\end{remark}
\begin{remark}
An important example of weakly ordinary varieties (in the sense of \cref{ordinary-unipotent-homotopy}) are ordinary varieties in the sense of Bloch--Kato \cite[Def.~7.2, Prop.~7.3]{BK86} and Illusie--Raynaud \cite[Def.~4.12, Thm.~4.13]{IR83}.
Moreover, the two notions of weakly ordinary and ordinary are equivalent for abelian varieties and K3 surfaces (\textit{cf.}~\cite[Prop.~9]{MR799251}, \cite[Lem.~1.1]{MR916481}).
However, in general, a weakly ordinary variety need not be ordinary;
see, e.g., \cite[Ex.~5.7]{MR1936581}. Let $X$ be a smooth projective variety of dimension $n$ for which $\omega_X \simeq \cO_X$.
If the pullback $F^* \colon H^n(X,\cO) \to H^n(X,\cO)$ by the absolute Frobenius is bijective, then $X$ is Frobenius split (\cite[Prop.~9]{MR799251}) and therefore $X$ is automatically weakly ordinary.\end{remark}

\newpage

\section{\texorpdfstring{$K(\pi^{\mathrm U},1)$}{K({\textpi}U,1)}-schemes: curves and abelian varieties}\label{curveabelian}
The unipotent homotopy type $\mathbf{U}(X)$ of a scheme $X$ is by definition a higher stack.
In this section, we show that when $X$ is a curve or an abelian variety over a field $k$, then $\UU(X)$ is a $1$-stack. In fact, in this case $\UU(X) \simeq B \pi_1^{\mathrm U}(X)$. In \cref{cold1}, we show that the unipotent fundamental group scheme admits a flat variation for families of curves and abelian schemes. 
The latter result is used to geometrically reconstruct the ``filtered circle" (see \cref{cold2}) from \cite{moulinos2019universal} which the authors introduced in order to explain the existence of the Hochschild--Kostant--Rosenberg filtration on Hochschild homology. In \cref{cold3}, we introduce the universal fundamental group scheme of genus $g$ curves as a group scheme over the moduli of curves. In \cref{cold6}, we construct examples to show that the unipotent fundamental group scheme is not a derived invariant for abelian varieties of dimension $\ge 3$.

\subsection{Unipotent homotopy type of curves}\label{spri1} 
In this subsection, we prove that if $X$ is a proper curve over a field $k$ (always equipped with a $k$-rational point, and assumed to be cohomologically connected but not assumed to be smooth), then $\UU(X) \simeq \Bb \pi_1^{\mathrm U}(X)$. In particular, $\pi_i ^{\mathrm U} (X)$ is the zero group scheme for $i \ge 2$. In other words, one can say that the unipotent homotopy type of the curve $X$ is
a $K(\pi^{\mathrm U},1)$. For a curve $X$, the dimension of $H^1 (X, \cO)$ as a $k$-vector space will be called genus. We begin by recalling certain definitions that will be useful in formulating and proving the results in this subsection. 

\begin{definition}[{\cite[\S~IV]{MR682517}}]\label{nonfgl}
We fix an integer $g \ge 1$.
Let $k  \llbrace \underline{X};\underline{Y}; \underline{Z} \rrbrace$ be the non-commutative formal power series ring in $X_1, \ldots, X_g, Y_1, \ldots, Y_g$, $ Z_1, \ldots, Z_g$, modulo the relations $X_i Y_j = Y_j X_i$, $X_i Z_j = Z_j X_i$, $X_i Z_j = Z_j X_i$ for all $i$ and $j$. A \textit{non-commutative formal group law} of dimension $g$ over $k$ is a collection $F(\underline{X};\underline{Y})= (F_1(\underline{X}; \underline{Y}), \ldots, F_g(\underline{X};\underline{Y}))$ with the $F_i(\underline{X},\underline{Y}) \in k\llbrace \underline{X};\underline{Y} \rrbrace$ such that 
\begin{enumerate}
    \item $F_i(\underline{X};0) = F_i (0;\underline{X}) = X_i$ for $1 \le i \le g$;
    \item $F(\underline{X};\underline{Y})= F(\underline{Y};\underline{X})$;
    \item $F(F(\underline{X};\underline{Y});\underline{Z})= F(\underline{X}; F(\underline{Y};\underline{Z})) $ in the ring $k  \llbrace \underline{X};\underline{Y}; \underline{Z} \rrbrace$.
\end{enumerate}
\end{definition}

The above definition can be interpreted more conceptually in terms of the following notion.
\begin{definition}[Linearly compact topological Hopf algebra]\label{lincompacth}
Let $k$ be a field and $\mathrm{Vect}_k^{\mathrm{fd}}$ denote the category of finite-dimensional $k$-vector spaces. Since $\mathrm{Vect}_k \simeq \mathrm{Ind} (\mathrm{Vect}_k^{\mathrm{fd}})$, the notion of duality in $\mathrm{Vect}_k^{\mathrm{fd}}$ extends to an equivalence $\Pro(\mathrm{Vect}_k^{\mathrm{fd}})^{\mathrm{op}} \simeq \mathrm{Vect}_k$. The tensor product in $\mathrm{Vect}_k$ equips the category $\Pro(\mathrm{Vect}_k^{\mathrm{fd}})$ with a symmetric monoidal structure. We define the category of linearly compact topological Hopf algebras over $k$ to be the category of Hopf algebra objects with respect to the symmetric monoidal structure on $\Pro(\mathrm{Vect}_k^{\mathrm{fd}})$.   
\end{definition}{}

\begin{remark}
    Note that $\Pro(\mathrm{Vect}_k^{\mathrm{fd}})$ is equivalent to the category of linearly compact topological vector spaces over $k$, and the symmetric monoidal structure can be identified with completed tensor product of such spaces.
    Moreover, under this identification the duality functor $\Pro(\mathrm{Vect}^{\mathrm{fd}}_k)^{\mathrm{op}} \to \mathrm{Vect}_k$ from \cref{lincompacth} is given by the continuous dual;
    see \cite[Ch.~I, \S~2]{MR332802}.
\end{remark}{}

\begin{remark}[Duality]\label{durevision}
By construction, the category of Hopf algebras over $k$ is anti-equivalent to the category of linearly compact topological Hopf algebras over $k$. In particular, under this duality, the category of affine group schemes over $k$ is equivalent to the category of cocommutative linearly compact topological Hopf algebras over $k$.
This duality will be made more explicit in \cref{dualizinggpsch} in the context of unipotent group schemes and will be used in \cref{cof0}.
\end{remark}{}

\begin{remark}
Any finite-dimensional $k$-algebra (not necessarily commutative) is an algebra object of $\Pro(\mathrm{Vect}_k^{\mathrm{fd}})$.
The non-commutative power series ring $k \llbrace \underline{X}  \rrbrace $ (or the commutative power series ring $k \llbracket \underline{X}  \rrbracket$) in the variables $X_1, \ldots, X_g$ can be naturally viewed as algebra object of $\Pro(\mathrm{Vect}_k^{\mathrm{fd}})$.
\end{remark}{}

\begin{definition}[Non-commutative formal Lie groups]\label{usethisdefasrefsays}
The category of non-commutative formal Lie groups over $k$ of dimension $g$ is defined to be the opposite category of cocommutative, linearly compact topological Hopf algebras over $k$ whose underlying algebra object of $\Pro(\mathrm{Vect}_k^{\mathrm{fd}})$ is isomorphic to the non-commutative power series ring $k \llbrace \underline{X}  \rrbrace $ in the variables $X_1, \ldots, X_g$.   
\end{definition}{}

\begin{remark}\label{revissremk}
 Note that a non-commutative formal group law in the sense of \cref{nonfgl} defines a non-commutative formal Lie group, with comultiplication given by $F$. Conversely, every non-commutative formal Lie group arises in this way.
 Under the duality from \cref{durevision}, the category of non-commutative formal Lie groups is anti-equivalent to a full subcategory of affine group schemes over $k$.
\end{remark}{}

\begin{proposition}\label{rhi1}Let $X$ be a proper curve over a field $k$. Then the dual of the unipotent group scheme $\pi_1^{\mathrm U}(X)$ is a non-commutative $g$-dimensional formal Lie group where $g= H^1(X, \cO_X)$.
\end{proposition}{}

\begin{proof}This follows from {\cite[\S~IV, Prop.~4]{MR682517}} since $\pi_1^{\mathrm{U}}(X) \simeq \pi_1^{\mathrm{U,N}}(X)$.
\end{proof}{}

\begin{proposition}\label{curve}Let $X$ be a proper curve over a field $k$. Then $\UU(X) \simeq  B\pi_1^{\mathrm U}(X)$.
\end{proposition}{}
\begin{proof}By definition, we have a natural map $\UU(X) \to \Bb\pi_1^{\mathrm U}(X)$. Since this is a morphism of affine stacks, to prove that they are equivalent it would be enough to show that the induced map $R\Gamma(\Bb \pi_1^{\mathrm U}(X), \cO) \to R\Gamma(\UU(X) , \cO)$ is an isomorphism. By \cref{postnikov}, this map is an isomorphism on cohomology groups $H^i(\,\cdot\,)$ for $i \le 1$. Since $X$ is a curve, $H^i (X, \cO) \simeq H^i (\UU(X), \cO) = 0 $ for $i \ge 2$. Therefore, it would be enough to show that $R\Gamma(\Bb \pi_1^{\mathrm U}(X), \cO) = 0$ for $i \ge 2$.

\vspace{2mm}
To show that $R\Gamma(\Bb \pi_1^{\mathrm U}(X), \cO) = 0$ for $i \ge 2$, we note that at the level of objects in the derived category, we have $R\Gamma(\Bb \pi_1^{\mathrm U}(X), \cO) \simeq R\Hom_{k\llbrace X_1, \ldots, X_g \rrbrace}(k,k)$ (see \cref{newyr}). Therefore, we are done by the following lemma.
\end{proof}
\begin{lemma}\label{nc-group-law-cohomology}
In the derived category, we have 
$$ R\Hom_{k\llbrace X_1, \ldots, X_g \rrbrace}(k,k) \simeq k \oplus k[-1]^{\oplus g}.$$

\end{lemma}{}

\begin{proof}We denote $A \colonequals k \llbrace X_1, \ldots, X_g \rrbrace$. We consider the module $A^{\oplus g}$ and let $u_1, \ldots, u_g$ denote a basis of $A^{\oplus g}$ over $A$. There is a map $A^{\oplus g} \to A$ which sends $u_i$ to $x_i$. This fits into an exact sequence $0 \to A^{\oplus g} \to A \to k \to 0$ which yields the claim.
\end{proof}{}

Note that in \cref{cold19}, we defined the unipotent homotopy type $\UU(X)$ for $X \in \mathcal{S}$. As a consequence of \cref{curve}, we will show that the unipotent homotopy type of a weakly ordinary curve of genus $g$ is isomorphic to the unipotent homotopy type of wedge product of $g$-circles $\bigvee^g S^1$.

\begin{lemma}\label{achin2}
Let $X$ be a proper curve over an algebraically closed field $k$ of characteristic $p > 0$. 
Then there is a natural isomorphism $R\Gamma_{\et}(X, k) \simeq R\Gamma(B\pi^{\et}_1(X)_p, \cO)$. Further, $\pi_1^{\et}(X)_p$ is the free pro-$p$-finite group scheme on $\dim H^1_{\et}(X, k)$ generators (see \cref{cold91}).
\end{lemma}

\begin{proof}
As in the proof of \cref{unipotent-etale-comparison}, we have a natural isomorphism $R\Gamma_{\et}(X,k) \simeq R\Gamma (X, \cO)_{\mathrm{perf}}$. Since $R\Gamma_{\et}(X,k)$ is naturally augmented and $H^0_{\et} (X,k) \simeq k$, it follows that $\Spec R\Gamma_{\et}(X,k)$ is a pointed connected stack. Moreover, by \cref{curve}, it follows that $\Spec R\Gamma_{\et}(X,k) \simeq B \pi_1^{\mathrm{U}}(X)^{\mathrm{perf}}$. By \cref{tired1}, $\pi_1^{\mathrm{U}}(X)^{\mathrm{perf}}$ is isomorphic to the maximal pro-\'etale quotient of $\pi_1^{U}(X)$; this is further isomorphic to $\pi_1^{\et}(X)_p$ by \cref{unipotent-etale-comparison}. This shows that $R\Gamma_{\et}(X, k) \simeq R\Gamma(B\pi^{\et}_1(X)_p, \cO)$ as desired. To show that $\pi_1^{\et}(X)_p$ is the free pro-$p$ group scheme on $\dim H^1_{\et}(X,k)$ generators, we note that $H^2 (B \pi_1^{\et}(X)_p, \cO) = H^2_{\et}(X,k) = 0$. Moreover, by \cref{compareetaleuni}, $\pi_1^{\et}(X)_p$ is the affine group scheme associated to a pro-$p$-finite group. The desired claim now follows from \cref{compare}, \cite[\S~4.1, Prop.~21]{Serre} and \cite[\S~4.2, Cor.~2]{Serre}.
\end{proof}{}

\begin{proposition}\label{ordinarycurve}
Let $X$ be a proper, weakly ordinary curve of genus $g$ over an algebraically closed field $k$ of characteristic $p > 0$.
Then there is a natural isomorphism $\UU(X) \simeq \UU( \bigvee ^g S^1 )$.
\end{proposition}{}

\begin{proof} Since $X$ is weakly ordinary, by \cref{achin2} and \cref{ordinary-unipotent-homotopy}, $\pi_1^{\et}(X)_p$ is a free pro-$p$-group scheme in $g$ generators and $R\Gamma (X, \cO) \simeq R\Gamma_{\et}(X, k) \simeq R\Gamma (B\pi_1^{\et}(X)_p, \cO)$. 
Since $\bigvee ^g S^1$ is an Eilenberg–MacLane space for the free group on $g$ generators (which is a good group in the sense of Serre \cite[\S~2.6]{Serre}), it follows that  $C^*(\bigvee ^g S^1, k) \simeq R\Gamma(B\pi_{\et}^1(X)_p, \cO)$. By combining these isomorphisms, we obtain that $\UU(X) \simeq \UU(\bigvee^g S^1)$.
\end{proof}{}

\begin{remark}\label{freegroups}
Let $X$ be a proper, weakly ordinary curve of genus $g$ over an algebraically closed field $k$ of characteristic $p > 0$. \cref{ordinarycurve} in particular implies that $\pi_1^{\mathrm U}(X)$ is the free pro-$p$ group scheme on $g$ generators. Consequently, the non-commutative formal Lie group attached to $X$ (see \cref{rhi1}) is given by $F(X;Y)= X_i + Y_i + X_i Y_i$, which can be thought of as the non-commutative analogue of $\widehat{\GG}_m^g$. 
Note that the analogue of the latter statement for ordinary abelian varieties is well-known (see \cref{fglabelian}): 
if $X$ is an ordinary abelian variety of dimension $g$, then the formal completion of $X^\vee$ at zero is representable by $\widehat{\GG}_m^g$.
\end{remark}

\begin{example}\label{usefulnoriwork}
We note that \cref{ordinarycurve} gives a generalization of 
{\cite[\S~IV, Ex.~1]{MR682517}}, since the curves obtained by identifying certain disjoint sets of points of $\PP^1_{k}$ produces a curve that is weakly ordinary. In particular, we get that an ordinary ellipitic curve or a nodal cubic curve has the same unipotent homotopy type, which is that of the circle $S^1$.
\end{example}{}

\begin{example}From {\cite[\S~IV, Ex.~2]{MR682517}}, it follows that for the cuspidal cubic curve $X$, the associated formal group law is $\widehat{\GG}_a$. 
This also shows that $\pi_1^{\mathrm U}(X) \simeq W[F]$, where $W[F]$ is the kernel of Frobenius on the ring scheme of $p$-typical Witt vectors.
\end{example}{}

\subsection{Unipotent homotopy type of abelian varieties}\label{spri2}
In this section, we prove that if $X$ is an abelian variety over a field $k$, then $\UU(X) \simeq \Bb \pi_1^{\mathrm U}(X)$. In particular, $\pi_i ^\mathrm{U} (X)$ is the zero group scheme for $i \ge 2$. In other words, one can say that the unipotent homotopy type of the abelian variety $X$ is
a $K(\pi^{\mathrm U},1)$. It follows from the product formula (\cref{productfor}) that $\pi_1^{\mathrm U}(X)$ is a commutative group scheme for an abelian variety $X$.

\begin{definition}\label{commfgl}
We fix an integer $g \ge 1$. Let $k \llbracket \underline{X};\underline{Y};\underline{Z} \rrbracket$ be the commutative formal power series ring in $X_1, \ldots, X_g, Y_1, \ldots, Y_g$, $ Z_1, \ldots, Z_g$. A \textit{commutative formal group law} of dimension $g$ over $k$ is a collection $F(\underline{X};\underline{Y})= (F_1(\underline{X}; \underline{Y}), \ldots, F_g(\underline{X};\underline{Y}))$ with the $F_i(\underline{X},\underline{Y}) \in k \llbracket \underline{X},\underline{Y} \rrbracket$ such that 
\begin{enumerate}
    \item $F_i(\underline{X};0) = F_i (0;\underline{X}) = X_i$ for $1 \le i \le g$;
    \item $F(\underline{X};\underline{Y})= F(\underline{Y};\underline{X})$;
    \item $F(F(\underline{X};\underline{Y});\underline{Z})= F(\underline{X}; F(\underline{Y};\underline{Z})) $ in the ring $k \llbracket \underline{X};\underline{Y};\underline{Z} \rrbracket$.
\end{enumerate}
\end{definition}

\begin{definition}[Commutative formal Lie groups]\label{usethisdefasrefsays2}
The category of commutative formal Lie groups over $k$ of dimension $g$ is defined to be the opposite category of commutative, cocommutative, linearly compact topological Hopf algebras over $k$ whose underlying algebra object of $\Pro(\mathrm{Vect}_k^{\mathrm{fd}})$ is isomorphic to the commutative power series ring $k \llbracket \underline{X}  \rrbracket$ in the variables $X_1, \ldots, X_g$.   
\end{definition}

\begin{remark}[Commutative formal groups and duality]\label{thisisnew}
Following \cite[\S~37.3.1]{MR2987372}, one may define the category of commutative formal groups over a field $k$ to be the opposite category of commutative, cocommutative, linearly compact topological Hopf algebras over $k$. 
Note that commutative formal groups can also be thought of as pro-representable functors from the category of finite-dimensional $k$-algebras to the category of abelian groups. By \cref{durevision}, the category of commutative affine group schemes over $k$ is anti-equivalent to the category of commutative formal groups over $k$ under duality (see also \cite[Thm.~37.3.12]{MR2987372}). 
\end{remark}

\begin{remark}
 Note that a commutative formal group law in the sense of \cref{commfgl} defines a commutative formal Lie group. 
 Conversely, every commutative formal Lie group arises in this way.
 Moreover, in dimension $1$, the notions of commutative and non-commutative formal Lie group coincide.
\end{remark}{}

\begin{proposition}\label{fglabelian}
Let $X$ be an abelian variety over a field $k$. Then the dual of the unipotent group scheme $\pi_1^{\mathrm U}(X)$ is naturally isomorphic to the commutative formal Lie group obtained by taking the formal completion of the dual abelian variety $X^\vee$ at zero.
\end{proposition}
\begin{proof}
 As already noted, by \cref{productfor}, $\pi^{\mathrm U}_1(X)$ is commutative. The claim in the proposition follows from \cite{MR682517}. Indeed, according to \cite[\S~IV, Prop.~6]{MR682517}, if $k$ has positive characteristic, then the dual of $\pi^{\mathrm U}_1(X)$ is given by the colimit over all local finite subgroup schemes of the dual abelian variety $X^\vee$ in the category of formal groups.
 This colimit is isomorphic to the formal completion of $X^\vee$ at zero and thus corresponds to a commutative formal Lie group of dimension $g$. If $k$ has characteristic zero, then (see \cite[\S~IV, Prop.~2]{MR682517}) we have $\pi^{\mathrm U}_1(X) \simeq \GG^g_a$ and the formal completion of $X^\vee$ is $\widehat{\mathbf G}_a^g$; thus the result follows directly.
\end{proof}{}
\begin{proposition}\label{abelianvar}Let $X$ be an abelian variety a field $k$. Then $\UU(X) \simeq B \pi_1^{\mathrm U}(X)$.
\end{proposition}{}

\begin{proof}
We need to show that the natural truncation map $X \to B\pi_1^{\mathrm U}(X)$ induces an isomorphism $R\Gamma (B\pi_1^{\mathrm U}(X), \cO) \to R\Gamma(X, \cO)$. 
By \cref{postnikov}, this map is an isomorphism on cohomology groups $H^i(\,\cdot\,)$ for $i \le 1$. Therefore, it would be enough to prove that $H^* (B \pi_1^{\mathrm U}(X), \cO)$ is an exterior algebra generated by the elements of degree $1$. To do so, we may view $R\Gamma (B \pi_1^{\mathrm U}(X), \cO)$ as an associative algebra object in the sense of \cite{luriehigher}.
Since by \cref{fglabelian}, the dual of $\pi_1^{\mathrm U}$ is a commutative formal Lie group of dimension $g = \dim X$, the associative algebra object $R\Gamma(B\pi_1^{\mathrm U}(X),  \mathscr O)$ can be computed as $R\Hom_{k \llbracket x_1, \ldots, x_g \rrbracket}(k,k)$ (see, e.g.,~\cref{paderborn1}). The latter can be computed via a standard Koszul complex which yields the claim.
\end{proof}{}

\begin{lemma}\label{achin02}
Let $X$ be an abelian variety over an algebraically closed field $k$ of characteristic $p > 0$.
Then there are natural isomorphisms $R\Gamma_{\et}(X, k) \simeq R\Gamma(B\pi^{\et}_1(X)_p, \cO)$. Further, $\pi_1^{\et}(X)_p$ is the free commutative pro-$p$ group scheme on $\dim H^1_{\et}(X, k)$ generators (see \cref{cold91}).
\end{lemma}

\begin{proof}
    The isomorphism $R\Gamma_{\et}(X, k) \simeq R\Gamma(B\pi^{\et}_1(X)_p, \cO)$ follows from \cref{abelianvar} in a way similar to the proof of \cref{achin2}. Since $\pi_1^{\mathrm{U}}(X)$ is dual to a formal Lie group (\cref{fglabelian}), it follows (e.g., by Dieudonn\'e theory) that $\pi^{\et}_1(X)_p \simeq \pi_1^{\mathrm{U}}(X)^{\mathrm{perf}} \simeq \pi_1^{\mathrm{U}}(X)_{\mathrm{red}}$ is also dual to a formal Lie group. Since $\pi_1^{\et}(X)_p$ is the affine group scheme associated with a pro-$p$-finite group (see \cref{compareetaleuni}), the latter assertion also follows.
\end{proof}{}

\begin{proposition}
Let $X$ be a weakly ordinary abelian variety of dimension $g$ over an algebraically closed field $k$ of characteristic $p>0$. Then $\UU(X) \simeq \UU((S^1) ^{\times g})$.
\end{proposition}{}
\begin{proof}
This follows in a way similar to the proof of \cref{ordinarycurve}.
\end{proof}

\subsection{Derived equivalent abelian varieties}\label{cold6}
Since the unipotent homotopy group scheme is an invariant of schemes, it is a natural question to ask if they are in fact derived invariants. More precisely, if $X$ and $Y$ are two smooth projective algebraic varieties defined over an algebraically closed field $k$ of characteristic $p>0$ such that $D_{\mathrm{perf}}(X)$ and $D_{\mathrm{perf}}(Y)$ are equivalent as $k$-linear triangulated categories, one can ask if $\pi_i^{\mathrm{U}} (X)$ and $\pi_i^{\mathrm{U}} (Y)$ are isomorphic. We will return to this theme in \cref{worldcup1} in greater detail, where this question will be studied for Calabi--Yau varieties. The goal of this subsection is to record examples showing that the unipotent fundamental group scheme is \emph{not} a derived invariant in general.
More precisely, we give examples of abelian threefolds $X$ such that $\pi_1^{\mathrm{U}}(X)$ and $\pi_1^{\mathrm{U}}(X^\vee)$ are not isomorphic as group schemes (see \cref{worldcup3}). This case is particularly interesting because many of the interesting numerical invariants are equal for $X$ and $X^\vee$. For example, in \cite[Thm.~1.2]{AB1}, it is shown that certain numerical invariants are derived invariants for \emph{all} smooth projective varieties of dimension $\le 3$. Also, as explained in \cref{worldcup4}, one cannot find such examples of abelian varieties of dimension $\le 2$ (in odd characteristic).
\vspace{2mm}

Our construction will be described in \cref{not-derived-invariant}. Let us mention the key ideas in our construction. We begin by understanding the moduli of certain Dieudonn\'e modules of dimension $3$, extending the work of Oda and Oort \cite[Prop.~4.1]{Ooo} to certain non-supersingular cases. Slightly more precisely, we construct a family of isomorphism classes of Dieudonn\'e modules of dimension $3$, height $6$ and with symmetric Newton polygons of slopes $\frac 1 3$ and $\frac 2 3$ that is parametrized by $\GG_m$. Further, we construct this family to be stable under the duality of Dieudonn\'e modules; 
the operation of taking duals corresponds to the endomorphism $z \mapsto \frac{(-1)^{p-1}}{z}$ of $\GG_m$. This is described in \cref{a=1}.
We then construct an abelian variety $X_z$ whose Dieudonn\'e module is isomorphic to one in the family mentioned above for some $z \in \GG_m(k)$ satisfying $z^2 \ne (-1)^{p-1}$ and show that it has the desired property that $\pi_1^{\mathrm{U}}(X_z) \not\simeq \pi_1^{\mathrm{U}}(X_z^\vee)$.
In particular, this gives a general strategy for producing many examples of abelian varieties which do not admit a principal polarization by looking inside the moduli space Dieudonn\'e modules.
\vspace{2mm}

We will begin by quickly recalling some necessary background on Dieudonn\'e theory which we will need later; see \cite{YManin}, \cite[\S~III.8]{dem}, \cite[\S~6]{Fon1}, \cite[\S~1]{Ooo}, \cite[\S~1]{MR1792294} for more details on the classical Dieudonn\'e theory (see \cite{Mon21} for a recent stacky approach). For this subsection, we assume that $k$ is a perfect field of characteristic $p > 0$. Let $W(k)$ denote the ring of $p$-typical Witt vectors of $k$ with the Witt vector Frobenius denoted as $\sigma \colon W(k) \to W(k)$.
\begin{definition}
The \emph{Dieudonn\'e ring} $\mathscr{D}_k$ is the quotient of the non-commutative polynomial ring $W(k)\{F,V\}$ by the relations $FV = VF = p$ and $F x = \sigma(x)F$ and $V  x = \sigma^{-1}(x)V$ for all $x \in W(k)$.
For us, a \emph{Dieudonn\'e module} will be a left $\mathscr{D}_k$-module whose underlying $W(k)$-module is free of finite rank.
\end{definition}
\begin{remark}\label{Dieudonne-ring-sum}
The underlying $W(k)$-module of $\mathscr{D}_k$ is a free module with a set of basis elements given by $\left \{1,F,F^2, \dotsc, V, V^2, \dotsc \right \}$. Similarly, the ideal $(F,V) \mathscr{D}_k \subset \mathscr{D}_k$ is a free $W(k)$-module with basis given by $\left \{p,F,F^2, \dotsc, V, V^2, \dotsc\right \}$. 
\end{remark}
\begin{theorem}[Dieudonn\'e]\label{covariant-Dieudonne}
There is a contravariant equivalence $\mathbf{M}$ between the category of $p$-divisible groups over $k$ and the category of Dieudonn\'e modules which satisfies the following properties:
\begin{enumerate}
    \item For any $p$-divisible group $G$ over $k$, we have $\rk_{W(k)} \mathbf{M}(G) = \height(G)$.
    \item\label{covariant-Dieudonne-dual} For any $p$-divisible group $G$ over $k$ with dual $G^\vee$, the Dieudonn\'e module $\mathbf{M}(G^\vee)$ is given by the $W(k)$-module $\mathbf{M}(G)^\vee \colonequals \Hom_{W(k)}(\mathbf{M}(G),W(k))$ with actions $F \cdot \theta \colonequals \sigma \circ \theta \circ V$ and $V \cdot \theta \colonequals \sigma^{-1} \circ \theta \circ F$.
\end{enumerate}
\end{theorem}

\begin{remark}
For a Dieudonn\'e module $M$, one defines $\dim M \colonequals \dim_k M/FM$.
\end{remark}{}

\begin{remark}\label{R_mn}
For positive integers $m,n$ such that $(m,n)=1$, the quotient $\mathscr{D}_{m,n} \colonequals \mathscr{D}_k/\mathscr{D}_k(V^n - F^m)$ is a Dieudonn\'e module of rank $n+m$ and dimension $n$. Note that under the notion of duality defined in \cref{covariant-Dieudonne}.\ref{covariant-Dieudonne-dual}, the dual of $\mathscr{D}_{m,n}$ is given by $\mathscr{D}_{n,m}$. The special cases when $m = 1$ and $n = 2$ and $m = 2$ and $n = 1$ will be important to us later on. Note that we have $\mathscr{D}_{1,2} \simeq W(k) \cdot 1 \oplus W(k) \cdot V \oplus W(k) \cdot F$. Similarly, $ \mathscr{D}_{2,1} \simeq W(k) \cdot 1 \oplus W(k) \cdot V \oplus W(k) \cdot F$. In $\mathscr{D}_{1,2}$, the $\mathscr{D}_k$-action satisfies $V \cdot V = F$ and $F \cdot F = pV$. Similarly, for $\mathscr{D}_{2,1}$, we have $V \cdot V = pF$ and $F \cdot F = V$.
Moreover, we have 
\[ (F,V) \mathscr{D}_{1,2} \simeq W(k) \cdot p \oplus W(k) \cdot V \oplus W(k) \cdot F; \; (F,V)^2 \mathscr{D}_{1,2} \simeq W(k) \cdot p \oplus W(k) \cdot pV \oplus W(k) \cdot F. \]
Similarly,
\[ (F,V) \mathscr{D}_{2,1} \simeq W(k) \cdot p \oplus W(k) \cdot F \oplus W(k) \cdot V; \; (F,V)^2 \mathscr{D}_{2,1} \simeq W(k) \cdot p \oplus W(k) \cdot pF \oplus W(k) \cdot V. \]
There are isomorphisms of Dieudonn\'e modules given by
\begin{equation}\label{back}
    (F,V) \mathscr{D}_{1,2} \simeq \mathscr{D}_{1,2}, \; p\mapsto F,\, F \mapsto V, \, V \mapsto 1; \quad (F,V) \mathscr{D}_{2,1} \simeq \mathscr{D}_{2,1}, \; p \mapsto V,\, V \mapsto F,\, F \mapsto 1.
\end{equation}
Consider the composition
\[ \beta \colon \mathscr{D}_k \times (F,V) \mathscr{D}_k \xrightarrow{\text{}} (F,V) \mathscr{D}_k \rightarrow pW(k) \simeq W(k) \]
in which the first morphism is the multiplication in $R$, the second morphism is the projection onto the direct summand $pW(k)$ in the decomposition arising from \cref{Dieudonne-ring-sum} and the third morphism is division by $p$.
Since $\beta(\mathscr{D}_k(V^2-F) \cdot (F,V) \mathscr{D}_k(V - F^2)) = \beta(\mathscr{D}_k(V-F^2) \cdot (F,V) \mathscr{D}_k(V^2 - F)) = 0$, $\beta$ descends to natural pairings $\bar{\beta}_1 \colon \mathscr{D}_{1,2} \times (F,V) \mathscr{D}_{2,1} \to W(k)$ and $\bar{\beta}_2 \colon \mathscr{D}_{2,1} \times (F,V) \mathscr{D}_{1,2} \to W(k)$ under which $\mathscr{D}_{1,2}$ and $\mathscr{D}_{2,1}$ are dual to each other (using \cref{back}).
\end{remark}
Up to isogeny, the category of Dieudonn\'e modules over an algebraically closed field is semisimple and the simple objects are those of \cref{R_mn}:
\begin{theorem}[Dieudonn\'e--Manin classification]\label{Dieudonne-Manin}
Assume that $k$ is algebraically closed.
The Dieudonn\'e module of each $p$-divisible group $G$ is isogenous to a direct sum $\oplus^r_{i=1} \mathscr{D}_{m_i,n_i}$ for unique pairs $(m_1,n_1),\dotsc,(m_r,n_r)$ with $(m_i,n_i) = 1$.
Moreover, there are no nonzero morphisms $\mathscr{D}_{m,n} \to \mathscr{D}_{m',n'}$ if $m \neq m'$ or $n \neq n'$.

\end{theorem}
The next statement describes moduli of certain Dieudonn\'e modules $M$ of dimension $3$, height $6$ and whose slopes are given by $\frac 1 3$ and $\frac 2 3$. Note that these Dieudonn\'e modules are \emph{not} supersingular. For a similar study of moduli of supersingular abelian varieties of dimension $3$, see \cite[Prop.~4.1]{Ooo}.
In this case, we can describe the duality functor from \cref{covariant-Dieudonne}.\ref{covariant-Dieudonne-dual} explicitly and see that most such Dieudonn\'e modules are not self-dual.
In \cref{not-derived-invariant}, this will be used to construct dual (and thus derived equivalent) abelian varieties with different unipotent homotopy group schemes.
\begin{proposition}\label{a=1}
Let $k$ be an algebraically closed field and $N \colonequals \mathscr{D}_{1,2} \oplus \mathscr{D}_{2,1}$. Note that by \cref{R_mn}, $W(k)^{\oplus 2}$ can be viewed as a $W(k)$-submodule of $N$. In this setup, we have the following:
\begin{enumerate}[label=(\arabic*)]
    \item\label{a=1-a=2} $N / (F,V) N \simeq k \oplus k$.
    \item\label{a=1-Teich} Let $(F,V) N \subset M \subset N$ be inclusions of Dieudonn\'e modules such that $$\dim_k M / (F,V)N = \dim_k N / M = 1.$$
    Then $M = (F,V) N + W(k)\cdot ([x_1],[x_2]) \subset N$ for some nonzero $(x_1,x_2) \in k \oplus k$. Here, $[\,\cdot \,] \colon k \to W(k)$ is the multiplicative lift. Moreover, if $x_1,x_2 \neq 0$, the corresponding $M$ satisfies $(F,V) M = (F,V) N$.
    \item\label{a=1-morphism} Let $(F,V) N \subset M \subset N$ and $(F,V) N \subset M' \subset N$ be such that $(F,V)M= (F,V) N$ and $(F,V)M'= (F,V) N$.
    Then any isomorphism $M \xrightarrow{\sim} M'$ induces a unique isomorphism $N \xrightarrow{\sim} N$ which fits into the following commutative diagram:
    \[ \begin{tikzcd}
    (F,V) N \arrow[r,hook] \arrow[d,"\simeq"] & M \arrow[r,hook] \arrow[d,"\simeq"] & N \arrow[d,"\simeq"] \\
    (F,V)N \arrow[r,hook] & M' \arrow[r,hook] & N.
    \end{tikzcd} \]
    \item\label{a=1-flag} Reduction modulo $(F,V) N$ induces a bijection$$\chi \colon \bigl\{ (F,V) N \subset M \subset N \suchthat \dim_k M / (F,V) N = \dim_k N / M = 1,\; M \not\simeq N \bigr\}_{/_{\simeq}}  \xrightarrow{\sim} \GG_m(k)/\FF^\times_{p^3}.$$
Here, we identify $\PP(N/(F,V) N) \simeq \PP^1_k$ via the isomorphism from \ref{a=1-a=2}, and $\FF^\times_{p^3}$ acts on $\GG_m(k)$ by scalar multiplication;
    moreover, $\GG_m(k)/\FF^\times_{p^3} \simeq \GG_m(k)$ via the $(p^3-1)$-th power map.
    \item\label{a=1-dual} The dual (in the sense of \cref{covariant-Dieudonne}.\ref{covariant-Dieudonne-dual}) of a chain of inclusions $(F,V) N \subset M \subset N$ is a chain of inclusions $(F,V) N \subset M^\vee \subset N$, and the induced morphism on $\GG_m(k)$ under $\chi$ and the isomorphism $\GG_m(k)/\FF^\times_{p^3} \simeq \GG_m(k)$ is given by $z \mapsto \frac{(-1)^{p-1}}{z}$.
\end{enumerate}
\end{proposition}
\begin{proof}

Since $\mathscr{D}_{1,2}/ (F,V) \mathscr{D}_{1,2} \simeq k$ and $\mathscr{D}_{2,1}/ (F,V) \mathscr{D}_{2,1} \simeq k$, we obtain \ref{a=1-a=2}.
\vspace{2mm}

Let $\pi \colon N \to k \oplus k$ denote the quotient map obtained by using \ref{a=1-a=2}. Note that there is an inclusion $W(k) \oplus W(k) \hookrightarrow N$ such that the composition with $\pi$ identifies with the map given by reduction modulo $p$.
Given $(F,V) N \subset M \subset N$, let $\overline{M} \colonequals \pi(M)$.
If $\dim_k M / (F,V) N = \dim_k N / M = 1$, we can find $0 \ne (x_1,x_2) \in k \oplus k$ such that $\overline{M} = k \cdot (x_1,x_2) \subset k \oplus k$.
Moreover, since $W(k) \cdot p \oplus W(k) \cdot p \subset (F,V) N$ (see \cref{R_mn}) and the multiplicative lift $[\,\cdot\,] \colon k \to W(k)$ defines a section of the reduction modulo $p$ map $W(k) \to k$, we obtain
\[ M = (F,V) N + W(k)\cdot ([x_1],[x_2]) \subset N. \] 
The computations in \Cref{R_mn} then show that as $W(k)$-submodules of $N$, we have
\begin{align*}
    (F,V) M 
    &= (F,V)^2 N + (F,V) W(k)\cdot ([x_1],[x_2]) \\
    &= (F,V)^2 N + W(k)\cdot ([x^p_1] F,[x^p_2] F) + W(k) \cdot([x^{1/p}_1]  V,[x^{1/p}_2] V) \\
    &\supseteq \bigl(W(k) \cdot p \oplus W(k) \cdot [x_1^{1/p}]V \oplus W(k) \cdot F\bigr) \oplus \bigl(W(k) \cdot p \oplus W(k) \cdot [x^p_2]F \oplus W(k) \cdot V \bigr).
\end{align*}
Thus, if $x_1,x_2 \neq 0$, then the elements $[x_1^{1/p}], [x_2^p] \in W(k)$ are units; therefore we have $(F,V) M = (F,V) N$, yielding \ref{a=1-Teich}.
\vspace{2mm}

Before proceeding further, we will need one additional preparation. To this end, note that we can combine the pairings $\bar{\beta}_1$ and $\bar{\beta}_2$ from \cref{R_mn} to a perfect pairing
\[ \bar{\beta} \colon N \times (F,V) N \to W(k) \]
which exhibits $N$ and $(F,V) N$ as duals in a natural way.
Tensoring with $K \colonequals W(k)[\frac 1 p]$, we obtain a perfect pairing on $N_K \colonequals N \otimes_{W(k)} K = (F,V) N \otimes_{W(k)} K$ denoted as 
\[ \bar{\beta}_K \colon N_K \times N_K \to K. \]
Since the $W(k)$-linear dual $\Lambda^\vee$ of a $W(k)$-lattice $\Lambda \subset N_K$ can be identified with $\Lambda^* \colonequals \{ y \in N_K \suchthat \bar{\beta}_K(\Lambda,y) \subseteq W(k) \}$, the chain of lattice inclusions $(F,V) N \subset M \subset N$ dualizes to $(F,V) N \subset M^\vee \subset N$ under $\bar{\beta}_K$.
Explicitly,
\begin{equation}\label{a=1-explicit-dual}\begin{aligned}
    M^\vee &\simeq \{ y \in N \suchthat \bar{\beta}_K(W(k) \cdot ([x_1],[x_2]),y) \subseteq W(k) \} \\
    &= (F,V) N + \{ (y_1 ,y_2) \in W(k) \oplus W(k)  \suchthat [x_1]y_1 + [x_2]y_2 \in pW(k) \} \\
    &= (F,V) N + W(k) \cdot ([-x_2] ,[x_1]) \subset N.
\end{aligned}\end{equation}
This shows the first part of \ref{a=1-dual}.
\vspace{2mm}

Next, let $(F,V) N \subset M \subset N$ and $(F,V) N \subset M' \subset N$ as in \ref{a=1-morphism} and $v \colon M \xrightarrow{\sim} M'$ be an isomorphism of Dieudonn\'e modules.
By \ref{a=1-Teich} and the $\mathscr{D}_k$-linearity of $v$, we have an induced isomorphism of submodules
\[ (F,V) N = (F,V) M \xrightarrow{\sim} (F,V) M' = (F,V) N. \]
By the preceding paragraph, we can apply the same reasoning to the duals and then pass to duals again to obtain an isomorphism $N \xrightarrow{\sim} N$ which restricts to $v$ (canonically identified with $v^{\vee\vee}$).
This shows \ref{a=1-morphism}.
\vspace{2mm}

Now, we move on to \ref{a=1-flag}.
Let $M$ be a Dieudonn\'e module with $(F,V) N \subset M \subset N$ such that $\dim_k M / (F,V) N = \dim_k N / M = 1$ and $M \not\simeq N$.
By the proof of \ref{a=1-Teich}, $\overline{M} = \pi(M) \subset N / (F,V) N$ is canonically identified with a line $k \cdot (x_1,x_2) \subset k \oplus k$ and hence a point of $\PP^1(k)$.
Since $\dim_k N/(F,V) N = 2$, the condition that $M \not\simeq N$ is by \ref{a=1-Teich} equivalent to $x_1,x_2 \neq 0$, so the point associated with $\overline{M}$ lies in $\PP^1(k) \smallsetminus \{ 0,\infty \} = \GG_m(k)$.
To see that $\chi$ is well-defined, we need to show that this point is invariant under isomorphisms up to $\FF^\times_{p^3}$-action. To this end, let $v \colon M \xrightarrow{\sim} M'$ be an isomorphism of Dieudonn\'e modules as above.
By \ref{a=1-morphism}, $v$ naturally extends to an isomorphism $\widetilde{v} \colon N \xrightarrow{\sim} N$.
By the Dieudonn\'e--Manin classification (\cref{Dieudonne-Manin}), $\widetilde{v} = \widetilde{v}_1 \oplus \widetilde{v}_2$ for some isomorphisms of Dieudonn\'e modules $\widetilde{v}_1 \colon \mathscr{D}_{1,2} \xrightarrow{\sim} \mathscr{D}_{1,2}$ and $\widetilde{v}_2 \colon \mathscr{D}_{2,1} \xrightarrow{\sim} \mathscr{D}_{2,1}$. 
Since $\mathscr{D}_{1,2} = \mathscr{D}_k/\mathscr{D}_k (V^2 - F)$, any left $\mathscr{D}_k$-module endomorphism of $\mathscr{D}_{1,2}$ is determined by an element $(a_1 + b_1 V + c_1 F)$ for some $a_1,b_1,c_1 \in W(k)$ under the decomposition from \cref{R_mn} which must satisfy $(V^2 - F) (a_1+ b_1V + c_1F) = 0$ in the ring $\mathscr{D}_{1,2}$.
The latter is equivalent to the conditions $\sigma^3 (a_1) = a_1$, $\sigma ^3 (b_1)= b_1$ and $\sigma^3 (c_1) = c_1$. Similarly, any left $\mathscr{D}_k$-module endomorphism of $\mathscr{D}_{2,1}$ is determined by an element $(a_2 + b_2 F + c_2 V)$ satisfying $\sigma^3 (a_2) = a_2$, $\sigma ^3 (b_2)= b_2$ and $\sigma^3 (c_2) = c_2$. Let $\overline{a}_1 \in k$ and $\overline{a}_2 \in k$ be the images of $a_1$ and $a_2$ modulo $p$, respectively. It follows that any endomorphism of $N$ induces an endomorphism $k \oplus k \to k \oplus k$ (by going modulo $(F,V) N $ and using \ref{a=1-a=2}) that sends $(x,y)$ to $(\overline{a}_1 x, \overline{a}_2 y)$. Since $\sigma^3 (a_1)= a_1$, it follows that $\overline{a}_1^{p^3} = \overline{a}_1$. Similarly, we have $\overline{a}_2^{p^3} = \overline{a}_2$. This checks that $\chi$ is well-defined.
\vspace{2mm}

Lastly, we need to define an inverse to $\chi$.
Given a line $\overline{M} \subset N / (F,V) N$, identified with $k \cdot (x_1,x_2) \subset k \oplus k$ such that $x_1,x_2 \neq 0$, set $M \colonequals \pi^{-1}(\overline{M}) = (F,V) N + W(k)\cdot ([x_1],[x_2]) \subset N$, which is naturally a $\mathscr{D}_k$-module.
For $\lambda_1,\lambda_2 \in \FF^\times_{p^3}$, multiplication by the lifts $[\lambda_1],[\lambda_2] \in W(k)$ on $N = \mathscr{D}_{1,2} \oplus \mathscr{D}_{2,1}$ defines an isomorphism of Dieudonn\'e submodules $M \xrightarrow{\sim} M'$.
Thus, the isomorphism class of $M$ does not change under the $\FF^\times_{p^3}$-action. This defines a well-defined inverse map to $\chi$ and finishes the proof of \ref{a=1-flag}.
The explicit description in \cref{a=1-explicit-dual} now shows that under $\chi$ and the isomorphism $\GG_m(k)/\FF^\times_{p^3} \simeq \GG_m(k)$, the duality functor identifies with the endomorphism of $\GG_m(k)$ given by $z \mapsto \frac{(-1)^{p^3-1}}{z} = \frac{(-1)^{p-1}}{z}$, finishing \ref{a=1-dual}.
\end{proof}

The following statement will be used in \cref{not-derived-invariant}.
Recall that the \emph{slope} of $\mathscr{D}_{m,n}$ is the rational number $\frac{n}{m+n}$.
\begin{proposition}\label{Oort}
Let $G$ be a $p$-divisible group with symmetric Newton polygon;
that is, the set of slopes $0 \le \lambda_1 \le \dotsb \le \lambda_r \le 1$ of the $\mathscr{D}_{m_i,n_i}$ appearing in the Dieudonn\'e--Manin classification of $\mathbf{M}(G)$ satisfy $\lambda_i = 1 - \lambda_{r-i+1}$.
Then there exists an abelian variety $X$ such that $X[p^\infty] \simeq G$.
\end{proposition}
\begin{proof}
By \cite[\S~5]{MR1792294} and \cref{Dieudonne-Manin}, there exists an abelian variety $Y$ and an isogeny $\alpha \colon Y[p^\infty] \to G$.
Let $K \colonequals \Ker\alpha$.
Then $X \colonequals Y/K$ has the desired property.
\end{proof}
\begin{proposition}\label{not-derived-invariant} 
There exists an abelian threefold $X$ over an algebraically closed field $k$ of characteristic $p>0$ such that $\pi_1^{\mathrm{U}}(X)$ and $\pi_1^{\mathrm{U}}(X^\vee)$ are non-isomorphic. 
\end{proposition}

\begin{proof}
In \cref{a=1}, we constructed a set of isomorphism classes of Dieudonn\'e modules in bijection with $\GG_m(k)$. Let $f(z) \in k[z]$ denote the polynomial $z^2 + (-1)^{p}$. Let $t \in k^*$ be such that $f(t) \ne 0$. Let $M$ be the Dieudonn\'e module whose isomorphism class corresponds to $t$ under \cref{a=1}. By \cref{covariant-Dieudonne}, there exists a $p$-divisible group $G$ such that $\mathbf{M}(G) \simeq M$. By construction, $G$ has a symmetric Newton polygon. Therefore, by \cref{Oort}, there exists an abelian variety $X$ such that $X[p^\infty] \simeq G$. Since $\dim M = 3$, we have $\dim X = 3$. By choice of $t$ and \cref{a=1}.\ref{a=1-dual}, it follows that $G$ and $G^\vee$ are non-isomorphic. Since $G^\vee \simeq X^\vee [p^\infty]$, it follows that $X[p^\infty]$ and $X^\vee [p^\infty]$ are not isomorphic. Since the slopes of $G$ are $\frac 1 3 $ and $\frac 2 3$, it follows that $G$ and $G^\vee$ are connected. This implies that the formal Lie groups obtained by formally completing $X$ and $X^\vee$ are non-isomorphic. Therefore, we are done by \cref{fglabelian}.
\end{proof}{}

\begin{remark}[Derived equivalent abelian threefolds]\label{worldcup3} Since the Fourier--Mukai transform associated with the Poincar\'e bundle defines a derived equivalence $D_\mathrm{perf}(X) \simeq D_\mathrm{perf}(X^\vee)$, \cref{not-derived-invariant} shows that the unipotent fundamental group scheme for abelian threefolds (or equivalently, the Dieudonn\'e module $H^1(X, W)$) is not a derived invariant.
\end{remark}{}
\begin{remark}\label{worldcup4}
Note that, at least in odd characteristic, the unipotent fundamental group scheme is a derived invariant for abelian varieties of dimension $\le 2$. To see this, we recall a few facts. By \cite[Thm.~2.19]{MR1921811}, two derived equivalent abelian varieties $X$ and $Y$ are related by an isomorphism $X \times X^\vee \simeq Y \times Y^\vee$.
Further, any abelian variety is isogenous to its dual. Consequently, if $X$ and $Y$ are derived equivalent elliptic curves, then they are are either both ordinary or both supersingular and therefore the claim holds.
Also, if $X$ and $Y$ are derived equivalent abelian surfaces, then the slopes of their associated $p$-divisible groups are identical and there are three possibilities for the slopes:
\begin{enumerate}
\item All the slopes are $0$ and $1$; in that case, $X^\vee$ and $Y^\vee$ are ordinary and therefore the associated formal Lie groups are isomorphic.
\item The slopes are $0$, $\frac{1}{2}$ and $1$; in that case $(X^\vee)[p^\infty] \simeq \QQ_p/\ZZ_p \times \mu_{p^\infty} \times H[p^\infty]$ where $H$ is a $1$-dimensional formal Lie group of height $2$, and similarly for $Y$.
But over an algebraically closed field, up to isomorphisms, there is only one formal Lie group of dimension $1$ and height $2$. Therefore, the formal Lie groups of $X^\vee$ and $Y^\vee$ are isomorphic.
\item The only slope is $\frac{1}{2}$; in that case $X$ and $Y$ are supersingular, derived equivalent abelian surfaces over a field of odd characteristic and therefore, according to \cite[Thm.~4.6]{LZ21a}, are isomorphic.
\end{enumerate}
\end{remark}

\begin{example}
A $5$-dimensional example similar to \cref{not-derived-invariant} can be obtained in a relatively easier way from the $p$-divisible groups appearing in \cite[\S~5]{MR1703336}.
For any coprime $m,n \in \ZZ_{>0}$, the authors define isogenous $p$-divisible groups $G_{m,n}$ (whose contravariant Dieudonn\'e module is the $\mathscr{D}_{n,m}$ from \cref{R_mn}) and $H_{m,n}$ of height $m+n$ which satisfy $G^\vee_{m,n} \simeq G_{n,m}$ and $H^\vee_{m,n} \simeq H_{n,m}$ \cite[\S\S~5.2--5.3]{MR1703336}. However, $G_{2,3}$ and $H_{2,3}$ are not isomorphic because their types are the semi-modules $\{0\} \cup (2 + \ZZ_{\ge 0})$ and $1 + \ZZ_{\ge 0}$, which are not integral translates of one another;
\textit{cf.}~\cite[\S\S~5.6--5.7]{MR1703336}.
Thus, $G \colonequals G_{2,3} \oplus H_{3,2}$ is a connected $p$-divisible group of height $10$ with $G^\vee \simeq G_{3,2} \oplus H_{2,3}$.
By the Dieudonn\'e--Manin classification, any isomorphism $G \xrightarrow{\sim} G^\vee$ would restrict to an isomorphism $G_{2,3} \xrightarrow{\sim} H_{2,3}$, which is impossible.
Thus, we can proceed as in \cref{not-derived-invariant} to construct an abelian variety $X$ of $\dim X = 5$ such that $\pi^{\mathrm U}_1(X) \not\simeq \pi^{\mathrm U}_1(X^\vee)$.
\end{example}

\subsection{Families of \texorpdfstring{$K(\pi^{\mathrm U},1)$}{K({\textpi}U,1)}-schemes}\label{cold1}In \cite{MR682517}, Nori considered the question of whether there is a flat variation of the unipotent fundamental group scheme for algebraic varieties defined in a family. Note that unlike the case of affine stacks over a field \cite[Thm.~2.4.5]{Toe}, the following example shows that the sheaf of homotopy groups of a general family of affine stacks is not representable.

\begin{example}[To\"en]\label{weirdexample} Let $\mathbf{Z}_p$ denote the ring of $p$-adic integers. Let $K$ denote the complex of $\mathbf{Z}_p$-modules given by $\mathbf{F}_p[1] \simeq (\mathbf{Z}_p \xrightarrow[]{\times p} \mathbf{Z}_p)$, where the $\mathbf{Z}_p$ on the right is in homological degree $1$.
It follows that the functor that sends an affine scheme $\Spec A$ over $\mathbf{Z}_p$ to $A \otimes_{\mathbf{Z}_p} \mathbf{F}_p [1]$ is naturally a pointed connected affine stack, which we denote as $X$. In this situation, one sees that $\pi_1 (X)$ is not representable. Let us mention the work of Hirschowitz \cite{Hirs1} and Simpson \cite{Simp11} which studies certain non-representable sheaf of groups in a related context.
\end{example}{}
In what follows, we construct a flat variation of the unipotent fundamental group scheme for families of curves and abelian varieties; see \cref{norisquestion}. The existence of such a construction in the case of curves was stated in \cite{MR682517} (without proof), and relying on it, a construction of a non-commutative formal Lie group defined in a family was sketched. Our method of construction uses the higher algebraic description of the unipotent fundamental group scheme obtained in \cref{algebraicdesc}.
\begin{proposition}\label{rent13}
Let $f \colon X \to \Spec A$ be a family of curves\footnote{i.e., $f\colon X \to \Spec A$ is a flat proper finitely presented morphism of relative dimension $1$ whose fibres are cohomologically connected.} of genus $g$ or abelian varieties of dimension $g$.
Let $\sigma \colon \Spec A \to X$ be a section of $f$, which equips $R\Gamma(X,\cO)$ with the structure of an augmented $E_\infty$-algebra over $A$.
Then the $E_\infty$-algebra $A \otimes_{R\Gamma(X,\cO)} A$ over $A$ is a discrete flat $A$-algebra.
\end{proposition}
\begin{proof}
The cohomological connectedness of the fibres gives $H^0 (X, \cO) \simeq A$ (see, e.g., \cite[\href{https://stacks.math.columbia.edu/tag/0E6B}{Tag~0E6B}]{stacks}). Further, note that the discreteness of $A \otimes_{R\Gamma(X,\cO)} A$ over a general base $\Spec A$ automatically implies flatness, by base changing to $\Spec A/I$ for every ideal $I \subset A$ and using \cite[\href{https://stacks.math.columbia.edu/tag/00M5}{Tag~00M5}]{stacks}.
\vspace{2mm}

We proceed by considering the bar resolution, which gives us the following simplicial object in $D(A)$
\begin{equation}\label{bar-resolution}
     \xymatrix{
 \cdots   R\Gamma(X,\cO) \otimes_A R\Gamma(X,\cO) \ar[r]<3pt>\ar[r]\ar[r]<-3pt>  &   R\Gamma(X,\cO) \ar[r]<1.5pt>\ar[r]<-1.5pt> & \ A\,
}
\end{equation}{}whose colimit (in the $\infty$-category $D(A)$) is the tensor product $A \otimes_{R\Gamma(X,\cO)} A$. Note that the terms of \cref{bar-resolution} are coconnective objects of $D(A)$ and have compatible increasing exhaustive $\mathbf N$-indexed filtrations coming from the truncation functors.
By taking colimits, we obtain an increasing exhaustive $\mathbf N$-indexed filtration on $A \otimes_{R\Gamma(X,\cO)} A$. One sees that the $n$-th graded piece of this filtration is given by $B_n [-n]$, where $B_0 \simeq A$ and for $n \ge 1$, $B_n$ is the colimit of the following simplicial object\begin{equation*}\label{bar-resolution-graded}
    C_n \colonequals \xymatrix{
 \cdots  H^n(R\Gamma (X, \cO)^{\otimes 3}) \ar[r]<4.5pt>\ar[r]<1.5pt>\ar[r]<-4.5pt>\ar[r]<-1.5pt> &  H^n(R\Gamma (X, \cO)^{\otimes 2})  \ar[r]<3pt>\ar[r]\ar[r]<-3pt>  &  H^n(X, \cO)  \ar[r]<1.5pt>\ar[r]<-1.5pt> & 0\ .
}
\end{equation*}We note that $V \colonequals  H^1 (X, \cO)$ is a projective module over $A$ of rank $g$ (e.g., \cite[\href{https://stacks.math.columbia.edu/tag/0E1J}{Tag~0E1J}]{stacks} in the case of curves and \cite[Lem.~2.5.3]{BBM} in the case of abelian varieties) and the simplicial object $C_1$ is of the form
\[  \xymatrix{
 \cdots V^{\oplus 3} \ar[r]<4.5pt>\ar[r]<1.5pt>\ar[r]<-4.5pt>\ar[r]<-1.5pt> &  V^{\oplus 2} \ar[r]<3pt>\ar[r]\ar[r]<-3pt>  &  V \ar[r]<1.5pt>\ar[r]<-1.5pt> & 0\ .
}\]
Considering $V$ as an abelian group, we observe that the above simplicial object is isomorphic to the classifying object $B V$. Now we divide the proof into two cases.
\vspace{2mm}

\noindent
\textit{Case 1.} If $f \colon X \to \Spec A$ is a family of curves of genus $g$, then by using the fact that the cup product structure on $H^* (X, \cO)$ is trivial, we see that the simplicial object corresponding to $V^{\otimes{n}}[n]$ under the Dold--Kan correspondence is homotopy equivalent to $C_n$. This implies that $B_n \simeq V^{\otimes{n}}[n]$ for $n \ge 1$. Therefore, $B_n[-n] \simeq  V^{\otimes{n}}$, which is, in particular, discrete.
Thus, the graded pieces of the filtration on $A \otimes_{R\Gamma(X, \cO)}A$ are all discrete. Hence $A \otimes_{R\Gamma(X, \cO)}A$ must also be discrete.
\vspace{2mm}

\noindent
\textit{Case 2.}
Let us now suppose that $f \colon X \to \Spec A$ is a family of abelian varieties of dimension $g$. 
Since $H^* (X, \cO) \simeq \wedge^* H^1 (X, \cO)$ (e.g., by \cite[Prop.~2.5.2.(ii)]{BBM}), it also follows that $H^* (R\Gamma(X, \cO)^{\otimes{k}}) \simeq \wedge^* H^1 (R\Gamma(X, \cO)^{\otimes{k}})$. Thus, we see that $C_n$ is computed by applying $\wedge^n$ termwise to the simplicial object $C_1$. Since $C_1 \simeq BV$, it follows that $B_n \simeq \wedge^n (V[1])$, where the $\wedge^i$ in the latter is taken in the derived sense. 
By the d\'ecalage formula (see \cite[Prop.~I.4.3.2.1]{Ill71}), we have $\wedge^n (V[1]) \simeq (\Gamma^n V) [n]$. Thus, we see that $B_n[-n] \simeq \Gamma^n V$, which is discrete. Similar to the above case, now we can conclude that $A \otimes_{R\Gamma(X, \cO)}A$ is discrete. This finishes the proof.
\end{proof}
\begin{proposition}\label{norisquestion}
Let $f \colon X \to \Spec A$ be a family of curves of genus $g$ or abelian varieties of dimension $g$ equipped with a section $\sigma \colon \Spec A \to X$. Then there exists a flat affine group scheme denoted as $\pi_1^{\mathrm{U}}(X/ \Spec A)$ over $\Spec A$ such that for all $s \in \Spec A$, we have $\pi_1^{\mathrm{U}}(X_s) \simeq \pi_1^{\mathrm{U}}(X/ \Spec A)_s$.
\end{proposition}

\begin{proof}
By \cref{rent13}, the $E_\infty$-ring $A \otimes_{R\Gamma(X, \cO)}A$ is a discrete flat algebra over $A$. Note that the functor correpresented by $A \otimes_{R\Gamma(X, \cO)}A$ in the category of $E_\infty$-algebras over $A$ is naturally valued in grouplike $E_1$-spaces. Since $A \otimes_{R\Gamma(X, \cO)}A$ is discrete, we see that $A \otimes_{R\Gamma(X, \cO)}A$ naturally has the structure of a Hopf algebra over $A$. We let $\pi_1^{\mathrm{U}}(X/ \Spec A) \colonequals \Spec (A \otimes_{R\Gamma(X, \cO)}A)$, equipped with the natural structure of a flat group scheme over $A$. For every map $s \colon \Spec k \to \Spec A$, where $k$ is a field, we see that $\pi_1^{\mathrm{U}}(X/\Spec A)_s$ is naturally isomorphic to $\Spec (k \otimes_{R\Gamma(X_s, \cO)}k)$, when the latter is also equipped with its natural group scheme structure over $k$. Let $Y \colonequals \mathbf{U}(X_s)$. Then $Y$ is naturally a pointed connected stack over $k$, and by \cref{rent14}, it follows that $k \otimes_{R\Gamma(X_s, \cO)}k \simeq R\Gamma (\Omega Y, \cO)$. However, by construction, $Y$ is an affine stack. Therefore it follows that $\Omega Y \simeq \Spec (k \otimes_{R\Gamma(X_s, \cO)}k)$. Thus, we have isomorphisms of group schemes $$\pi_1^{\mathrm{U}} (X_s) \simeq \pi_1 (Y) \simeq \pi_0 (\Omega Y) \simeq \Spec (k \otimes_{R\Gamma(X_s, \cO)}k).$$ This finishes the proof.
\end{proof}{}

\begin{remark}\label{akhil78}
The construction of $\pi_1^{\mathrm{U}}(X/\Spec A)$ above is functorial and glues to yield a flat group scheme $\pi_1^{\mathrm{U}}(X/S)$ for any flat family $X \to S$ of curves of genus $g$ (equipped with a section) or abelian schemes of dimension $g$. It is also true that $\mathbf{U}(X) \simeq B \pi_1^{\mathrm{U}}(X/\Spec A)$; however, since we are working over a general base, the conclusion does not follow from \cref{connn}. Instead, the claim follows from the notion of ``coconnectively faithfully flat" maps introduced in \cite[Prop.~2.10, Prop.~3.3]{soon}. In particular, it would imply that $\mathbf{U}(X)$ is a pointed connected $1$-stack.
\end{remark}

\begin{example}[The filtered circle]\label{cold2}
Let $\mathscr{X} \to \mathbf{A}^1$ denote the standard flat family of curves that degenerates nodal curves to a cusp (given by the equation $y^2z = x^3 + txz^2$,  where $t$ denotes the chosen coordinate of $\mathbf{A}^1$). Note that $\mathrm{Pic}^0$ of this family provides a degeneration of $\mathbf{G}_m$ to $\mathbf{G}_a.$ In particular, it follows that in this case, $\pi_1^{\mathrm{U}}(\mathscr{X}/\mathbf{A}^1)$ from \cref{norisquestion} is the affine group scheme dual to the formal group law $X+ Y+ tXY$ over $\mathbf{A}^1$. This implies that $\pi_1^{\mathrm{U}}(\mathscr{X}/ \mathbf{A}^1) \to \mathbf{A}^1$ is $\GG_m$-equivariant. Therefore, one obtains a group scheme $\pi_1^{\mathrm{U}}(\mathscr{X}/ \mathbf{A}^1)_{\mathrm{Fil}}$ over the stack $[\mathbf{A}^1/ \GG_m]$ whose classifying stack recovers the key definition of the \emph{filtered circle} from \cite{moulinos2019universal}, which the authors introduced in order to explain the existence of the Hochschild--Kostant--Rosenberg filtration on Hochschild homology. In other words, one sees that the unipotent homotopy type of the family of curves $\mathscr{X} \to \mathbf{A}^1$, when viewed with the natural structure of a filtered group stack, recovers the filtered circle. This gives a simple new realization of the filtered circle in the world of algebraic geometry, even though it cannot be directly realized as a space. In \cite[\S~3.3]{moulinos2019universal}, the authors identify the underlying stack of the filtered circle as affinization of the circle $\SSS^1$: in our construction, this can be seen directly from \cref{usefulnoriwork}. 
\end{example}

\begin{example}[Moduli of curves]\label{cold3}
Let $\overline{\cM}_{g,1}$ be the moduli space of stable curves with one marked point and let $q_1 \colon \overline{\cC}_{g,1} \to \overline{\cM}_{g,1}$ be the universal curve over it.
The marking gives rise to a natural section $\sigma \colon \overline{\cM}_{g,1} \to \overline{\cC}_{g,1}$.
For any affine $\Spec A \to \overline{\cM}_{g,1}$, \cref{norisquestion} gives a group scheme $\pi^\rU_1\bigl((\overline{\cC}_{g,1} \times \Spec A)/\Spec A\bigr)$ over $\Spec A$.
Since this association is functorial in $A$, we can glue the various group schemes to obtain a unipotent fundamental group scheme $\pi^\rU_1(\overline{\cC}_{g,1}/\overline{\cM}_{g,1})$ over $\cM_{g,1}$, whose fibre over a point $s \in \overline{\cM}_{g,1}(k)$ corresponding to a marked stable curve $(C,c)$ over $k$ is $\pi^\rU_1(\overline{\cC}_{g,1}/\overline{\cM}_{g,1})_s \simeq \pi^\rU_1(C)$. A natural filtration on $\pi^\rU_1(\overline{\cC}_{g,1}/\overline{\cM}_{g,1})$ is induced by the lower central series
\[ \pi^\rU_1(\overline{\cC}_{g,1}/\overline{\cM}_{g,1}) = \pi^\rU_1(\overline{\cC}_{g,1}/\overline{\cM}_{g,1})^{[1]} \ge \pi^\rU_1(\overline{\cC}_{g,1}/\overline{\cM}_{g,1})^{[2]} \ge \dotsb \ge \pi^\rU_1(\overline{\cC}_{g,1}/\overline{\cM}_{g,1})^{[n]} \ge \dotsb,\]
where we recursively define  $\pi^\rU_1(\overline{\cC}_{g,1}/\overline{\cM}_{g,1})^{[n+1]} \colonequals \bigl[\pi^\rU_1(\overline{\cC}_{g,1}/\overline{\cM}_{g,1})^{[n]},\pi^\rU_1(\overline{\cC}_{g,1}/\overline{\cM}_{g,1})\bigr]$.
The associated graded sheaf of rings $\gr^* \pi^\rU_1(\overline{\cC}_{g,1}/\overline{\cM}_{g,1}) \colonequals \bigoplus_{n \in \NN} \gr^n \pi^\rU_1(\overline{\cC}_{g,1}/\overline{\cM}_{g,1})$ with
\[ \gr^n \pi^\rU_1(\overline{\cC}_{g,1}/\overline{\cM}_{g,1}) \colonequals \pi^\rU_1(\overline{\cC}_{g,1}/\overline{\cM}_{g,1})^{[n]} / \pi^\rU_1(\overline{\cC}_{g,1}/\overline{\cM}_{g,1})^{[n+1]} \]
is again generated by $\gr^1 \pi^\rU_1(\overline{\cC}_{g,1}/\overline{\cM}_{g,1}) \simeq R^1q_{1*}\cO$.
It would be interesting to find a description of $\gr^* \pi^\rU_1(\overline{\cC}_{g,1}/\overline{\cM}_{g,1})$ in terms of the geometry of $\overline{\cC}_{g,1} \to \overline{\cM}_{g,1}$.

\end{example}

\newpage

\section{Formal Lie groups via  unipotent homotopy theory}
In this section, we begin by establishing several foundational results on unipotent group schemes which extend certain results from \cite{Serre} for pro-$p$-finite groups (see \cref{cof0}). These results are used in \cref{cof1} to obtain a purely algebraic perspective on (potentially non-commutative) formal Lie groups and construct new formal Lie groups by dualizing unipotent homotopy group schemes of certain higher stacks (see \cref{cof6}). In \cref{cof3}, we show that the formal Lie groups we construct recover the ones constructed by Artin--Mazur \cite{MR457458} in many cases of interest.

\subsection{Some homological results on unipotent group schemes}\label{cof0}

Let $G$ be a unipotent affine group scheme over a field $k$. In this section, our first goal is to establish necessary and sufficient criteria such that the dual of $G$ defines a non-commutative formal Lie group according to \cref{nonfgl}. In \cref{bgcurve1}, we show that the dual of $G$ is a non-commutative formal Lie group of dimension $g$ if and only if $\dim H^1 (BG, \cO)=g$ and $\dim H^2 (BG, \cO) =0$. Note that if $k$ has characteristic $p>0$, and $G$ is an affine group scheme corresponding to a pro-$p$-finite group, then according to \cref{bgcurve1}, the conditions $\dim H^1(BG, \cO)=g$ and $H^2(BG, \cO)=0$ are equivalent to $G$ being the free pro-$p$ group on $g$ generators (\cref{cold91}); this recovers \cite[\S~4.2, Cor.~2]{Serre} (\textit{cf}.~\cite[\S~4.1, Prop.~21]{Serre}). Our strategy for proving \cref{bgcurve1} is to formally treat $X = BG$ as a \textit{curve of genus} $g$ and adapt Nori's proof of the fact that dual of the unipotent fundamental group scheme of a curve is a non-commutative formal Lie group (see {\cite[\S~IV, Prop.~4]{MR682517}}). 
\vspace{2mm}

 In \cref{fine}, we show that the dual of a commutative unipotent group scheme $G$ is representable by a commutative formal Lie group of dimension $g$ if and only if $\dim H^1 (BG, \cO)=g$ and $\mathrm{Ext}^1 (G, \GG_a)=0$; for another criterion, see \cref{commgp}. The proof, in this case, is very different from the non-commutative case. We make use of the cotangent complex \cite{Ill71,Ill72}, the related notion of the co-Lie complex (\cite{Il1}), as well as some inputs from commutative algebra regarding Betti numbers for complete local rings (\cite{Gull1}, \cite{Gull2}).
\vspace{2mm}

 We will begin by describing the dual of an affine group scheme as a linearly compact topological Hopf algebra (\cref{lincompacth}) following \cite{{MR682517}}.

\begin{construction}[Dualizing group schemes]\label{dualizinggpsch}
Let $G$ be an affine group scheme over  a field $k$. Let $G = \Spec R$ for a $k$-algebra $R$. Let $\mu \colon R \to R \otimes_k R$ denote the comultiplication. We look at the vector space dual $A \colonequals R^*$ which has a (not necessarily commutative) algebra structure induced from the comultiplication on $R$. 
Let $V$ be a finite-dimensional subspace of $R$ such that $\mu (V) \subseteq V \otimes V$. One defines $V^{\perp} \colonequals \left \{ a \in A \suchthat a(v) = 0 \text{ for all } v \in V \right \}$. We note that there is a natural map $A = R^* \to V^*$ obtained by restricting a linear functional to the subspace $V$. This map is surjective with kernel $V^\perp$. 
Thus $A/V^{\perp}$ is finite-dimensional. Since $V$ is also $\mu$-stable it follows that $A/V^{\perp}$ is a finite-dimensional algebra. The collection of all such $\mu$-stable finite-dimensional subspaces $V$ of $R$ provides a system of neighbourhoods of zero in the algebra $A$ and $A$ is complete with respect to the topology induced from this, i.e., $A = \varprojlim_{V} A/ V^\perp$. In particular, by taking $V=k$, we obtain a two-sided ideal of $A$ given by $k^\perp$ which will be denoted as $\mathfrak{m}$. It follows that $A/ \mathfrak{m} \simeq k$. 
If $G$ is further assumed to be unipotent, it follows that $A/V^{\perp}$ is an artinian local ring with maximal ideal $\mathfrak{m}$. We set $J_n \colonequals \bigcap _{V} \mathfrak{m}^n + V^{\perp}$. Since $\mathfrak{m}= k^{\perp}$, it follows that $J_1 = \mathfrak{m}$. We have the following proposition.

\end{construction}

\begin{proposition} Let $G = \Spec R$ be a unipotent affine group scheme over $k$ such that $\dim H^1 (BG, \cO)= \dim \Hom(G, \GG_a)= g$ for some $g$. Then (in the above notations), we have 
\begin{enumerate}
\item Each $J_n$ is an open two sided ideal of $A= R^*$ and $A = \varprojlim_{n} A/J_n$.
\item $(J_1/J_2)^* = H^1 (BG, \cO)$.
\item $\dim A/J_n \le 1 + g+ g^2 + \ldots + g^{n-1}$.
\end{enumerate}
\end{proposition}{}
\begin{proof} This follows from \cite[Lem.~IV.3]{MR682517} (and its proof).
\end{proof}{}A quasi-coherent sheaf on $BG$ is the same as an $R$-comodule. Any quasi-coherent sheaf on $BG$ is a filtered colimit of vector bundles, which corresponds to a finite-dimensional representation of the affine group scheme $G$. This implies that quasi-coherent sheaves on $BG$ is the same as left modules over $A$ such that each element of the module is killed by some $V^\perp$ where $V$ again varies over finite-dimensional $\mu$-stable subspaces of $R$. Denoting the category of such $A$-modules by $A\w{-nilMod}$, we have an equivalence of categories $C \colon \QCoh(BG) \simeq A\w{-nilMod}$. Note that for a vector bundle $\mathscr{V}$ on $BG$, we have $\rk \mathscr{V} = \dim_k C (\mathscr{V})$. We will begin by recalling certain universal constructions on vector bundles on $BG$.

\begin{definition}[Universal vector extensions]
Let $\mathscr{V}$ be a vector bundle on $BG$, where $G$ is unipotent and $H^1 (BG, \cO)$ is assumed to be finite-dimensional. An exact sequence $ 0 \to \cO^m \to U(\mathscr{V}) \to \mathscr{V} \to 0$ is called a universal vector extension of $\mathscr{V}$, if for every exact sequence $0 \to \cO^n \to W' \to \cV \to 0$ there is a unique map $f$ which fits into a diagram 
\begin{center}
\begin{tikzcd}
0 \arrow[r] & {\cO^m} \arrow[r] \arrow[d, "f"] & U(\mathscr{V}) \arrow[r] \arrow[d] & \cV \arrow[r] \arrow[d,equals] & 0 \\
0 \arrow[r] & \mathscr O ^n \arrow[r]                                   & W' \arrow[r]                       & \cV \arrow[r]                 & 0
\end{tikzcd}
\end{center}{}

We note that under our assumptions, every vector bundle $\mathscr{V}$ on $BG$ has a universal vector extension. Indeed, since $\mathscr{V}^*$ corresponds to a representation of $G$, and $G$ is unipotent, it follows that $\mathscr{V}^*$ is an iterated extension of $\cO$ as vector bundle on $BG$. Therefore, by induction, it follows that $H^1 (BG, \mathscr{V}^*)$ is finite-dimensional. Let us choose $v_1, \ldots, v_m$ as a basis of $H^1 (BG, \mathscr{V}^*)$. Then $\theta= (v_1, \ldots, v_m) \in H^1 (BG, \mathscr{V}^*)^{\oplus m}= H^1 (BG, \mathscr{H}om (\mathscr{V}, \cO^m))= \Ext^1(\mathscr{V}, \cO^m)$ defines the required extension $ 0 \to \cO^m \to U(\mathscr{V}) \to \mathscr{V} \to 0$. More canonically, one can write this extension as $0 \to H^1(BG, \mathscr{V}^*)^* \otimes \cO \to U(\mathscr{V}) \to \mathscr{V}\to 0$. 
\end{definition}{}
\begin{remark}\label{univvanish}
By a long exact sequence chase, it follows from the construction that the map $H^1 (BG, \mathscr{V}^*) \to H^1 (X, U(\mathscr{V})^*)$ is the zero map. 
The assumptions that $G$ is unipotent and that $H^1(BG, \cO)$ is finite-dimensional both play an important role in the construction of universal vector extensions for an arbitrary vector bundle on $BG$.
\end{remark}{}

Similarly, for $A$-modules $M$, one can define a universal vector extension $0 \to k^n \to U(M) \to M \to 0$ to be an extension with analogues universal properties. Let $\mathscr{V}$ be a vector bundle on $BG$ as above. Let $C(\mathscr{V})$ be an object of $A\w{-nilMod}$ via the equivalence of categories $C \colon \QCoh(BG) \simeq A\w{-nilMod}$. Then it follows from the categorical equivalence that $U (C(\mathscr{V})) \simeq C (U(\mathscr{V}))$. Let us inductively define a sequence of vector bundles on $BG$ starting with $\mathscr{V}_1 \colonequals \cO$ and $\mathscr{V}_{n+1} \colonequals U(\mathscr{V}_n)$. We note the following proposition from Nori.

\begin{proposition}\label{c(v)}Under the equivalence of categories $C \colon \QCoh(BG) \simeq A\mathrm{-nilMod}$, it follows that $C (\mathscr{V}_n)= A/J_n$, where $\mathscr{V}_n$ is as constructed above.

\end{proposition}{}

\begin{proof}The $n=1$ case is clear. For $n \ge 1$, this amounts to saying that $U(A/J_n)= A/J_{n+1}$. The proposition is proven in \cite[Lem.~IV.7]{MR682517}.
\end{proof}{}

Before delving into determining cohomological criteria for the representability of duals of unipotent group schemes by a formal Lie group, we record the following calculations.

\begin{proposition}\label{newyr} 
Let $G$ be a group scheme over $k$ whose dual is a non-commutative formal Lie group of dimension $g$. Then in $D(k)$, we have an isomorphism $$R\Gamma(BG, \cO) \simeq R\Hom_{k\llbrace X_1, \ldots, X_g \rrbrace}(k,k).$$
\end{proposition}{}

\begin{proof}
Using the equivalence of categories states in \cref{c(v)} and \cref{nc-group-law-cohomology}, it would be enough to prove that $H^i (BG,\cO)=0$ for $i>1$. Note that there is an explicit complex which computes $R\Gamma(BG, \cO)$ obtained by applying fpqc descent along the map $* \to BG$. Further, that complex depends only on the coalgebra structure of the Hopf algebra underlying the ring of global sections on $G$. Therefore, it is enough to prove the cohomology vanishing assertion in the special case when $G$ is the dual of the non-commutative formal group law $F_i (X,Y) \colonequals X_i+Y_i +X_iY_i$. In this case, $G$ is the free pro-$p$ group scheme on $g$ generators. Thus the claim follows from the classical computation of group cohomology of free groups and \cref{compare}.
\end{proof}{}

\begin{proposition}\label{paderborn1}
Let $G$ be a group scheme over $k$ whose dual is a commutative formal Lie group of dimension $g$. Then we have an isomorphism
$$R\Gamma(BG, \cO) \simeq R\Hom_{k \llbracket x_1, \ldots, x_g \rrbracket}(k,k)$$ of associative algebra objects in $D(k)$. 
\end{proposition}{}
\begin{proof}
The equivalence of categories from \cref{c(v)} extends to a fully faithful functor from the $\infty$-category of dualizable objects of $D_{\mathrm{qc}}(BG)$ to the derived $\infty$-category $D(k \llbracket x_1, \ldots, x_g \rrbracket)$. Under this categorical equivalence, the structure sheaf $\cO$ is sent to the $k \llbracket x_1, \ldots, x_g \rrbracket$-module $k$. This yields the desired result.
\end{proof}{}

\begin{remark}\label{retire4}
Let $G$ be a group scheme over $k$ whose dual is a commutative formal Lie group of dimension $g$. Then $\mathrm{Ext}^1(G, \mathbf{G}_a)=0$. In order to see this, note that by \cref{gait1}, $\mathrm{Ext}^1(G, \mathbf{G}_a) \simeq H^3 (K(G,2), \cO)$. By \cref{paderborn1}, $H^* (BG, \cO) \simeq \wedge^* H^1 (BG, \cO)$. Applying descent along $\Spec k \to K(G, 2)$, and directly carrying out the strategy of \cite[\S~3.3]{Mon21}, it follows that the cohomology ring $H^*(K(G,2), \cO)$ is a symmetric algebra in degree $2$. In particular, $H^3 (K(G,2), \cO)=0$, which gives the desired claim.
The same argument carries over to any affine base;
therefore, we also obtain $\mathscr{E}xt^1(G,\GG_a) = 0$.
\end{remark}{}

Now we are ready to prove the following.
\begin{proposition}\label{bgcurve1}Let $G$ be a unipotent affine group scheme over $k$ such that $\dim H^1 (BG, \cO)=g$. Then the dual of $G$ is a non-commutative formal Lie group of dimension $g$ if and only if $ H^2 (BG, \cO) =0$.
\end{proposition}{}

\begin{proof}
The only if direction follows from \cref{nc-group-law-cohomology} and \cref{newyr}.
\vspace{2mm}

We now prove the converse. Let $B \coloneqq k \llbrace X_1, \ldots, X_g \rrbrace$. Let $\mathfrak{m}^\prime$ denote the two sided ideal $(X_1,..., X_g)$ of $B$. 
We choose $x_1, \ldots, x_g \in J_1$ such that they are a basis of $J_1/J_2$ (see \cref{dualizinggpsch}). Sending $X_i$ to $x_i$ defines a surjection $B \twoheadrightarrow A$, where $A$ is defined as the dual of the Hopf algebra underlying $G$ as in \cref{dualizinggpsch}. We will be done if we prove that $B/ \mathfrak{m}^{\prime n} \to A/J_n$ is an isomorphism for all $n \ge 1$. Since these maps are already surjective and $\dim B/ \mathfrak{m}^{\prime n} = 1+ g + \ldots g^{n-1}$, it would be enough to show that under our assumptions, $\dim A/J_n = 1+ g + \ldots + g^{n-1}$.

\vspace{2mm}
Now we will make use of our preparation involving universal vector extensions. Defining $\mathscr{V}_n$ as in the set up of \cref{c(v)}, we note that (as a consequence of \cref{c(v)}), we have that $\rk \mathscr{V}_n= \dim A/J_n$. We note the following lemma.

\begin{lemma}In this set up, $H^2 (BG, \mathscr{V}_n^*)=0$ for all $n \ge 1$.

\end{lemma}{}

\begin{proof}This follows by induction on $n$ by using the assumption $H^2(BG, \cO)=0$ and the exact sequence $0 \to \cO^m \to \mathscr{V}_{n+1} \to \mathscr{V}_n \to 0$.
\end{proof}{}
We note that in the exact sequence $0 \to \cO^m \to \mathscr{V}_{n+1} \to \mathscr{V}_n \to 0$, we must have $m = \dim H^1 (BG, \mathscr{V}_n^*)$ by construction. Therefore, in order to prove that $\rk \mathscr{V}_n = 1+g + \ldots + g^{n-1}$, by using induction on $n$, it would be enough to prove that $\dim H^1 (BG, \mathscr{V}_n^*)= g^n$ for all $n \ge 1$. To show this, we look at the dual exact sequence $0 \to \mathscr{V}_n^* \to \mathscr{V}_{n+1}^* \to \cO^m \to 0$. As noted in \cref{univvanish}, the map $H^1(BG, \mathscr{V}_n^*) \to H^1(BG, \mathscr{V}_{n+1}^*)$ is zero. Therefore, by using the long exact sequence in cohomology and the above lemma, we get that $\dim H^1 (BG, \mathscr{V}_{n+1}^*)= \dim H^1 (BG, \cO ^m)= m \cdot \dim H^1(BG, \cO)= mg= g \dim H^1 (BG, \mathscr{V}_n^*)$.
This inductively shows that indeed $H^1(BG, \mathscr{V}_n^*)= g^n$, which finishes the proof.
\end{proof}{}

Note that in \cref{bgcurve1}, $G$ is not assumed to be commutative. In fact, if $G$ were assumed to be commutative, then under the assumptions, the dual of $G$ will still be forced to be a non-commutative formal Lie group of dimension $g$, which implies that $g=1$. We also note that if $G$ is commutative, then $BG$ has the structure of an abelian group stack. In such a case, the fact that $g$ is forced to be $1$ is formally analogous to the fact that if a smooth curve $X$ of genus $g\ge 1$ has the structure of an abelian variety, then $g=1$. Indeed, $X$ being an abelian variety formally forces the canonical line bundle to be trivial, which implies $g= \dim H^1(X, \cO)= \dim H^0 (X, \cO)=1$

\vspace{2mm}
However, in the case where $G$ is a commutative group scheme, one can look at $\Ext^1 (G, \GG_a)$ instead of $H^2 (BG, \cO)$. In this case, $\Ext^1 (G, \GG_a)$ embeds inside $H^2 (BG, \cO)$. One can further formulate and prove the following proposition.

\begin{proposition}\label{onedimensionalfgl}
Let $G$ be a unipotent commutative affine group scheme over a field $k$. We assume that $\dim H^1 (BG, \cO)=1$ and $\Ext^1 (G, \GG_a)=0$. Then the dual of $G$ is a commutative formal Lie group of dimension $1$.
\end{proposition}{}

\begin{proof}
Since $\dim H^1 (BG, \cO)=1$, we can construct (as in the proof of \cref{bgcurve1}) a natural continuous surjection $\alpha \colon k \llbracket x \rrbracket \twoheadrightarrow A$, where $A$ is the dual of the Hopf algebra underlying $G$ (see \cref{dualizinggpsch}). In the case where $\alpha$ is an injective map of rings, we will show that $\alpha$ in fact happens to be an isomorphism of topological rings, which would prove that the dual of $G$ is a formal Lie group. Indeed, if $\alpha$ is injective, it would be a ring theoretic isomorphism. To prove that it is an isomorphism of topological rings, it would be enough to prove that $k\llbracket x \rrbracket/x^n \to A/J_n$ is an isomorphism for all $n \ge 1$. We already know that it is a surjection and hence $\dim A/J_n \le n$. Thus, it would be enough to show that $\dim A/J_n = n$. Since $\alpha$ is an isomorphism of rings, $A$ is a discrete valuation ring with maximal ideal $\mathfrak{m}= k^{\perp}$. By construction it follows that $\mathfrak{m}^n \subseteq J_n=\mathfrak{m}^n + E$ for any small enough open ideal $E$ in $A$. However, that implies that $\mathfrak{m}^n = J_n$ and thus $\dim A/J_n = \dim A/ \mathfrak{m}^n = n$.
\vspace{2mm}

Now we show that $\alpha$ must always be injective, which will finish the proof of the proposition. Otherwise, assume for the sake of contradiction that $\alpha$ is not injective.
Since $k \llbracket x \rrbracket$ is a discrete valuation ring with uniformizer $x$, we must then have $\Ker\alpha = (x^n)$ and $A \simeq k\llbracket x \rrbracket/x^n$ (as $k$-algebras) for some $n \in \NN$.
In other words, $G$ must be a finite flat unipotent group scheme over $k$ whose Cartier dual $G^\vee$ has underlying scheme $\Spec k[x]/x^n$.
Since $H^1(BG,\cO)=1$, we must furthermore have $n>1$.
\vspace{2mm}

Let $\LL_{G^\vee/k}$ denote the cotangent complex. Next, we will compute the co-Lie complex $\ell_{G^\vee} \colonequals Le^* \LL_{G^\vee/k}$, where $e \colon \Spec k \to G^\vee$ is the unit section \cite{Il1}.
Since $G^\vee$ is a complete intersection over $\Spec k$, its cotangent complex is given by the complex
\[ \LL_{G^\vee} \simeq (A \cdot x^n \to A \cdot dx) \] with boundary map $a \cdot x^n \mapsto nax^{n-1} \cdot dx$.
The morphism $k[x]/x^n \to k$ corresponding to the unit section $e$ must contain nilpotents in its kernel and is thus given by $x \mapsto 0$.
Therefore,
\[ \ell_{G^\vee} \simeq e^* (A \cdot x^n \to A \cdot dx) \simeq (k \xrightarrow{0} k) \simeq k[1] \oplus k[0]. \]
By Grothendieck's formula \cite[\S~14.1]{grobeau},
\[ R\Hom_k(\ell_{G^\vee},\cO) \simeq \tau_{\ge -1} R\Hom(G,\GG_a) \]
we must then have $\Ext^1(G,\GG_a) \simeq \Hom(k,k) \simeq k$, contradicting the assumption.
\end{proof}{}

We will now generalize the above result when $H^1(BG, \cO)= g >1$ in the case where $G$ is a commutative unipotent affine group scheme $G$. We note that unlike the non-commutative case studied in \cref{bgcurve1}, in the commutative case, it is \textit{impossible} to have $H^2 (BG, \cO)=0$, i.e., adding the assumption that $G$ is commutative increases the size of $H^2(BG, \cO)$. In fact, as we will see later, if $G$ is commutative then $\dim H^2(BG, \cO) \ge \binom{g}{2}$. In fact, $\dim H^n (BG, \cO) \ge \binom{g}{n}$. In this case, we have the following proposition.

\begin{proposition}\label{commgp}Let $G$ be a unipotent commutative affine group scheme over a field $k$. We assume that $\dim H^1 (BG, \cO) = g$. Then $\dim H^n (BG, \cO) \ge \binom{g}{n}$. Further, if $$\binom{g}{n} \le \dim H^n (BG, \cO) < \sum_{i \ge 0} \binom{g}{n-2i}$$ for \textit{some} $n>1$, then dual of $G$ is a commutative formal Lie group of dimension $g$; moreover, in such a case $\dim H^n (BG, \cO) = \binom{g}{n}$ for all $n$.
\end{proposition}{}

We will defer the proof of \cref{commgp} since it will be deduced from the somewhat finer proposition proven below. As before, instead of looking at $H^2 (BG, \cO)$, one can contemplate the subspace $\Ext^1 (G, \GG_a)$ and prove the following.

\begin{proposition}\label{fine}Let $G$ be a unipotent affine commutative group scheme over a field $k$ such that $\dim H^1 (BG, \cO) = g$. Then the dual of $G$ is a commutative formal Lie group of dimension $g$ if and only if $\Ext^1 (G, \GG_a)=0$.

\end{proposition}{}

\begin{proof} 
The only if part follows from \cref{retire4}. For the converse, let $G= \Spec R$ and $A \colonequals R^*$. Similar to the argument at the beginning of the proof of \cref{bgcurve1}, we have a surjection $k \llbracket x_1, \ldots, x_g \rrbracket \to A$; thus it follows that $A$ is a commutative noetherian local $k$-algebra.  Since $G$ is a commutative group scheme, it also follows that $A$ is a quotient of the power series ring $k \llbracket x_1, \ldots, x_g \rrbracket$ by a regular sequence. Let us write $A = k \llbracket x_1, \ldots, x_g \rrbracket / I$ for an ideal $I$. The maximal ideal $\mathfrak{m}$ of $A$ is $(\overline{x_1}, \ldots, \overline{x_g})$. Let $\mathfrak{m}^{[p^i]} \colonequals (\overline{x_1}^{p^n}, \ldots, \overline{x_g}^{p^n})$. We note that $\Spec A/ \mathfrak{m}^{[p^i]}$ is naturally a commutative group scheme over $k$ and the induced maps $\Spec A/ \mathfrak{m}^{[p^i]} \to \Spec A/\mathfrak{m}^{[p^{i+1}]}$ are group scheme homomorphisms. The latter claim follows from the observation that $\Spec A/ \mathfrak{m}^{[p^i]}$ is dual to the group scheme $G/V^i$, where $V$ is the Verschiebung operator on $G$. Since $G$ is unipotent, it follows that $G \simeq \varprojlim G/V^i$. This implies that $A \simeq \varprojlim_{i} A/ \mathfrak{m}^{[p^i]}$. Since $\mathfrak{m}$ is finitely generated, the system of ideals $(\mathfrak{m}^{[p^i]})_{i \ge 1}$ generate the same topology as the $\mathfrak{m}$-adic topology on $A$. This implies that $A$ is $\mathfrak{m}$-adically complete. 
\vspace{2mm}

We note that $H^i (BG, \cO) \simeq \Ext^i_A (k,k)$ by duality (\textit{cf.}~\cref{c(v)}). In this situation, by \cite{Gull2}, $\dim \Ext^2_A (k,k)= \dim \w{Tor}_2(k,k) = \binom{g}{2} + b$ for some integer $b \ge 0$. Therefore, if $\mathrm{Ext}^1(G, \GG_a)=0$, then by \cref{ext1} proven below, it follows that $b=0$. Further, if $b=0$, then (see \cite{Gull2}) $A$ is a regular complete local ring of dimension $\dim \mathrm{Tor}_1(k,k) = \dim \mathrm{Ext}^1_A(k,k) = \dim H^1(BG, \cO) = g$. This implies that the surjection $k \llbracket x_1, \ldots, x_g \rrbracket \to A$ constructed above must be an isomorphism (of complete local rings), which finishes the proof.
\end{proof}
\begin{lemma}\label{ext1} Let $G$ be a unipotent commutative group scheme over $k$ such that $\dim H^1(BG, \cO)=g$. Then $\dim H^2(BG, \cO)= \binom{g}{2} + \dim \Ext^1 (G, \GG_a). $

\end{lemma}{}

\begin{proof}In order to prove this, we can base change to a perfect field. Therefore, without loss of generality, we will now assume that $k$ is perfect. We note that there is a natural map $$s_G \colon H^2(BG, \cO) \to \Hom (G \wedge G, \GG_a)$$ which concretely sends a $2$-cocycle $u \colon G \times G \to \GG_a$ to the associated alternating bilinear form $v \colon G \times G \to \mathbf G_a$ defined as $v(x,y) \colonequals u(x,y) - u(y,x)$ for scheme theoretic points $x,y$ of $G$.
The kernel of this map is naturally identified with $\Ext^1 (G, \GG_a)$. On the other hand, there is a map $\w{Bil} (G \times G, \GG_a) \to H^2(BG , \cO)$ which just regards a bilinear form as a $2$-cocycle. Composition with $s_G$ yields a map $$t_G \colon \w{Bil} (G \times G, \GG_a) \to \Hom(G \wedge G, \GG_a).$$ Since $\dim \Hom (G, \GG_a)=g$, it follows that $\dim \w{Im}(t_G) = \binom{g}{2}$. Therefore, we would be done if we prove that $\w{Im}(t_G) = \w{Im}(s_G)$ as subspaces of $\Hom(G \wedge G, \GG_a)$. We note that when $k$ has characteristic $2$, it can happen that $\dim \Hom(G \wedge G,\GG_a) > \binom{g}{2}$. Thus the equality $\w{Im}(t_G) = \w{Im}(s_G)$ does not simply follow from dimension count in that case. Instead, we argue as below.
\vspace{2mm}

Since $\dim H^1 (BG, \cO)$ is finite, by \cref{profinitenesschar}, it follows that $G$ is profinite. Therefore, by (co)limit considerations, in order to prove that $\mathrm{Im}(t_G) = \mathrm{Im}(s_G)$, we can additionally assume that $G$ is a finite group scheme. Therefore the Cartier dual $G^\vee$ is a local finite group scheme over a perfect field $k$ and it follows that $\displaystyle{\cO(G^\vee) \simeq \frac{k[x_1, \ldots, x_g]}{(x_1^{p^{v_1}}, \ldots, x_s^{p^{v_g}})}}$ (\cite[Sec. 14.4]{waterhouse}); here $g = \dim \Hom(G, \GG_a)= \dim T_{e}(G^\vee)= \dim H^1(BG, \cO)= \dim \mathrm{Ext}^1_{\cO(G^\vee)}(k,k) = \dim \mathrm{Tor}_1(k,k)$. Since $\cO(G^\vee)$ is an Artinian local ring, we have $\dim H^2(BG, \cO) = \dim \mathrm{Ext}^2_{\cO(G^\vee)}(k,k) = \dim \mathrm{Tor}_2(k,k) = \binom{g}{2} + b$, where the last equality follows from \cite{Gull2}. Since $\cO(G^\vee)$ is also a complete intersection, it follows from \cite[Rmk.~3.2]{Gull2} that $0 = \dim \cO(G^\vee) = g-b$. On the other hand, by Grothendieck's formula \cite[\S~14.1]{grobeau}, $$ R\Hom_k(\ell_{G^\vee},\cO) \simeq \tau_{\ge -1} R\Hom(G,\GG_a).$$ Since $\displaystyle{\cO(G^\vee) \simeq \frac{k[x_1, \ldots, x_g]}{(x_1^{p^{v_1}}, \ldots, x_s^{p^{v_g}})}}$, it follows from explicitly computing the co-Lie complex $\ell_{G^\vee}$ that $\dim \mathrm{Ext}^1(G, \GG_a) = g. $ Combining the equalities obtained so far, we see that $\dim H^2(BG, \cO) = \binom{g}{2} + b = \binom{g}{2} + g = \binom{g}{2} + \dim \mathrm{Ext}^1(G, \GG_a)$. The latter implies that $\mathrm{Im}(t_G) = \mathrm{Im}(s_G)$, which finishes the proof.
\end{proof}

\begin{proposition}\label{genfunc}
Suppose that $G$ be a unipotent commutative group scheme over $k$ such that $\dim H^1 (BG, \cO)= g$. Then $H^n(BG, \cO)$ is finite-dimensional for all $n$.
Moreover, the generating function
$$p(t) \colonequals \sum_{n=0}^{\infty} \dim H^n (BG, \cO) \cdot t^n$$
has the closed-form expression
$$ p(t) = \frac{(1+t)^{g}}{(1-t^2)^{\dim \mathrm{Ext}^1(G, \GG_a)}}.$$
\end{proposition}{}

\begin{proof}
If $k$ has characteristic zero, the claim follows directly (see \cite[IV, \S~2, Prop.~4.2.b)]{MR0302656});
therefore we will assume that $k$ has characteristic $p > 0$.
By base change, we can moreover assume that $k$ is perfect.

\vspace{2mm}
First, we claim that $H^*(BG, \cO)$ is generated by elements of degree $1$ and $2$. Since $H^1 (BG, \cO)$ is finite-dimensional, by \cref{profinitenesschar}, it follows that $G$ is profinite. Therefore, by considering filtered colimits, it is enough to show that $H^*(BG, \cO)$ is generated by elements of degree $1$ and $2$ when $G$ is additionally assumed to be finite. The Cartier dual $G^\vee$ of $G$ is a local finite group scheme over $k$. Since $k$ is perfect, it follows that $\displaystyle{\cO(G^\vee) \simeq \frac{k[x_1, \ldots, x_s]}{(x_1^{p^{v_1}}, \ldots, x_s^{p^{v_s}})}}$ (\cite[Sec. 14.4]{waterhouse}). The latter is isomorphic to the group algebra of the finite abelian group $\mathbf{Z}/p^{v_1}\mathbf{Z} \times \ldots \times \mathbf{Z}/p^{v_s}\mathbf{Z}$. Therefore, by Cartier duality, $$H^* (BG, \cO) \simeq \mathrm{Ext}^*_{\frac{k[x_1, \ldots, x_s]}{(x_1^{p^{v_1}}, \ldots, x_s^{p^{v_s}})}}(k,k) \simeq H^* (B (\mathbf{Z}/p^{v_1}\mathbf{Z} \times \ldots \times \mathbf{Z}/p^{v_s}\mathbf{Z}), \cO).$$
It is thus enough to prove that the group cohomology $H^* (\mathbf{Z}/p^{v_1}\mathbf{Z} \times \ldots \times \mathbf{Z}/p^{v_s}\mathbf{Z}, k)$ is generated in degree $1$ and $2$. One can further reduce to the case of $H^* (\mathbf{Z}/p^v \mathbf{Z},k )$, where the claim follows by considering Tate resolutions (see, e.g., \cite[Prop.~4.5.1]{carl}). 
\vspace{2mm}

Returning to the setting of the proposition, let us write $G = \Spec R$ and $A = R^*$. As in the proof of \cref{fine}, one knows that $A$ is a noetherian complete local ring with residue field $k$. By duality, $H^* (BG, \cO) \simeq \mathrm{Ext}^*_A(k,k)$. In particular, $H^i (BG, \cO)$ is a finite-dimensional vector space. Since $H^* (BG, \cO)$ is generated by elements of degree $1$ and $2$, it follows that the function $i \mapsto \dim H^i (BG, \cO)$ has polynomial growth, i.e., there exists a polynomial $f(t) \in \mathbf{Z}[t]$ such that $\dim H^i (BG, \cO) \le f(i)$ for $i \gg 0$.\footnote{
For example, if $\dim H^1(BG, \cO)=g$ and $\dim H^2 (BG, \cO) = h$ and $p \ne 2$, one can find such a polynomial $f(t)$ such that $\deg f(t) < h$; if $p =2$, one can find such a polynomial $f(t)$ satisfying $\deg f(t) < g+h$.}
\vspace{2mm}

Since we have $\dim \mathrm{Tor}_i (k,k) = \dim \mathrm{Ext}^i_A(k,k) = \dim H^i (BG, \cO)$, we conclude that the function $i \mapsto \dim \mathrm{Tor}_i (k,k)$ has polynomial growth. By \cite[Sec.~2,~Thm.~2.3]{Gull2} and \cite{Gull1}, it follows that we must have $\displaystyle{\sum_{i=0}^{\infty} \dim \mathrm{Tor}_i (k,k)t^i = \frac{(1+t)^g}{(1-t^2)^l}}$. Therefore, $\dim H^2(BG, \cO) = \dim \mathrm{Ext}^2_A(k,k) = \dim \mathrm{Tor}_2(k,k) = \binom{g}{2} +l$. By \cref{ext1}, we have $l= \dim \mathrm{Ext}^1(G, \GG_a)$; thus we obtain $$\displaystyle{\sum_{i=0}^{\infty} \dim H^i (BG, \cO)t^i = \frac{(1+t)^g}{(1-t^2)^{\dim \mathrm{Ext}^1(G, \GG_a)}}}.$$ This finishes the proof.
\end{proof}{}

\begin{proof}[Proof of \cref{commgp}] Follows immediately from \cref{fine} and \cref{genfunc}. Indeed, \cref{genfunc} implies that for an arbitrary $n>1$, $\dim H^n (BG, \cO) \ge \binom{g}{n}$ with equality if and only if $\mathrm{Ext}^1(G, \GG_a)=0$. Also, by \cref{genfunc}, $\dim H^n(BG, \cO) < \sum_{i \ge 0} \binom{g}{n-2i}$ implies that $\dim \mathrm{Ext}^1(G, \GG_a)=0$.
\end{proof}{}

\begin{example}
Suppose that $k$ has characteristic $2$. Let $G$ be dual of the formal Lie group $\widehat{\GG}_a$. Then there is a surjection $G \to \alpha_2$. In this case, it follows that $\dim \mathrm{Hom}(G\wedge G, \GG_a)=1$ (see \cref{zvi1}). However, we also have $H^2 (BG, \cO)=0$. Since (e.g., by using a spectral sequence analogous to \cite[Prop.~3.12]{Mon21}) one has an exact sequence $0 \to \mathrm{Ext}^1 (G, \GG_a)\to H^2 (BG, \GG_a) \to \mathrm{Hom}(G \wedge G, \GG_a) \to \mathrm{Ext}^2 (G, \GG_a)$ it follows that $\mathrm{Ext}^2 (G, \GG_a) \ne 0$ in this case.
\end{example}{}

\begin{example}\label{coffee150}
Let $G$ be a commutative group scheme over a perfect field $k$ whose dual is a $1$-dimensional formal Lie group of height $>1$. Then it follows that the Verschiebung operator on $G^\vee$ induces the zero map on tangent space of $G^\vee$. This implies that $\Hom(G^{\w{perf}}, \GG_a) = 0$. Since $G^\w{perf}$ is unipotent, it follows that $G^\w{perf}$ is trivial. In other words, the maximal pro-\'etale quotient of $G$ is trivial.
\end{example}{}

We end this subsection with the following lemma that will be used in \cref{worldcup1}.

\begin{lemma}\label{newyr2}
Let $G$ be a unipotent affine group scheme over a field $k$ whose dual is a non-commutative formal Lie group.
Let $H$ be a unipotent affine commutative group scheme over $k$.
Then $H^i(BG,H) = 0$ for all $i > 2$.
\end{lemma}
\begin{proof}
First, we consider the case $H \simeq \GG^{ I}_a$ for some index set $I$.
By \cref{newyr} and \cref{nc-group-law-cohomology}, we then have
\[ H^i(BG,H) \simeq H^i(BG,\cO)^{ I} \simeq \Ext^i_{k\llbrace X_1, \ldots, X_g \rrbrace}(k,k)^{ I} \simeq 0 \]
for all $i > 1$.
If $k$ is of characteristic zero, any unipotent affine commutative group scheme is of this form by \cite[IV, \S~2, Prop.~4.2.b)]{MR0302656}, so we are done.
Thus, we may from now on assume that $k$ has characteristic $p>0$.
\vspace{2mm}

Next, we consider the case where the Verschiebung $V \colon H^{(p)} \to H$ is $0$.
Then $H$ fits into a short exact sequence of affine commutative group schemes
\[ 0 \to H \to \GG^{ I}_a \to \GG^{ J}_a \to 0 \]
for some index sets $I$ and $J$ by \cite[IV, \S~3, Thm.~6.6, Cor.~6.8]{MR0302656}.
Since $H^i(BG,\GG_a^{ I}) = H^i(BG,\GG_a^{ J}) = 0$ for all $i > 1$ by the previous case, the associated long exact sequence in cohomology shows that $H^i(BG,H) = 0$ for all $i > 2$.
\vspace{2mm}

Now, we consider the case when the $n$-fold Verschiebung $V^{n} \colon H^{(p^n)} \to H$ is $0$ for some $n > 0$.
Then $H$ admits a finite decreasing filtration
\[ H \supseteq V(H^{(p)}) \supseteq \dotsb \supseteq V^{(n-1)}(H^{(p^{n-1})}) \supseteq V^{n}(H^{(p^n)}) = 0, \]
in which the objects $V^{ j}(H^{(p^j)})$ and the graded pieces $V^{ j}(H^{(p^j)})/V^{(j+1)}(H^{(p^{j+1})})$ are unipotent affine commutative group schemes that are annihilated by $V^{n-j}$ and $V$, respectively.
Using the long exact sequences in cohomology for the short exact sequences
\[ 0 \to V^{(j+1)}(H^{(p^{j+1})}) \to V^{ j}(H^{(p^j)}) \to V^{ j}(H^{(p^j)})/V^{(j+1)}(H^{(p^{j+1})}) \to 0 \]
and induction on $n$, we obtain that $H^i(BG,H) = 0$ for all $i > 2$.
\vspace{2mm}

Finally, we consider the general case of unipotent group schemes over a field $k$ of characteristic $p>0$.
As before, we consider the decreasing filtration
\[ H \supseteq V(H^{(p)}) \supseteq \dotsb \supseteq V^{(N-1)}(H^{(p^{N-1})}) \supseteq V^{ j}(H^{(p^j)}) \supseteq \dotsb.\]
Let $H_j \colonequals H/V^{ j}(\pi^{(p^j)})$. Since $H$ is unipotent, we have $H \simeq \varprojlim_{j}H_j$. Since the transition morphisms $H_{j+1} \to H_j$ are surjective for all $j$, we have $\varprojlim_{j} H_j \simeq R\varprojlim_{j} H_j$ by \cite[Ex.~3.1.7, Prop.~3.1.10]{MR3379634}.
We conclude that$$R\Gamma(BG,H) \simeq R\Gamma(BG,R\varprojlim_{j}H_j) \simeq R\varprojlim_{j}R\Gamma(BG,H_j);$$the associated Milnor sequences take the form
\begin{equation}\label{Milnor-exact-sequence}
0 \to {\varprojlim_j}^1H^{i-1}(BG,H_j) \to H^i(BG,H_j) \to \varprojlim_{j }H^i(BG,H_j) \to 0
\end{equation}
for all $i \ge 0$.
Since the Verschiebung acts as zero on the unipotent affine commutative group schemes $V^{ j}(H^{(p^j)})/V^{(j+1)}(H^{(p^{j+1})})$, the previous case applied to the long exact sequence in cohomology for the short exact sequences
\[ 0 \to V^{ j}(H^{(p^j)})/V^{(j+1)}(H^{(p^{j+1})}) \to H_{j+1} \to H_j \to 0 \]
shows that the map $H^{i-1}(BG,H_{j+1}) \to H^{i-1}(BG,H_j)$ is surjective for all $i > 2$ and all $j \in \NN$.
In particular, the $H^{i-1}(BG,H_j)$ form a Mittag-Leffler system for any fixed $i > 2$; thus we see that ${{{\varprojlim^1}}}H^{i-1}(BG,H_j) = 0$. Now, again by the previous case, it follows that $H^i(BG,H_j) = 0$ for $i>2$. Therefore, by \cref{Milnor-exact-sequence}, it follows that $H^i (BG, H)=0$ for $i>2$, which finishes the proof.
\end{proof}

\subsection{Construction of formal Lie groups}\label{cof1}
Let $k$ be a field. 
In this section, we establish equivalences between certain kinds of objects of $\mathrm{DAlg}^{\mathrm{ccn}}_k$ and formal Lie groups (see \cref{classifynoncommfgl}, \cref{classifycommfgl}, \cref{classifycommfgl2}). In particular, \cref{classifynoncommfgl} proposes an answer to a problem posed by Nori \cite[p.~75]{MR682517}. In a different context, when $k= \mathbf{F}_p$, a classification for differential graded algebras with exterior (co)homology was obtained earlier by Dwyer, Greenlees and Iyengar \cite[Thm.~1.1]{Iye} in terms of complete discrete valuation rings with residue field $\mathbf{F}_p$. In the $g=1$ case, \cref{classifynoncommfgl} provides a refinement of \cite[Thm.~1.1]{Iye}: the differential graded algebra $R\mathrm{Hom}_{\mathbf{F}_p\llbracket x\rrbracket} (\mathbf{F}_p, \mathbf{F}_p)$ can further be equipped with the structure of a different (augmented) ``commutative algebra'' structure for each $1$-dimensional formal Lie group over $\mathbf{F}_p$. The proofs of \cref{classifynoncommfgl}, \cref{classifycommfgl} and \cref{classifycommfgl2} use a geometric approach. More specifically, we use the notion of affine stacks and the results regarding the representability of duals of unipotent group schemes by formal Lie groups established in \cref{cof0}. 

\vspace{2mm}

In \cref{cof6}, we will construct formal Lie groups as duals of unipotent homotopy group schemes of certain higher stacks that will be used in \cref{cof3} to recover the Artin--Mazur formal groups \cite{MR457458} in many cases of interest.

\begin{proposition}\label{classifynoncommfgl}
The full subcategory of the $\infty$-category $(\mathrm{DAlg}^{\mathrm{ccn}}_k)_{/k}$ spanned by $B \in (\mathrm{DAlg}^{\mathrm{ccn}}_k)_{/k} $ that satisfies $H^0 (B)= k$, $\dim H^1(B) = g$ and $H^i(B)=0$ for $i >1$ is equivalent to the category of non-commutative formal Lie groups over $k$ of dimension $g$ via the functor that sends a non-commutative formal Lie group $$E \mapsto R\Gamma (B E^\vee, \cO).$$
\end{proposition}{}

\begin{proof}
Given a non-commutative formal Lie group $E$ of dimension $g$, one can dualize and obtain a possibly non-commutative affine unipotent group scheme $G \colonequals E^\vee$. We set $B \colonequals R\Gamma(BG, \cO)$. The map $ \Spec k \to BG$ provides a natural augmentation $B \to k$. Further, it also follows from \cref{newyr} that $H^0(B) = k$, $\dim H^1(B) = g$ and $H^i (B) = 0$ for $i>1$. This shows that the functor described in the above proposition indeed lands in the desired subcategory of $(\mathrm{DAlg}^{\mathrm{ccn}}_k)_{/k}$.

\vspace{2mm}
Now we describe a functor in the opposite direction. Given a $B \in (\mathrm{DAlg}^{\mathrm{ccn}}_k)_{/k}$ satisfying the properties in the proposition, one can look at $\Spec B$ as an affine stack; this is naturally a pointed and connected higher stack. It follows that $\pi_1 (\Spec B)$ is representable by a possibly non-commutative unipotent affine group scheme $H$. 
There is a natural map of pointed stacks $\Spec B \to BH$ arising via the $1$-truncation. By \cref{postnikov}, the natural map $H^1(BH, \cO) \to H^1(B)$ is an isomorphism and $H^2(BH, \cO) \to H^2(B)=0$ is injective. This implies that $\dim H^1(BH, \cO) = g$ and $\dim H^2 (BH, \cO) = 0$. Therefore by \cref{bgcurve1}, it follows that dual of $H$ is a non-commutative formal Lie group. By the discussion in the previous paragraph, it also follows that $H^n (BH, \cO)=0$ for $n>1$ so that the map $\Spec B \to BH$ induces an isomorphism on cohomology with coefficients in the structure sheaf. 
This implies that the natural map $\Spec B  \to BH$ is an isomorphism since the target is an affine stack. Using \cref{lemma:group-homomorphism-classifying-stack}, one sees that the full subcategory of $(\mathrm{DAlg}^{\mathrm{ccn}}_k)_{/k}$ as in the proposition is equivalent to an ordinary $1$-category. Therefore, the functor $$B \mapsto \pi_1 (\Spec B)^\vee$$ gives an inverse to the functor described in the proposition, proving the desired equivalence.
\end{proof}{}

\begin{proposition}\label{classifycommfgl}
The full subcategory of the $\infty$-category $(\mathrm{DAlg}^{\mathrm{ccn}}_k)_{/k}$ spanned by $B \in (\mathrm{DAlg}^{\mathrm{ccn}}_k)_{/k} $ that satisfies $H^0 (B)= k,\, \dim_k H^1(B) = g$, $ H^*(B) \simeq \wedge^* H^1(B)$ and $\pi_1 (\Spec B)$ is commutative is equivalent to the category of commutative formal Lie groups over $k$ of dimension $g$ via the functor that sends a formal Lie group $$E \mapsto R\Gamma (B E^\vee, \cO).$$
\end{proposition}{}

\begin{proof}
By \cref{paderborn1}, we see that the functor described in the above proposition indeed lands in the desired subcategory of $(\mathrm{DAlg}^{\mathrm{ccn}}_k)_{/k}$.
\vspace{2mm}

Now we describe a functor in the opposite direction. Given a $B \in (\mathrm{DAlg}^{\mathrm{ccn}}_k)_{/k}$ satisfying the properties in the proposition, we look at $\Spec B$ which is naturally equipped with the structure of a pointed connected higher stack. It follows from the hypothesis that $\pi_1 (\Spec B)$ is representable by a commutative unipotent affine group scheme $G$. There is a natural map of pointed stacks $\Spec B \to BG$ arising via $1$-truncation. By \cref{postnikov}, the natural map $H^1 (BG, \cO) \to H^1 (B)$ is an isomorphism and $H^2 (BG, \cO) \to H^2 (B)$ is injective. Therefore, $\dim_k H^2 (BG, \cO) \le \binom{g}{2}$. By \cref{commgp}, it follows that $G^\vee$ is a formal Lie group of dimension $g$. By \cref{paderborn1}, it follows that the natural map $\Spec B \to BG$ induces an isomorphism on cohomology with coefficients in the structure sheaf and therefore is an isomorphism since the target is an affine stack. Using \cref{lemma:group-homomorphism-classifying-stack}, one sees that the full subcategory of $(\mathrm{DAlg}^{\mathrm{ccn}}_k)_{/k}$ as in the proposition is equivalent to an ordinary $1$-category. Therefore, the functor $$B \mapsto \pi_1 (\Spec B)^\vee$$ gives an inverse to the functor described in the proposition, proving the desired equivalence.
\end{proof}{}

\begin{proposition}\label{classifycommfgl2}The full subcategory of the $\infty$-category $(\mathrm{DAlg}^{\mathrm{ccn}}_k)_{/k}$ spanned by $B \in (\mathrm{DAlg}^{\mathrm{ccn}}_k)_{/k} $ that satisfies $H^0(B) = k$, $\dim_k H^2 (B) = g$ and $H^* (B) \simeq \Sym^*H^2 (B)$ is equivalent to the category of commutative formal Lie groups over $k$ of dimension $g$ via the functor that sends a formal Lie group $$ E \mapsto R\Gamma(K(E^\vee, 2), \cO).$$
\end{proposition}{}

\begin{proof}
Let $G= E^\vee$. First we show that $H^* (K(G,2), \cO) \simeq \Sym^* H^2 (K(G,2), \cO)$. Note that by \cref{paderborn1}, we have $H^* (BG, \cO) \simeq \wedge^* H^1 (BG, \cO)$. By applying faithfully flat descent along $\Spec k \to K(G,2)$, we obtain a spectral sequence $$E_{1}^{i,j} = H^j ((BG)^i, \cO) \implies H^{i+j} (K(G,2), \cO),$$ where $(BG)^i$ denotes the $i$-fold fibre product of $BG$ with itself. One can analyze this spectral sequence in a manner entirely analogous to \cite[\S~3.3]{Mon21}. Indeed, similar to \cite[Lem. 3.25]{Mon21}, we obtain that the complex $E_1^{\bullet, 1} \simeq H^1 (BG)[-1]$. Similar to \cite[Lem. 3.27]{Mon21}, one obtains that $E_1^{\bullet, n} \simeq \Sym^n (H^1 (BG, \cO))[-n]$. Thus we obtain the desired calculation $H^* (K(G,2), \cO) \simeq \Sym^* H^2 (K(G,2), \cO)$. The object $R\Gamma(K(G,2), \cO) \in \mathrm{DAlg}^{\mathrm{ccn}}_k$ is also naturally augmented via the map induced from $\Spec k \to K(G,2)$. This shows that the functor $E \mapsto R\Gamma (K(E^\vee, 2), \cO)$ indeed lands in the desired subcategory of $(\mathrm{DAlg}^{\mathrm{ccn}}_k)_{/k}$.
\vspace{2mm}

Now we describe a functor in the opposite direction. Given a $B \in (\mathrm{DAlg}^{\mathrm{ccn}}_k)_{/k}$ satisfying the properties in the proposition, we look at $\Spec B$, which is naturally a pointed connected higher stack. By \cref{hurewicz}, since $H^1(B)=0$, it follows that $\tau_{\le 2} \Spec B \simeq K(H,2)$ for some unipotent commutative group scheme $H$. 
By \cref{postnikov}, it follows that $H^3 (K(H,2), \cO)$ naturally embeds into $H^3 (B)$. However, by assumption the latter is zero. Therefore $H^3 (K(H,2), \cO) = 0$. By \cref{gait1}, this implies that $\mathrm{Ext}^1 (H, \GG_a)=0$ and $\dim_k \mathrm{Hom}(H, \GG_a) = \dim_k H^2 (K(H,2), \cO) = \dim_k H^2 (B)=g$. By \cref{fine}, it follows that $H^\vee$ is a formal Lie group of dimension $g$. By the calculation of cohomology of $K(H,2)$ when $H^\vee$ is a formal Lie group discussed in the previous paragraph, it follows that the natural map $\Spec B \to K(H,2)$ is an isomorphism since the target is an affine stack. 
By delooping and using \cref{lemma:group-homomorphism-classifying-stack}, one sees that the full subcategory of $(\mathrm{DAlg}^{\mathrm{ccn}}_k)_{/k}$ as in the proposition is equivalent to an ordinary $1$-category. Therefore, the functor $$B \mapsto \pi_2 (\Spec B)^\vee$$ gives an inverse to the functor described in the proposition, which proves the desired equivalence.
\end{proof}{}

\begin{remark}
In \cref{classifycommfgl2}, $B \colonequals R\Gamma (K(E^\vee, 2), \mathscr{O})$ acquires a natural coalgebra structure, which is induced from the coalgebra structure of $\mathscr{O}(E^\vee)$ via a two-fold application of the cobar construction \cite[\S~5.2.3]{luriehigher}. To identify this coalgebra structure, note that there is a canonical isomorphism of cocommutative coalgebras $\mathscr{O}(E^\vee) \simeq \Gamma^* (V)$ where $V$ is the vector space of dimension $g$ given by the tangent space of the formal group $E$. This equivalence, under the two-fold cobar construction, produces an isomorphism $B \simeq \Gamma^* (V[-2])$ of coalgebras.\footnote{
Here we use that $\Gamma^* (V_1 \oplus V_2) \simeq \Gamma^* (V_1) \otimes_k \Gamma^* (V_2)$, and the two-fold cobar construction applied to $V$ viewed as a cogroup via the diagonal map is equivalent to $V[-2]$.}
Since the tangent space as an abstract vector space does not depend on the formal group, the cocommutative coalgebra structure of $\Gamma^*(V)$ is also independent of the formal group. As a consequence, the isomorphism class of the coalgebra structure on $B$ does not depend on the formal group $E$ and is given by $\Gamma^* (V[-2]) \simeq \Gamma^* (H^2(B)[-2])$. Roughly speaking, this should be seen as arising from the fact that the topological algebra structure of the formal group $E$ is $\Sym^* (V)^\wedge$, which does not depend on the formal group $E$. However, the augmented \emph{algebra} structure of $B$ depends crucially on the formal group $E$, as \cref{classifycommfgl2} shows. Similarly, in the context of \cref{classifycommfgl}, the coalgebra structure on $B'\colonequals R\Gamma (BE^\vee, \mathscr{O})$ is given by $\Gamma^* (H^1 (B')[-1])$.
\end{remark}{}

\begin{proposition}\label{formalgroup}
Let $n \ge 1$ be a fixed integer. Let $B \in (\mathrm{DAlg}^{\mathrm{ccn}}_k)_{/k}$ be such that $H^0(B) = k, H^i(B) = 0$ for $0<i < n$, $\dim H^n (B)= g$ and $H^{n+1}(B) = 0$. Then the dual of $\pi_n (\Spec B)$ is a commutative formal Lie group of dimension $g$ for $n>1$ and is a non-commutative formal Lie group of dimension $g$ if $n=1$.
\end{proposition}{}

\begin{proof} 
It follows that $\Spec B$ is naturally a pointed and connected affine stack and $\pi_n (\Spec B)$ is representable by a unipotent affine group scheme $G$, which is commutative when $n>1$.
Since $H^i (B) = 0$ for all $0<i<n$, \cref{hurewicz} shows that $\pi_i (\Spec B)$ is trivial for all $0 \le i < n$. In other words, $\tau_{\le n}\Spec B \simeq K(G,n)$. Therefore, we have a natural map $\Spec B \to K(G,n)$. Further, by \cref{postnikov}, the induced map $H^n (K(G,n), \cO) \to H^n (B)$ is an isomorphism and $H^{n+1}(K(G,n), \cO) \to H^{n+1}(B)=0$ is injective. This implies that $H^{n+1} (K(G,n), \cO)=0$. By \cref{gait1}, it follows that $\dim \Hom(G, \GG_a)=g$ and if $n>1$, then $\Ext^1 (G, \GG_a)=  H^{n+1}(K(G,n), \cO) =0$. By applying \cref{bgcurve1} when $n=1$ and \cref{fine} when $n>1$, we obtain the desired statements.
\end{proof}{}

\begin{construction}[Formal Lie groups via unipotent homotopy group schemes]\label{cof6}
Let $n \ge 1$ be a fixed positive integer. Let $X$ be a pointed higher stack over $k$ such that $H^0 (X, \cO) \simeq k, H^i (X, \cO)=0$ for $0 <i < n$, $\dim H^n (X, \cO)= g$ and $H^{n+1}(X, \cO) = 0$. By letting $B \colonequals R\Gamma(X, \cO)$, as a consequence of \cref{formalgroup}, it follows that $\pi_n(X)^{\vee}$ is a commutative formal Lie group of dimension $g$ for $n>1$ and is a non-commutative formal Lie group of dimension $g$ if $n=1$.
\end{construction}{}

\subsection{Recovering the Artin--Mazur formal groups}\label{cof3}
In this section, we show that our unipotent homotopy group schemes recover the Artin--Mazur formal groups from \cite{MR457458} for many cases of interest.
Throughout this section, we fix an integer $n \ge 1$. Further, let $X$ be a proper scheme over an algebraically closed field $k$ of characteristic $p > 0$ (equipped with a $k$-rational point) satisfying the conditions
\begin{equation}\label{1/2-CY}
 H^0 (X, \cO)\simeq k, \quad  H^i(X,\cO) = 0 \text{ for all } 0 < i < n, \quad \text{and} \quad H^{n+1}(X, \cO)=0.\tag{$\ast$}
\end{equation}
Examples of such varieties include curves, Calabi--Yau varieties of dimension $n$ (see \cref{CY}) and their blow-ups. When $X$ is a surface and $n=2$, the condition $(*)$ translates to the condition that the irregularity of $X$ is $0$. One can also obtain more examples by considering suitable products of schemes that satisfy (\ref{1/2-CY}). Let us now begin by recalling the relevant formal Lie groups constructed by Artin--Mazur \cite{MR457458}.
\begin{definition}[{\cite[\S~II.1, \S~II.4]{MR457458}}]\label{latex}
Let $n \ge 1$ be an integer and $X$ be a scheme satisfying (\ref{1/2-CY}).
Then the contravariant functor 
\[ \Phi^n_X \colon (\Art/k)^{\op} \to \Ab, \quad S \mapsto \Ker\bigl(H^n_{\et}(X_S,\GG_m) \to H^n_{\et}(X,\GG_m)\bigr) \]
from the category of finite Artinian $k$-schemes to the category of abelian groups is pro-representable by a commutative formal Lie group over $k$ of dimension $\dim_kH^n(X,\cO_X)$.
We call $\Phi^n_X$ the \emph{Artin--Mazur formal Lie group of X}.
\end{definition}

Let us now recall the following construction based on unipotent homotopy group schemes.
\begin{construction}\label{paderborn2}
Let $n \ge 1$ be a fixed integer and let $X$ be a proper scheme as in (\ref{1/2-CY}). By \cref{cof6}, $\pi_n^\mathrm{U}(X)^{\vee}$ is a commutative formal Lie group of dimension $\dim_kH^n (X,\cO)$ for $n>1$ and a non-commutative formal Lie group of dimension $\dim_kH^n (X,\cO)$ if $n=1$.
\end{construction}{}
 We will show in \cref{cycomp} that \cref{paderborn2} recovers the Artin--Mazur formal Lie group. One of the concrete advantages of  \cref{paderborn2} is that it does not require using the \'etale cohomology groups; in particular, when $n=1$, dual of $\pi_1^\mathrm{U}(X)$ gives a \textit{non-commutative} formal Lie group attached to $X$, which would not be seen by the Artin--Mazur formal Lie group. In other words, \cref{paderborn2} recovers the commutative formal Lie group constructed by Artin--Mazur \cite{MR457458} as well as the non-commutative formal Lie group constructed by Nori (for curves) in {\cite[\S~IV, Prop.~4]{MR682517}}.
\vspace{2mm}

To state the main result about Artin--Mazur formal Lie groups that we will use, we need the following notions:
\begin{enumerate}
    \item Serre's \emph{Witt vector cohomology} $\Hh^n(X,W) \colonequals \varprojlim_r \Hh^n_{\mathrm{Zar}}(X,W_r)$ of a proper scheme $X$ from \cite[\S~5]{MR0098097}, where $W_r$ is the sheaf of truncated Witt vectors of length $r$ on $X$
    \item \emph{Duality} between the categories of commutative formal groups and commutative affine group schemes over a field from \cref{thisisnew}.
\end{enumerate}
Using these notions, one can describe the Dieudonn\'e module of the Artin--Mazur formal Lie group, as we note below. First, we recall that for a commutative formal Lie group $E$ over a perfect field $k$ of characteristic $p>0$, its Dieudonn\'e module is defined to be (\textit{cf}.~\cite[Thm.~4.15]{MR767090})
\begin{equation}\label{defofdmfgl}
\mathbf{M}(E)\colonequals \Hom(\widehat{W},E),    
\end{equation}where $\widehat{W}$ is the infinite-dimensional, commutative formal group obtained from formally completing the group scheme of $p$-typical Witt vectors $W$ at zero. The following proposition expresses the Dieudonn\'e module of the Artin--Mazur formal Lie group of $X$ in terms of Serre's Witt vector cohomology.
\begin{proposition}[{\cite[Cor.~4.3]{MR457458}}]
Let $n \ge 1$ be an integer and $X$ be a proper scheme as in (\ref{1/2-CY}). Then there is a natural isomorphism $\mathbf{M}(\Phi^n_X) \simeq \Hh^n(X,W)$.
\end{proposition}

Before we proceed further, let us note the following lemma, which expresses the Dieudonn\'e module of a formal Lie group in terms of its dual group scheme.

\begin{lemma}\label{dieudonnemodu12}
Let $\mathbf{M}(E)$ denote the Dieudonn\'e module of a commutative formal Lie group $E$ over a perfect field $k$ of characteristic $p>0$. Then we have a natural isomorphism 
    \[ \mathbf{M}(E)  \simeq \Hom(E^\vee,W), \]where $E^\vee$ is the (unipotent) group scheme dual to $E$.
\end{lemma}{}

\begin{proof}
This amounts to the statement that $(\widehat{W})^\vee \simeq W$. In order to see this, let $A$ be an arbitrary $k$-algebra. Then the group of $p$-typical curves on
$({\widehat{\mathbf G}}_m)_A$ is naturally
isomorphic to $\mathrm{Hom}(\widehat{W}_A, (\widehat{\GG}_{m})_A)$, where the homomorphisms are taken in the category of formal groups. Now, by the Artin--Hasse exponential (\cite[Prop.~15.3.8, Prop.~17.4.23]{MR2987372}), one has a natural isomorphism $W(A) \simeq \mathrm{Hom}(\widehat{W}_A, (\widehat{\GG}_{m})_A)$. This proves the desired statement.
\end{proof}{}

\begin{remark} 
Let $E$ be a commutative formal group of finite height (i.e., $p \colon E \to E$ is an isogeny of height $h$) over a perfect field $k$ of characteristic $p>0$. Then $\mathbf{M}(E)$ from \cref{defofdmfgl} is naturally isomorphic to $\sigma^*\mathbf{M}(E[p^\infty]^\vee)$ from \cref{covariant-Dieudonne} where $E[p^\infty]$ is the $p$-divisible group $(E[p^n])_{n \ge 1}$ of height $h$ obtained from $E$; see the corollary on \cite[p.~255]{Fon1} or \cite[Thm.~3.34]{MR4190563}.

\end{remark}{}

\begin{proposition}\label{cycomp}
Let $n \ge 1$ be an integer and let $X$ be a proper scheme as in (\ref{1/2-CY}). 
Then if $n > 1$, $\Phi^n_X$ is naturally isomorphic to the dual $\pi^{\mathrm{U}}_n(X)^\vee$ of the $n$-th unipotent homotopy group scheme of $X$. 
If $n=1$, $\Phi^n_X$ is naturally isomorphic to $(\pi_1^{\mathrm{U}}(X)^{\mathrm{ab}})^{\vee}$.
\end{proposition}
\begin{proof}
Let us first suppose that $n > 1$. It will suffice to exhibit a natural isomorphism $$\mathbf{M}(\pi^{\mathrm U}_n(X)^\vee) \xrightarrow{\sim} \mathbf{M}(\Phi^n_X) \simeq \Hh^n(X,W)$$ of the associated Dieudonn\'e modules.
Since (by \cref{dieudonnemodu12}) we have $$\mathbf{M}(\pi^{\mathrm U}_n(X)^\vee) \simeq \Hom(\pi^{\mathrm U}_n(X),W) \simeq \varprojlim_r \Hom(\pi^{\mathrm U}_n(X),W_r),$$ it is in fact enough to give compatible natural isomorphisms
\[ v_r \colon \Hom(\pi^{\mathrm U}_n(X),W_r) \xrightarrow{\sim} \Hh^n_{}(X,W_r), \]
where the homomorphisms are taken in the category of commutative group schemes over $k$ and $\Hh^n_{}(X,W_r)$ denotes $\Hh^n_{\mathrm{Zar}}(X,W_r) \simeq \Hh^n_{\mathrm{fpqc}}(X,W_r)$.
\vspace{2mm}

We begin with the construction of $v_r$.
By \cref{hurewicz2}, the assumption (\ref{1/2-CY}) guarantees that $\tau_{\le n} \mathbf{U}(X) \simeq K(\pi^{\mathrm U}_n(X),n)$.
Since $K(W_r,n)$ is $n$-truncated, we obtain maps
\begin{IEEEeqnarray*}{rCl}
 \Hom(\pi^{\mathrm U}_n(X),W_r) & \to & \pi_0 \Map(K(\pi^{\mathrm U}_n(X),n),K(W_r,n)) \simeq \pi_0 \Map(\UU(X),K(W_r,n))
\end{IEEEeqnarray*}
Since $W_r$ is unipotent, $K(W_r, n)$ is an affine stack (\cref{thmoftoen}). Therefore, we have $$\pi_0 \Map(\UU(X),K(W_r,n)) \simeq \pi_0 \Map(X,K(W_r,n)) \simeq H^n _{\mathrm{}}(X, W_r).$$ Composing these maps, we obtain the desired maps $$v_r \colon \mathrm{Hom}(\pi_n^{\mathrm{U}}(X), W_r) \to H^n_{\mathrm{}}(X, W_r), $$
which are compatible in $r$. 
We will now inductively show that
\begin{enumerate}[label=($\alph*_r$)]
    \item\label{Witt-cohomology} $\Hh^{n+1}(X,W_r) = 0$,
    \item\label{Witt-Ext} $\Ext^1(\pi^{\mathrm U}_n(X),W_r) = 0$, and 
    \item\label{Witt-isom} $v_r$ is an isomorphism
\end{enumerate}
for all $r \ge 1$.
For $r=1$, we have $W_1 \simeq \GG_a$; therefore the assertions follow directly from the fact that $H^{n+1}(X, \cO)=0$, \cref{fine} and \cref{hurewicz2}. For the induction step, let us assume that \ref{Witt-cohomology}, \ref{Witt-Ext} and \ref{Witt-isom} hold.
The short exact sequences of commutative group schemes
\begin{equation}\label{Wittsequence}
0 \to W_r \to W_{r+1} \to \GG_a \to 0
\end{equation} 
induces an exact sequence 
\[ \Hh^{n+1}(X,W_r) \to \Hh^{n+1}(X,W_{r+1}) \to \Hh^{n+1}(X,\GG_a), \]
which yields ($a_{r+1}$). The short exact sequence \cref{Wittsequence} also induces an exact sequence
\[ \Ext^1(\pi^{\mathrm U}_n(X),W_r) \to \Ext^1(\pi^{\mathrm U}_n(X),W_{r+1}) \to \Ext^1(\pi^{\mathrm U}_n(X),\GG_a), \]
giving $(b_{r+1})$. We also obtain a diagram
\begin{center}
\[ \begin{tikzcd}
0 \arrow[r]  & \Hom(\pi^{\mathrm U}_n(X),W_r) \arrow[r] \arrow[d,"v_r"] & \Hom(\pi^{\mathrm U}_n(X),W_{r+1}) \arrow[d,"v_{r+1}"] \arrow[r] & \Hom(\pi^{\mathrm U}_n(X),\GG_a) \arrow[d,"v_1"] \arrow[r] & 0  \\
0 \arrow[r] & \Hh^n(X,W_r) \arrow[r] & \Hh^n(X,W_{r+1}) \arrow[r] & \Hh^n(X,\cO) \arrow[r] & 0.
\end{tikzcd} \]
\end{center}{}

In order to see that the rows above are indeed short exact, we recall that $\Hh^{n-1}(X,\cO) = 0$ by assumption (\ref{1/2-CY}) and $\Ext^1(\pi^{\mathrm U}_n(X),W_r) = \Hh^{n+1}(X,W_r) = 0$ by \ref{Witt-Ext} and \ref{Witt-cohomology}.
Thus, $v_{r+1}$ is an isomorphism, which yields $(c_{r+1})$ and finishes the proof. The case $n=1$ follows exactly the same way by working with $\pi_1^{\mathrm{U}}(X)^{\mathrm{ab}}$.
\end{proof}

\begin{construction}\label{amhclass}Let $X$ be a variety over an algebraically closed field $k$ equipped with a $k$-rational point $x$ satisfying the conditions as in (\ref{1/2-CY}). Let $E= \Phi_X^n$. By the proof of \cref{cycomp}, we obtain maps $X \to \UU(X) \xrightarrow{} \tau_{\le n} \UU(X) \xrightarrow{\sim}K(\pi_n^{\mathrm{U}}(X),n) \xrightarrow{\sim} K(E^\vee, n)$ of higher stacks. This defines a canonical class (does not depend on the choice of $x$)  $\xi^{\mathrm{AM}}_{\mathrm{H}} \in \pi_0 \mathrm{Maps} (X, K(E^\vee,n)) \simeq H^n (X, E^\vee)$ that we call the \emph{Artin--Mazur--Hurewicz} class. By construction, the class $\xi^{\mathrm{AM}}_{\mathrm{H}}$ induces a canonical isomorphism $\tau_{\le n}\UU(X) \simeq K(E^\vee , n)$ in the homotopy category of higher stacks.
\end{construction}{}

\begin{remark}\label{multidR}Note that the proof of \cref{cycomp} shows more generally that in the situation of \cref{cof6}, the Dieudonn\'e module of the (abelianization) of the formal Lie group $\pi_n^{\mathrm{U}}(X)^\vee$ is given by $H^n (X, W)$.
    \end{remark}{}

\begin{construction}[Formal groups associated to de Rham cohomology]\label{multideRham1}
Let $X$ be a smooth proper scheme over an algebraically closed field $k$ of characteristic $p > 0$ and $n \ge 1$ be an integer such that $H^i_{\mathrm{dR}} (X)=0$ for $0<i<n$ and $H^{n+1}_{\mathrm{dR}}(X)=0$;
for $n=2$, an interesting class of examples is supplied by $K3$ surfaces over $k$.
Let $X^{\mathrm{dR}}$ be the mod $p$ reduction of the stack $W(X^{\mathrm{perf}})/\mathscr{G}$, where we use $X^{\mathrm{perf}}$ to denote the perfection of $X$ and $\mathscr{G}$ is the flat affine groupoid from \cite{drinfeld2018stacky}. By \cite[Thm.~2.4.2.(iii)]{drinfeld2018stacky}, $H^i_{\mathrm{dR}}(X) \simeq H^i (X^{\mathrm{dR}}, \cO)$.
Therefore, by \cref{cof6}, (the abelianization of) $\pi_n^{\mathrm{U}}(X^{\mathrm{dR}})^{\vee}$ is a formal Lie group of dimension $\dim_k H^n_{\mathrm{dR}}(X)$, whose Dieudonn\'e module is given by $H^n (X^{\mathrm{dR}},W)$.
\end{construction}

\begin{remark}\label{multiderhamargumentrevise}
In \cite[\S~III]{MR457458}, Artin and Mazur constructed certain formal groups $\Phi^n_{\mathrm{DR}}(X/k, \mathbf{G}_m)$ from the multiplicative de Rham complex. Our \cref{multideRham1} recovers these in the above situation, i.e., we have $\Phi^n_{\mathrm{DR}}(X/k, \mathbf{G}_m) \simeq (\pi_n^{\mathrm{U}}(X^{\mathrm{dR}})^{\mathrm{ab}})^{\vee}$.
To prove this, it would be enough to show that the Dieudonn\'e module of $\Phi^n_{\mathrm{DR}}(X/k, \mathbf{G}_m)$ is isomorphic to $H^n (X^{\mathrm{dR}}, W)$. Since this does not appear in \cite{MR457458} or elsewhere in the literature, we sketch an argument that relies on some $p$-adic Hodge theory. 
To this end, let $B$ denote a smooth $k$-algebra and let $F_{\mathrm{dR}^\times (B)} (k[t]/t^r)$ denote the complex of abelian groups 
\begin{equation}\label{multicomplexderham}
  (1+ t (k[t]/t^r)\otimes_k B) \xrightarrow{\mathrm{dlog}} t (k[t]/t^r) \otimes_k \Omega^1_{B/k} \xrightarrow{\mathrm{d}}  t (k[t]/t^r) \otimes_k \Omega^2_{B/k} \xrightarrow{\mathrm{d}} \dotsb. 
\end{equation}
Set $\mathfrak{m} \colonequals t k[t]/t^r$. Note that for any non-unital commutative ring $I$ over $k$, $(1+ \mathfrak{m}\otimes_k I)$ has a natural abelian group structure induced from multiplication as polynomials, and we view $(1+ \mathfrak{m}\otimes_k B)$ above with this group structure. 
By Cartier's theory of $p$-typical curves (see e.g.\ \cite[\S~I.3]{MR457458}), for the desired isomorphism of Dieudonn\'e modules, it would be enough to show that 
\[ F_{\mathrm{dR}^\times (B)}(k[t]/t^r) \simeq R\Gamma ((\Spec B)^{\mathrm{dR}}, 1 + tk[t]/t^r \mathscr{O}) \equalscolon R\Gamma ((\Spec B)^{\mathrm{dR}}, 1 + \mathfrak{m} \cO) \]
in the derived category of abelian groups. 
Note that we have a fiber sequence $$1 + \mathfrak{m} \Fil^1_{\mathrm{Hodge}} \Omega^\bullet_B \to R\Gamma ((\Spec B)^{\mathrm{dR}}, 1 +\mathfrak{m} \mathscr{O}) \to 1 +\mathfrak{m} B. $$
By \cref{revisionlemma} below, in the derived category, we have an isomorphism $\exp \colon \mathfrak{m} \Fil^1_{\mathrm{Hodge}} \Omega^\bullet_B \simeq 1 + \mathfrak{m} \Fil^1_{\mathrm{Hodge}} \Omega^\bullet_B$.
This gives a natural map 
\[ 1+ \mathfrak{m}B \to \mathfrak{m} \Fil^1_{\mathrm{Hodge}} \Omega^\bullet_B[1] \simeq \Bigl[ \mathfrak{m} \otimes_k \Omega^1B \xrightarrow{\mathrm{d}} \mathfrak{m} \otimes_k \Omega^2B \xrightarrow{\mathrm{d}} \dotsb \Bigr], \]
necessarily induced by a map of abelian groups $1+\mathfrak{m} B \to \mathfrak{m} \otimes_k \Omega^1B$ that identifies with $\mathrm{dlog}$. 
This identifies the fiber of the map $1 + \mathfrak{m}B \to \mathfrak{m}\Fil^1_{\mathrm{Hodge}} \Omega^\bullet_B[1]$ with the complex \cref{multicomplexderham}, which finishes the proof.
\end{remark}

In the above remark, $\Fil^1_{\mathrm{Hodge}} \Omega^\bullet_B$ can be realized as a cosimplicial non-unital ring since by descent along $B \to B_{\mathrm{perf}}$ (\textit{cf}.~\cite[Rmk.~8.15]{BMS2}) one has 
\begin{equation}\label{cosimprevision}
\Fil^1_{\mathrm{Hodge}}\Omega^\bullet_B \simeq \lim \left (\xymatrix{
 \Fil^1_{\mathrm{PD}} \AA_{\mathrm{crys}}(B_{\mathrm{perf}})/p \ar[r]<1.5pt>\ar[r]<-1.5pt>   & \Fil^1_{\mathrm{PD}} \AA_{\mathrm{crys}}(B_{\mathrm{perf}}\otimes_B B_{\mathrm{perf}})/p   \ar[r]<3pt>\ar[r]\ar[r]<-3pt> & 
 \dotsb} \right);
\end{equation}
see \cite[Prop.~8.12]{BMS2} and \cite[Ex.~3.5.5]{Monda}. 
This equips $1+ \mathfrak{m}\Fil^1_{\mathrm{Hodge}} \Omega^\bullet_B$ with the structure of a cosimplicial abelian group, which can be viewed as an object of the derived category, as used in \cref{multiderhamargumentrevise}. The following lemma was also used above.

\begin{lemma}\label{revisionlemma} In the above set up, there is a natural isomorphism
\[ \exp \colon \mathfrak{m} \Fil^1_{\mathrm{Hodge}} \Omega^\bullet_B \simeq 1 + \mathfrak{m} \Fil^1_{\mathrm{Hodge}} \Omega^\bullet_B. \]   
\end{lemma}{}
\begin{proof}
Let $A$ be a commutative $k$-algebra and $I \subset A$ be an ideal equipped with divided powers. This induces a divided power structure $\gamma_i \colon I[t]/t^r \to A[t]/t^r$ on the ideal $I[t]/t^r\subset A[t]/t^r$.
We therefore obtain a map $\exp \colon tI[t]/t^r \to 1+ tI[t]/t^r$ defined by $\exp(c) \colonequals \sum_{i \ge 0} \gamma_i (c)$, which is well-defined since $t$ is nilpotent in $A[t]/t^r$. 
The map $1+ tI[t]/t^r \to tI[t]/t^r$ given by $1+ c \mapsto \log (1+c) \colonequals \sum_{i \ge 1} (-1)^{i+1} (i-1)!\, \gamma_i (x)$ is an inverse to $\exp$. Setting $\mathfrak{m} \colonequals tk[t]/t^r$ as before, this gives an isomorphism $\exp \colon \mathfrak{m}I \simeq 1+ \mathfrak{m}I$. 

Now let $S$ be a quasi-regular semiperfect ring (\cite[Def.~4.20]{BMS2}). Then $\Fil^1 \AA_{\mathrm{crys}}(S)/p$ is an ideal of $\AA_{\mathrm{crys}}(S)/p$ equipped with divided powers. Since the rings $B_{\mathrm{perf}}\otimes_B \dotsb \otimes_B B_{\mathrm{perf}}$ appearing in \cref{cosimprevision} are all quasi-regular semiperfect, applying the discussion from the previous paragraph yields the desired isomorphism.
\end{proof}{}

\newpage

\section{Calabi--Yau varieties and homotopy of spheres}

\subsection{Unipotent homotopy groups of weakly ordinary Calabi--Yau varieties}\label{spri4}
In this section, we apply the unipotent homotopy theory developed in our paper to the study of Calabi--Yau varieties. 
We will begin by proving that the $p$-adic \'etale homotopy type of a weakly ordinary $n$-dimensional Calabi--Yau variety is equivalent to the $p$-completed $n$-sphere. 
This implies (\cref{ordinaryCYhom}) that the unipotent homotopy groups of a weakly ordinary Calabi--Yau variety are isomorphic to (the affine group schemes associated with) the unstable $p$-adic homotopy groups of the $n$-sphere.

\begin{definition}\label{CY}
Let $k$ be an algebraically closed field of characteristic $p > 0$.
A Calabi--Yau variety over $k$ is a smooth, projective variety $X$ over $k$ such that
\begin{enumerate}
    \item the canonical bundle $\omega_X \simeq \cO_X$;
    \item $\Hh^i(X,\cO) = 0$ for $0 < i < \dim X$.
\end{enumerate}
\end{definition}

\begin{remark}
A Calabi--Yau variety is weakly ordinary if and only if the Artin--Mazur formal Lie group $\Phi_X^n$ is isomorphic to $\widehat{\mathbf G}_m$; this follows from the isomorphism $H^n(X, \cO) \simeq \mathrm{Hom}((\Phi_X^n)^\vee, \mathbf{G}_a)$ and \cref{coffee150} (by using duality).
\end{remark}{}

The following result shows that the $p$-adic \'etale homotopy type of a weakly ordinary Calabi--Yau variety is isomorphic to the $p$-completed $n$-sphere $(S^n)^\wedge_p$; our proof will freely make use of the notions introduced in \cite{etalehomo};
see also the discussion following \cref{defetalehom}.
\begin{proposition}\label{ordinaryCY}
Let $X$ be a weakly ordinary Calabi--Yau variety of dimension $n \ge 2$.
Then $\Et(X)_p^{\wedge}$ is isomorphic to the $p$-completed $n$-sphere $(\SSS^{n})_p^{\wedge }$, considered as an object of $\Pro(\mathcal{S}^{p\mhyphen\mathrm{fc}})$.
\end{proposition}

\begin{proof}
Since the $X$ is weakly ordinary, by \cref{ordinary-unipotent-homotopy}, we have
\[ \Spec R\Gamma_{\et}(X, k) \simeq \Spec R\Gamma (X, \cO) \simeq \mathbf{U}(X). \]
As $H^i (X, \cO) = 0$ for $0 < i <n$, it follows from \cref{hurewicz2} that $\pi_i^\mathrm{U}(X) \simeq \pi^\et_i(X)_p$ is the trivial group scheme for $0 <i <n$.
Combined with \cref{compareetaleuni}, this implies that $\pi_i (\Et(X)_p^{\wedge})$ is trivial  as a pro-$p$-finite group for $0<i<n$.
Further, we have $\mathrm{Hom}(\pi_n^\mathrm{U}(X), \GG_a) \simeq H^n (X, \cO)$, which is nonzero. By \cref{compareetaleuni}, this implies that the pro-$p$-finite group $\pi_n(\Et(X)_p^{\wedge })$ is nonzero. Moreover, since $X$ is a weakly ordinary Calabi--Yau variety, it follows that the Artin--Mazur formal Lie group of $X$ is isomorphic to $\widehat{\mathbf G}_m$. By \cref{cycomp}, this implies that $\pi^{\mathrm U}_n(X)$ is isomorphic to the profinite group scheme $\mathbf{Z}_p$. 
Another application of \cref{compareetaleuni} yields that $\pi_n(\Et(X)_p^{\wedge })$ is isomorphic to the profinite group $\mathbf{Z}_p$. 
\vspace{2mm}

This implies that there is a map $\SSS^n \to \Et(X)_p^{\wedge}$ in $\Pro(\mathcal{S})$ which induces the map $\mathbf{Z} \xrightarrow[]{1 \mapsto 1} \mathbf{Z}_p$ on $\pi_n$; 
here, $\SSS^n$ is considered as a constant pro-object. 
By the universal property of pro-$p$-finite completion, we get a map $(\SSS^n)_p^{\wedge} \to \Et(X)_p^{\wedge}$ in $\Pro(\mathcal{S}^{p\mhyphen\mathrm{fc}})$. This induces a map $C^* (\Et(X)_p^{\wedge}, k) \to C^* ((\SSS^n)_p^{\wedge}, k)$. We will show that this map is an isomorphism. To do so, we recall that $H^i(\Et(X)_p^{\wedge}, k) \simeq H_{\et}^i (X, k) \simeq H^i (X, \cO)$ for all $i \ge 0$. Therefore, it would be enough to show that the induced map $H^n (\Et(X)_p^{\wedge}, k) \to H^n ((\SSS^n)_p^{\wedge}, k)$ is nonzero (since they are both $1$-dimensional $k$-vector spaces). For the latter, it would be enough to show that the composition $H^n (\Et(X)_p^{\wedge}, k) \to H^n ((\SSS^n)_p^{\wedge}, k) \to H^n (\SSS^n, k)$ is nonzero. However, we recall that by construction, the map $\SSS^n \to \Et(X)_p^{\wedge}$ induces the map $\mathbf{Z} \xrightarrow[]{1 \mapsto 1} \mathbf{Z}_p$ of pro-homotopy groups. By \cite[Cor.~4.5]{etalehomo}, this implies that the map on pro-homology groups is also given by $\mathbf{Z} \xrightarrow[]{1 \mapsto 1} \mathbf{Z}_p$. Since $H_{n-1}(\Et(X)_p^{\wedge})$ and $H_{n-1}(\SSS^n)$ are both zero as pro-objects, it follows from the universal coefficient theorem that $H_{n}(\Et(X)_p^{\wedge}, k)$ and $H_{n}(\SSS^n, k)$ are both pro-isomorphic to $k$. Further, the map $H_{n}(\SSS^n, k) \to H_{n}(\Et(X)_p^{\wedge}, k)$ is identified with the map $k \xrightarrow[]{1 \mapsto 1} k$ of $k$-vector spaces. This implies that the map $H^n(\Et(X)_p^{\wedge}, k) \to H^n (\SSS^n, k)$ is nonzero, which establishes that the map $C^* (\Et(X)_p^{\wedge}, k) \to C^* ((\SSS^n)_p^{\wedge}, k)$ is an isomorphism. Now, applying \cref{p-adichomotopytheory} proves the claim in the proposition. 
\end{proof}

\begin{remark}\label{ordinaryCYhom}
Via \cref{ordinary-unipotent-homotopy}, \cref{ordinaryCY} has the following consequence:
Let $X$ be a weakly ordinary Calabi--Yau variety of dimension $n \ge 2$. 
Then $\pi^{\mathrm U}_i(X)$ is the group scheme associated with the profinite group $\pi_i((S^{n})_p^\wedge)$ for all $i \ge 0$. 
In fact, over an algebraically closed field $k$ of characteristic $p>0$, the unipotent homotopy type of an ordinary Calabi--Yau variety of dimension $n$ is isomorphic to the unipotent homotopy type of the $n$-sphere.
\end{remark}{}

\subsection{Derived equivalent Calabi--Yau varieties.}\label{worldcup1}
Let $X$ and $Y$ be two smooth, projective algebraic varieties over a field $k$ such that the categories $D_{\mathrm{perf}}(X)$ and $D_{\mathrm{perf}}(Y)$ of perfect complexes on $X$ and $Y$ are equivalent as $k$-linear triangulated categories. In such a case, we say that $X$ and $Y$ are \textit{derived equivalent} or Fourier--Mukai
equivalent. A very natural question arising from the foundational work of Bondal--Orlov \cite{MR1818984} is 
the following: if $X$ and $Y$ are derived equivalent, what can we say about the algebraic varieties $X$ and $Y$? 
Bondal and Orlov proved that if two smooth projective varieties $X$ and $Y$ with ample (anti-)canonical bundle are derived equivalent, then $X$ and $Y$ are isomorphic as varieties.

\vspace{2mm}
We say that an invariant $h$ of algebraic varieties is a \textit{derived invariant} if $h(X) \simeq h(Y)$ whenever $X$ and $Y$ are derived invariant. There are many questions and conjectures regarding whether certain invariants of algebraic varieties, such as the Hodge numbers $h^{i,j}$, are derived invariants or not. 
Such questions can be asked not only for numerical but also for categorical invariants.
For example, a conjecture of Orlov \cite{Orlov2} asks whether the rational Chow motive $M(\cdot)_{\mathbf
Q}$ of a smooth projective variety is a derived invariant.

\vspace{2mm}
Since our paper introduces the notion of unipotent homotopy types of algebraic varieties, it is natural to ask the following question.
\begin{question}\label{orlov?}
For what class of smooth projective varieties $X$ and $Y$ over an algebraically closed field $k$ does $D_{\mathrm{perf}}(X) \simeq D_{\mathrm{perf}} (Y)$ as $k$-linear triangulated categories imply that $\mathbf{U}(X) \simeq \mathbf{U}(Y)?$
\end{question} 
If $X$ and $Y$ have ample (anti-)canonical bundles, then the answer is clearly yes because $X$ and $Y$ are isomorphic as algebraic varieties by \cite{MR1818984}.
However, the following example coming from the recent work of Addington--Bragg \cite{examplebyAB} would quickly show that this question is too strong to have a positive answer in general.

\begin{example}[Addington--Bragg]\label{Addington-Bragg}
There exist smooth projective threefolds $X$ and $M$ over the field $\overline{\mathbf{F}}_{3}$ which are derived equivalent but satisfy $\dim H^{i} (X, \mathscr O)=0$ whereas $\dim H^i (M, \cO)=1$ for $1 \le i \le 2$.
This shows that $\mathbf{U}(X) \simeq \mathbf{U}(M)$ cannot be true since $R\Gamma(X, \cO)$ cannot be isomorphic to $R\Gamma(Y, \cO)$.
In fact, in the language of unipotent fundamental group schemes, it follows that in this case, $\pi_1^{\mathrm U}(X)$ is trivial, whereas $\pi_1^{\mathrm U}(M)$ is nontrivial. 
\end{example}

We remind the reader that according to \cref{CY}, the variety $M$ is not a Calabi--Yau variety.
Moreover, as we saw in \cref{worldcup3}, \cref{orlov?} is too strong even under the additional assumption that $\dim H^i(X,\cO) = \dim H^i(Y,\cO)$.
The main goal of this subsection is to prove \cref{thmCY} which addresses \cref{orlov?} affirmatively for Calabi--Yau varieties, and \cref{thmCY1}, which gives a much stronger functorial refinement of \cref{thmCY}.

\begin{theorem}\label{thmCY}
Let $X$ and $Y$ be two Calabi--Yau varieties of dimension $n$ (\textit{cf.} \cref{CY}) over an algebraically closed field $k$ of characteristic $p>0$ such that $D_\mathrm{perf}(X) \simeq D_\mathrm{perf}(Y)$. Then there is an isomorphism $\mathbf{U}(X) \simeq \mathbf{U}(Y)$ of unipotent homotopy types of $X$ and $Y$.
\end{theorem}{}

\begin{theorem}\label{thmCY1}
Let $k$ be an algebraically closed field of characteristic $p>2$ and $n>2$ be an integer. Let $\UU \colon \mathrm{CY}_n \to \mathrm{h}\mathrm{Shv}(k)$ be the functor obtained by sending $X$ to the unipotent homotopy type $\UU(X)$. Then there is a canonical functor $\widetilde{\UU} \colon \mathcal{N}\mathrm{CY}_n^{\mathrm{op}}\to \mathrm{h}\mathrm{Shv}(k)$ such that $\widetilde{\UU}\circ\mathcal{N}$ is naturally equivalent to $\UU$. (See \cref{defncy}.)
\end{theorem}{}

\begin{remark}
   Note that since the canonical bundle of a Calabi--Yau variety is trivial, the techniques from \cite{MR1818984} do not yield useful conclusions. Instead, the techniques we use to prove the above two results are more of a homotopical nature and work more generally for smooth proper $n$-dimensional varieties that satisfy the conditions (\ref{1/2-CY}) from \cref{cof3}. 
\end{remark}{}

Let us give an outline of the proof of \cref{thmCY} (\cref{thmCY1} will need additional ingredients that will be discussed later). In their recent work \cite{AB1}, Antieau and Bragg prove that if $X$ and $Y$ are derived equivalent Calabi--Yau varieties of dimension $n$, then the Artin--Mazur formal Lie groups $\Phi^n_X$ and $\Phi^n_Y$ are naturally isomorphic. The key ingredient in \cite{AB1} is the use of topological Hochschild homology, which attaches a $p$-typical cyclotomic spectrum $\mathrm{THH}(X)$ to the variety $X$. By definition, $\mathrm{THH}(X)$ is manifestly a derived invariant since its definition only depends on $D_\mathrm{perf}(X)$. 
To any $p$-typical cyclotomic spectrum $\mathcal C$, there is a construction of Hesselholt \cite{TR} that attaches a new spectrum $\mathrm{TR}(\mathcal C)$ with an $\SSS^1$-action and natural endomorphisms $F$ and $V$. For the Calabi--Yau variety $X$, one sets $\mathrm{TR}(X) \colonequals  \mathrm{TR}(\mathrm{THH}(X))$. In their paper, using the descent spectral sequence, Antieau and Bragg observe that the Dieudonn\'e module of $\Phi^n_X$ is naturally isomorphic to $\pi_{-n}\mathrm{TR}(X)$, which shows that $\Phi^n_X$ is a derived invariant.

\vspace{2mm}
In our paper, we gave a reconstruction of $\Phi^n_X$ based on the unipotent homotopy theory developed here. In particular, \cref{cycomp} shows that $\Phi^n_X$ is dual to the unipotent group scheme $\pi_n^\w{U} (X)$. When $X$ is a ordinary Calabi--Yau, one knows that $\pi_n^\w{U} (X)$ is the profinite group scheme $\mathbf{Z}_p$, which is dual to the formal Lie group $\widehat{\GG}_m$. By \cref{ordinaryCY}, we know that in this case $\mathbf{U}(X)$ is isomorphic to the unipotent homotopy type of the ($p$-completed) $n$-sphere $\SSS^n$. Motivated by this, we take the following approach in order to prove the derived invariance of the unipotent homotopy type of Calabi--Yau varieties.

\begin{enumerate}
    \item Construct a notion of ``formal $n$-sphere" $\SSS^n_{E}$ for every $1$-dimensional formal Lie group $E$ defined over $k$. $\SSS^n_{E}$ will be defined to be a pointed higher stack.
    \item Prove that $\mathbf{U}(X) \simeq \mathbf{U}(\SSS^n_{\Phi_n(X)})$. 
\end{enumerate}{}

We now proceed towards realizing the two steps above. First, we focus on some preliminaries. As before, let $\Shv(k)$ denote the category of (higher) stacks over $k$. Let $\Shv(k)_*$ denote the category of pointed (higher) stacks over $k$. Recall that there is a functor
\[ \Omega \colon \Shv(k)_* \to \Shv(k)_*, \quad X \mapsto \Omega X \colonequals * \times_X *. \]
This functor has a left adjoint $\Sigma \colon \Shv(k)_* \to \Shv(k)_*$, called the suspension functor.

\begin{lemma}\label{usefullemma113} Let $G$ be a commutative group scheme over $k$. Let $X \in \Shv(k)_*$. Suppose that $H^n (\Spec k, G)=0$ for all $n>0$. Then $H^n (X, G) = \pi_0 \mathrm{Map}_{\Shv(k)_* } (X, K(G, n))$ for $n>0$.
\end{lemma}{}

\begin{proof} We know that $H^n (X, G) = \pi_0 \mathrm{Map}_{\Shv(k)} (X, K(G, n))$ for $n \ge 0$. Note that there is a natural forgetful functor $\Shv(k)_* \to \Shv(k)$ which forgets the base point. This functor has a left adjoint $\Shv(k) \to \Shv(k)_*$ which adds a disjoint base point, i.e., it sends $Y$ to $Y \coprod \left \{ * \right \}$.

\vspace{2mm}
We note that there is a natural map $\mathrm{Map}_{\Shv(k)_*} (X, K(G, n)) \to \mathrm{Map}_{\Shv(k)} (X, K(G, n))$. By adjunction, we have $\mathrm{Map}_{\Shv(k)_*} (X \coprod \left \{ * \right \}, K(G, n)) \simeq \mathrm{Map}_{\Shv(k)} (X, K(G, n))$. This defines a map 
\begin{equation}\label{forgetbase}
    \mathrm{Map}_{\Shv(k)_*} (X, K(G, n)) \to \mathrm{Map}_{\Shv(k)_*} (X \coprod \left \{ * \right \}, K(G, n)).
\end{equation}Since $X$ is pointed, there is a natural map $\left \{ * \right \} \to X$ which induces a map
\begin{equation}\label{forgetbase1}
X \coprod \left \{ *  \right \} \to X.    
\end{equation}in $\Shv(k)_*$. We note that the map in \cref{forgetbase} is induced by the map \cref{forgetbase1}.
The pushout diagram $$\left \{ * \right \} \coprod_{X \coprod \left \{ * \right \}} X \simeq \SSS^1$$ in $\Shv(k)_*$ therefore induces a fibre sequence
\[ \mathrm{Map}_{\Shv(k)_*} (\SSS^1, K(G, n+1)) \to \mathrm{Map}_{\Shv(k)_*} (X, K(G, n+1)) \to \mathrm{Map}_{\Shv(k)} (X, K(G, n+1)). \]
We note that $\mathrm{Map}_{\Shv(k)_*} (\SSS^1, K(G, n+1)) \simeq \mathrm{Map}_{\Shv(k)_*} (\Sigma (\left \{ * \right \} \coprod  \left \{ * \right \})  , K(G, n+1))$. By adjunction, the latter term is isomorphic to $\mathrm{Map}_{\Shv(k)}(\Spec k, K(G, n) )$.
Thus, for $n>1$, we obtain that the map $\pi_1\mathrm{Map}_{\Shv(k)_*} (X, K(G, n+1)) \to \pi_1\mathrm{Map}_{\Shv(k)} (X, K(G, n+1))$ is an isomorphism since $H^{r}(\Spec k, G)=0$ for $r>0$. This implies that the map $\pi_0\mathrm{Map}_{\Shv(k)_*} (X, K(G, n)) \to \pi_0\mathrm{Map}_{\Shv(k)} (X, K(G, n))$ is an isomorphism for $n>1$. 
For $n=1$, since the map
$$\pi_0 \mathrm{Map}_{\Shv(k)}(X, G) \to \pi_0 \mathrm{Map}_{\Shv(k)}(\Spec k, G ) \simeq G(k)$$ is surjective, the long exact sequence of homotopy groups associated to the above fibre sequence shows that the map
$\pi_0\mathrm{Map}_{\Shv(k)_*} (X, K(G, n)) \to \pi_0\mathrm{Map}_{\Shv(k)} (X, K(G, n))$ is an isomorphism. 
This finishes the proof.
\end{proof}{}

\begin{lemma}\label{cohomologyvan}
Let $k$ be an algebraically closed field and $G$ be a commutative unipotent affine group scheme over $k$. Then $H^n(\Spec k, G)=0$ for $n>0$.
\end{lemma}{}

\begin{proof}
First, we show the claim when $G$ is of finite type over $k$.
If $G$ is of finite type and unipotent, then there is a filtration on $G$ where the graded pieces are isomorphic to $\mathbf{Z}/p \mathbf{Z}$, $\alpha_p$ or $\GG_a$, for which the claim can be checked directly. Now if $G$ is not assumed to be finite type, one still has an injection $f \colon G \to \prod_{i \in I} G_i$, where $(G_i)_{i \in I}$ are all the finite type quotients of $G$. Consider the resulting exact sequence 
\begin{equation}\label{pas1}
    0 \to G \xrightarrow{f} \prod_{i \in I} G_i \to \Coker(f) \to 0.
\end{equation}
By \cref{pas2} and the fact that $H^1 (\Spec k, \prod_{i \in I} G_i)=0$ for $i>0$, we obtain $H^1 (\Spec k, G)=0$. 
Let $n \ge 1$ be an integer and assume that $H^i(\Spec k, G)=0$ has been proven for any unipotent group scheme $G$ for $i \le n$. Since $\Coker(f)$ is unipotent, we must have $H^n (\Spec k, \Coker(f))=0$. Since $H^i (\Spec k, \prod_{i \in I} G_i)=0$ for all $i>0$, we obtain $H^{n+1}(\Spec k, G)=0$. Therefore, we are done by induction.
\end{proof}{}
The following lemma was used in the proof above.
\begin{lemma}\label{pas2} Let $k$ be an algebraically closed field. Let $\varphi \colon G \to H$ be a surjection of commutative group schemes over $k$.
Then the map $G(k) \to H(k)$ on $k$-valued points is surjective.
\end{lemma}{}

\begin{proof}
We thank Peter Scholze for suggesting the idea of the proof. Let us fix $x \in H(k)$. Let $\mathcal{C}$ denote the category whose objects are diagrams of group schemes \begin{center}
    \begin{tikzcd}
                                                                           & \tilde{H} \arrow[d, "u"] \\
G \arrow[r, "\varphi", two heads] \arrow[ru, "\tilde{\varphi}", two heads] & H  ,                     
\end{tikzcd}
\end{center}where $\tilde{H}$ is further equipped with a $k$-rational point $\tilde{x}$ such that $u(\tilde{x}) = x$. We will simply denote this data by $(G \twoheadrightarrow \tilde{H})$ if no confusion is likely to occur. The morphisms between $(G \twoheadrightarrow \tilde{H_1}) \to (G \twoheadrightarrow \tilde{H_2})$ are group homomorphisms $\tilde{H_1} \to \tilde{H_2}$ which are compatible with all the extra data (including the $k$-rational points). We note that the category $\mathcal{C}$ is essentially a partially ordered set. By considering inverse limits in the category of affine group schemes, it follows that every chain has an upper bound. Therefore, by Zorn's lemma, $\mathcal{C}$ must have maximal elements. Let $(G \twoheadrightarrow H')$ denote such a maximal element. We note that the map $\varphi' \colon G \twoheadrightarrow H'$ is faithfully flat. The induced functor $\mathrm{Rep}_k(H') \to \mathrm{Rep}_k (G)$ of the category of finite-dimensional representations is fully faithful. If this functor was also an equivalence, then the map $\varphi'$ would be an isomorphism, which would finish the proof of the lemma. Therefore, we assume on the contrary that the functor $\mathrm{Rep}_k(H') \to \mathrm{Rep}_k (G)$ is not an equivalence. Let $X \in \mathrm{ob}(\mathrm{Rep}_k (G))$ that is not in the essential image of the latter functor. We let $\mathcal{N}$ denote the neutral Tannakian category (see \cite[Def.~2.19]{tannaka}) generated by the essential image of $\mathrm{Rep}_k(H') \to \mathrm{Rep}_k (G)$ and $X$ under tensor, dual, direct sum and subquotients. It follows that there is an affine group scheme $\underline{H}'$ such that $\mathrm{Rep}(\underline{H}') \simeq \mathcal{N}$ and the fully faithful functors $\mathrm{Rep}_k(H') \to \mathcal{N} \to \mathrm{Rep}_k (G)$ are induced by surjections $G \twoheadrightarrow \underline{H}' \twoheadrightarrow H'$ (see \cite[Cor.~2.9,~Thm.~2.11,~and~Prop.~2.21]{tannaka}). Let $K$ denote the kernel of the map $\underline{H}' \to H'$. Then it follows from construction that $\mathrm{Rep}_k (K)$ is generated by one object under tensor, dual, direct sum and subquotients. Therefore, by \cite[Prop.~2.20]{tannaka}, $K$ is a finite type group scheme over $k$. Let $\underline{H}'_{x'}$ denote the scheme theoretic fibre of the map $\underline{H}' \to H'$ over the $k$-rational point $x' \in H'(k)$. Since $K$ is finite type, it follows that $\underline{H}'_{x'}$ is also finite type over $k$. But since $k$ is algebraically closed, there exists a $k$-rational point of $\underline{H}'_{x'}$; this defines a $k$-rational point $\underline{x}' \in \underline{H}'(k)$ that maps to $x'$ under the map $\underline{H}' \to H'$. But this contradicts the maximality of $(G \twoheadrightarrow H')$, which finishes the proof.
\end{proof}{}

Combining the above two lemmas, we obtain the following proposition, which will be used frequently in the following parts of the paper.

\begin{proposition}\label{usefullemma1}
Let $G$ be a commutative unipotent affine group scheme over an algebraically closed field $k$. Let $X \in \Shv(k)_*$. Then $H^n (X, G) = \pi_0 \mathrm{Map}_{\Shv(k)_* } (X, K(G, n))$ for $n>0$.
\end{proposition}{}
\begin{proof}
Follows directly from \cref{usefullemma113} and \cref{cohomologyvan}.
\end{proof}{}

\begin{lemma}\label{lemmaaboutstuff} Let $G$ and $H$ be commutative affine group schemes over an algebraically closed field $k$. Assume that $H$ is unipotent. Then $\tau_{\le n-1} \mathrm{Map} (K(G,n), K(H,n)) \simeq  \mathrm{Hom} (G, H)$ for $n \ge 1$. In other words, $\tau_{\le n-1} \mathrm{Map} (K(G,n), K(H,n))$ is $0$-truncated.
\end{lemma}{}

\begin{proof}
Note that $\pi_0 \mathrm{Map} (K(G,n), K(H,n)) \simeq \mathrm{Hom}(G, H)$, by \cref{usefullemma1}. Thus we have a natural map $$\tau_{\le n-1} \mathrm{Map} (K(G,n), K(H,n)) \to  \mathrm{Hom} (G, H).$$ Since $\mathrm{Map} (K(G,n), K(H,n))$ is naturally a grouplike $E_\infty$-space, it is enough to show that $\pi_i \mathrm{Map} (K(G,n), K(H,n))$ is trivial for $0<i<n$. To this end, we note that $$\pi_i \mathrm{Map} (K(G,n), K(H,n)) \simeq \pi_0 \mathrm{Map} (K(G,n), K(H, n-i)) \simeq H^{n-i} (\Spec k, H);$$ the latter is trivial by \cref{cohomologyvan}, as desired.
\end{proof}{}

Finally, we are ready to introduce the key construction of the formal $n$-sphere.
\begin{construction}[Formal $n$-sphere] Let $E$ be a $1$-dimensional commutative formal Lie group. Let us consider the unipotent affine group scheme $E^\vee$. The classifying stack $BE^\vee$ is naturally a pointed connected stack. We set $$\SSS^n_E \colonequals \Sigma^{n-1} B E^\vee. $$ It follows that $\SSS^n_E$ is a pointed connected stack.
\end{construction}{}

\begin{remark} The construction of the formal sphere is motivated by the fact that $\SSS^n \simeq \Sigma^{n-1} \SSS^1 \simeq \Sigma^{n-1} B \mathbf{Z}$.
\end{remark}{}

\begin{proposition}[Cohomology of the formal sphere]\label{usefullemma2} Let $E$ be a $1$-dimensional formal Lie group and $\SSS^n_E$ be the associated formal $n$-sphere. Then $H^0 (\SSS^n_E, \cO) \simeq k$, $H^n (\SSS^n_E, \cO) \simeq k$ and $H^i (\SSS^n_E, \cO) = 0$ for $n \ne \left \{0,n  \right \}$. 
\end{proposition}{}

\begin{proof}
We note that $H^0 (\SSS^n_{E}, \cO) \simeq k$ since the formal $n$-sphere is connected as a stack. When $n=1$, the desired computation is a special case of \cref{classifynoncommfgl} since $\SSS^1_E \simeq BE^\vee$. For $n \ge 2$, we have $$H^1 (\SSS^n_E, \mathscr O) \simeq \pi_0 \w{Map}_{\Shv(k)_*}(\SSS^n_E, K(\GG_a,1)).$$ However, since $n \ge 2$, it follows that $\tau_{\le 1} \SSS^n_E \simeq \left \{ * \right \}$. Therefore, $H^1 (\SSS^n_E, \cO)=0$, when $n \ge 2$. Further, for $n \ge 2$ and $i \ge 2$, we have
$$H^i (\SSS^n_{E}, \cO) \simeq \pi_0 \w{Map}_{\Shv(k)_*}(\Sigma \, \SSS^{n-1}_{E}, K(\GG_a,i)) \simeq \pi_0 \w{Map}_{\Shv(k)_*}(\SSS^{n-1}_{E}, \Omega\,K(\GG_a,i)) \simeq H^{i-1}(\SSS^{n-1}_E, \cO).$$ Therefore, by induction on $n$, we obtain the computation as desired.
\end{proof}{}

\begin{proposition}[Unipotent homotopy type of a Calabi--Yau variety]\label{hach} Let $X$ be a Calabi--Yau variety of dimension $n \ge 2$ over an algebraically closed field $k$ equipped with a $k$-rational point. Let $\Phi_X^n$ be the Artin--Mazur formal Lie group of $X$. Then

\begin{enumerate}
    \item There exists an isomorphism $$  \mathbf{U}(\SSS^n_{\Phi_X^n}) \simeq \mathbf{U}(X)$$ of higher stacks over $k$; in the homotopy category, this isomorphism can be chosen to induce the canonical isomorphism $ \tau_{\le n}  \mathbf{U}(\SSS^n_{\Phi_X^n}) \simeq \tau_{\le n}\UU(X)$ induced by the Artin--Mazur--Hurewicz class $\xi^{\mathrm{AM}}_{\mathrm{H}}$ (see \cref{amhclass}) on the $n$-truncations.
    
    \item The set of isomorphisms $$  \mathbf{U}(\SSS^n_{\Phi_X^n}) \simeq \mathbf{U}(X)$$ in the homotopy category of higher stacks over $k$ that induces the isomorphism induced by $\xi^{\mathrm{AM}}_{\mathrm{H}}$ on the $n$-truncations is a torsor under the group $H^2 (B \pi_n^{\mathrm{U}}(X), \pi_{n+1}^\mathrm{U}(X))$.
\end{enumerate}{}
\end{proposition}{}

\begin{proof}
We note that $\pi_i^{\mathrm{U}}(X)$ is trivial for $i <n$ by \cref{hurewicz}. For simplicity, let us use $E$ to denote the formal Lie group $\Phi_X^n$. The class $\xi^{\mathrm{AM}}_{\mathrm{H}}$ (see \cref{amhclass}) induces an isomorphism $\tau_{\le n} \mathbf{U}(X) \simeq K(E^\vee, n)$ that is canonical in the homotopy category. By \cref{f}, it also follows that $\tau_{\le n} \SSS^n_{E} \simeq K(E^\vee, n)$. Therefore, we obtain the following maps $$\SSS^n_E \to \mathbf{U}(\SSS^n_E) \to \tau_{\le n} \SSS^n_{E} \xrightarrow[]{\sim} \tau_{\le n} \mathbf{U}(X).$$ Suppose that we are given a map $G_r \colon \SSS^n_E \to \tau_{\le n+r} \mathbf{U}(X)$ for some $r \ge 0$. We will lift the given map $G_r \colon \SSS^n_E \to \tau_{\le n+r}\mathbf U(X) $ along the Postnikov truncation $\tau_{\le n+r+1}\mathbf U(X) \to \tau_{\le n+r}\mathbf U(X)$. \vspace{2mm}

Now, we note that $\pi_1^{\mathrm U}(X) \simeq \pi_1 (\mathbf{U}(X)) \simeq \left \{ * \right \}$ which implies (by the discussion before \cite[Prop.~1.2.2]{Toe} for example) that for every $r \ge 0$, there is a pullback diagram 
\begin{equation}\label{ddlj1}
\begin{tikzcd}
\tau_{\le n+r+1} \mathbf{U}(X) \arrow[rr] \arrow[d] &  & \tau_{\le n+r} \mathbf{U}(X) \arrow[d] \\
* \arrow[rr]                                        &  & {K (\pi_{n+r+1}\mathbf{U}(X), n+r+2).}              
\end{tikzcd}
\end{equation}
This gives an exact sequence
$$\pi_0 \mathrm{Map} (\SSS^n_E, \tau_{\le n+r+1} \mathbf{U}(X)) \to \pi_0 \mathrm{Map} (\SSS^n_E, \tau_{\le n+r} \mathbf{U}(X)) \to \pi_0 \mathrm{Map}(\SSS^n_E,  K (\pi_{n+r+1}\mathbf{U}(X), n+r+2)).$$In order to prove that $G_r$ lifts along $ \tau_{\le n+r+1} \mathbf{U}(X) \to \tau_{\le n+r} \mathbf{U}(X)$, it is thus enough to prove that \begin{equation}\label{ddlj2}
    \pi_0 \mathrm{Map}(\SSS^n_E,  K (\pi_{n+r+1}\mathbf{U}(X), n+r+2)) \simeq H^{r+3} (BE^\vee, \pi_{n+r+1}^{\mathrm U} (X))
\end{equation}
vanishes; here, the isomorphism follows from \cref{usefullemma1} and the adjunction of $(\Sigma , \Omega)$. To see the desired vanishing, we may note that $\pi_{n+r+1}^\mathrm{U} (X)$ is a unipotent group scheme and apply \cref{newyr2}, which gives that $H^{r+3} (BE^\vee, \pi_{n+r+1}^\mathrm{U} (X))= 0$ since $r \ge 0$. Therefore, we can lift the morphism $G_r \colon \SSS_E^n \to \tau_{\le n+r} \mathbf{U}(X)$ to obtain a map $G_{r+1} \colon \SSS_E^n \to \tau_{\le n+r+1} \mathbf{U}(X)$. In other words, we have the following commutative diagram in the homotopy category of stacks over $k$ 
\begin{center}
    \begin{tikzcd}
                                                                &  & \tau_{\le n+r+1}\mathbf{U}(X) \arrow[d] \\
\SSS^n_E \arrow[rr, "G_r"] \arrow[rru, "G_{r+1}"] &  & \tau_{\le n+r}\mathbf{U}(X).            
\end{tikzcd}
\end{center}{}
Proceeding by induction on $r$, we may obtain a map $G \colon \SSS^n_E \to \varprojlim_{r} \tau_{\le n+r} \mathbf{U}(X)$. Since $\mathbf{U}(X) \simeq \varprojlim_{r} \tau_{\le n+r} \mathbf{U}(X)$, this constructs a map $\SSS^n_E \to \mathbf{U}(X)$. By construction and \cref{usefullemma1}, it follows that the induced map $H^i (\mathbf{U}(X), \mathscr O) \to H^i (\SSS^n_E, \cO)$ is an isomorphism for $i \le n$. 
However, since $X$ is a Calabi--Yau variety of dimension $n$, the cohomology groups $H^i (\mathbf{U}(X), \cO) \simeq H^i (X, \mathscr O) =0$ for $i >n$. On the other hand, $H^i (\SSS^n_E, \cO)=0$ for $i >n$ by \cref{usefullemma2}. This implies that the induced maps $H^i (\mathbf{U}(X), \mathscr O) \to H^i (\SSS^n_E, \cO)$ are isomorphisms for all $i \ge 0$. Thus the map $\SSS^n_E \to \mathbf{U}(X)$ induce an isomorphism $\mathbf{U}(\SSS^n_E) \simeq \mathbf{U}(X)$. This proves the first part of the proposition. \vspace{2mm}

Now we prove the second part of the proposition. Note that $\mathbf{U}(X) \simeq \varprojlim_{r} \tau_{\le n+r} \mathbf{U}(X)$. Therefore, $$ \w{Map}(\SSS^n_E, \mathbf{U}(X)) \simeq \varprojlim_{r}\w{Map} (\SSS^n_E, \tau_{\le n+r} \mathbf{U}(X)).$$Let $y \in \pi_0 \mathrm{Map} (\SSS^n_E, \mathbf{U}(X))$. 
We can lift $y$ to obtain a map $\widetilde{y} \colon \left \{* \right \} \to \mathrm{Map}(\SSS^n_E, \UU(X))$, which induces maps $\widetilde{y}_{r} \colon \left \{* \right \} \to \mathrm{Map}(\SSS^n_E, \tau_{\le n+r} \UU(X))$.
By using the associated Milnor sequence, we have an exact sequence of pointed sets
$$0 \to {\varprojlim_{r}}^1\pi_1(\w{Map} (\SSS^n_E,  \tau_{\le n+r} \mathbf{U}(X)),\widetilde{y}_r) \to \pi_0 \w{Map}(\SSS^n_E, \mathbf{U}(X)) \to \varprojlim_{r} \pi_0 \w{Map} (\SSS^n_E, \tau_{\le n+r} \mathbf{U}(X)) \to 0.$$Using the diagram \cref{ddlj1}, we obtain a pullback diagram
\begin{equation}\label{ddlj10}
\begin{tikzcd}
\mathrm{Map} (\SSS^n_E, \tau_{\le n+r+1} \mathbf{U}(X)) \arrow[rr] \arrow[d] &  & \mathrm{Map}(\SSS^n_E,\tau_{\le n+r} \mathbf{U}(X) )\arrow[d] \\
* \arrow[rr]                                        &  & \mathrm{Map}(\SSS^n_E, {K (\pi_{n+r+1}\mathbf{U}(X), n+r+2)).}              
\end{tikzcd}
\end{equation}
We claim that for $r \ge 1$, the maps $$\pi_1 (\mathrm{Map} (\SSS^n_E, \tau_{\le n+r+1} \mathbf{U}(X)), \widetilde{y}_r) \to \pi_1 \mathrm{Map} ((\SSS^n_E, \tau_{\le n+r} \mathbf{U}(X)), \widetilde{y}_r)$$ are surjective. Given the claim, we have $\varprojlim_{r}^1 \pi_1(\mathrm{Map} (\SSS^n_E,  \tau_{\le n+r} \mathbf{U}(X)), \widetilde{y}_r) \simeq 0$, which implies that we have an isomorphism \begin{equation}\label{ddlj3}
    \pi_0 \w{Map}(\SSS^n_E, \mathbf{U}(X)) \simeq \varprojlim_{r} \pi_0 \w{Map} (\SSS^n_E, \tau_{\le n+r} \mathbf{U}(X)).
\end{equation}In order to see the surjectivity as in the above claim, we apply the long exact sequence in homotopy groups associated to \cref{ddlj10}. Indeed, since \cref{ddlj2} vanishes, it would be enough to show that for $r \ge 1$, 
\begin{equation}\label{ddlj9}
  \pi_1 \mathrm{Map}(\SSS^n_E,  K (\pi_{n+r+1}\mathbf{U}(X), n+r+2)) \simeq 0. 
\end{equation}To this end, note that $$\pi_1 \mathrm{Map}(\SSS^n_E,  K (\pi_{n+r+1}\mathbf{U}(X), n+r+2)) \simeq \pi_0 \mathrm{Map}(\SSS^n_E,  K (\pi_{n+r+1}\mathbf{U}(X), n+r+1));$$by applying \cref{usefullemma1} and the $(\Sigma, \Omega)$-adjunction, the latter is $H^{r+2} (B E^\vee, \pi_{n+r+1}\UU(X))$, which vanishes for $r \ge 1$ by \cref{newyr2}, as desired. A similar argument using \cref{ddlj10} and the vanishing of \cref{ddlj2} and \cref{ddlj9}
proves that for $r \ge 1$, the maps $$\pi_0 \mathrm{Map} (\SSS^n_E, \tau_{\le n+r+1} \mathbf{U}(X)) \to \pi_0 \mathrm{Map} (\SSS^n_E, \tau_{\le n+r} \mathbf{U}(X))$$ are isomorphisms. Together with \cref{ddlj3}, this implies that  $$\pi_0 \w{Map}(\SSS^n_E, \mathbf{U}(X)) \simeq \pi_0 \w{Map} (\SSS^n_E,  \tau_{\le n+1} \mathbf{U}(X)).$$ Using the long exact sequence associated with \cref{ddlj10} for $r=0$, we obtain an exact sequence$$\pi_1 \mathrm{Map}(\SSS^n_E, K(\pi_{n+1}^\mathrm{U}(X), n+2)) \to \pi_0 \mathrm{Map} (\SSS^n_E, \tau_{\le n+1} \mathbf{U}(X)) \to \pi_0 \mathrm{Map} (\SSS^n_E, \tau_{\le n} \mathbf{U}(X));$$ the last map is surjective by the vanishing of \cref{ddlj2}. The first map is injective since by \cref{lemmaaboutstuff} $$\tau_{\le 1} \mathrm{Map} (\SSS^n_E, \tau_{\le n} \UU(X)) \simeq \tau_{\le 1} \mathrm{Map} (K(E^\vee, n), K(E^\vee,n)) \simeq \mathrm{Hom}(E^\vee, E^\vee)$$ is $0$-truncated. Finally, by applying \cref{usefullemma1} and the $(\Sigma, \Omega)$-adjunction again, we have $\pi_1 \mathrm{Map}(\SSS^n_E, K(\pi_{n+1}^\mathrm{U}(X), n+2)) \simeq  H^2 (BE^\vee, \pi_{n+1}^{\mathrm{U}}(X))$. This finishes the proof.
\end{proof}{}

\begin{proof}[Proof of \cref{thmCY}] 
This is a consequence of \cref{hach}. Indeed, the Dieudonn\'e module of $\Phi_X^n$ is naturally isomorphic to $H^n (X, W)$, which is identified with $\pi_{-n} \mathrm{TR}(X)$ via the descent spectral sequence (see \cite[Def.~3.5(a), Thm.~5.1]{AB1}). Similarly, the Dieudonn\'e module of $\Phi_Y^n$ is isomorphic to $\pi_{-n}\mathrm{TR}(Y)$. By construction, $\mathrm{TR}(\,\cdot \,)$ is a derived invariant; therefore, $\pi_{-n}\mathrm{TR}(X) \simeq \pi_{-n}\mathrm{TR}(Y)$. Since we are working over an algebraically closed field $k$, by choosing $k$-rational points on $X$ and $Y$ and applying \cref{hach}, it follows that $\UU(X) \simeq \UU(Y)$.
\end{proof}{}

In order to prove \cref{thmCY1}, we need to obtain a concrete understanding of certain unipotent homotopy group schemes of Calabi--Yau varieties. This will require some additional preparations. 

\begin{proposition}[Some homotopy groups of the formal sphere]\label{homotopyformalsphere}Let $\SSS^n_E$ be the formal $n$-sphere for a fixed formal Lie group $E$ over $k$. Then

\begin{enumerate}
    \item We have $\pi_{i}(\mathbf{U}(\SSS^n_E)) \simeq \left \{* \right \} $ for $i <n$ and $\pi_{n}(\mathbf{U}(\SSS^n_E))\simeq E^\vee$.
    
    \item Further, $\pi_{3}(\mathbf{U}(\SSS^2_E)) \simeq (E^\vee \otimes E^\vee)^{\mathrm{uni}}$ and $$\pi_{n+1}(\mathbf{U}(\SSS^n_E)) \simeq (E^\vee \curlywedge E^\vee)^{\mathrm{uni}}$$for all $n \ge 3$ (\textit{cf.}~\cref{weakwedge}). 
\end{enumerate}{}

\end{proposition}{}

\begin{proof}
The first part follows from the fact that $\tau_{\le n} (\SSS^n_E) \simeq K(E^\vee, n)$; since the latter is an affine stack, we obtain that $\tau_{\le n}(\mathbf{U}(\SSS^n_E)) \simeq K(E^\vee, n)$.
\vspace{2mm}

We proceed to prove the second part. We begin by showing that $\pi_{3}(\mathbf{U}(\SSS^2_E)) \simeq (E^\vee \otimes E^\vee)^{\mathrm{uni}}$. We note that $\tau_{\le 2}(\SSS^2_E) \simeq K(E^\vee, 2)$, which is an affine stack. Therefore, we can appeal to \cref{freu1} to compute $\pi_{3}(\mathbf{U}(\SSS^2_E))$. Indeed, by \cref{lentil}, $\pi_3 (\SSS^2_E) \simeq E^\vee \otimes_{\mathbf{Z}} E^\vee$. Therefore, by \cref{freu1}, it follows that $\pi_{3}(\mathbf{U}(\SSS^2_E)) \simeq (E^\vee \otimes E^\vee)^{\mathrm{uni}}$, as desired.
\vspace{2mm}

By construction of the formal $n$-sphere and the Freudenthal suspension theorem in unipotent homotopy theory (\cref{freudenthalsusp}), it is enough to prove the claim when $n=3$,
i.e., we only need to prove that $\pi_4 (\mathbf{U}(\SSS^3_E)) \simeq (E^\vee \curlywedge E^\vee)^{\mathrm{uni}}$. By the proof of \cref{freudenthalsusp1} and \cref{c}, we have an exact sequence of group schemes 
\begin{equation}\label{bauhaus}
\left (\pi_2 (\mathbf{U}(\SSS^2_E)) \otimes \pi_2 (\mathbf{U}(\SSS^2_E)) \right )^{\mathrm{uni}} \xrightarrow[]{W} \pi_3 (\mathbf{U}(\SSS^2_{E})) \xrightarrow[]{} \pi_4 (\mathbf{U}(\Sigma (\mathbf{U}(\SSS^2_E)))) \xrightarrow[]{} 0.
\end{equation}
Further, the map $W$ is induced by the Whitehead product 
$$W_{2,2} \colon \pi_2^{\mathrm{U}}(\SSS^2_E) \times  \pi_2^{\mathrm{U}}(\SSS^2_E) \to \pi_3^{\mathrm{U}} (\SSS^2_E)$$ constructed in \cref{wh}. Since $\pi_{3}(\mathbf{U}(\SSS^2_E)) \simeq (E^\vee \otimes E^\vee)^{\mathrm{uni}}$, by the explicit computation of the Whitehead product from \cref{computewhite}, it follows that the map $$W \colon \left (\pi_2 (\mathbf{U}(\SSS^2_E)) \otimes \pi_2 (\mathbf{U}(\SSS^2_E)) \right )^{\mathrm{uni}} \xrightarrow[]{} \pi_3 (\mathbf{U}(\SSS^2_{E}))$$ above identifies with the map $(E^\vee \otimes E^\vee)^{\mathrm{uni}} \to (E^\vee \otimes E^\vee)^{\mathrm{uni}} $ that is induced by $x \otimes y \mapsto x \otimes y + y \otimes x$. Therefore, by the exact sequence \cref{bauhaus} above and \cref{garten}, $\pi_4 (\mathbf{U}(\SSS^3_E)) \simeq \pi_4 (\mathbf{U}(\Sigma (\mathbf{U}(\SSS^2_E)))) \simeq (E^\vee \curlywedge E^\vee)^{\mathrm{uni}}$. This finishes the proof.
\end{proof}{} 

\begin{corollary}[Some homotopy groups of Calabi--Yau varieties]Let $X$ be a Calabi--Yau variety of dimension $n$ and $\Phi_X^n$ be the Artin--Mazur formal Lie group of $X$. As a consequence of \cref{hach} and \cref{homotopyformalsphere}, we obtain the following results.

\begin{enumerate}
    \item If $X$ is a Calabi-Yau variety of dimension $2$, i.e., if $X$ is a K3 surface, then $$\pi_3^{\mathrm{U}} (X) \simeq ((\Phi_X^n)^\vee \otimes (\Phi_X^n)^\vee)^{\mathrm{uni}}.$$
    
    \item 
 If $X$ is a Calabi--Yau variety of $\dim X = n \ge 3$, then $$\pi_{n+1}^{\mathrm U}(X) \simeq ((\Phi_X^n)^\vee \curlywedge (\Phi_X^n)^\vee)^{\mathrm{uni}}.$$ 
\end{enumerate}{}

\end{corollary}{}

We will show that $\pi_3 (\mathbf{U}(\SSS^2_E))^{\vee}$ as computed by \cref{homotopyformalsphere} is a formal Lie group of dimension $1$. To this end, we will prove a more general assertion. 

\begin{proposition}\label{mathew1}
Let $E_1$ and $E_2$ be two commutative formal Lie groups of dimension $m$ and $n$ over a perfect field $k$ of characteristic $p>0$. Then $((E_1^\vee \otimes E_2^\vee)^{\mathrm{uni}})^{\vee}$ is a commutative formal Lie group of dimension $mn$.
\end{proposition}{}

\begin{proof}We work in the derived category of abelian sheaves in the fpqc site. Let $\underline{E_1^\vee}$ and $\underline{E_2^\vee}$ denote the sheaves represented by $E_1^\vee$ and ${E_2^\vee}$ and $\underline{E_1^\vee} \otimes \underline{E_2^\vee}$ denote their ordinary tensor product. Let $K \colonequals (\underline{E_1^\vee} \otimes^{L} \underline{E_2^\vee} )$. Computing $H^1 (R\Hom(K, \GG_a))$ via a spectral sequence (\cite[\href{https://stacks.math.columbia.edu/tag/07A9}{Tag~07A9}]{stacks}), we see that $\w{Ext}^1 (\underline{E_1^\vee} \otimes \underline{E_2^\vee}, \GG_a)$ is a subspace of $H^1 (R\Hom(K, \GG_a))$.
On the other hand, the latter term is isomorphic to $$H^1 (R\Hom(\underline{E_1^\vee}, \tau_{\ge -1}R{\mathscr{H}\w{om}}(\underline{E_2^\vee}, \GG_a))).$$ By \cref{retire4}, we see that $$H^1 (R\Hom(\underline{E_1^\vee}, \tau_{\ge -1}R{\mathscr{H}\w{om}}(\underline{E_2^\vee}, \GG_a)))=0.$$ Consequently, it follows that $\w{Ext}^1 (\underline{E_1^\vee} \otimes \underline{E_2^\vee}, \GG_a) \simeq 0$. Note that $$\w{Ext}^1 (\underline{E_1^\vee} \otimes \underline{E_2^\vee}, \GG_a) \simeq H^3 (B^2(\underline{E_1^\vee} \otimes \underline{E_2^\vee}), \cO).$$ 
However, by \cref{postnikov}.\cref{postnikov-cohom} (and \cref{whoknew}), the vector space $$H^3 (\tau_{\le 2} \UU(B^2(\underline{E_1^\vee} \otimes \underline{E_2^\vee})), \cO )\simeq H^3 (B^2 (E_1^\vee \otimes E_2^\vee)^{\mathrm{uni}}, \cO) \simeq  \w{Ext}^1 ((E_1^\vee \otimes E_2^\vee)^{\w{uni}}, \GG_a)$$ embeds into $H^3 (B^2(\underline{E_1^\vee} \otimes \underline{E_2^\vee}), \cO) \simeq \w{Ext}^1 (\underline{E_1^\vee} \otimes \underline{E_2^\vee}, \GG_a) = 0$. Therefore, $\w{Ext}^1 ((E_1^\vee \otimes E_2^\vee)^{\w{uni}}, \GG_a) = 0$. Finally, we note that $\dim_k \mathrm{Hom}(E_1^\vee \otimes E_2^\vee, \GG_a) = \dim_k \mathrm{Hom}(E_1^\vee, \GG_a) \otimes_{k} \mathrm{Hom}(E_2^\vee, \GG_a) = mn$. Thus, we are done by applying \cref{fine}.
\end{proof}{}

\begin{remark}\label{heightformula}
Commutative formal Lie groups over a perfect field can be classified by their Dieudonn\'e modules. At the level of Dieudonn\'e modules, the equivalent statement of \cref{mathew1} was also observed by Drinfeld and Mathew.  More precisely, if $M_1$ and $M_2$ denote the covariant Dieudonn\'e module of $E_1$ and $E_2$ respectively, then the Dieudonn\'e module of $((E_1^\vee \otimes E_2^\vee)^{\mathrm{uni}})^{\vee}$ is isomorphic to $M_1 \widehat{\boxtimes} M_2$, where we refer to \cite[Sec.~4]{ANI} for the definition of $M_1 \widehat{\boxtimes} M_2$. We thank Drinfeld for sharing a proof of the latter fact. The tensor product $M_1 \widehat{\boxtimes} M_2$ was originally introduced in \cite{Goe} and has been studied in \cite{BL07} and \cite{TM18}. We point out that in \cite{ANI}, the tensor product $M_1 \widehat{\boxtimes} M_2$ has been shown to be induced from the symmetric monoidal structure on cyclotomic spectra (see \cite[Sec.~4.1]{ANI}). \cref{mathew1} gives a direct geometric approach to these constructions by using group schemes and higher stacks without passing to associated Dieudonn\'e modules.
\end{remark}{}

\begin{definition}\label{retire2}
If $E_1$ and $E_2$ are two commutative formal Lie groups over a perfect field $k$, we define $$E_1 \boxtimes E_2 \colonequals ((E_1^\vee \otimes E_2^\vee)^{\w{uni}})^{\vee}.$$ By \cref{mathew1}, we have $\dim (E_1 \boxtimes E_2) = \dim E_1 \cdot \dim E_2$.
\end{definition}{}

\begin{remark}\label{haircutbad}
We note that $(\, \cdot \,) \boxtimes (\, \cdot\,)$ defines a symmetric monoidal operation on the category of commutative formal Lie groups over $k$. One can check (for example, by \cref{dieudonnemodu12} and Dieudonn\'e theory) that the formal Lie group $\widehat{\GG}_m$ is the unit object under this operation. In particular, $\widehat{\GG}_m \boxtimes \widehat{\GG}_m \simeq \widehat{\GG}_m$.
\end{remark}{}

\begin{proposition}\label{haircut}
Let $E_1$ and $E_2$ be two $1$-dimensional formal Lie groups over a perfect field $k$ such that height of $E_1$ and $E_2$ are both $>1$. Then $E_1 \boxtimes E_2 \simeq \widehat{\GG}_a$.
\end{proposition}{}

\begin{proof}
Since $E_1$ and $E_2$ both have height $>1$, then there are surjections ${E_1}^\vee \twoheadrightarrow \alpha_p$ and ${E_2}^\vee \twoheadrightarrow \alpha_p$; for example, this can be deduced from \cref{coffee150}.
These maps induce a surjection $({E_1}^\vee \otimes {E_2}^\vee)^{\mathrm{uni}} \twoheadrightarrow (\alpha_p \otimes \alpha_p)^{\mathrm{uni}} \simeq W[F]$. Here, the latter isomorphism follows from \cref{compute098}. Moreover, since $\mathrm{Ext}^1(W[F], \GG_a)=0$, a similar argument as in \cref{compute098} shows that the surjection $({E_1}^\vee \otimes {E_2}^\vee)^{\mathrm{uni}} \twoheadrightarrow W[F]$ must be an isomorphism. Now, taking dual yields the claim in the proposition.
\end{proof}{}

\begin{corollary}\label{example09}
Let $E$ be a $1$-dimensional formal Lie group over $k$. Then 
$$
\pi_3^{\mathrm{U}}(\SSS^2_E) \simeq
\begin{cases}
  \mathbf{Z}_p & \text{if height of}\,\, E=1,\,\, \text{i.e.,}\,\, E = \widehat{\GG}_m,\\

W[F] & \text{if height of} \,\,E\,\, \text{is}\,\, >1.
\end{cases}
$$
\end{corollary}

\begin{proof}
This follows from \cref{homotopyformalsphere}, \cref{haircutbad} and \cref{haircut}.
\end{proof}{}

\begin{remark}
Note that if $E = \widehat{\GG}_m$, then the fact that $\pi_3 ^{\mathrm{U}}(\SSS^2_E) \simeq \mathbf{Z}_p$ is consistent with the calculation from the $p$-adic (unstable) homotopy group of the topological $2$-sphere $\SSS^2$. However, \cref{example09} shows that in general, $\pi_3^{\mathrm{U}}(\SSS^2_E)$ behaves quite differently.
\end{remark}{}
\begin{corollary}
Let $X$ be a K3 surface over an algebraically closed field $k$ of characteristic $p>0$. Then
$$
\pi_3^{\mathrm{U}}(X) \simeq
\begin{cases}
  \mathbf{Z}_p & \text{if}\,\, X\,\, \text{is weakly ordinary,} \\

W[F] & \text{otherwise.}
\end{cases}
$$
\end{corollary}

\begin{proof}
Follows from \cref{hach} and \cref{example09}.
\end{proof}{}

Next, we proceed towards computing the weak wedge product (\cref{weakwedge}) in order to obtain an understanding of $\pi_{n+1}(\mathbf{U}(\SSS^n_E))$ for $n \ge 3$. Before we do that, let us begin by understanding the wedge product first in this context.

\begin{proposition}\label{mpi1}
Let $G$ be a $1$-dimensional formal Lie group over a perfect field $k$ of characteristic $p >0$. Then$$
(G^\vee \wedge G^\vee)^{\mathrm{uni}} \simeq
\begin{cases}
W[F] & \text{if} \,\,p=2 \,\text{and height of} \,\,G\,\, \text{is}\,\, >1, \\

0 & \text{otherwise.}
\end{cases}
$$
\end{proposition}{}

\begin{proof}
First, we recall that $\dim _k \Hom(G^\vee, \GG_a)=1$. Therefore, the space of bilinear maps $G^\vee \times G^\vee \to \GG_a$ is also $1$-dimensional.
As a consequence, any nonzero bilinear map $u \colon G^\vee \times G^\vee \to \GG_a$ is necessarily given by $u(x,y) = c f(x) f(y)$ where $c\in k^\times$ and $f \colon G^\vee \to \GG_a$ is a nonzero map of group schemes. The bilinear map $u$ is alternating if and only if $u(x,x) = 0$ for all scheme theoretic points $x$ of $G^\vee$. In such a case, we must have $f(x) ^2 = 0$ for all $x$. However, $f$ is also a group homomorphism. Therefore, for all $x$ and $y$, we must have
$$0 = f(x+y)^2 = (f(x) + f(y))^2 = 2 f(x)f(y)= 2 c^{-1} u(x,y). $$In the case when $p \ne 2$, this implies that any alternating bilinear map $u \colon G^\vee \times G^\vee \to \GG_a$ must be zero. In other words, $\Hom(G^\vee \wedge G^\vee, \GG_a)=0$, implying $(G^\vee \wedge G^\vee)^{\mathrm{uni}}=0$, as desired.
\vspace{2mm}

Now let us suppose that $p=2$ and the height of $G$ is $>1$. In that case, we have a surjective map $\varphi \colon G^\vee \to \alpha_2$. By \cref{rightexact} and \cref{compute098}, we have a surjection $(G^\vee \otimes G^\vee)^{\w{uni}} \twoheadrightarrow (\alpha_2 \otimes \alpha_2)^{\w{uni}} = W[F]$. Since $\mathrm{Ext}^1(W[F], \GG_a)=0$, it follows that the latter map is an isomorphism. 
Consider the induced bilinear map $G^\vee \times G^\vee \to W[F]$. Since we are working in the case when $p=2$, by \cref{compute098}, it follows that this bilinear map is alternating.\footnote{
Here, we use that for any $\FF_2$-algebra $S$ and any $s \in S$, we have $[s]^2 = F([s])$.}
This implies that the isomorphism $(G^\vee \otimes G^\vee)^{\mathrm{uni}} \simeq W[F]$ factors through the surjection $(G^\vee \otimes G^\vee)^{\mathrm{uni}} \twoheadrightarrow (G^\vee \wedge G^\vee)^{\mathrm{uni}}$. This shows that $(G^\vee \wedge G^\vee)^{\mathrm{uni}} \simeq W[F]$.
\vspace{2mm}

In the remaining case when $p=2$ and $G = \widehat{\GG}_m$, we know that $G^\vee \simeq \mathbf{Z}_2$, which is a reduced group scheme.
In particular, there is no nonzero map of group schemes $f \colon \ZZ_2 \to \GG_a$ with $f(x)^2 = 0$ for all scheme theoretic points $x$ of $\mathbf{Z}_2$; because then $f$ would factor through a map $\ZZ_2 \to \alpha_2$, which is necessarily zero.
Therefore, by the discussion in the first paragraph of the proof, we see that $(\mathbf{Z}_2 \wedge \mathbf{Z}_2)^{\mathrm{uni}} =0$.
\end{proof}{}

\begin{proposition}\label{haircut11}
Let $G$ be a $1$-dimensional formal Lie group over a perfect field $k$ of characteristic $p >0$. Then$$
(G^\vee \curlywedge G^\vee)^{\mathrm{uni}} \simeq
\begin{cases}
W[F] & \text{if} \,\,p=2 \,\text{and height of} \,\,G\,\, \text{is}\,\, >1, \\
\mathbf{Z}/2\mathbf{Z} & \text{if} \,\,p=2 \,\text{and} \,\,G \simeq \widehat{\GG}_m, \,\text{i.e., height of} \,\,G=1, \\

0 & \text{otherwise.}
\end{cases}
$$

\end{proposition}{}

\begin{proof}
As noted in the proof of \cref{mpi1}, any nonzero bilinear map $u \colon G^\vee \times G^\vee \to \GG_a$ is necessarily given by $u(x,y) = c f(x)f(y)$ where $c \in k^\times$ and $f \colon G^\vee \to \GG_a$ is a nonzero map of group schemes. Since $\GG_a$ is a commutative ring scheme, it follows that the bilinear map $u$ is symmetric. This bilinear map would correspond to an element of $\mathrm{Hom}(G^\vee \curlywedge G^\vee, \GG_a)$ if and only if $u(x,y) + u(y,x) = 0$. However, since $u$ is symmetric, the latter condition implies that $2 u(x,y) =0$. In the case when $p \ne 2$, this implies that $u(x,y) = 0$, which further implies that $\mathrm{Hom}(G^\vee \curlywedge G^\vee, \GG_a)=0$. Therefore, $(G^\vee \curlywedge G^\vee)^{\mathrm{uni}} = 0$, as desired. \vspace{2mm}

Now let us suppose that $p=2$ and height of $G$ is $>1$. The second paragraph of the proof of \cref{mpi1} shows that the natural map $(G^\vee \otimes G^\vee)^\mathrm{uni} \to (G^\vee \wedge G^\vee)^\mathrm{uni} $ is an isomorphism and they are both further isomorphic to $W[F]$. However, the natural map $(G^\vee \otimes G^\vee)^\mathrm{uni} \to (G^\vee \wedge G^\vee)^\mathrm{uni} $ factors through the natural surjection $(G^\vee \otimes G^\vee)^\mathrm{uni} \twoheadrightarrow (G^\vee \curlywedge G^\vee)^\mathrm{uni}$, which shows that $(G^\vee \curlywedge G^\vee)^\mathrm{uni} \simeq W[F]$, as well. \vspace{2mm}

Now we are left with the case when $p =2$ and $G = \widehat{\GG}_m$. In this case, $G^\vee \simeq \mathbf{Z}_2$. Note that $\mathbf{Z}_2$ naturally has the structure of a commutative ring scheme which induces a map $m \colon (\mathbf{Z}_2 \otimes \mathbf{Z}_2)^{\mathrm{uni}} \to \mathbf{Z}_2$ that is an isomorphism (\cref{haircutbad}).
\Cref{garten} implies that $(\mathbf{Z}_2 \curlywedge \mathbf{Z}_2)^{\mathrm{uni}}$ is identified (under the isomorphism $m$) with the cokernel of the multiplication by $2$ map on $\mathbf{Z}_2$. This shows that $(\mathbf{Z}_2 \curlywedge \mathbf{Z}_2)^{\mathrm{uni}} \simeq \mathbf{Z}/2 \mathbf{Z}$, as desired.
\end{proof}{}

\begin{corollary}\label{haircut12}
Let $E$ be a $1$-dimensional formal Lie group over a perfect field $k$ of characteristic $p >0$. Then, for $n \ge 3$,$$ \pi_{n+1}^{\mathrm{U}}(\SSS^n_{E})
 \simeq
\begin{cases}
W[F] & \text{if} \,\,p=2 \,\text{and height of} \,\,E\,\, \text{is}\,\, >1, \\
\mathbf{Z}/2\mathbf{Z} & \text{if} \,\,p=2 \,\text{and} \,\,E \simeq \widehat{\GG}_m, \,\text{i.e., height of} \,\,E=1, \\

0 & \text{otherwise, i.e., if}\,\, p \ne 2.
\end{cases}
$$

\end{corollary}{}

\begin{proof}
Follows from \cref{homotopyformalsphere} and \cref{haircut11}.
\end{proof}{}

\begin{remark}
Note that if $E = \widehat{\GG}_m$ and $n \ge 3$, then the fact that $\pi_{n+1} ^{\mathrm{U}}(\SSS^n_E) \simeq \mathbf{Z}/2 \mathbf{Z}$ when $p=2$ and is zero when $p \ne 2$ is consistent with the calculation from the homotopy group of the topological $n$-sphere $\SSS^n$. Indeed, $\pi_{n+1}(\SSS^n) \simeq \mathbf{Z}/2 \mathbf{Z}$ for $n \ge 3$. As we saw above, the computation of the $(n+1)$-th unipotent homotopy group scheme of a general formal $n$-sphere requires a significant amount of additional ingredients, which includes establishing a more general form of the Freudenthal suspension theorem in unipotent homotopy theory (\cref{freudenthalsusp}) as we did in our paper. 
\end{remark}{}

\begin{corollary}\label{retire1}
Let $X$ be a Calabi--Yau variety of dimension $n \ge 3$ over an algebraically closed field $k$ of characteristic $p>0$. Then 
$$ \pi_{n+1}^{\mathrm{U}}(X)
 \simeq
\begin{cases}
W[F] & \text{if} \,\,p=2 \,\text{and} \,\,X\,\, \text{is not weakly ordinary,}\,\, \\
\mathbf{Z}/2\mathbf{Z} & \text{if} \,\,p=2 \,\text{and} \,\,X \,\text{is weakly ordinary,} \\

0 & \text{otherwise, i.e., if}\,\, p \ne 2.
\end{cases}
$$

\begin{proof}
Follows from \cref{hach} and \cref{haircut12}.
\end{proof}{}

\end{corollary}{}
Now we are ready to prove the final result of our paper.
\begin{proof}[Proof of \cref{thmCY1}]
Note that sending $X \mapsto \pi_{-n} \mathrm{TR}(X)$ induces a functor from $\mathcal{N}CY$ to the category of Dieudonn\'e modules. By \cref{hach} and \cref{retire1}, there is a unique isomorphism $\mathbf{U}(\SSS^n_{\Phi_X^n}) \simeq \mathbf{U}(X)$ determined by the Artin--Mazur--Hurewicz class. This finishes the proof (\textit{cf.} the proof of \cref{thmCY}).
\end{proof}{}

\bibliographystyle{amsalpha}
\bibliography{references1}

\end{document}